\documentclass[a4paper,12pt]{ociamthesis}

\usepackage[english]{babel}
\usepackage{array}
\usepackage[utf8]{inputenc}
\usepackage{graphicx}
\usepackage{amsmath,latexsym,amsfonts,amssymb,amscd,euscript,epic,eepic,times,amsthm,stmaryrd}
\let\olddegree\degree
\let\degree\relax   % temporarily make it undefined
\usepackage{mathabx}
\let\degree\olddegree
\usepackage{xypic}
\usepackage{nicefrac}
\usepackage{faktor}
\usepackage{pdfsync}
\usepackage{xcolor}
\usepackage{hyperref}
\usepackage{tikz-cd}
\tikzset{
    labl/.style={anchor=north, rotate=90, inner sep=.5mm}
}

\theoremstyle{plain}  % default
\newtheorem{theorem}{Theorem}[section]

\newtheorem*{theorem*}{Theorem}
\newtheorem{corollary}[theorem]{Corollary}
\newtheorem{lemma}[theorem]{Lemma}
\newtheorem{proposition}[theorem]{Proposition}

\theoremstyle{definition}
\newtheorem{definition}[theorem]{Definition}

\theoremstyle{remark}

\newtheorem{remark}[theorem]{Remark}
\newtheorem*{remark*}{Remark}

\newtheorem*{claim*}{Claim}

\DeclareFontFamily{OT1}{rsfs}{}
 \DeclareFontShape{OT1}{rsfs}{n}{it}{<->rsfs10}{}
 \DeclareMathAlphabet{\curly}{OT1}{rsfs}{n}{it}

\newcommand{\llambda}{\Gamma}
\newcommand{\ambda}{\gamma}

\newcommand{\forevery}{\;\text{for}\;\text{every}\;}
\newcommand{\andd}{\quad\text{and}\quad}

\newcommand{\suhthat}{\mid}

\renewcommand{\le}{\leqslant}

\renewcommand{\ge}{\geqslant}
\renewcommand{\setminus}{\smallsetminus}

\newcommand{\R}{\mathbb{R}}

\newcommand{\Z}{\mathbb{Z}}
\newcommand{\C}{\mathbb{C}}

\newcommand{\PPP}{\curly{P}}

\newcommand{\calR}{\mathcal{R}}
\newcommand{\wcalR}{\widetilde{\mathcal{R}}}

\newcommand{\lie}{\mathfrak}
\newcommand{\xra}{\xrightarrow}

\newcommand{\PU}{\mathrm{PU}}

\newcommand{\U}{\mathrm{U}}
\newcommand{\OO}{\mathrm{O}}
\newcommand{\GL}{\mathrm{GL}}
\newcommand{\SL}{\mathrm{SL}}

\newcommand{\SO}{\mathrm{SO}}
\newcommand{\NSO}{\mathrm{NSO}}

\newcommand{\Sp}{\mathrm{Sp}}
\newcommand{\NSp}{\mathrm{NSp}}
\newcommand{\Spin}{\mathrm{Spin}}
\newcommand{\Pin}{\mathrm{Pin}}

\newcommand{\Cl}{\mathrm{Cl}}

\DeclareMathOperator{\ad}{ad}
\DeclareMathOperator{\Ad}{Ad}

\DeclareMathOperator{\tr}{tr}

\DeclareMathOperator{\Hom}{Hom}
\DeclareMathOperator{\End}{End}

\DeclareMathOperator{\id}{Id}
\DeclareMathOperator{\Id}{Id}

\DeclareMathOperator{\Aut}{Aut}
\DeclareMathOperator{\Int}{Int}
\DeclareMathOperator{\Out}{Out}

\DeclareMathOperator{\gal}{Gal}

\newcommand{\fr}{\mathfrak B^{\C}}

\newcommand{\x}{\mathfrak{X}}

\renewcommand{\phi}{\varphi}

\newcommand{\gl}{\mathfrak{gl}}
\newcommand{\lieg}{\mathfrak{g}}

\renewcommand{\phi}{\varphi}

\newcommand{\alg}{\alpha_{\gamma}}

\newcommand{\oo}{\mathcal{O}}

\newcommand{\autg}{\Aut(G)}
\newcommand{\homm}[2]{\Hom\left(#1,#2\right)}
\newcommand{\outg}{\Out(G)}
\newcommand{\homgh}{\homm{\Gamma}{\autg}}

\newcommand{\gt}{G^{\theta}}
\newcommand{\gtt}{G^{\theta'}}
\newcommand{\gs}{G_{\theta}}
\newcommand{\pt}{p_{\theta}}

\newcommand{\tg}{\theta_{\gamma}}
\newcommand{\etag}{\eta_{\gamma}}

\newcommand{\zg}{z_{\gamma}}

\newcommand{\cct}{c_{\theta}}
\newcommand{\ct}{c^{\theta}}
\newcommand{\qt}{q_{\theta}}
\newcommand{\qqt}{\qt}
\newcommand{\hg}{h_{\gamma}}
\newcommand{\oh}{\overline{h}}

\newcommand{\od}{\overline{\delta}}
\newcommand{\ohg}{\overline{h}_{\gamma}}
\newcommand{\bg}{\beta_{\gamma}}
\newcommand{\fg}{f_{\gamma}}

\newcommand{\taug}{\tau_{\gamma}}
\newcommand{\mug}{\mu_{\gamma}}

\newcommand{\ag}{a_{\gamma}}
\newcommand{\sg}{s_{\gamma}}
\newcommand{\taut}{\tau^{\theta}}

\newcommand{\ctt}{\tilde c_{\theta}}
\newcommand{\gamt}{\Gamma_{\theta}}
\newcommand{\gamtt}{\widehat{\Gamma}_{\theta}}
\newcommand{\gamtz}{\widehat{\Gamma}^{\theta}}
\newcommand{\liegm}{\mathfrak{g}^{\theta}_{\mu}}
\newcommand{\liegt}{\mathfrak{g}_{\theta}}
\newcommand{\ol}{\ambda\Delta}

\newcommand{\zt}{Z^1_{a}(\Gamma,Z)}

\DeclareMathOperator{\Fun}{Fun}
\newcommand{\fun}[2]{\Fun(#1,#2)}

\newcommand{\fung}{\fun{\Gamma}{G/Z}}

\newcommand{\cent}{Z_{\Gamma}(\Gamma')}

\newcommand{\sett}{\Gamma_{\theta}}
\newcommand{\setp}{\Gamma_{\theta}^Z}

\newcommand{\wt}{\widetilde{\theta}}

\newcommand{\wf}{\widetilde{f}}

\newcommand{\glt}{\GL(n,\C)^{\theta}}
\newcommand{\gls}{\GL(n,\C)_{\theta}}

\newcommand{\md}{M}

\newcommand{\cM}{\mathcal{M}}
\newcommand{\wcM}{\widetilde{\mathcal{M}}}
\newcommand{\mdl}{\mathcal{M}}

\newcommand{\mdlsg}{\mdl_{ss}(X,G)^{\Gamma}}

\newcommand{\gam}{\Gamma_Y}

\newcommand{\xg}{X_{\Gamma}}
\newcommand{\xd}{X_{\Delta}}
\newcommand{\pg}{p_{\Gamma}}
\newcommand{\pd}{p_{\Delta}}

\newcommand{\mtau}{M^{\tau}}
\newcommand{\mtauu}{M^{\tau'}}
\newcommand{\mttauu}{M^{\tau\tau'}}

\newcommand{\wchi}{\widetilde{\chi}}

\newcommand{\gltt}{\GL(2n,\C)^{\theta}}

\newcommand{\spt}{\Sp(2n,\C)^{\theta}}
\newcommand{\sps}{\Sp(2n,\C)_{\theta}}

\DeclareMathOperator{\irr}{irr}

\newcommand{\homtc}{\Hom_{\theta,\tau,c}(\wgame,\Aut(G))}
\newcommand{\outc}{\Hom_{\theta,\tau,c}(\wgame,\Out(G))}
\newcommand{\etc}{e_{\tau,c}}

\newcommand{\wgam}{\widehat{\Gamma}_{Y}}
\newcommand{\wgamm}{\widehat{\Gamma}}
\newcommand{\wga}{\widehat{\gamma}}
\newcommand{\wgame}{\wgamm_{\eta}}

\newcommand{\otau}{\overline{\tau}}
\newcommand{\zzgtw}{Z^2_{\tau}(\wgam,Z(\gt_0))}
\newcommand{\zzgt}{Z^2_{\tau}(\gamt,Z(\gt_0))}

\newcommand{\wtg}{\widetilde{\theta}_{\gamma}}
\newcommand{\wtb}{\widetilde{\theta}_{\beta}^{-1}}

\newcommand{\fb}{\eta_{\beta}}
\newcommand{\hb}{h_{\beta}}
\newcommand{\mub}{\mu_{\beta}}
\newcommand{\ttg}{\theta'_{\gamma}}

\newcommand{\halpha}{\hat\alpha}

\newcommand{\olambda}{\overline{\Lambda}}
\newcommand{\gtl}{\gt_{\Lambda}}
\newcommand{\gtll}{\gt_{\olambda}}

\newcommand{{\zerocochain}}{g}
\newcommand{{\onecochain}}{f}

\newcommand{\ug}{\underline G}

\newcommand{\cc}[1]{\mathcal C_{#1}}
\newcommand{\hcc}[1]{\hat{\mathcal C}_{#1}}

\newcommand{\zk}{\lie z_{\lie k}}

\newcommand{\pair}[2]{\langle#1,#2\rangle}

\DeclareMathOperator{\Hol}{Hol}

\newcommand{\rhog}{\rho_{\Gamma}}
\newcommand{\rhomu}{\rho_{\mu}}
\newcommand{\rhotm}{\rho_{\tilde\tau,\mu}}
\newcommand{\cdotv}[1]{\cdot \rhog(#1)}

\newcommand{\twistedr}{$(\theta,c,\rhog)$-twisted $\Gamma$-equivariant }

\newcommand{\da}{\kappa}
\newcommand{\daga}{\da_{\gamma}}
\newcommand{\dab}{\delta}
\newcommand{\dac}{\epsilon}
\newcommand{\tda}{\tilde{\dac}}
\newcommand{\rholm}{\rho_{\tda,\mu}}
\newcommand{\lievm}{V^{\da}_{\mu}}

\title{Fixed points in Higgs bundle moduli spaces and the Prym--Narasimhan--Ramanan construction}
\author{Guillermo Barajas Ayuso}

\begin{document}

\thispagestyle{empty}

\begin{center}

\

\

\

\

    { \Huge {\bfseries {Fixed points in Higgs bundle moduli spaces and the Prym--Narasimhan--Ramanan construction}} \par}
{\large \vspace*{15mm} {
  \includegraphics[scale=0.2]{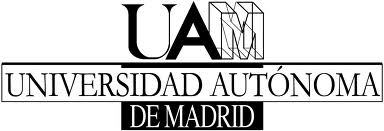}

\par} \vspace*{15mm}}
    {{\Large Guillermo Barajas Ayuso} \par}
{\large \vspace*{1ex}
    {{ICMAT (CSIC-UAM-UC3M-UCM)} \par}
\vspace*{1ex}{\vspace*{15mm}
 \includegraphics[scale=.5]{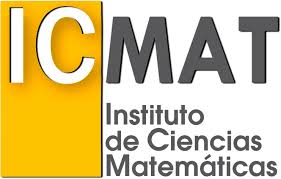}
\par }
\vspace*{15mm}
    {Tesis presentada para la obtenci\'on del t\'itulo de \par}
\vspace*{1ex}
    {\it {Doctor en Matem\'aticas} \par}
\vspace*{1ex}
    {por la Universidad Aut\'onoma de Madrid. \par}
\vspace*{1ex}
    {Director: Óscar Garc\'ia Prada. \par}
\vspace*{2ex}
    8 de marzo de 2023}
\end{center}

\newpage
\thispagestyle{empty}
\

\newpage
\thispagestyle{empty}

\

\vspace{7cm}
\

\begin{flushright}
{\it A papá y mamá}
\end{flushright}

\newpage
\thispagestyle{empty}

\

\begin{abstract}
Let $X$ be a compact Riemann surface and $G$ a connected reductive complex Lie group with centre $Z$. Consider the moduli space $\mdl(X,G)$ of polystable $G$-Higgs bundles on $X$. The group of isomorphism classes of $Z$-bundles on $X$, which is isomorphic to $H^1(X,Z)$, acts on $\mdl(X,G)$ via extension of structure group by the multiplication homomorphism
$Z\times G\to G$. The group $\Aut(G)$ also acts on $\mdl(X,G)$ by extension of structure group, and so does the group $\Aut(X)$ of holomorphic automorphisms by pullback. Finally, $\C^*$ acts by multiplying the Higgs field. Combining these provides an action of $H^1(X,Z)\rtimes(\Aut(G)\times\Aut(X))\times\C^*$ on $\mdl(X,G)$, where $\Aut(G)$ and $\Aut(X)$ act on $H^1(X,Z)$ by extension of structure group and pullback, respectively.  

Let $\Gamma$ be a finite subgroup of $H^1(X,Z)\rtimes(\Aut(G)\times\Aut(X))\times\C^*$.
The goal of this thesis is to find a Prym--Narasimhan--Ramanan-type construction to describe the fixed points of the action of $\Gamma$ on  $\mdl(X,G)$. More precisely, we show that fixed points correspond to twisted equivariant Higgs pairs over certain étale covers of $X$. Our results generalize García-Prada--Ramanan, where $\Gamma$ was considered to be cyclic, and Narasimhan--Ramanan, who only consider actions of cyclic subgroups of $H^1(X,\C^*)$ for $G=\GL(n,\C)$.
\end{abstract}

\

\

\

\

\

\noindent
{\Huge {\bfseries Acknowledgements}  }

\

\noindent First of all I want to thank my thesis supervisor Oscar García-Prada. His love for mathematics involves him in a number of projects requiring a lot of time and effort, thus his time has been invaluable to me. Each meeting with him, no matter its duration, has provided me with tools that I believe will be essential in my scientific career. Even a fraction of this thesis would not have been possible without his guidance.

\

Secondly I want to thank Peter Gothen, who bore with me during my stay in Porto. It was an enriching and fruitful experience, and I am very proud of the publication ---my first one!--- that came out of it.

\

I want to thank ``la Caixa'' Foundation (ID 100010434) for supporting this research through an INPhINIT fellowship (code LCF/BQ/DI19/11730022) fellowship and for the financial help I have enjoyed during my stay in Porto and the conferences and seminars that I have attended. In this regard I also wish to aknowledge the grants PID2019-109339GB-C31 and RED2018-102810-T, awarded by the Spanish Ministry of Science and
Innovation. 

\

Muchas gracias a Mario García, mi tutor, que ha lidiado con el papeleo que tanto nos gusta año a año.
En la parte investigadora, ha ejercido de ancla para nuestro grupo de trabajo sobre el libro de Donaldson--Kronheimer, ha dado grandes charlas informales en el ICMAT y ha sido una parte importante en muchas discusiones interesantes, sobre temas matemáticos pero también sobre la carrera investigadora. Eres toda una referencia para mí.

\

Aunque hay muchos Guilles en investigación últimamente (espero que nadie se entere de nuestra secreta conspiración de monopolizar las mates), hay uno muy especial para mí y ese es mi compañero de doctorado Guillermo Gallego. Cuando llegué a Madrid me ayudaste enormemente a ponerme al día al comienzo del doctorado. Juntos hemos sufrido las confusiones derivadas de leer notación complicada o de no conocer teoremas básicos. Realmente he disfrutado de las numerosas discusiones contigo, y espero que siga siendo así y podamos realizar colaboraciones oficiales en el futuro. Para mí siempre serás un maestro de las referencias matemáticas.

\

I would like to thank Nigel Hitchin for taking the time to know about my thesis and sharing mathematical ideas. I would also like to thank Tomas L. Gomez, Alfonso Zamora, David Alfaya, Andre G. Oliveira and Luis Alvarez-Consul, for several and very useful mathematical discussions.

\

Mil gracias a Juan Álvaro por tantas discusiones interesantes y unas cuantas quedadas para hablar de la vida. No tenía suficiente con los invariantes de Seiberg--Witten, así que vino desde Estados Unidos para estudiar los invariantes de Donaldson.

\

Muchas gracias a Ángel González por su voluntad incansable de hacer matemáticas y sus explicaciones intuitivas; no pensé que la transformada de Fourier--Mukai me sería tan útil cuando la explicaste. No sé cómo, pero de alguna forma estás en todos los fregados.

\

Muchas gracias a Emilio Franco por la facilidad que me ha dado para hablar con él en general, y en particular para discutir temas de {\it mirror symmetry}. Espero que consigas la opción vegetariana que te mereces en el CBM.

\

Thanks a lot to Mathieu Ballandras, who has been a great fit for our research team and has introduced me to tecno music. Our discussions in ICMAT and Oberwolfach have been priceless. I wish I had your help at the beginning of my PhD!

\

I would like to thank my colleagues from Mario's group, better known as the {\it super Mario bros}. Andoni Arriba, Raúl González and Roberto Téllez have been a big support and I have really enjoyed their company in ICMAT and during our stay in Lisbon. Raúl and Roberto were also victims of the study group on Donaldson polynomial invariants. Thank you Arpan for bringing me closer to understanding the physics of mirror symmetry, and for great discussions on a variety of mathematical and non-mathematical topics. I really enjoyed visiting the Galileo museum with you! Thank you King for coming so far to bring us your expertise on toric varieties. It has been a pleasure to see how the super Mario bros group has grown over time.

\

Gracias a mis compañeros de beca: los Víctors, Arce, Arnau, las Anas, Marina, Alejandra, Fernando... ¡Echaré de menos las barras libres de vinos del primer training! Gracias en particular a Víctor por hacerme ver El Padrino, enseñarme a jugar al pádel y acogerme en su casa.

\

Gracias a Enrique, Bilson, Celia y Fernando, que me acompañaron en el grupo de estudio del libro de Hartshorne.

\

Gracias a mis compañeros de congresos: Alex, Mario, Guille Sánchez, Laura... Ha habido muy buenas charlas, cenas y noches.

\

Gracias a mis grandes amigos Adri, Alex, Álvaro, Ángel, Aarón, Carlos, César, Javi, Luis, Sergio... por valorar lo que hacía sin entenderlo del todo, y sacarme de vez en cuando del universo matemático en el que trabajo.

\

Gracias a Maleena y Sandra, que me han escuchado hablar sobre mates con todo el cariño del mundo y me han apoyado siempre.

\

Gracias a mis compañeros de basket: David, Ángel, Paul, Marino, Vico, Sánchez, Crivil, Pelu, Adán, Dave... Gracias a mis compañeros y profes de baile, especialmente a Nico por hacer que ame la salsa y a mi compañera de fatigas Ana. También a Nacho, Nan, Mary, Sofía, Mica, Fran, Fer, Aye, Kevin, Raquel, Sergio, Laura... Cuando juego al baloncesto o bailo me olvido de todo lo demás, y estoy seguro de que eso me ha ayudado a poner en perspectiva la importancia de las cosas.

\

Gracias a mi hermana, mis tíos, mis primos y mis abuelos. Mi abuela siempre ha dicho que me busque un trabajo de verdad, pero yo sé lo que dice por mi bien. 

\

Y gracias a mis padres, que me han apoyado incondicionalmente desde que nací. Sin vuestros sacrificios no sería quien soy ni estaría donde estoy, y espero poder devolveros parte de vuestros esfuerzos con el tiempo.

\

\

\

\

\newpage

\thispagestyle{empty}

\

\newpage 

\

\newpage

\tableofcontents

\newpage

\

\chapter{Introduction}
\vspace{30 pt}

\noindent{\bf\large Background and statement of the problem}
\vspace{10 pt}

\noindent The aim of this thesis is to study fixed points of finite group actions on moduli spaces of $G$-Higgs bundles over a compact Riemann surface $X$. Here $G$ is a connected semisimple complex Lie group. A $G$-Higgs bundle is a pair $(E,\phi)$ consisting of a principal $G$-bundle $E$ over $X$ and a global section $\phi$ of $E(\mathfrak{g})\otimes K_X$, where $E(\mathfrak{g})$ is the bundle associated to $E$
via the adjoint representation of $G$ in $\lieg$ and $K_X$ is the canonical bundle of $X$ ---the section $\phi$ is called the Higgs field. There is a moduli space $\mdl(X,G)$ classifying isomorphism classes of polystable $G$-Higgs bundles, see Chapter \ref{chapter-higgs-bundles}. When $G$ is classical, $G$-Higgs bundles correspond to vector bundles equipped with an endomorphism tensored by $K_X$ and some extra structure.

Since their introduction by Hitchin more than 35 years ago \cite{hitchin1987,hitchin:duke}, moduli spaces of Higgs bundles have been of
tremendous interest in geometry, topology and theoretical physics. They have an extremely rich
geometry coming from the fact that they are hyperK\"ahler, they define completely integrable systems
and, by the non-abelian Hodge correspondence, they are identified with character varieties of the fundamental group of the surface.
Important aspect of the geometry of $\mathcal{M}(X,G)$ is of course reflected by the group of its holomorphic automorphisms, and in particular by the action of finite subgroups of the group of automorphisms. 

We consider several natural group actions on $\mdl(X,G)$. First, the group $\C^*$ acts on $\mdl(X,G)$ by rescaling the Higgs field. Fixed points of this action, called Hodge bundles, play an essential role in the study of the topology of the moduli space via localization. This has been considered when $G=\GL(n,\C)$ for rank and degree coprime, first for rank 2 and fixed determinant by Hitchin \cite{hitchin1987}, second for rank 3 by Gothen \cite{gothen} and finally for arbitrary rank by García-Prada--Heinloth--Schmitt \cite{oscar-heinloth-schmitt} and García-Prada--Heinloth \cite{oscar-heinloth}. Hodge bundles are also involved in the study of variations of Hodge structures by Simpson \cite{simpson-local-systems} and Biquard--Collier--García-Prada--Toledo \cite{oscar-biquard-collier-toledo}. A broader context where these objects appear is the theory of Higgs bundles for real forms and Higher Teichmüller components, see Bradlow--García-Prada--Gothen \cite{oscar-bradlow-gothen}, Aparicio--Bradlow--Collier--García-Prada--Gothen--Oliveira \cite{aparicio}, García-Prada--Oliveira \cite{García-prada-oliveira,García-prada-oliveira-sp}, etc. ---for more references, see the survey paper by García-Prada \cite{oscar-survey}.

In this thesis we are only concerned with actions of finite subgroups of $\C^*$. The simplest example is the involution $\iota$ of $\mdl(X,G)$ sending $(E,\phi)$ to $(E,-\phi)$, which is already considered by Hitchin \cite{hitchin1987} for $G=\GL(2,\C)$ and García-Prada--Ramanan \cite{PR} in general. The fixed points of $\iota$ correspond to Higgs pairs with structure group $\gt$, where $\theta$ runs over the inner involutions of $G$. Here $\gt$ is the subgroup of fixed points of $\theta$ and the Higgs field takes values in the $-1$-eigenspace of the automorphism of $\lie g$ induced by $\theta$. These fixed points correspond to representations of the fundamental group of $X$ in the real form $G^{\sigma}$, where $\sigma$ is an antiholomorphic involution of $G$ of Hodge type, that is, inner equivalent to a compact involution.

Secondly, the group of automorphisms $\Aut(G)$ of $G$ acts on $\mdl(X,G)$ by extension of structure group. More precisely, given a holomorphic automorphism $\theta$ of $G$ and a $G$-Higgs bundle $(E,\phi)$ over $X$, we may twist the $G$-action on the total space of $E$ by $\theta^{-1}$ to obtain a new $G$-bundle, which we call $\theta(E)$, and the isomorphism of vector bundles $E(\lie g)\cong \theta(E)(\lie g)$ induced by $\theta$ produces a Higgs field $\theta(\phi)$ for $\theta(E)$. The action of the group of inner automorphisms $\Int(G)$ respects isomorphism classes: for each $g\in G$, the map $E\to E$ given by multiplication by $g$ induces an isomorphism from $(E,\phi)$ to $\theta(E,\phi)$. Hence, we get a left action of $\Out(G):=\Aut(G)/\Int(G)$. 

For example, if $a$ is a non-trivial element of $\Out(G)$ such that $a^2=1$, the fixed points correspond to $G^{\theta}$-Higgs bundles, where $\theta$ is an outer involution of $G$ lifting $a$ (see \cite{PR}). Since the actions of $\Out(G)$ and $\C^*$ commute, we may also consider the involution combining both $a$ and $-1$ sending $(E,\phi)$ to $(\theta(E),-\theta(\phi))$, in which case fixed points correspond to representations of $\pi_1(X)$ in certain real forms of $G$ which are no longer of Hodge type. In fact, varying $a$, all the possible real forms of $G$ appear. 

Another important group acting on the moduli space of $G$-Higg bundles is $H^1(X,Z)$, the group of isomorphism classes of $Z$-bundles, where $Z$ is the centre of $G$. Given a $G$-Higgs bundle $(E,\phi)$ and a $Z$-bundle $L$, we may define the tensorization $E\otimes L$ to be the $G$-bundle obtained from the $G\times Z$-bundle $E\times L$ via extension of structure group by the multiplication homomorphism. Since the adjoint action of $Z$ on $\lie g$ is trivial, the Higgs field $\phi$ may also be regarded as a Higgs field of $E\otimes L$. 

It can be seen that the actions of $\Out(G)$ and $H^1(X,Z)$ do not commute, but rather the former twists the latter by extension of structure group. Thus we get a combined right action of $H^1(X,Z)\rtimes\Out(G)\times\C^*$ on $\mdl(X,G)$, where the semidirect product is defined using the aforementioned action of $\Out(G)$ on $H^1(X,Z)$, such that an element $(\alpha,a,\mu)$ sends a $G$-Higgs bundle $(E,\phi)$ to $\theta^{-1}(E\otimes\alpha,\mu\phi)$, where $\theta$ is any automorphism of $G$ lifting $a$. 
\cite{PR} gives a description of the fixed points of the action of a general finite cyclic subgroup of $H^1(X,Z)\rtimes\Out(G)\times\C^*$ which is analogous to the ones for the above cases.

 An important construction related to the action of $H^1(X,Z)$ in the context of vector bundles, which we generalize to Higgs bundles, is the famous Narasimhan--Ramanan description of fixed points of the action of a finite cyclic subgroup of the Jacobian. Let $L\to X$ be a line bundle  of finite order $r$, which determines an automorphism of the moduli space of vector bundles of rank $n$ and degree $d$ via tensorization. Narasimhan and Ramanan \cite{narasimhan-ramanan} find that the stable fixed point locus is contained in the pushforward of the moduli space of vector bundles over $X_L$ of rank $n/r$ and degree $d$, where $X_L$ is the étale cover of $X$ determined by $L$. Nasser \cite{nasser} improves this result by showing that the whole polystable fixed point locus is equal to the pushforward of this moduli space.
In particular, if $r=n$ then the fixed point subvariety is isomorphic to the pushforward of the Jacobian of $X_L$, and its intersection with the subvariety of vector bundles having a fixed determinant is isomorphic to the pushforward of the Prym variety.

The last group action on $\mdl(X,G)$ that we consider is that of holomorphic automorphisms $\Aut(X)$ of $X$ via pullback. This twists the action of $H^1(X,Z)$ via pullback and so at the end we get a right action of the group
 \begin{equation}\label{eq-big-group}
     H^1(X,Z)\rtimes(\Aut(X)\times \Out(G))\times \mathbb{C}^*,
 \end{equation}
where the semidirect product is defined by the actions of $\Aut(X)$ and $\Out(G)$ on $H^1(X,Z)$ given by pullback and extension of structure group, respectively. This is given explicitly by
$$
  (E,\phi)\cdot (\alpha,\eta,a,\mu):=(\eta^*\theta^{-1}(E\otimes\alpha), \mu\eta^*\theta^{-1}(\phi)),
$$
for each $(\alpha,\eta,a,\mu)$ in (\ref{eq-big-group}) and any automorphism  $\theta$ of $G$ lifting $a$.  

Fixed points of finite subgroups of $\Aut(X)$ have been studied by several authors including Andersen--Grove \cite{andersen-fixed-points}, Andersen \cite{andersen} and García-Prada--Wilkin \cite{ow}. Fixed points of general finite subgroups of $\Aut(X)\times \Out(G)\times \C^*$ are described by García-Prada--Basu \cite{oscar-suratno} in terms of twisted equivariant Higgs bundles, which we cover in Chapter \ref{chapter-twisted-equivariant-higgs-pairs}.

This thesis undertakes the task of describing the fixed points of the action of an arbitrary finite subgroup $\Gamma$ of (\ref{eq-big-group}) on $\mdl(X,G)$. In particular, we generalize \cite{PR} to any finite subgroup and unify their results with \cite{narasimhan-ramanan}, getting a general Prym--Narasimhan--Ramanan construction. We also generalize \cite{oscar-suratno} by adding the $H^1(X,Z)$-action to their work, and we improve their results in light of the Prym--Narasimhan--Ramanan construction. Our general answer is Theorem \ref{th-prym-narasimhan-ramanan-general}, but in order to give the reader a good grasp of how it works we have isolated its main ingredients into more special cases where their role can be understood best. The particular case of finite group actions of $Z$-bundles on the moduli space of holomorphic $G$-bundles is treated in a joint paper with García-Prada \cite{prym-narasimhan-ramanan}. The formalism of twisted equivariant principal bundles, which play a central role in this study, is developed in a joint paper with García-Prada, Gothen and Mundet i Riera \cite{GGM}. 

\vspace{20 pt}
\noindent{\bf\large Fixed points when the action does not involve tensorizarion}
\vspace{10 pt}

\noindent Let $\Gamma$ be a finite subgroup of $\Aut(X)\times\Out(G)\times\C^*$. In this situation, the fixed point variety $\mdl(X,G)^{\Gamma}$ is given in terms of twisted equivariant Higgs bundles over $X$, which are introduced in Chapter \ref{chapter-twisted-equivariant-higgs-pairs}. The corresponding results are explained in Chapter \ref{chapter-alpha-trivial}. Let $\eta$, $a$ and $\mu$ be the natural homomorphisms from $\Gamma$ to $\Aut(X)$, $\Out(G)$ and $\C^*$ given by projection, and take a homomorphism $\theta:\Gamma\to\Aut(G)$ lifting $a$. Take a map $c:\Gamma\times\Gamma\to Z$. A (right) $(\theta,c)$-twisted $\Gamma$-equivariant action on a $G$-bundle $E$ over $X$, which we denote `$\cdot$', is a lift of the right action $\eta^{-1}$ of $\Gamma$ on $X$ to an action of $\Gamma$ on $E$ which twists the bundle $G$-action by $\theta^{-1}$ and such that the composition of the actions of $\gamma_1$ and $\gamma_2\in\Gamma$ is equal to the composition of the actions of $c(\gamma_1,\gamma_2)$ and $\gamma_1\gamma_2$. The pair $(E,\cdot)$ is called a $(\theta,c)$-twisted $\Gamma$-equivariant $G$-bundle. 

More generally, we may define (left or right) $(\theta,c)$-twisted $\Gamma$-actions on a set $M$ which is equipped with a group $G$-action.
Such actions are associative if and only if $c$ is a 2-cocycle of $\Gamma$ with values in $Z$, in the sense of Galois cohomology \cite{serre-galois}, and we assume so hereafter. We call $Z^2_{a}(\Gamma,Z)$ to the group of 2-cocycles. We may also define the second cohomoly group $H^2_a(\Gamma,Z)$ as the quotient of $Z^2_{a}(\Gamma,Z)$ by an action of the group of maps $\Gamma\to Z$. The theory of twisted equivariant actions is developed in our joint paper with García-Prada, Gothen and Mundet i Riera \cite{GGM}, including a \v{C}ech cohomology interpretation of the set of isomorphism classes of these objects. We review some of our results in Chapter \ref{chapter-twisted-equivariant-bundles}.

Given a representation $\rho:G\to\GL(V)$, a $(\theta,c)$-twisted left $\Gamma$-action $\rhog$ on $V$ and a $(\theta,c)$-twisted $\Gamma$-equivariant $G$-bundle $E$, there is a right $\Gamma$-equivariant action on the associated bundle $E(V)$. A triple $(E,\cdot,\phi)$ consisting of a twisted $\Gamma$-equivariant $G$-bundle $E$ over $X$ and a $\Gamma$-invariant Higgs field $\phi\in H^0(E(V)\otimes K_X)$ is called a $(\theta,c,\rhog)$-twisted $\Gamma$-equivariant $(G,V)$-Higgs pair. There are notions of (poly,semi)stability for twisted equivariant Higgs pairs and a moduli space $\mdl(X,G,\Gamma,\theta,c,V,\rhog)$ classifying isomorphism classes of polystable objects, see Chapter \ref{chapter-twisted-equivariant-higgs-pairs}. This setting may be applied to the case where $V=\lie g$, $\rho$ is the adjoint representation and $\rhog=\mu^{-1}\theta$. In this situation we omit $\lie g$ in the notation, denoting by $\mdl(X,G,\Gamma,\theta,c,\mu^{-1}\theta)$ the moduli space. We have a natural forgetful morphism $\mdl(X,G,\Gamma,\theta,c,\mu^{-1}\theta)\to\mdl(X,G)$ omitting the $\Gamma$-action, whose image we call $\widetilde{\mdl}(X,G,\Gamma,\theta,c,\mu^{-1}\theta)$.

\vspace{5 pt}
{\bf Theorem A} (Theorem \ref{main}). {\it The union $\bigcup_{[c]}\widetilde{\mdl}(X,G,\Gamma,\theta,c,\mu^{-1}\theta)$ is contained in $\mdl(X,G)^{\Gamma}$, where $[c]$ runs over $H^2_a(\Gamma,Z)$, and the simple and stable fixed point locus $\mdl_{ss}(X,G)^{\Gamma}$ is contained in $\bigcup_{[c]}\widetilde{\mdl}(X,G,\Gamma,\theta,c,\mu^{-1}\theta)$.}
\vspace{5 pt}

Actually, in Theorem \ref{main} we decompose each piece $\mdl(X,G,\Gamma,\theta,c,\mu^{-1}\theta)$ according to the action of $\Gamma$ at isotropy points.

\vspace{20 pt}
\noindent{\bf\large Fixed points for trivial action on $X$}
\vspace{10 pt}

\noindent In chapter \ref{chapter-prym-narasimhan-ramanan} we consider the action of an arbitrary finite subgroup $\Gamma$ of $H^1(X,Z)\rtimes\Out(G)\times \mathbb{C}^*$. Projections on $\Out(G)$ and $\C^*$ provide homomorphisms $a$ and $\mu$, and projection on $H^1(X,Z)$ yields a map $\alpha:\Gamma\to H^1(X,Z)$ which satisfies
$
\alpha_{\gamma\gamma'}=\alpha_\gamma \alpha^{\gamma}_{\gamma'}
$
for each $\gamma$ and $\gamma'$ in $\Gamma$,
where the superscript denotes the left action of $\Gamma$ on $H^1(X,Z)$ induced by $a$. Maps $\alpha$ satisfying this equation are called 1-cocycles, and we write $Z^1_a(\Gamma,H^1(X,Z))$ for the group of 1-cocycles. In general, given a left action $b:\Gamma\to\Aut(A)$ on a group $A$, we have a notion of 1-cocycle and we denote the set of 1-cocycles by $Z^1_{b}(\Gamma,A)$.

By de Siebenthal \cite{de-siebenthal} there exists a homomomorphism $\theta:\Gamma\to\Aut(G)$ lifting $a$. Call $\gs$ to the subgroup of $G$ where the automorphism $\tg$ is equivalent to multiplying by some element of the centre $Z$, for each $\gamma$ in $\Gamma$. By definition we have a homomorphism $\gs\to Z^1_{a}(\Gamma,Z)$ whose kernel is the subgroup of fixed points $\gt\le G$. This, in turn, induces a map from the set of isomorphism classes of $\gs$-bundles $H^1(X,\underline{\gs})$ to $H^1(X,Z^1_a(\Gamma,Z))$. Since $G$ is semisimple, $Z$ is finite and so $X$ and $\Gamma$ may be exchanged, thus providing a map $\ctt:H^1(X,\underline{\gs})\to Z^1_a(\Gamma,H^1(X,Z))$ which may be thought of as a ``characteristic class". We may also introduce the $\mu$-weight space $\liegm$ of the automorphism $\theta$ of $\liegm$. The restriction of the adjoint action to $\gs$ preserves $\liegm$ and so we have a moduli space $\mdl_{\alpha}(X,\gs,\liegm)$ of $(\gs,\liegm)$-Higgs pairs $(F,\psi)$ such that $\ctt(F)\cong\alpha$. 

We also have an extension of structure group morphism $\cM_{\alpha}(X,\gs,\liegm)\to\mdl(X,G)$, whose image we denote with a tilde. This image is independent of the class $[\theta]$ of $\theta$ in $\Hom(\Gamma,\Aut(G))/\sim$, where $\sim$ is the equivalence relation such that $\theta\sim \theta'$ if and only there exists $g$ in $G$ satisfying $\theta\sim \Int_g\theta'\Int_g^{-1}$. 

\vspace{5 pt}
{\bf Theorem B} (Theorem \ref{th-fixed-points-oscar-ramanan-higgs}). {\it The union $\bigcup_{[\theta]}\wcM_{\alpha}(X,\gs,\liegm)$, where $[\theta]$ runs over the class of $\theta$ in $\Hom(\Gamma,\Aut(G))/\sim$, is contained in $\mdl(X,G)^{\Gamma}$, and the simple and stable fixed point locus $\mdlsg$ is contained in this union.}
\vspace{5 pt}

The actual statement of Theorem \ref{th-fixed-points-oscar-ramanan-higgs} is slightly different, since it uses a bijection between lifts $\beta\theta$ of $a$ and 1-cocycles $\beta\in Z^1_{\theta}(\Gamma,\Int(G))$ in the sense of Galois cohomology \cite{serre-galois}, inducing a bijection between the set of classes of lifts in $\Hom(\Gamma,\Aut(G))/\sim$ and the first group cohomology set $H^1_{\theta}(\Gamma,\Int(G))$. Here $\theta$ is any fixed lift.

\vspace{20 pt}
\noindent{\bf\large The Prym--Narasimhan--Ramanan construction}
\vspace{10 pt}

\noindent Theorem \ref{th-prym-narasimhan-ramanan-oscar-ramanan-higgs} takes Theorem B one step further, describing each component of the decomposition $\widetilde{\cM}_{\alpha}(X,G_{\theta},\lieg^{\theta}_{\mu})$ as a union of finite group-quotients of moduli spaces of twisted $\gal(Y/X)$-equivariant $(\gt_0,\liegm)$-Higgs pairs over certain étale covers $Y$ of $X$, where $\gt_0$ is the connected component of $\gs$. This is an application of an equivalence of categories between Higgs pairs over $X$ with non-connected reductive structure group $\gs$ and twisted equivariant Higgs pairs over an étale cover $Y\to X$ with structure group equal to the connected component of the identity $\gt_0<\gs$ (see Chapter \ref{chapter-twisted-equivariant-higgs-pairs}). It also holds when $G$ is reductive and $a$ is trivial, something that we exploit in Chapter \ref{chapter-jacobian} when $G=\GL(n,\C)$.

More precisely, let $\gamtt:=\gs/\gt_0$. By Proposition \ref{prop-extensions-isomorphic-twisted-group} we may find a map $t:\gamtt\to \gs$ which chooses an element of $\gs$ in each connected component and which is a homomorphism up to multiplication by elements of the centre $Z(\gt_0)$ of $\gt_0$. On the one hand, the composition $\Int\vert_{\gt_0}\circ t$, where $\Int\vert_{\gt_0}:\gs\to\Aut(\gt_0)$ is the action of $\gs$ on $\gt_0$ by conjugation, is a homomorphism $\tau:\Gamma\to\Aut(\gt_0)$ lifting the characteristic homomorphism of the extension $\gs$ of $\gt_0$ by $\gamtt$. On the other, we can measure the failure of $t$ to be a homomorphism by the map $c:\Gamma\times\Gamma\to Z(\gt_0)$ which sends a pair $(\gamma,\gamma')$ to the difference between $t_{\gamma}t_{\gamma'}$ and $t_{\gamma\gamma'}$. We say that $t$ is a $c$-twisted homomorphism. Associativity of the group multiplication on $\gs$ implies that $c$ is a 2-cocycle in $Z^2_{\tau}(\gamtt,Z(\gt_0))$. Moreover, the map $t$ induces an isomorphism of group extensions $\gs\cong \gt_0\times_{\tau,c}\gamtt$, where $\gt_0\times_{\tau,c}\gamtt$ is the set $\gt_0\times\gamtt$ equipped with a group multiplication involving $\tau$ and $c$ ---e.g., if $c=1$ then this is a semidirect product, see (\ref{eq-def-twisted product}).

Now, given an $\gs$-bundle $E\to X$, the bundle projection morphism can be factored through $E/\gt_0$, which is a $\gamtt$-bundle over $X$. Let $Y$ be a connected component of $E/\gt_0$ with structure group $\gam\le\gamtt$. The reductiveness of $\gs$ implies that $\gamtt$ is finite, hence $Y$ may be regarded as an étale cover of $X$ with Galois group $\gam$. Via the map $t$ we can further equip $E$ with a $\gamtt$-action which is $(\tau,c)$-twisted. This provides a bijection between the set of isomorphism classes of $\gs$-bundles over $X$ and a finite group-quotient of the set of isomorphism classes of $(\tau,c)$-twisted $\gam$-equivariant $\gt_0$-bundles over $Y$ (see Theorem \ref{th-prym-narasimhan-ramanan-principal}). The finite group is equal to the centralizer of $\gam$ in $\gamtt$, which is equal to the centre of $\gamtt$ if  $E/\gt_0$ is connected. All this theory was first developed in our joint paper with García-Prada, Gothen and Mundet i Riera \cite{GGM}. An analogous result for Higgs pairs is given by Theorem \ref{th-prym-narasimhan-ramanan-higgs}. Extension of structure group provides a morphism from the moduli space of twisted equivariant $(\gt_0,\liegm)$-Higgs pairs over $Y$ to $\mdl(X,G)$, whose image we denote with a tilde.

\vspace{5 pt}
{\bf Theorem C} (Theorem \ref{th-prym-narasimhan-ramanan-oscar-ramanan-higgs}). {\it Let $\qqt:H^1(X,\gamtt)\to Z^1_a(\Gamma,H^1(X,Z))$ be the composition of $\ctt$ with extension of structure group by $\gamtt\to\gamt:=\gs/\gt$. The union $\bigcup_{[\theta],Y}\widetilde{\mdl}(Y,\gt_0,\gam,\tau,c, \liegm)/Z_{\gamtt}(\gam)$ is contained in $\cM(X,G)^{\Gamma}$. Here $[\theta]$ runs over classes of lifts of $a$ in $\Hom(\Gamma,\Aut(G))/\sim$ and $Y$ runs over étale covers of $X$ with Galois group $\gam\le\gamtt$ such that $\qqt(Y)\cong \alpha$. Conversely, $\cM_{ss}(X,G)^{\Gamma}$ is contained in this union.}
\vspace{5 pt}

We show how Theorem C generalizes the Prym--Narasimhan--Ramanan construction in Section \ref{section-example-generalize-narasimhan}: let $G=\GL(n,\C)$ and $\Gamma\le J(X)$ be generated by a line bundle $L$ of finite order $r$. In this situation, the homomorphism $a:\Gamma\to\Out(\GL(n,\C))$ is trivial. It can be seen that there is only one class $[\theta]\in\Int(\GL(n,\C))/\sim$ in the decomposition of the fixed point locus of Theorem B, namely the class of the conjugation by the diagonal matrix $M$ whose diagonal contains every $r$-th root of unity with multiplicity $m:=n/r$. In particular, $r$ must divide $n$. In this setting $\GL(n,\C)^{\theta}\cong\GL(m,\C)^{\times r}$ and $\GL(n,\C)_{\theta}\cong\GL(n,\C)^{\theta}\rtimes_{\tau}(\Z/r\Z)$, where the action $\tau$ of $\Z/r\Z$ on $\GL(m,\C)^{\times r}$ permutes the different copies of $\GL(m,\C)$. 

Let $p_L:X_L\to X$ be the étale cover determined by $L$, which has Galois group $\Z/r\Z$. Theorem C implies that $\mdl_{ss}(X,\GL(n,\C))^{L}$ is isomorphic to an open subvariety in $\mdl(X_L,\GL(m,\C)^{\times r},\Z/r\Z,\tau,1)/(\Z/r\Z)$. Let us forget the Higgs field and translate this into the language of vector bundles: let $E$ be a $(\tau,1)$-twisted $\Z/r\Z$-equivariant $\GL(m,\C)^{\times r}$-bundle over $X_L$. The associated vector bundle is a direct sum $E_1\oplus\dots\oplus E_r$, where $E_i$ is a vector bundle of rank $m$. There is an induced $\Z/r\Z$-equivariant action which permutes the summands, hence $E_i\cong\zeta^{*i}E_1$, where $\zeta$ is a generator of $\gal(X_L/X)\cong \Z/r\Z$. The quotient of $E$ by this action is the pushforward $p_{L*}E_1$. Thus we have identified $\mdl_{ss}(X,\GL(n,\C))^{L}$ with an open subvariety of the pushforward of the moduli space of Higgs bundles of rank $m$ over $X_L$, as required.

A generalization of this result to any finite subgroup $\Gamma$ of the Jacobian $J(X)$ is given in Chapter \ref{chapter-jacobian}. A homomorphism $l:\Gamma\to\Gamma^*:=\Hom(\Gamma,\C^*)$ is called antisymmetric if the pairing of any element of $\Gamma$ with itself is equal to 1. For each such antisymmetric pairing choose a maximal subgroup $\Delta\le\Gamma$ where the pairing is trivial, which we call isotropic. This is of course equipped with an embedding in $\Hom(\Delta,H^1(X,\C^*))$, which by swapping $\Delta$ and $X$ provides a $\Delta^*$-bundle $\pd:\xd\to X$. Assume that $\vert\Delta\vert$ divides $n$. Given a Higgs bundle $(E,\phi)$ of rank $n/\vert\Delta\vert$ over $\xd$ and an element $\gamma$ in $\Gamma$, we may construct a new Higgs bundle $(l(\gamma)\vert_{\Delta}^*(E\otimes\pd^*\gamma),l(\gamma)\vert_{\Delta}^*\phi)$, where $l(\gamma)\vert_{\Delta}\in\Delta^*$ is regarded as an element of $\gal(\xd/X)$. This defines a $l(\Gamma)$-action on $\mdl(\xd,\GL(n/\vert\Delta\vert,\C))$.

Theorem \ref{th-finite-group-jacobian-higgs} states that the union $\bigcup_{l}p_{\Delta*}\mdl(\xd,\GL(n/\vert\Delta\vert,\C))^{l(\Gamma)}$ is contained in $\mdl(X,\GL(n,\C))^{\Gamma}$, and the simple and stable fixed point locus $\mdl_{ss}(X,\GL(n,\C))^{\Gamma}$ is contained in the union. Here $l$ runs over antisymmetric pairings of $\Gamma$, and we declare a component to be empty if $\vert\Delta\vert$ does not divide $n$. Note that, if $\vert\Gamma\vert$ divides $n$, the component corresponding to trivial $l$ is equal to $p_{\Gamma*}\mdl(X,\GL(n/\vert\Gamma\vert))$, which makes it clear that this generalizes the result for cyclic $\Gamma$.

We also apply Theorem C when $G=\Sp(2n,\C)$ and $\Gamma$ is a finite subgroup of $H^1(X,Z)$. In this case the centre $Z$ is isomorphic to $\Z/2\Z$, hence we may identify elements of $H^1(X,Z)$ with line bundles of order $2$ over $X$. Let us keep the notation of the previous paragraph. Given an element $q$ in $\Delta^*$, consider the set of isomorphism classes of triples $(E,\phi,\psi)$, where $(E,\phi)$ is a Higgs bundle of rank $2n/\vert\Delta\vert$ over $\xd$ and $\psi$ is an isomorphism $(E,\phi)\xrightarrow{\sim}q^*(E^*,\phi^*)$ satisfying $q^*\psi^*=-\psi$. There are (poly,semi)stability notions for these objects, and a moduli space $\mdl(\xd,\GL(2n/\vert\Delta\vert,\C),q)$ parametrizing polystable triples may be constructed. In particular, $\mdl(\xd,\GL(2n/\vert\Delta\vert,\C),1)$ is isomorphic to $\mdl(\xd,\Sp(2n/\vert\Delta\vert,\C))$.

We have a pushforward morphism 
\begin{equation*}
p_{\Delta*}:\mdl(\xd,\GL(2n/\vert\Delta\vert,\C),q)\to\mdl(X,\Sp(2n,\C)),
\end{equation*}
sending each triple $(E,\phi,\psi)$ to the Higgs bundle $(p_{\Delta*}E,p_{\Delta*}\phi)$ equipped with the symplectic form $p_{\Delta*}\psi:p_{\Delta*}E\to p_{\Delta*}q^*E^*\cong p_{\Delta*}E^*.$
Recall that this yields an $\Sp(2n,\C)$-Higgs bundle by taking the bundle of frames of $p_{\Delta*}E$ and then the reduction of structure group to $\Sp(2n,\C)$ given by the symplectic form.
An action of $l(\Gamma)$ may also be defined as in the previous paragraph.
% : it acts on $(E,\phi)$ and its pullback by $q$, hence it acts on the set of isomorphisms $\psi$ between them.
Theorem \ref{th-finite-group-jacobian-sp-higgs} states that $\bigcup_{l,q}p_{\Delta*}\mdl(\xd,\GL(2n/\vert\Delta\vert,\C),q)^{l(\Gamma)}$ is contained in $\mdl(X,\Sp(2n,\C))^\Gamma$, and the simple and stable locus $\mdl_{ss}(X,\Sp(2n,\C))^\Gamma$ is contained in the union. Here $l$ runs over antisymmetric pairings, and $q$ runs over elements of $\Delta^*$. For example, when $\Gamma$ is generated by a single line bundle $L$, the aforementioned union only contains two components, namely $p_{L*}\mdl(X_L,\Sp(n,\C))$ (if $n$ is even) and $p_{L*}\mdl(X_L,\GL(n,\C),-1)$ ---see Section \ref{section-example-sp-cyclic} for a detailed analysis of this case.

We analyse a few more examples of finite cyclic group actions in Chapter \ref{chapter-cyclic}. In Section \ref{section-dualization} we consider the involution of $\mdl(X,\SL(n,\C))$ sending $(E,\phi)$ to $(E^*\otimes L,\phi^*)$, where $L$ is a line bundle of finite order. Then the fixed point components detected by Theorem C are images of moduli spaces of twisted equivariant $\SO(n,\C)$ and $\Sp(2m,\C)$-Higgs bundles over $X_L$, where the last ones only appear if $n=2m$ is even. When $L$ is trivial, these are just $\SO(n,\C)$ and $\Sp(2m,\C)$-Higgs bundles.

In Section \ref{section-spin} we consider the action of the group generated by a $Z$-bundle $L$ on $\mdl(X,\Spin(n,\C))$, where $Z\cong \Z/2\Z\times\Z/2\Z$ is the centre of $\Spin(n,\C)$. The application of Theorem C yields Proposition \ref{prop-spin}, which identifies the relevant fixed point components as images of twisted equivariant $\Spin(p,\C)\times\Spin(q,\C)$-Higgs bundles over $X_L$. Here $p$ and $q$ take a set of values satisfying $p+q=n$, which depend on the monodromy group of $L$ and the divisibility of $n$ by $4$. 

In Section \ref{section-e7} we find that the fixed point locus of the action of the group generated by a line bundle of order 2 on $\mdl(X,E_7)$, where $E_7$ is the corresponding simply connected exceptional group, contains the image of a moduli space of twisted equivariant $(E_6\times\C^*)/(\Z/3\Z)$-Higgs bundles. Here $\Z/3\Z<\C^*$ acts on $E_6\times\C^*$ by multiplication on both factors.

\vspace{20 pt}
\noindent{\bf\large Fixed points for general action}
\vspace{10 pt}

\noindent The fixed point locus $\mdl(X,G)^{\Gamma}$ when $\Gamma$ is any finite subgroup of (\ref{eq-big-group}) is studied in Chapter \ref{chapter-general}. Projections on $\Aut(X)$, $\Out(G)$ and $\C^*$ provide homomorphisms $\eta,a$ and $\mu$, whereas projection on $H^1(X,Z)$ yields a 1-cocycle $\alpha\in Z^1_{a,\eta}(\Gamma,H^1(X,Z))$, where $\Gamma$ acts on $H^1(X,Z)$ by sending $L$ to $\eta^{*-1}a(L)$. To understand the fixed point description in this general context, we first find a result in the case of trivial $\alpha$ "combining" Theorems A and B. More precisely, we push Theorem A to account for the potential non-injectivity of $\eta$, so that we first perform a reduction of structure group applying Theorem B to the action of $\ker\eta$ and then we get a twisted equivariant action of
the finite group of automorphisms of an étale cover of $X$ lifting $\eta(\Gamma)$.

Assume that $\alpha$ is trivial. Fix a homomorphism $\theta:\ker\eta\to\Aut(G)$ lifting $a\vert_{\ker\eta}$ and let $\gamtz:=\gt/\gt_0$ be the group of connected components of $\gt$. Take a connected étale cover $p:Y\to X$ with Galois group $\gamtz$ and denote by $\wgam$ the group of automorphisms of $Y$ lifting elements of $\eta(\Gamma)$. Let $\wgame$ be the set of pairs $(\gamma,\wga)$ in $\Gamma\times\wgam$ such that $\etag=p(\wga)$, where $p(\wga)$ is the induced automorphism of $X$. Note that the projections of $\wgame$ on the first and second factors have kernels $\gamtz$ and $\ker\eta$ respectively. Consider a homomorphism $\tau:\wgam\to\Aut(\gt_0)$ and a 2-cocycle $c\in Z^2_{\tau}(\wgam,Z(\gt_0))$ whose restrictions to $\gamtz$ fit in an isomorphism 
\begin{equation}\label{eq-extension-gt}
    \gt_0\times_{\tau,c}\gamtz\cong\gt
\end{equation}
of extensions, as above. Then we have a $c$-twisted homomorphism $e_{\tau,c}:\wgam\to\Aut(\gt_0)$, given by composing the obvious map $\wgam\to\gt_0\times_{\tau,c}\wgam$ with the conjugation action of $\gt_0\times_{\tau,c}\wgam$ on itself. We call $\Hom_{\theta,\tau,c}(\wgame,\Aut(G))$ to the set of $c$-twisted homomorphisms $\wgame\to\Aut(G)$ whose restriction to $\ker\eta$ is $\theta$ and which preserve $\gt$ and induce $e_{\tau,c}$. Note that the first two conditions in this definition indeed imply that there is an induced map $\wgam\to\Aut(\gt)$, since $\ker\eta$ acts trivially on $\gt$ via $\theta$.

Note that an element $\tilde\tau\in \Hom_{\theta,\tau,c}(\wgame,\Aut(G))$ provides a $c$-twisted homomorphism $\mu^{-1}d\tilde\tau:\wgame\to\Hom(\liegm,\lie g)$, which in turn induces a $c$-twisted homomorphism $ \rhotm:\wgam\to\Hom(\liegm,\lie g)$ because the action of $d\tilde\tau\vert_{\ker\eta}$ on $\liegm$ is equal to $\mu$. Thus we have a notion of $(\tau,c,\rhotm)$-twisted $\wgam$-equivariant $(\gt_0,\liegm)$-Higgs pair.

The formalism of the above two paragraphs may be generalized to the case when the Galois group of $Y$ is any subgroup of $\gamtz$. Consider the union $$U:=\bigcup_{[\theta],Y,[\tau],[c],\tilde\tau}\wcM(Y,G^{\theta}_0,\wgam,\tau,c,\lie g^{\theta}_{\mu},\rhotm).$$ This runs over connected étale covers $p:Y\to X$ with Galois group equal to a subgroup of $\gamtz$, classes $[\theta]\in\Hom(\Gamma,\Aut(G))/\sim$ of lifts of $a\vert_{\ker\eta}$, cohomology classes $[\tau]$ and $[c]$ such that (\ref{eq-extension-gt}) holds and elements $\tilde\tau\in\Hom_{\theta,\tau,c}(\wgame,\Aut(G))$. The tilde over $\cM$ denotes the image of the morphism given by Proposition \ref{prop-prym-narasimhan-ramanan-alpha-trivial-higgs}. Then Theorem \ref{th-prym-narasimhan-ramanan-alpha-trivial-higgs} states that $U$ is contained in $\mdl(X,G)^{\Gamma}$, and $\mdl_{ss}(X,G)^{\Gamma}$ is contained in $U$.

The general case, where $\alpha$ is not trivial, requires replacing $\gt$ with $\gs$ and $\gamtz$ with $\gamtt:=\gs/\gt_0$. Given a $\gamtt$-bundle $p:Y\to X$, the pullback $p^*\alpha$ defines an étale cover $Y_{\alpha}$ of $Y$ which, composed with $p$, provides an étale cover of $X$. Replace $Y$ with $Y_{\alpha}$ in the previous paragraphs, so that $\wgam$ is the lift of $\eta(\Gamma)$ to $Y_{\alpha}$, etc. We may then define $U$ as in the previous paragraph, but now we impose a further condition which relates $\tilde\tau$, $c$ and $\alpha$ (see (\ref{eq-2-cocycle-vs-alpha})). The general theorem is the following.

\vspace{5 pt}
{\bf Theorem D} (Theorem \ref{th-prym-narasimhan-ramanan-general}). {\it The union $U$ is contained in $\mdl(X,G)^{\Gamma}$, and the smooth fixed point locus $\mdl_{ss}(X,G)^{\Gamma}$ is contained in $U$.}

\vspace{20 pt}
\noindent{\bf\large Fixed points in character varieties}
\vspace{10 pt}

\noindent The non-abelian Hodge correspondence provides a homeomorphism between $\mdl(X,G)$ and the character variety $\calR(X,G)$ parametrizing $G$-conjugacy classes of reductive representations $\pi_1(X)\to G$. The theory has a analogue for twisted equivariant Higgs bundles: fix an action of a finite group $\Gamma$ on a compact Riemann surface $X$, a homomorphism $\theta:\Gamma\to\Aut(G)$ and a 2-cocycle $c\in Z^2_{\theta}(\Gamma,Z)$. Let $\calR(X,G,\Gamma,\theta,c)$ be the moduli space of $G$-conjugacy classes of reductive representations of the equivariant fundamental group $\pi_1(X,\Gamma,x)$ of $X$ in $G\times_{\theta,c}\Gamma$. Then $\calR(X,G,\Gamma,\theta,c)$ is homeomorphic to $\mdl(X,G,\Gamma,\theta,c)$. Thus we may translate our results to give a description of the fixed point loci of certain finite group actions on character varieties. 

More precisely, an element $(\alpha,\eta,a)$ in $H^1(X,Z)\rtimes(\Aut(X)\times\Out(G))$ sends a representation $\rho:\pi_1(X,x)\to G$ to $\theta^{-1}\circ(\eta_*(\rho\otimes\alpha))$. Here $\theta\in\Aut(G)$ is a lift of $a$, $\eta_*$ is induced by the homomomorphism $\eta_*:\pi_1(X,x)\to\pi_1(X,\eta(x))$ and we are calling $\alpha$ to its holonomy representation $\pi_1(X,x)\to Z$ by abuse of notation. With notation as above, Theorem \ref{th-prym-narasimhan-ramanan-character-general} states that the variety of irreducible and simple representations $\calR_{ss}(X,G)^{\Gamma}$ which are fixed by $\Gamma$ is contained in a union of images of moduli spaces $\calR(Y,G^{\theta}_0,\wgam,\tau,c)$, and this union is contained in $\calR(X,G)^{\Gamma}$. Particular cases of this theorem can be found in Sections \ref{section-prym-narasimhan-ramanan-character-varieties} and \ref{section-character-variety-alpha-trivial}.

% We aim to consider further actions on $\mdl(X,G)$ when $\mu$ us non-trivial in future work. For example, fixed points of the action of groups $\Gamma$ of order two with $\mu$ non-trivial can be expressed in terms of moduli spaces of $G^{\R}$-Higgs bundles, where $G^{\R}$ is a real form of $G$. The corresponding actions on $\calR(X,G)$ are antiholomorphic and fixed points are described in terms of real representations of the fundamental group, see \cite{PR}.

\vspace{20 pt}
\noindent{\bf\large Fixed points in moduli spaces of Higgs pairs}
\vspace{10 pt}

\noindent The final Chapter \ref{chapter-pairs} considers finite group actions on moduli spaces of $(G,V)$-Higgs pairs, where now we are given an arbitrary representation $\rho:G\to\GL(V)$. In this situation the actions of $\Aut(X)$ and $\C^*$ still make sense, and the action of $H^1(X,Z)$ is well defined as long as $Z$ is in the kernel of $\rho$ ---otherwise we may consider $H^1(X,Z\cap\ker\rho)$ instead. However, we must replace $\Out(G)$ with the group $\Out(G,V)$, which we define next. Let $\GL_G(V)$ be the subgroup of $\GL(V)\times\Aut(G)$ consisting of pairs $(\da,\theta)$ such that $\da$ induces an isomorphism between $\rho$ and $\rho\circ\theta$. We have a homomorphism $\rho_G:G\to\GL_G(V)$ mapping $g$ to $(\rho(g),\Int_g)$, whose image $\rho_G(G)$ is a normal subgroup of $\GL_G(V)$. Set $\Out(G,V):=\GL_G(V)/\rho_G(G)$. Then the action of $\GL_G(V)$ on $\mdl(X,G,V)$ such that $(\da,\theta)$ sends $(E,\phi)$ to $(\theta(E),\da(\phi))$, where $\da(\phi)\in H^0(X,E\times_{\da\circ\rho}V\otimes K_X)\cong H^0(X,\theta(E)\times_{\rho}V\otimes K_X) $ is the Higgs field induced by $\da^{-1}$, is a left action whose restriction to $\rho_G(G)$ preserves the isomorphism class of $(E,\phi)$. Combining everything, we get a right action
\begin{equation}\label{eq-def-big-group-pairs}
    \mdl(X,G,V)\righttoleftarrow H^1(X,Z)\rtimes(\Aut(X)\times\Out(G,V))\times\C^*.
\end{equation}

Let $\Gamma$ be a finite subgroup of (\ref{eq-def-big-group-pairs}). Projections on the second, third and fourth factors provide homomorphisms $\eta,a$ and $\mu$, and projection on the first factor provides a 1-cocycle $\alpha\in Z^1_{a,\eta}(\Gamma,H^1(X,Z))$, where $\Gamma$ acts on $H^1(X,Z)$ via $\eta$ and the projection of $a$ to $\Out(G)$.
The fixed point Theorems for moduli spaces of Higgs pairs are similar to those for Higgs bundles, but we have to replace lifts $\theta:\Gamma\to\Aut(G)$ with homomorphisms $(\da,\theta):\Gamma\to\GL_G(V)$ lifting $a$, the equivalence relation $\sim$ by conjugation by elements of $\rho_G(G)$, and the weight spaces $\liegm$ with $\lievm$. The analogues of Theorems A, B, C and D are Theorems \ref{main-pairs}, \ref{th-fixed-points-oscar-ramanan-pairs}, \ref{th-prym-narasimhan-ramanan-oscar-ramanan-pairs} and \ref{th-prym-narasimhan-ramanan-general-pairs}, respectively.

\vspace{20 pt}
\noindent{\bf\large Motivations and further developments}
\vspace{10 pt}

\noindent We intend to apply our results to several settings in future work. For example, Atiyah--Bott-like fixed point theorems may be applied to study the topology of the moduli space. In this direction, Narasimhan--Ramanan \cite{narasimhan-ramanan} calculate the $y$-genus of the moduli space of vector bundles of fixed rank and determinant, showing in particular that the Euler characteristic and the signature vanish. 
Andersen \cite{andersen} describes the simple fixed point locus of a finite order automorphism $\eta$ of $X$ in the moduli space of $G$-bundles using twisted equivariant $G$-bundles, and applies this to the calculation of Witten--Reshetikhin--Turaev invariants of the mapping torus of $\eta$.
The main obstruction to apply this formalism to moduli spaces of Higgs bundles is their non-compactness and non-smoothness, which may be overcome by restricting our attention to the cohomology of Hodge bundle subvarieties, which are compact. Via localization, this may lead to results on the $\C^*$-equivariant cohomology of the moduli space.

Fixed point descriptions in terms of parabolic Higgs pairs have been achieved by Ander-\\sen--Grove \cite{andersen-fixed-points} for vector bundles of rank $2$ and García-Prada--Wilkin \cite{ow} for arbitrary $G$.
These rely on the correspondence between equivariant $G$-Higgs bundles and parabolic bundles. Extending this to twisted equivariant pairs would allow for generalizations of their results in our setting.

An important  motivation for this thesis is  
the  identification of  hyperk\"ahler or Lagrangian subvarieties of $\cM(X,G)$, which are the support of branes in the context of mirror 
symmetry and Langlands duality as introduced by Kapustin and Witten \cite{kapustin-witten}. 
% More precisely, recall that the smooth locus of $\cM(X,G)$ is a hyperkähler variety with distinguished complex structures $I$ and $J$ induced by the complex structures of $X$ and $G$ respectively. Several types of branes arise depending A brane of type $B$ with respect to $I$ is a hyperholomorphic vector bundle with hyperkähler suport, whereas a brane of type $A$ is a flat vector bundle with Lagrangian support. Similar definitions apply for $J$. We express the types of a brane with respect to $I, J$ and $IJ$ with a triple of letters in the same order. 
For example, if the projection of a finite subgroup $\Gamma$ of (\ref{eq-big-group}) on $\C^*$ is trivial, the smooth fixed point locus is hyperkähler and so it is the potential support of BBB-branes. However, if $\Gamma$ has order two and the corresponding projection on $\C^*$ is non-trivial, the smooth fixed point locus is the potential support of BAA-branes. We expect the Prym--Narasimhan--Ramanan construction to provide examples of fully equipped branes: this has been achieved by Franco--Gothen--Oliveira--Peón-Nieto \cite{franco-branes} in the case of the action of finite cyclic subgroup of the Jacobian.
% , who construct Narasimhan--Ramanan BBB-branes over the fixed point locus using the Prym--Narasimhan--Ramanan construction 

Examples of BAA-branes corresponding to $\U(n,n)$-Higgs bundles inside the moduli space of vector bundles of rank $2n$, which are fixed points of multiplying the Higgs field by $-1$, have been considered by Hitchin \cite{hitchin2013}. The conjectural mirrors have support over the subvariety of $\Sp(2n,\C)$-bundles. It would be interesting to consider a finite extension of $\Sp(2n,\C)$ instead, whose objects correspond by our theory to twisted equivariant $\Sp(2n,\C)$-Higgs bundles over certain étale covers of $X$. We could then study what happens on the $\U(n,n)$ side, and the compatibility mirror symmetry with the Prym--Narasimhan--Ramanan construction. 

There are no instances of supports for branes of type ABA and AAB in this thesis. Some of these arise from actions involving antiholomorphic involutions of $G$ and $X$, which we plan to consider in the future. Biswas--Calvo--García-Prada \cite{biswas-calvo-García-prada} study real $G^{\R}$-Higgs bundles, which are the real version of our twisted equivariant Higgs bundles. Other references are Biswas--García-Prada \cite{biswas-García-prada} and Biswas--García-Prada--Hurtubise \cite{bhp,biswas-García-prada-hurtubise2,biswas-García-prada-hurtubise3}.

The realization of $\mdl(X,G)$ as the Hitchin fibration \cite{hitchin:duke} over a vector space is crucial in the study of its geometry. 
We expect our fixed point descriptions to be 
useful for the study of the Hitchin fibres of different fixed point subvarieties. Some references in this direction are Heller--Schaposnik \cite{heller-schaposnik}, Schaffhauser \cite{schaffhauser} and Schaposnik \cite{schaposnik-thesis}.

We also plan to extend our fixed point description to moduli spaces of parabolic Higgs bundles in future work. These correspond to representations of the fundamental group of punctured surfaces via non-abelian Hodge theory \cite{oscar-biquard-mundet}. The parabolic setup is geometrically richer and perhaps more natural from the physics point of view of mirror symmetry\cite{gukov-witten,kapustin-witten}.

Finally, our description of fixed points in moduli spaces of Higgs pairs may be applied to the study of fixed point subvarieties for Higgs bundles associated to real forms, which may lead to results on the topology of Higher Teichmüller components \cite{oscar-survey}.

\newpage 

\chapter{Higgs bundles on compact Riemann surfaces}\label{chapter-higgs-bundles}

Higgs bundles are the main characters of our story. To introduce them we will need two ingredients: a compact Riemann surface $X$ with canonical bundle $K_X$ and a reductive complex Lie group $G$ with centre $Z$ and Lie algebra $\lie g$. For now we assume that $G$ is connected, but we will study the case when $G$ is non-connected in Section \ref{section-twisted-and-non-connected-higgs}.

\section{Higgs pairs}\label{section-higgs-pairs}

 Let $V$ be a complex vector space equipped with a holomorphic representation
$$\rho:G\to\GL(V).$$
Given a $G$-bundle over $X$, there is an associated vector bundle $E(V):=E\times_{\rho}V$. 

\begin{definition}\label{def-higgs-pair}
A \textbf{$(G,V)$-Higgs pair} over $X$ is a pair $(E,\phi)$, where $E$ is a holomorphic principal $G$-bundle and $\phi$ is a section of $E(V)\otimes K_X$, called the \textbf{Higgs field}. Two $(G,V)$-Higgs pairs $(E,\phi)$ and $(E',\phi')$ are \textbf{isomorphic} if there is an isomorphism
$f:E\to E'$ of $G$-bundles such that the induced isomorphism $E(\mathfrak{g})\otimes K_X\to E'(\mathfrak{g})\otimes K_X$
sends $\phi$ to $\phi'$.
\end{definition}

\begin{remark}
    We may generalize the notion of $(G,V)$-Higgs pair by allowing tensorization by any line bundle $L\to X$ (or even vector bundles, see \cite{guille}). The result is $L$-twisted Higgs pairs, which also appear naturally in the literature (see \cite{bradlow-oscar}, for example).
\end{remark}

We recall the stability notions for $(G,V)$-Higgs pairs using the approach in \cite{PBI}. Fix a maximal compact subgroup $K$ of $G$ with Lie algebra $\lie k$. Choose a non-degenerate pairing $\langle\cdot,\cdot\rangle$ on $\lie g$ extending the Killing form of a Levi subgroup. Every element $s\in i\lie k$ determines a parabolic subgroup $P_s$ with Lie algebra $\lie p_s$ of $G$, namely
\begin{equation}\label{eq-def-Ps}
    P_s:=\{g\in G\suhthat e^{ts}ge^{-ts}\,\text{remains bounded as}\;t\to\infty\}.
\end{equation}
If $L_s$ is its Levi subgroup then $K_s:=K\cap L_s$ is a maximal compact subgroup of $L_s$ and its inclusion in $P_s$ is a homotopy equivalence. Now let $E$ be a $G$-bundle with a holomorphic reduction $\tau\in H^0(X,E(G/P_s))$, where $E(G/P_s)$ is the $G/P_s$-bundle associated to $E$ via the natural left action of $G$ on $G/P_s$. We call $E_{\tau}$ to the corresponding $P_s$-bundle. Then there is a smooth reduction $\tau'\in\Omega^0(X,E_{\tau}/K_s)$, and we may equip the corresponding $K_s$-bundle with a connection $A$ with curvature $F_A$. We define
\begin{equation}\label{eq-def-deg}
    \deg E(\tau,s):=\frac{i}{2\pi}\int_X\chi_s(F_A),
\end{equation}
where $\chi_s$ is the image of $s$ under the isomorphism $\lie g\cong\lie g^*$ induced by the non-degenerate pairing.

For each $s\in i\lie k$ we may also define
\begin{align}\label{eq-def-V_s}
    &V_s:=\{v\in V\suhthat \rho(e^{ts})v\,\text{remains bounded as}\;t\to\infty\}\andd\\\nonumber
    & V^0_s=\{v\in V\suhthat \lim_{t\to\infty} \rho(e^{ts})v=v\}.
\end{align}
Given a reduction $\tau\in H^0(X,E(G/P_s))$, we may define a sub-bundle $E(V)_{\tau,s}:=E_{\tau}\times_{P_s}V_s\subseteq E(V)$. Given a further reduction $\tau'\in H^0(X,E_{\tau}(P_s/L_s))$, we also have a sub-bundle $E(V)^0_{\tau',s}:=E_{\tau'}\times_{L_s}V^0_s\subseteq E(V)_{\tau,s}$.

\begin{definition}\label{def-stability-higgs-pair}
Let $z\in i\zk$, where $\zk$ is the centre of $\lie k$. A $(G,V)$-Higgs pair $(E,\phi)$ over $X$ is:

\begin{itemize}
    \item \textbf{$z$-semistable} if $\deg E(\tau,s)\ge \pair{z}s$ for any $s\in i\lie k$ and any reduction of structure group $\tau\in H^0(X,E(G/P_s)\otimes K_X)$ such that $\phi\in H^0(X,E(V)_{\tau,s}\otimes K_X)$.
    \item \textbf{$z$-stable} if $\deg E(\tau,s)> \pair{z}s$ for any $s\in i\lie k$ and any reduction of structure group $\tau\in H^0(X,E(G/P_s))$ such that $\phi\in H^0(X,E(V)_{\tau,s}\otimes K_X)$.
    \item \textbf{$z$-polystable} if it is $z$-semistable and, if $\deg E(\tau,s)=\pair{z}s$ for some $s\in i\lie k$ and a reduction $\tau\in H^0(X,E(G/P_s))$ such that $\phi\in H^0(X,E(V)_{\tau,s}\otimes K_X)$, there is a further holomorphic reduction of structure group $\tau'\in H^0(X,E_{\tau}(P_s/L_s))$ with $\phi\in H^0(X,E(V)^0_{\tau',s}\otimes K_X)$
\end{itemize}
\end{definition}

There is a moduli space $\mdl(X,G,V)$ classifying isomorphism classes of ($z$-)polystable $(G,V)$-Higgs pairs \cite{schmitt:2008}, where $z$ runs over $i\zk$.

\section{Moduli space of \texorpdfstring{$G$}{G}-Higgs bundles}\label{section-moduli-space-higgs-bundles}

When $V=\lie g$, the Lie algebra of $G$, and $\rho=\Ad:G\to\GL(\lie g)$ is the adjoint representation, Definition \ref{def-higgs-pair} yields the notion of a \textbf{$G$-Higgs bundle}.
 We will sometimes denote the adjoint bundle
$E(\mathfrak{g})$ by $\ad(E)$.
A $G$-Higgs bundle $(E,\phi)$ is said to be {\bf simple} if $\Aut(E,\phi)\cong Z$ where $\Aut(E,\phi)$
is the group of Higgs bundle automorphisms of $(E,\phi)$ and $Z\subset G$ 
is the centre of $G$. 

\begin{remark}
    When $G$ is classical we may regard a $G$-bundle as a vector bundle equipped with some extra structure. This is realized via extension of structure group by an embedding of $G$ in $\GL(n,\C)$, followed by taking the associated vector bundle. Similarly, in this context a $G$-Higgs bundle may be regarded as a pair $(E,\phi)$, where $E$ is a vector bundle with some extra structure and $\phi$ is a section of $\End(E)\otimes K_X$.
\end{remark}

Definition \ref{def-stability-higgs-pair} yields (poly,semi)stability notions for $G$-Higgs bundles. Recall that these depend on a parameter $z\in i\zk$, where $\zk$ is the centre of a maximal compact subalgebra of $\lie g$. There exists a moduli space classifying isomorphism classes of $z$-polystable $G$-Higgs bundles, which we call $\mdl_z(X,G)$ \cite{schmitt:2008}. We call $\mdl(X,G)$ to the union of all the moduli spaces $\mdl_z(X,G)$ as $z$ runs through $i\zk$. This is a complex quasi-projective algebraic variety. It is not smooth in general, and its smooth locus is usually the open subvariety consisting of stable and simple points, which we call $\mdl_{ss}(X,G)$. We will simply call this the \textbf{smooth locus} by abuse of notation.

\begin{remark}
    Let $E$ be a $G$-bundle. Given a maximal compact subalgebra $\lie k\subset\lie g$, $s\in i\lie k$ and $\tau\in H^0(X,E(G/P_s))$, we have $E(V)_{\tau,s}=E_{\tau}(\lie p_s)$, the adjoint bundle of the $P_s$-bundle $E_{\tau}$ determined by $\tau$. Moreover, given a further reduction of structure group $\tau'\in H^0(X,E_{\tau}(P_s/L_s))$, we have $E(V)^0_{\tau',s}=E_{\tau'}(\lie l_s)$, the adjoint bundle of the $L_s$-bundle $E_{\tau'}$ determined by $\tau'$. These are all ingredients involved in the definition of (poly,semi)stability.
\end{remark}

A very important tool to study $\mdl(X,G)$ is the \textbf{Hitchin fibration}. This is a realization of $\mdl(X,G)$ as a fibration over a vector space $B(G)$ whose generic fibres are abelian varieties, making $\mdl(X,G)$ an algebraic completely integrable system.
For example, when $G=\GL(n,\C)$ the vector space $B(G)$ is equal to $\bigoplus_{k=1}^nH^0(X,K_X^{\otimes k})$ and the projection $\mdl(X,G)\to B(G)$ is given by taking "characteristic polynomials" of the Higgs field, thought of locally as an endomorphism of a vector bundle.

Another crucial feature of $\mdl(X,G)$ is its hyperk\"ahler structure. In other words, it has a natural symplectic structure $\omega$ and two anticommuting complex structures $I$ and $J$ such that both $\omega(\cdot,I\cdot)$ and $\omega(\cdot,J\cdot)$ are K\"ahler metrics. The complex structure $I$ is induced by the complex structure of the Riemann surface $X$, whereas $J$ is induced by the complex structure of $G$.

\section{Non-abelian Hodge theory}

The moduli space of $G$-Higgs bundles is closedly related to the character variety of $\pi_1(X)$ with values in $G$. The proof of this result involves two intermediate steps: the Hitchin--Kobayashi correspondence and the Donaldson--Corlette theorem.

Let $K\subset G$ be a maximal compact subgroup
of $G$. Let $(E,\phi)$ be a $G$-Higgs bundle on $X$. Let $h$ be a reduction of structure group of $E$ from $G$
 to $K$, and $F_h$ be the curvature of  the {\bf Chern connection} --- the unique connection on $E$ compatible with $h$ and the holomorphic structure of $E$.
 Let $\sigma_h:\Omega^{1,0}(X,E(\lieg))\to \Omega^{0,1}(X,E(\lieg))$
be the $\C$-antilinear map defined by the reduction $h$ and the conjugation 
between $(1,0)$- and $(0,1)$-forms on $X$.
Consider the {\bf Hitchin equation} 
\begin{equation}\label{hitchin-equation}
F_h+[\phi,\sigma_h(\phi)]=-i2\pi z\omega,
\end{equation}
where $z\in i\zk$ and $\omega\in\Omega^2(X)$ is a Kähler form on $X$ with total volume one.
One has 
the following (see \cite{hitchin1987,simpson,PBI}). 

\begin{theorem}\label{EH1}
Let $(E,\phi)$ be a $G$-Higgs bundle on $X$ and let $z\in i\zk$.
Then $(E,\phi)$ is $z$-polystable if and only if the $G$-bundle $E$ 
admits a metric $h$ satisfying (\ref{hitchin-equation}). 
\end{theorem}

\begin{remark}
Let $(E,\phi)$ be a $\GL(n,\C)$-Higgs bundle, or the associated Higgs bundle of rank $n$. Equation (\ref{hitchin-equation}) implies that $\frac{i}{2\pi}\int_X\tr F_h=nz$. The left hand side is the \textbf{degree} of $E$, a topological invariant, which means that $z$ is completely determined by the topology of $E$.
\end{remark}

By a {\bf representation} we mean a 
 homomorphism $\rho:\pi_{1}(X)\to G$. The set of all such homomorphisms, 
$\Hom(\pi_{1}(X),G)$, is an affine variety (this follows from the fact that $G$ is linear algebraic, see \cite{goldman}). The group $G$ acts on $\Hom(\pi_{1}(X),G)$ by conjugation
 \[ g\cdot \rho(\gamma):=g^{-1}\rho(\gamma)g.\] for $g\in G$, $\rho\in \Hom(\pi_{1}(X),G)$ and $\gamma\in \pi_{1}(X)$.
 The GIT quotient $\Hom(\pi_{1}(X),G)\sslash G$ is defined by restricting to the subspace
$\Hom^{+}(\pi_{1}(X),G)$ of reductive representations: otherwise the quotient is not Hausdorff. A representation  $\rho\in \Hom(\pi_{1}(X),G)$ is called \textbf{irreducible} if its stabilizer in $G$ is equal to $Z$. It is
 {\bf reductive} if 
 its composition with the adjoint representation of $G$ in $\lie{g}$ is a direct sum of irreducible 
representations, i.e. representations which are the compositions of irreducible representations $\pi_1(X)\to G$ with the adjoint representation.
 We define the {\bf character variety} to be the orbit space
 \[\calR(X,G):=\Hom^{+}(\pi_{1}(X),G)/G=\Hom(\pi_{1}(X),G)\sslash G.\]
 This is an affine algebraic variety, since $\Hom(\pi_{1}(X),G)$ is affine (see \cite{newstead} or \cite{richardson}). 
 
 Let $(E,\phi)$ be a $G$-Higgs
 bundle. Theorem \ref{EH1} states that the polystability of $(E,\phi)$
 is equivalent to a reduction $h$ satisfying the Hitchin equation with $z=0$. A simple computation shows that if $\nabla_h$ is the Chern connection of $h$,
 \[D=\nabla_h+\phi-\sigma_h(\phi)\] 
is a flat connection on the  $G$-bundle 
$E$. Moreover, its holonomy defines a reductive representation of $\pi_{1}(X)$ in  $G$. By a theorem of Donaldson 
\cite{donaldson} and Corlette \cite{corlette}, all reductive representations $\rho:\pi_{1}(X)\to G$ arise in this way. 
More concretely one has the {\bf non-abelian Hodge correspondence} given by the 
following.
 
 \begin{theorem}[Non-abelian Hodge correspondence]\label{rep1}
  Let $G$ be a reductive complex  Lie group. There is a real analytic isomorphism $\cM_0(X,G)\cong \calR(X,G)$ between the moduli space of polystable $G$-Higgs bundles and the character variety of $G$. Under 
this isomorphism, the irreducible representations
are in correspondence with the stable and simple $G$-Higgs bundles.
 \end{theorem}

 \begin{remark}
     When the parameter $z\in i\zk$ defining (poly,semi)stability is not zero, there is also a version of Theorem \ref{rep1} involving representations of the universal central extension of $\pi_1(X)$ (see \cite{narasimhan-seshadri} for a version with zero Higgs field).
 \end{remark}

\section{Group actions on the moduli space of \texorpdfstring{$G$}{G}-Higgs bundles}\label{section-action}

Let $(E,\phi)$ be a $G$-Higgs bundle. An automorphism $\theta$ of $G$ provides another $G$-Higgs 
bundle $(\theta(E),\theta(\phi))$, as follows: the bundle $\theta(E)$ is the holomorphic principal $G$-bundle with total space $E$ and $G$-action 
$$E\times G\to E;\,(e,g)\mapsto e\theta^{-1}(g),$$
where we have written the action of $G$ on $E$ by adjoining elements of $G$ on the right. Alternatively, this is just the $G$-bundle obtained from $E$ by the extension of structure group induced by $\theta$. If $\phi$ is locally equal to $(e,v)\otimes k$ for some $v\in \lieg$ and local sections $e$ and $k$ of $E$ and $K_X$ respectively, $\theta(\phi)$ is locally equal to $(e,\theta(v))\otimes k$. This is well defined, since 
\begin{align*}
(e g,\theta(\Ad_{g^{-1}}v))&=(e \theta^{-1}(\theta(g)),\Ad_{\theta(g^{-1})}\theta(v))
\\&=(e\cdot\theta(g),\Ad_{\theta(g^{-1})}\theta(v))
\\&=
(e,\theta(v))
\end{align*}
for each $g\in G$,
where the presence or absence of the dot denotes the $G$-action on $\theta(E)$ or $E$ respectively. 

Note that this defines a left action of $\Aut(G)$ on the set of $G$-Higgs bundles. We sometimes write $\theta(E,\phi)$ instead of $(\theta(E),\theta(\phi))$. 

Recall that $\Int(G)$ is the group of inner automorphisms of $G$. There is a surjection
$$\Int:G\to\Int(G);\,g\mapsto\Int_g,$$
with kernel equal to $Z$, the centre of $G$. The quotient $\Out(G):=\Aut(G)/\Int(G)$ is the group of outer automorphisms of $G$. The natural left action of $\Aut(G)$ induces a right action of $\outg$ on the set of isomorphism classes of $G$-Higgs bundles as follows: given a class $a\in\outg=\Aut(G)/\Int(G)$, we may choose a representative $\theta$ in $\Aut(G)$ and consider the automorphism of the set of $G$-Higgs bundles sending $(E,\phi)$ to $\theta^{-1}(E,\phi)$. The isomorphism class of $\theta^{-1}(E,\phi)$ is independent of the choice of $\theta$, since $(E,\phi)$ is isomorphic to $\Int_s(E,\phi)$ for every $s\in G$: an isomorphism is multiplication by $s$, since
\begin{equation*}
    egs=es\Int_s^{-1}(g)
\end{equation*}
for each $e\in E$ and $g\in G$. 
% {\color{red}When talking about isomorphism classes of $G$-Higgs bundles, we sometimes write $(a^{-1}(E),a^{-1}(\phi))$ or $a^{-1}(E,\phi)$ instead of $(\theta^{-1}(E),\theta^{-1}(\phi))$.}

We also have an action of the group of $Z$-bundles $H^1(X,Z)$ over $X$ on the set of isomorphism classes of $G$-Higgs bundles: let $\alpha\in H^1(X,Z)$ and a $G$-Higgs bundle $(E,\phi)$. Construct the fibre product $E\times_X\alpha$ with respect to $X$, which is associated to the pullback diagramme
\[\begin{tikzcd}
E\times_X\alpha\arrow{r}\arrow{d} &
E\arrow{d}\\
\alpha\arrow{r} & X
\end{tikzcd}.
\]
There is a $Z$-action on $E\times_X\alpha$, such that $z\in Z$ sends $(e,a)\in E\times\alpha$ to $(ez,az^{-1})$, so we set $E\otimes\alpha:=(E\times_X\alpha)/Z$. When equipped with the right $G$-action given on the left factor $E$, this is a $G$-bundle over $X$ whose isomorphism class only depends on the isomorphism classes of $E$ and $\alpha$. Since the adjoint action of $Z$ on $\lie g$ is trivial we have $(E\otimes\alpha)\times_{\Ad}\lie g=E\times_{\Ad}\lie g$, so that $\phi$ may also be regarded as a section of $E(\lie g)\otimes K_X$. This defines an action of $H^1(X,Z)$ on the set of isomorphism classes of $G$-bundles which is both right and left, since $H^1(X,Z)$ is abelian.

Finally, there is an action of the group of complex automorphisms $\Aut(X)$ of $X$ given by pullback and an action of $\C^*$ given by multiplying the Higgs field. The first one is a right action and the second one is right and left, since $\C^*$ is abelian.

Given elements $\alpha$ and $\alpha'$ in $H^1(X,Z)$, $a$ and $a'$ in $\Out(G)$, $\eta$ and $\eta'\in\Aut(X)$ and $\mu$ and $\mu'$ in $\C^*$, together with a $G$-Higgs bundle $(E,\phi)$ over $X$, we have
\begin{align*}
    (\eta'^*\theta'^{-1}(\eta^*\theta^{-1}(E\otimes\alpha)\otimes\alpha'),\mu'\mu \theta'^{-1}\theta^{-1}(\phi))=\\
((\eta\eta')^*(\theta\theta')^{-1}(E\otimes\alpha\otimes \eta^{*-1}a(\alpha')),\mu'\mu(\theta\theta')^{-1}(\phi)),
\end{align*}
where $\theta$ and $\theta'$ are elements of $\Aut(G)$ lifting $a$ and $a'$ respectively. Therefore, the three actions fit together to provide a right action of the group $H^1(X,Z)\rtimes(\outg\times\Aut(X))\times\C^*$ on the set of isomorphism classes of $G$-Higgs bundles on $X$, where the left action of $\outg\times\Aut(X)$ on $H^1(X,Z)$ which defines the semidirect product is given by 
\begin{equation*}
    (\outg\times\Aut(X))\times H^1(X,Z)\to H^1(X,Z);\,((a,\eta),\alpha)\mapsto \eta^{*-1}a(\alpha).
\end{equation*}
Here we are directly writing $a(\alpha)$, which is well defined even at the level of $Z$-bundles themselves because $\Int(G)$ acts trivially on $Z$.
Explicitly, $(\alpha,a,\eta,\mu)\in H^1(X,Z)\rtimes(\outg\times\Aut(X))\times\C^*$ sends $(E,\phi)$ to $(\eta^*\theta^{-1}(E\otimes\alpha),\mu \phi)$, where $\theta\in\Aut(G)$ lifts $a$.
Since this action preserves ($z$-)(poly)stability and simplicity it induces an action on $\cM(X,G)$ which restricts to the locus of simple and stable points $\cM_{ss}(X,G)$.

\begin{remark}
    The restriction of the action of $H^1(X,Z)\rtimes(\outg\times\Aut(X))\times\C^*$ to $H^1(X,Z)\rtimes\outg\times\C^*$ does not coincide with the action considered in \cite{PR}. Indeed, we have defined a right action such that $(\alpha,a,\mu)\in H^1(X,Z)\rtimes\outg\times\C^*$ sends each Higgs bundle $(E,\phi)$ to $(\theta^{-1}(E\otimes\alpha),\mu\theta^{-1}(\phi))$ for any lift $\theta\in\Aut(G)$ of $a$, whereas the action in \cite{PR} is on the left and sends $(E,\phi)$ to $(\theta(E)\otimes\alpha,\mu\theta(\phi))$. The reason is that we want all the group actions to be on the right, including the action of $G$ on $G$-bundles.
\end{remark}

\begin{remark}\label{remark-induced-iso}
    It is important to notice that, given an isomorphism of $G$-Higgs bundles
$$f:(E,\phi)\to (F,\psi)$$
and an element $(\alpha,\theta,\eta,\mu)\in H^1(X,Z)\rtimes(\Aut(G)\times\Aut(X))\times\C^*$, we
may define an induced isomorphism
$$\eta^*f\otimes \id:\eta^*\theta^{-1}(E\otimes\alpha,\mu\phi)\to \eta^*\theta^{-1}(F\otimes\alpha,\mu\psi)$$
via the natural biholomorphisms between $E$ and $\theta^{-1}(E)$, and $F$ and $\theta^{-1}(F)$.
\end{remark}

\newpage

\chapter{Twisted equivariant bundles}\label{chapter-twisted-equivariant-bundles}
To present the Prym--Narasimhan--Ramanan construction in its full generality we need to go one step further and consider $G$-Higgs bundles equipped with a certain type of twisted action of a group of automorphisms of the curve. The theory of twisted equivariant bundles is developed in our joint paper with García-Prada, Gothen and Mundet i Riera \cite{GGM}.

\section{Twisted equivariant actions}\label{section-twisted-action}
Let $G$ be a connected reductive complex Lie group. Fix a finite group $\Gamma$ and a homomorphism $a:\Gamma\to\Out(G)$. 

\begin{definition}\label{def-2-cocycle}
    A \textbf{$2$-cochain} $c\in
C^2_{a}(\Gamma,Z)$ is a map
\begin{displaymath}
  c\colon\Gamma\times\Gamma\to Z
\end{displaymath}
which satisfies $c(\gamma,1)=c(1,\gamma)=1$ for all $\gamma\in\Gamma$. The set of 2-cochains inherits a group multiplication from $Z$. The subgroup $Z^2_{a}(\Gamma,Z)$ of
\textbf{$2$-cocycles} consists of those $c$ which satisfy the \textbf{cocycle
condition}
\begin{equation}\label{eq-2-cocycle-condition}
  \ag(c(\gamma',\gamma''))c(\gamma,\gamma'\gamma'')
  =c(\gamma,\gamma')c(\gamma\gamma',\gamma'')
\end{equation}
for $\gamma,\gamma',\gamma''\in\Gamma$.
\end{definition}

By \cite{de-siebenthal} there exists a lift $\Out(G)\to\Aut(G)$ of the natural quotient, hence in particular we have a homomorphism $\theta:\Gamma\to\Aut(G)$ fitting in the commutative diagramme
\begin{equation}\label{eq-lift-automorphism}
   \begin{tikzcd}
\Aut(G)\arrow[r]  & \Out(G)\\
  & \Gamma\arrow[lu,dotted,"\theta"]\arrow[u,"a"]
\end{tikzcd}. 
\end{equation}
We will sometimes write $Z^2_{\theta}(\Gamma,Z)$ instead of $Z^2_{a}(\Gamma,Z)$.

\begin{definition}\label{def-twisted-action}
    Let $M$ be a smooth (complex) manifold equipped with a smooth (holomorphic) right $G$-action. A \textbf{$(\theta,c)$-twisted right action} of $\Gamma$ on $M$ is a choice of a smooth (complex) automorphism of $M$ for each $\gamma\in\Gamma$, which we denote $\bullet\cdot\gamma$, satisfying:
    \begin{equation}\label{eq-twisted-equivariant-axioms}
    (mg)\cdot\gamma=(m\cdot\gamma)\theta_{\gamma}^{-1}(g)\andd(m\cdot\gamma)\cdot\gamma'=(mc(\gamma,\gamma'))\cdot(\gamma\gamma')
    \end{equation}
    for every $m\in M$, $g\in G$ and $\gamma$ and $\gamma'\in\Gamma$. We are sometimes interested in looking at the pair consisting of the $G$-action and the twisted $\Gamma$-action, which we call a \textbf{$(\theta,c)$-twisted right $(G,\Gamma)$-action} on $M$.
\end{definition}

Now fix a right holomorphic action of $\Gamma$ on a complex manifold $X$.

\begin{definition}
  \label{def-twisted-equivariant-bundle}
  Let $E$ be a holomorphic principal $G$-bundle over $X$. A \textbf{$(\theta,c)$-twisted $\Gamma$-equivariant right structure/action} on $E$ is a $(\theta,c)$-twisted right action of $\Gamma$ on $E$ descending to the action of $\Gamma$ on $X$, i.e. fitting in the commutative diagramme
  \begin{equation*}
      \begin{tikzcd}
        E\arrow[r,"\cdot\gamma"]\arrow[d] & E\arrow[d]\\
        X\arrow[r,"\cdot\gamma"] & X
      \end{tikzcd}
  \end{equation*}
  for each $\gamma\in\Gamma$.
  The pair $(E,\cdot)$ is called a \textbf{$(\theta,c)$-twisted $\Gamma$-equivariant $G$-bundle}. When it is clear from the context we will omit $\cdot$ from the notation.

A \textbf{morphism} between $(\theta,c)$-twisted $\Gamma$-equivariant $G$-bundles over $X$ is a $\Gamma$-equivariant homomorphism of $G$-bundles. Sometimes we call this a \textbf{$(\theta,c)$-twisted $\Gamma$-equivariant morphism}, or just a \textbf{$\Gamma$-equivariant morphism}.
\end{definition}

Given a $(\theta,c)$-twisted $\Gamma$-equivariant $G$-bundle $E$ over $X$ and an element $\gamma$ in the centre $Z(\Gamma)$, there is an induced $(\theta,c)$-twisted $\Gamma$-equivariant structure on $\gamma^*E$ so that each $\gamma'\in\Gamma$ sends $\gamma^*e\in\gamma^*E$ to $\gamma^*(e\cdot\gamma')$. This descends to the action of $\Gamma$ on $X$, since $\gamma$ commutes with $\gamma'$. We define a \textbf{$Z(\Gamma)$-isomorphism} between two twisted equivariant bundles $E$ and $E'$ to be a $\Gamma$-equivariant isomorphism of $G$ bundles $E\to\gamma^*E'$ for some $\gamma\in Z(\Gamma)$. Equivalently, this is a map $f:E\to E'$ of complex manifolds fitting in the commutative diagramme
    \begin{displaymath}
      \begin{CD}
      E @>{f}>> E'\\
      @VVV @VVV \\
      X @>{\cdot\gamma}>> X
      \end{CD}
    \end{displaymath}
for some $\gamma\in Z(\Gamma)$. To distinguish between isomorphisms as defined earlier and $Z(\Gamma)$-isomorphisms we sometimes call the former ones \textbf{fibre-preserving isomorphisms}.

The rest of this section is a discussion on how the choice of the lift $\theta$ of $a:\Gamma\to\Out(G)$ affects the category of $(\theta,c)$-twisted $\Gamma$-equivariant $G$-bundles. 
First we follow \cite{PR} to describe the different equivalence classes of  lifts of 
$a:\Gamma\to \Out(G)$ to 
$\Aut(G)$ in terms of {\bf non-abelian cohomology}. 

\begin{definition}\label{def-1-cocycle}
    Let $\Gamma$ be a group and  $A$ 
another group acted upon by  $\Gamma$ via a homomorphism
$\theta:\Gamma\to \Aut(A)$, that is, 
every  $\gamma\in \Gamma$ defines an automorphism of $A$ that
we denote by $\theta_\gamma$.
We define a $1$-{\bf cocycle} 
of $\Gamma$ in $A$ 
as a map $\gamma\mapsto a_\gamma$ of $\Gamma$ to $A$ such that
\begin{equation}\label{eq-def-1-cocycle}
a_{\gamma\gamma'}=a_\gamma\theta_\gamma(a_{\gamma'})\;\; \mbox{for}\;\;
\gamma,\gamma'\in \Gamma.
\end{equation}
The set of cocycles is denoted by $Z^1_\theta(\Gamma,A)$.
Two cocycles $a,a'\in Z^1_\theta(\Gamma,A)$ are  said to be  {\bf cohomologous}
if there is $b\in A$ such that
\begin{equation}\label{eq-cohomologous}
a'_\gamma=b^{-1}a_\gamma \theta_\gamma(b).
\end{equation}
 This is an equivalence relation in $Z^1_\theta(\Gamma,A)$ and the quotient is 
denoted by $H^1_\theta(\Gamma,A)$. This is the {\bf first cohomology set of
$\Gamma$ in $A$}. 
\end{definition}

Coming back to our problem, let $S_a$ be the set of lifts  $\theta:\Gamma\to \Aut(G)$ of $a:\Gamma\to \Out(G)$. By \cite{de-siebenthal} this is non-empty. If we fix one element in $\theta\in S_a$ then every other lift $\theta'$ is equal to $\beta\theta$ for some map $\beta:\Gamma\to\Int(G)$. For every $\gamma$ and $\gamma'\in\Gamma$ we have
\begin{equation*}
    \beta_{\gamma\gamma'}\theta_{\gamma\gamma'}=
    \theta'_{\gamma\gamma'}=
    \theta'_{\gamma}\theta'_{\gamma'}=
    \beta_{\gamma}\tg\beta_{\gamma'}\theta_{\gamma'}=
    \beta_{\gamma}\tg(\beta_{\gamma'})\theta_{\gamma\gamma'}.
\end{equation*}
Comparing the first and last terms we conclude that $\beta\in Z^1_{\theta}(\Gamma,\Int(G))$, where by abuse of notation we regard $\theta$ as an automorphism of $\Int(G)$. Conversely, such a 1-cocycle provides a lift.

Since $\Int(G)$ is a normal subgroup of $\Aut(G)$, its conjugation action on $\Aut(G)$ preserves the set of lifts of $a$. It is straightforward to check that this action induces the action of $\Int(G)$ on $Z^1_{\theta}(\Gamma,\Int(G))$ given by Definition \ref{def-1-cocycle}. Thus we conclude:

\begin{lemma}\label{lemma-lifts-vs-non-abelian-cohomology}
    Given a lift $\theta$ of $a$, there is a $\Int(G)$-equivariant bijection
\begin{equation}\label{eq-bijection-Sa-cocycles}
    \{\text{Lifts of $a$}\}\leftrightarrow Z^1_{\theta}(\Gamma,\Int(G));\,\beta\theta\mapsto\beta,
\end{equation}
where the action on the left hand side is given by conjugation and the action on the right hand side is given by (\ref{eq-cohomologous}).
In particular, it induces a bijection
\begin{equation}\label{eq-bijection-Sa-cohomology-pairs}
    \{\text{Lifts of $a$}\}/\Int(G)\leftrightarrow H^1_{\theta}(\Gamma,\Int(G)).
\end{equation}

\end{lemma}

Thus, given another homomorphism $\theta':\Gamma\to\Aut(G)$ lifting $a$, there exists a map $s:\Gamma\to G$ such that $\Int_s\in Z^1_{\theta}(\Gamma,\Int(G))$ and $\theta'=\Int_s\theta$. Here $\Int_s$ is the homomorphism of $\Gamma$ to $\Int(G)$ defined by
$(\Int_s)_{\gamma}=\Int_{s_{\gamma}}$. We may assume that $s(1)=1$. Since $\Int(G)$ acts trivially on $Z$, there is a natural identification $Z^2_{\theta}(\Gamma,Z)\cong Z^2_{\theta'}(\Gamma,Z)$ and so we talk about 2-cocycles as elements in any of these sets indistinctly. We may define a map 
\begin{equation}\label{eq-def-cs}
    c_s:\Gamma\times\Gamma\to Z;\,(\gamma,\gamma')\mapsto s_{\gamma}\tg(s_{\gamma'})s_{\gamma\gamma'}^{-1}.
\end{equation}
This is in fact a 2-cocycle, since
\begin{align}\label{eq-calculation-cs-cocycle}
    \theta_{\gamma}(c_s(\gamma',\gamma''))c_s(\gamma,\gamma'\gamma'')&=
    \theta_{\gamma}( s_{\gamma'}\theta_{\gamma'}(s_{\gamma''})s_{\gamma'\gamma''}^{-1}) s_{\gamma}\theta_{\gamma}(s_{\gamma'\gamma''})s_{\gamma\gamma'\gamma''}^{-1}\\\nonumber&
    =\theta_{\gamma}( s_{\gamma'}\theta_{\gamma'}(s_{\gamma''})s_{\gamma'\gamma''}^{-1}) \theta_{\gamma}(s_{\gamma'\gamma''})s_{\gamma\gamma'\gamma''}^{-1}s_{\gamma}
    \\\nonumber
    &=\theta_{\gamma}( s_{\gamma'})\theta_{\gamma\gamma'}(s_{\gamma''}) s_{\gamma\gamma'\gamma''}^{-1}s_{\gamma}\\\nonumber
    &=
    s_{\gamma}\theta_{\gamma}(s_{\gamma'})s_{\gamma\gamma'}^{-1}s_{\gamma\gamma'}\theta_{\gamma\gamma'}(s_{\gamma''})s_{\gamma\gamma'\gamma''}^{-1}
  \\\nonumber
    &=c_s(\gamma,\gamma')c_s(\gamma\gamma',\gamma''),
\end{align}
and
$$c_s(\gamma,1)=s_{\gamma}\tg(s(1))s_{\gamma}^{-1}=1=s(1)\theta_1(s_{\gamma})s_{\gamma}^{-1}=c_s(1,\gamma).$$
Thus the product $cc_s$ is also a 2-cocycle.

\begin{remark}
    In (\ref{eq-calculation-cs-cocycle}) we have used that, given two elements $a,b\in G$ such that $ab\in Z$, we have $ab=ba$, since $b^{-1}ab=ab^{-1}b=a$. This fact is applied in the second and fourth equations.
\end{remark}

We have the following:

\begin{proposition}[Proposition 2.15 in \cite{GGM}]\label{prop-different-theta}
Let $s:\Gamma\to\Int(G)$ be a map such that $\Int_s\in Z^1_{\theta}(\Gamma,\Int(G))$. Let $\cc{}(\theta,c)$ be the category of $(\theta,c)$-twisted $\Gamma$-equivariant $G$-bundles. Then the categories $\cc{}(\theta,c)$ and $\cc{}(\Int_s\theta,cc_s)$ are equivalent.
\end{proposition}
\begin{proof}
Set $\theta':=\Int_s\theta$ as above. Let $(E,\cdot)$ be a $(\theta,c)$-twisted $\Gamma$-equivariant $G$-bundle. We have a natural choice of $(\theta',cc_s)$-twisted $\Gamma$-equivariant action on $E$, namely
\begin{equation}
    E\times\Gamma\to E;\,(e,\gamma)\mapsto e*\gamma:=es_{\gamma}\cdot\gamma.
\end{equation}
This satisfies Definition \ref{def-twisted-action}:
\begin{equation*}
    (eg)*\gamma=(e\sg\sg^{-1}g\sg)\cdot\gamma=(e\sg)\cdot\gamma\theta_{\gamma}^{-1}(\sg^{-1}g\sg)=e*\gamma\theta'^{-1}_{\gamma}(g)
\end{equation*}
and
\begin{align*}
    (e*\gamma)*\gamma'&=(e\sg\cdot\gamma)s_{\gamma'}\cdot\gamma'\\
    &=((e\cdot\gamma)\cdot\gamma')\theta_{\gamma\gamma'}^{-1}(\sg)\theta_{\gamma'}^{-1}(s_{\gamma'})\\
    &=
    ((ec(\gamma,\gamma'))\cdot\gamma\gamma')\theta_{\gamma\gamma'}^{-1}(\sg\theta_{\gamma}(s_{\gamma'}))\\
    &=
    ((ec(\gamma,\gamma'))\cdot\gamma\gamma')\theta_{\gamma\gamma'}^{-1}(c_s(\gamma,\gamma')s_{\gamma\gamma'})\\
    &=
    (ec(\gamma,\gamma')c_s(\gamma,\gamma')s_{\gamma\gamma'})\cdot\gamma\gamma'\\
    &=
    (ec(\gamma,\gamma')c_s(\gamma,\gamma'))*\gamma\gamma'
\end{align*}
for each $g\in G$, $\gamma,\gamma'\in\Gamma$ and $e\in E$.

A $(\theta,c)$-twisted $\Gamma$-equivariant morphism of $G$-bundles $f:E\to E'$ is also $(\theta',cc_s)$-twisted $\Gamma$-equivariant. Indeed, we have
\begin{equation*}
    f(e*\gamma)=f(es_{\gamma}\cdot\gamma)=f(e\cdot\gamma)\tg^{-1}(s_{\gamma})=f(e)\cdot\gamma\tg^{-1}(s_{\gamma})=f(e)s_{\gamma}\cdot\gamma=f(e)*\gamma
\end{equation*}
for each $e\in E$ and $\gamma\in\Gamma$.

Finally it is clear that the functor is invertible, namely we get the action $\cdot$ by composing the action $*$ with multiplication by $s(\bullet)^{-1}$.
\end{proof}

We may focus our attention on the ambigüity concerning the 2-cocycle $c$: two 2-cocycles $c$ and $c'\in Z^2_a(\Gamma,Z)$ are \textbf{cohomologous} if there exists a map $s:\Gamma\to Z$ such that 
\begin{equation}\label{eq-def-cohomologous-2-cocycle}
c'(\gamma,\gamma')=c(\gamma,\gamma')s_{\gamma\gamma'}^{-1}s_{\gamma}\tg(s_{\gamma'})
\end{equation}
for each $\gamma$ and $\gamma'\in\Gamma$. We define the \textbf{second Galois cohomology group} $H^2_a(\Gamma,Z)$ to be the set of equivalence classes, equipped with the group structure inherited from $Z$. 

\begin{corollary}\label{cor-different-c}
    If $c$ and $c'$ are cohomologous, there is an equivalence of categories between $\cc{}(\theta,c)$ and $\cc{}(\theta,c')$.
\end{corollary}
\begin{proof}
    Note that $cc_s=c'$ and $\Int_s=1$, so that $\Int_s\theta=\theta$. Hence we may use Proposition \ref{prop-different-theta}.
%     Let $(E,\cdot)$ be a \twisted $G$-bundle over $X$ and let $\beta:\Gamma\to Z$ be a map satisfying (\ref{eq-def-cohomologous-2-cocycle}). Then the following
%     \begin{equation*}
%         E\times\Gamma\to E;\,(e,\gamma)\mapsto e*\gamma:=(e\beta_{\gamma})\cdot\gamma 
%     \end{equation*}
%     defines a ($\theta,c'$)-twisted $\Gamma$-equivariant action on $E$. Indeed, the fact that $\beta_{\gamma}$ commutes with every element of $G$ for each $\gamma\in\Gamma$ implies that the first equation in (\ref{eq-twisted-equivariant-axioms}) is satisfied, and the second equation follows from:
%     \begin{align*}
%         (e*\gamma)*\gamma'&=
%         (((e\beta_{\gamma})\cdot\gamma)\beta_{\gamma'})\cdot\gamma'\\
%         &=((e\beta_{\gamma}\tg(\beta_{\gamma'}))\cdot\gamma)\cdot\gamma'\\
%         &=(e\beta_{\gamma}\tg(\beta_{\gamma'})c(\gamma,\gamma'))\cdot(\gamma\gamma')\\
%         &=(e\beta_{\gamma}\tg(\beta_{\gamma'})c(\gamma,\gamma')\beta_{\gamma\gamma'}^{-1})*(\gamma\gamma')\\
%         &=(ec'(\gamma,\gamma'))*(\gamma\gamma').
%     \end{align*}
% As in the proof of Proposition \ref{prop-different-theta} we may show that a morphism $(E,\cdot)\to (E',\cdot)$ is $(\theta,c)$-twisted $\Gamma$-equivariant if and only if it is $(\theta,c')$-twisted $\Gamma$-equivariant. The functor is clearly invertible.
\end{proof}

\section{Non-connected groups and twisted equivariant structures}\label{section-twisted-and-non-connected-principal}

There is a very explicit relation between principal bundles with non-connected structure group and twisted equivariant bundles which is crucial in our Prym--Narasimhan--Ramanan construction and we explain next, following \cite[Section 4]{GGM}. Let $G_0$ be the connected component of a (not necessarily connected) reductive complex Lie group $G$, $Z$ the centre of $G_0$ and $\Gamma:=G/G_0$. We have a short exact sequence
\begin{equation}\label{eq-general-extension}
    1\to G_0\to G\to\Gamma\to 1,
\end{equation}
making $G$ an \textbf{extension} of $G_0$ by $\Gamma$. Let $a:\Gamma\to \Out(G_0)$ be the characteristic homomorphism of (\ref{eq-general-extension}) and take a lift $\theta:\Gamma\to \Aut(G_0)$. 

\begin{proposition}\label{prop-associativity-2-cocycle}
    The multiplication
    \begin{equation}\label{eq-def-twisted product}
        (G_0\times\Gamma)\times(G_0\times\Gamma)\to G_0\times\Gamma;\,[(g,\gamma),(g',\gamma')]\mapsto (g\tg(g')c(\gamma,\gamma'),\gamma\gamma')
    \end{equation}
    is associative if and only if $c\in Z^2_a(\Gamma,Z)$. Consequently, if $c$ is a 2-cocycle then (\ref{eq-def-twisted product}) defines a group multiplication.
\end{proposition}
\begin{proof}
    We have
\begin{align*}
[(g,\gamma)(g',\gamma')](g'',\gamma'')
&=(c(\gamma,\gamma')g\tg(g'),\gamma\gamma')(g'',\gamma'')\\&
=(c(\gamma,\gamma')c(\gamma\gamma',\gamma'')g\tg(g')\theta_{\gamma\gamma'}(g''),\gamma\gamma'\gamma'')
\end{align*}
and
\begin{align*}
(g,\gamma)[(g',\gamma')(g'',\gamma'')]&
=(g,\gamma)(c(\gamma',\gamma'')g'\theta_{\gamma'}(g''),\gamma'\gamma'')\\&
=(c(\gamma,\gamma'\gamma'')\tg(c(\gamma',\gamma''))g\tg(g')\theta_{\gamma\gamma'}(g''),\gamma\gamma'\gamma'').
\end{align*}
Thus the product is associative if and only if (\ref{eq-2-cocycle-condition}) holds, as required.

Now assume that $c$ is a 2-cocycle. It is easy to see that $(1,1)$ is a neutral element for the multiplication, so in order to end the proof of the proposition it is left to show existence of inverses. We check that the inverse of $(g,\gamma)$ is $(c(\gamma^{-1},\gamma)^{-1}\tg^{-1}(g)^{-1},\gamma^{-1})$ for each $g\in G_0$ and $\gamma\in \Gamma$:
\begin{align*}
    (c(\gamma^{-1},\gamma)^{-1}\tg^{-1}(g)^{-1},\gamma^{-1})(g,\gamma)
    =(c(\gamma^{-1},\gamma)^{-1}\tg^{-1}(g)^{-1}\tg^{-1}(g)c(\gamma^{-1},\gamma),\gamma^{-1}\gamma)=(1,1)
\end{align*}
and
\begin{align*}
    (g,\gamma)(c(\gamma^{-1},\gamma)^{-1}\tg^{-1}(g)^{-1},\gamma^{-1})
    =
    (gg^{-1}c(\gamma,\gamma^{-1})\tg(c(\gamma^{-1},\gamma))^{-1},\gamma\gamma^{-1})
    =(1,1),
\end{align*}
where the last equation follows from
\begin{equation*}
\tg(c(\gamma^{-1},\gamma))=\tg(c(\gamma^{-1},\gamma))c(\gamma,1)\stackrel{(\ref{eq-2-cocycle-condition})}{=}c(\gamma,\gamma^{-1})c(1,\gamma)=c(\gamma,\gamma^{-1}).
\end{equation*}
\end{proof}

\begin{definition}\label{def-twisted-product}
Given a 2-cocycle $c$, we define the \textbf{$(\theta,c)$-twisted product} of $G_0$ by $\Gamma$, written $G_0\times_{(\theta,c)}\Gamma$, to be the group which is equal to $G_0\times\Gamma$ as a set and has multiplication (\ref{eq-def-twisted product}).
\end{definition}

Recall that there is an equivalence relation on the set of extensions of $G_0$ by $\Gamma$ making another extension $G'$ equivalent to $G$ if and only if there is an isomorphism $G\cong G'$ fitting in a commutative diagramme
$$
\begin{tikzcd}[cong/.style = {draw=none,"\cong" description,sloped}, eq/.style = {draw=none,"=" description,sloped}]
1\arrow[r]  & G_0 \arrow[r]\arrow[d,eq] & G \arrow[r]\ar[d,cong] & \Gamma \arrow[r]\arrow[d,eq] & 1\\
1\arrow[r]  & G_0 \arrow[r] & G' \arrow[r] & \Gamma \arrow[r] & 1.
\end{tikzcd}
$$

\begin{proposition}\label{prop-extensions-isomorphic-twisted-group}
There exists a 2-cocycle $c\in Z^2_{a}(\Gamma,Z)$ such that the extensions of $G_0$ given by $G$ and $G_0\times_{(\theta,c)}\Gamma$ are equivalent.
\end{proposition}
\begin{proof}
Take a section $t:\Gamma\to G$ of (\ref{eq-general-extension}) whose composition with the natural homomorphism $G\to\Int(G)$ restricts to $\theta$. Every $g\in G$ can be uniquely written in the form $g_0t_{\gamma}$, where $\gamma$ is the connected component where $g$ lies and $g_0\in G_0$. This determines a map $G\to G_0\times \Gamma$. Let $c:\Gamma\times\Gamma\to Z$ be the map satisfying $t_{\gamma}t_{\gamma'}=c(\gamma,\gamma')t_{\gamma\gamma'}$. It is straightforward to check that the group multiplication of $G$ induces the product (\ref{eq-def-twisted product}) on $G_0\times\Gamma$, which is therefore a group multiplication, and we have an isomorphism of extensions $G\cong G_0\times_{(\theta,c)}\Gamma$. In particular associativity of the group multiplication of $G$ implies the associativity of (\ref{eq-def-twisted product}), which in turn implies that $c\in Z^2_a(\Gamma,Z)$ by Proposition \ref{prop-associativity-2-cocycle}.
\end{proof}

Choose a 2-cocycle $c$ as in Proposition \ref{prop-extensions-isomorphic-twisted-group}. Slightly abusing notation, we consider the groups $G$ and $G_0\times_{(\theta,c)}\Gamma$ to be equal. 

\begin{proposition}[Proposition 2.7 in \cite{GGM}]\label{prop-bijection-G,Gamma-action}
    There is a bijective correspondence between holomorphic right $G$-actions on a complex manifold $M$ and $(\theta,c)$-twisted right $(G_0,\Gamma)$-actions on $M$ (Definition \ref{def-twisted-action}),
    given as follows:
\begin{itemize}
\item Given an action of $G$ on $M$, the action of the subgroup
  $G_0\subseteq G$ on $M$ is defined by restriction, and the automorphism determined by
  $\gamma\in \Gamma$ on $M$ is defined to be the one given by
  $(1,\gamma)\in G$.
\item Given a $(\theta,c)$-twisted $(G,\Gamma)$-action on $M$, we define the $G$-action on $M$ by
  \begin{displaymath}
    m\cdot(g,\gamma) = (m\cdot g)\cdot\gamma.
  \end{displaymath}
\end{itemize}
\end{proposition}

\begin{proof}
  Suppose we have an action of $G$ on $M$. We check the
  conditions of Definition~\ref{def-twisted-action} with the given
  actions of $G_0$ and $\Gamma$. Let $g\in G_0$ and
  $\gamma,\gamma',\gamma''\in\Gamma$. Then we have identities in
  $G$,
  \begin{align*}
    &(g,1)(1,\gamma) = (g,\gamma)
      = (1,\gamma)(\tg^{-1}(g),1)\quad\text{and}\\
    &(1,\gamma')(1,\gamma'') = (c(\gamma',\gamma''),\gamma'\gamma'')
      = (c(\gamma',\gamma''),1) (1,\gamma'\gamma''),
  \end{align*}
  which show that (\ref{eq-twisted-equivariant-axioms}) is satisfied.

  Conversely, suppose we have a $(G_0,\Gamma)$-twisted
  action on $M$, and define the action of $G$ as in the statement of
  the proposition (it is worth noting that this definition is forced
  upon us by the identity
  \begin{math}
    (g,\gamma) = (g,1)(1,\gamma)
  \end{math}).
  We must check that this in
  fact defines an honest $G$-action. In the case of a
  left action we have, for $g,g'\in G$ and $\gamma,\gamma'\in\Gamma$:
  \begin{align*}
    ( m\cdot(g,\gamma))\cdot (g',\gamma')
      &=((m\cdot g)\cdot\gamma)\cdot(g',\gamma')\\
      &=(((m\cdot g)\cdot\gamma)\cdot g')\cdot\gamma'\\
      &=((m\cdot g\tg(g'))\cdot\gamma)\cdot\gamma'\\
      &=((m\cdot g\tg(g')c(\gamma,\gamma'))\cdot(\gamma\gamma')\\
      &=m\cdot (g\tg(g')c(\gamma,\gamma'),\gamma\gamma')\\
      &=m\cdot ((g,\gamma)(g',\gamma')),\\
  \end{align*}
  as required.
\end{proof}

Let $E$ be a $G$-bundle over $X$ and set $p_Y:Y:=E/G_0\to X$, which is a principal $\Gamma$-bundle over $X$. Assume that the total space of $Y$ is connected or, equivalently, that $E$ is connected. Then $Y$ is an étale cover of $X$ with Galois group $\Gamma$. The pullback $F:=p_Y^*E$ has a reduction of structure group $F$ to $G_0$ given by the tautological section of $F/G_0\cong p_Y^*(E/G_0).$
Regarded as a $G_0$-bundle over $Y$, $F$ is isomorphic to
$E\to E/G_0\cong Y.$ Since the total spaces of $F$ and $E$ are equal, we have a holomorphic $G$-action on $F$. By Proposition \ref{prop-bijection-G,Gamma-action} this provides a right $G_0$-action (making it a holomorphic $G_0$-bundle over $Y$) and a $(\theta,c)$-twisted $\Gamma$-action. The action of $\Gamma$ is given by the restriction of the action of $G$ to the subset $\{(1,\gamma)\}_{\gamma\in\Gamma}\subset G$, which descends to the natural $\Gamma=\gal(Y/X)$-action on $Y$. Hence $F$ inherits a ($\theta,c$)-twisted $\Gamma$-equivariant structure. Conversely, given a $(\theta,c)$-twisted $\Gamma$-equivariant $G_0$-bundle $F$ on $Y$, the $G$-action on the total space of $F$ given by Proposition \ref{prop-bijection-G,Gamma-action} provides a holomorphic $G$-bundle $E$ over $X.$

\begin{definition}\label{def-categories}
  We denote by $\mathcal{C}_1$ the category whose
  \begin{itemize}
  \item \textbf{objects} are $(\theta,c)$-twisted $\Gamma$-equivariant
    principal $G_0$-bundles $E\to Y$ such that the twisted
    $\Gamma$-action descends to the action of $\Gamma$ as covering
    transformations of the fixed $\Gamma$-covering $E/G_0\cong Y\to X$,
    and whose
  \item \textbf{morphisms} are $Z(\Gamma)$-isomorphisms of twisted equivariant bundles, i.e. holomorphic
    maps $f\colon E\to E'$ both $G$ and $\Gamma$-equivariant such that the diagramme
    \begin{displaymath}
      \begin{CD}
      E @>{f}>> E'\\
      @VVV @VVV \\
      Y @>{\bar{f}}>> Y
      \end{CD}
    \end{displaymath}
    commutes and the induced map $\bar{f}\colon Y\to Y$  is
    a covering transformation which belongs to $Z(\Gamma)$.
  \end{itemize}

    We denote by $\mathcal{C}_2$ the category whose
  \begin{itemize}
  \item \textbf{objects} are principal $G$-bundles on $X$ such
    that there is an isomorphism $E/G_0\cong Y$ covering the identity on
    $X$
    and whose
  \item \textbf{morphisms} are morphisms of principal
    $G$-bundles.
  \end{itemize}
\end{definition}

Summing up, we get the following.

\begin{proposition}[Proposition 4.5 in \cite{GGM}]\label{prop-twisted-equivariant-bundles-one-to-one}
The categories $\cc1$ and $\cc2$ are equivalent. 
\end{proposition}
\begin{proof}
    It is left to define the corresponding functor on morphisms. An isomorphism of $G$-bundles $E\cong E'$ over $X$ induces an isomorphism of $\Gamma$-bundles $E/G_0\cong E'/G_0$. This is just an automorphism of $Y$ over the identity on $X$ that commutes with the action of the Galois group, i.e. an element of $Z(\Gamma)$. This shows that the induced map of twisted equivariant $G_0$-bundles over $Y$, which is $G_0$-equivariant (since it is $G$-equivariant) covers an element of $Z(\Gamma)$. Conversely, a morphism of twisted equivariant $G_0$-bundles over $Y$ which covers an element of $Z(\Gamma)$ is both $G_0$ and $\Gamma$-equivariant as a map of $G$-bundles, thus it induces an isomorphism of $G$-bundles over $X$ as required.
\end{proof}

\begin{remark}
Note that the equivalence of categories is not true if we replace the category of twisted equivariant bundles with $Z(\Gamma)$-isomorphisms with the subcategory of twisted equivariant bundles with fibre-preserving morphisms, since an automorphism of a $G$-bundle $E$ on $X$ may not induce the identity on $E/G_0$.
\end{remark}

\section{Twisted equivariant bundles and monodromy}\label{section-monodromy}

Proposition \ref{prop-twisted-equivariant-bundles-one-to-one} assumes the $G$-bundle $E$ to be connected, which may be too restrictive. We would like to consider the interaction between general $G$-bundles and twisted equivariant $G_0$-bundles.

Let ${E}\to X$ be a principal ${G}$-bundle (we do not assume
that ${E}$ is connected this time). We obtain a principal $\Gamma$-bundle
\begin{displaymath}
  Y:={E}/G_0 \to X,
\end{displaymath}
where the $\Gamma={G}/G_0$-action is induced by the ${G}$-action.
Since $\Gamma$ is discrete, $Y$ is a (possibly non-connected) étale cover of $X$ with covering group $\Gamma$ which we call $p\colon Y\to X$. Notice that $\Gamma$ acts on $Y$ on the
right.  Write
\begin{equation}
  \label{eq:tilde-p}
  \tilde{p}\colon{E}\to Y={E}/G_0
\end{equation}
for the quotient map. Since $G_0$ is connected, $\tilde{p}$
induces an isomorphism $\pi_0{E}\xra{\cong}\pi_0Y$.

Choose compatible base points in $X$ and $Y$. Fundamental groups will
be taken with respect to these base points and can be identified with
the corresponding covering groups of the (common) universal covering of $X$ and
$Y$. Let $Y'$ be the
connected component of $Y$ containing the base point.
The $\Gamma$-action on $Y$ induces a $\Gamma$-action on $\pi_0Y$. Let
$\Gamma'\subseteq\Gamma$ be the kernel of the corresponding
homomorphism $\Gamma\to\Aut(\pi_0Y)$. Then $Y'\to X$ is a connected
$\Gamma'$-covering. Moreover, we have the exact sequence
\begin{displaymath}
  1 \to \pi_1Y' \to \pi_1(X) \xrightarrow{w} \Gamma' \to 1.
\end{displaymath}
Identifying $\pi_1(X)$ with the covering group of the universal covering
$\tilde{X}\to X$, the \textbf{monodromy} $w\colon\pi_1(X)\to\Gamma'$ takes a
covering transformation of $\tilde{X}\to X$ to the induced covering
transformation of $Y'\to X$. Equivalently, if $[\alpha]\in\pi_1(X)$, then
$w([\alpha])\in\Gamma'$ is the unique element relating the endpoints
of the lift of the loop $\alpha$ starting at the base point of
$Y'$. We shall sometimes refer to
$w\colon\pi_1(X)\to\Gamma'\subset\Gamma$ as the \textbf{monodromy} of the
  ${G}$-bundle ${E}\to X$ and to $\Gamma'$ as the
\textbf{monodromy group} of ${E}$.

\begin{proposition}
  \label{prop:connected-reduction}
  Let ${E}\to X$ be a principal ${G}$-bundle with monodromy
  $w\colon\pi_1(X)\to \Gamma'\subseteq\Gamma$. Then ${E}$ admits a
  reduction of structure group to ${G}'\subseteq {G}$, where
  ${G}':=G_0\times_{\theta,c}\Gamma'$ is defined by
  restricting $\theta$ and $c$ to $\Gamma'$. Moreover, the total space
  of the corresponding ${G}'$-bundle ${E}'\subseteq {E}$ is
  connected, and $Y'={E}'/G_0\to X$ is a connected
  $\Gamma'$-covering with surjective monodromy
  $w\colon\pi_1(X)\to\Gamma'$.
\end{proposition}

\begin{proof}
  Let ${E}'=\tilde{p}^{-1}(Y)$, where $\tilde{p}$ was defined in
  \eqref{eq:tilde-p}. Then ${E}'$ is connected and, by
  construction, the ${G}$-action on ${E}$ restricts to a
  ${G}'$-action on ${E}'$ which makes ${E}'\to X$ into a
  principal ${G}'$-bundle.
\end{proof}

We have the following immediate corollary.

\begin{corollary}
  \label{cor:connected-reduction}
  Let ${E}\to X$ be a principal ${G}$-bundle. Then ${E}$
  admits a reduction to the connected component of the identity
  $G_0\subseteq{G}$ if and only if its monodromy is trivial. \qed
\end{corollary}

Let $H^1(X,\ug)$ and $H^1(Y',\underline{G_0},\Gamma',\theta,c)$ denote the sets of isomorphism classes of $G$-bundles over $X$ and $(\theta,c)$-twisted $\Gamma'$-equivariant $G_0$-bundles over $Y'$, respectively. 

\begin{proposition}\label{prop-action-centralizer}
    Let $Z_{\Gamma}(\Gamma')$ be the centralizer of $\Gamma'$ in $\Gamma$. There is a natural action of $Z_{\Gamma}(\Gamma')$ on $H^1(Y',\underline{G_0},\Gamma',\theta,c)$ on the left given as follows: take a $(\theta,c)$-twisted $\Gamma'$-equivariant $G_0$-bundle $(E,\cdot)$ over $Y'$ and an element $z\in Z_{\Gamma}(\Gamma')$. 
    \begin{itemize}
        \item We have another $G_0$-bundle $\theta_z(E)$ given by extension of structure group (see Section \ref{section-action}.
        \item There is a natural $(\theta,c)$-twisted $\Gamma'$-action on $\theta_z(E)$ given by 
        \begin{equation}\label{eq-action-centralizer-on-twisted-equivariant}
            e*\gamma:=[ec(z^{-1},z)^{-1}c(z^{-1},\gamma)c(z^{-1}\gamma,z)]\cdot\gamma.
        \end{equation}
    \end{itemize}
    This induces an action of $\cent$ on the set of isomorphism classes $H^1(X,\underline{G'})_{Y'}$ of $G'$-bundles $E'$ such that $E'/G_0\cong Y'$ via Proposition \ref{prop-twisted-equivariant-bundles-one-to-one}, which is just extension of structure group by $\Int_{(1,z)}$ for each $z\in \cent$.
\end{proposition}
\begin{proof}
Probably the best way to think of the seemingly random formula (\ref{eq-action-centralizer-on-twisted-equivariant}) is to take the $G':=G_0\times_{\theta,c}\Gamma'$-bundle $E_{G'}$ over $X$ associated to $E$ according to Proposition \ref{prop-twisted-equivariant-bundles-one-to-one}, and then define an alternative $G'$-action by
\begin{equation}\label{eq-action-centralizer-on-twisted-equivariant-equivalent}
    E_{G'}\times G'\to E_{G'};\,(e,(g,\gamma))\mapsto e*(g,\gamma):=e(1,z)^{-1}(g,\gamma)(1,z),
\end{equation}
i.e. consider the extension of structure group $\Int_{(1,z)}(E_{G'})$. Recall that the total space of $E_{G'}$ and $E$ are the same. The restriction of the action to $G_0$ is then given by 
\begin{equation*}
    e*g=e(1,z)^{-1}g(1,z)=e\theta_z^{-1}(g)
\end{equation*}
for each $e\in E_{G'}$ and $g\in G_0$, which is precisely the $G_0$-action on the total space of the underlying $G_0$-bundle $E$ defining the action of $G_0$ on $\theta_z(E)$. Definition (\ref{eq-action-centralizer-on-twisted-equivariant}) for the action of $\Gamma'$ on $\theta_z(E)$ is equivalent to the $\Gamma'$-action given by restriction of (\ref{eq-action-centralizer-on-twisted-equivariant-equivalent}), since
\begin{align*}
    (1,z)^{-1}(1,\gamma)(1,z)&=
    (c(z^{-1},z)^{-1},z^{-1})(1,\gamma)(1,z)\\&=
    (c(z^{-1},z)^{-1}c(z^{-1},\gamma),z^{-1}\gamma)(1,z)\\&=
    (c(z^{-1},z)^{-1}c(z^{-1},\gamma)c(z^{-1}\gamma,z),z^{-1}\gamma z)\\&=
    (c(z^{-1},z)^{-1}c(z^{-1},\gamma)c(z^{-1}\gamma,z),\gamma),
\end{align*}
where the last equation follows from the fact that $z$ commutes with every element of $\Gamma'$. Thus the pair consisting of the $G_0$-action and the $\Gamma'$-action (\ref{eq-action-centralizer-on-twisted-equivariant}) on $E$ is equivalent to the natural $G'$-action on $\Int_{(1,z)}(E_{G'})$ by the construction of Proposition \ref{prop-bijection-G,Gamma-action}. The same proposition implies that (\ref{eq-action-centralizer-on-twisted-equivariant}) defines a $(\theta,c)$-twisted action.
\end{proof}

\begin{theorem}\label{th-prym-narasimhan-ramanan-principal}
Let $Y\in H^1(X,\Gamma)$ with monodromy group $\Gamma'$ and corresponding connected component $Y'$, and denote by $H^1(X,\ug)_Y$ the set of isomorphism classes of $G$-bundles $E$ over $X$ such that $E/G_0\cong Y$. We have a bijection
\begin{equation}\label{eq-narasimhan-ramanan}
    H^1(Y',\underline{G_0},\Gamma',\theta,c)/Z_{\Gamma}(\Gamma')\xrightarrow{\sim} H^1(X,\ug)_{Y},
\end{equation}
where $Z_{\Gamma}(\Gamma')$ is the centralizer of $\Gamma'$ in $\Gamma$, which acts on $H^1(Y,\underline{G_0},\Gamma',\theta,c)$ as in Proposition \ref{prop-action-centralizer}. The bijection is given by Proposition \ref{prop-twisted-equivariant-bundles-one-to-one} and extension of structure group from $G'$ to $G$.
\end{theorem}
\begin{proof}
Given a $G$-bundle $E_G$ over $X$ such that $E_G/G_0\cong Y$, take a connected component $E\subset E_G$ such that $E/G_0\cong Y'$. Then $E$ is a reduction of structure group to $G'$, and so surjectivity follows by Proposition \ref{prop-twisted-equivariant-bundles-one-to-one}. 

To show that the morphism is well-defined, consider a $G'$-bundle $E$ and $z\in\cent$. The element $s:=(1,z)\in G$ determines an automorphism $\beta:=\Int_{s^{-1}}$ of $G'$ which defines an extension of structure group $\beta(E)$. Let $E_G$ be the extension of structure group of $E$ by the embedding $G'\hookrightarrow G$. Then the stabilizer of $E$ under the $G$-bundle action is equal to $G'$, which implies that the stabilizer of $Es\subset E_G$ is equal to $s^{-1}G' s=G'$. In other words, $Es$ determines a reduction of structure group of $E_G$ to $G'$. Moreover, the map
$$E\to Es;\,e\mapsto es$$
induces an isomorphism of $G'$-bundles $\beta(E)\cong Es$. Indeed,
recall from Section \ref{section-action} that $\beta(E)$ may be regarded as the $G'$-bundle which
has the same total space as $E$ and $G$-action determined by
$$E\times G'\to E;\, (e,g)\mapsto e\beta^{-1}(g).$$
But we have
$$e\beta^{-1}(g)s=esgs^{-1}s=esg,$$
which shows that the induced map $\beta(E)\to Es$ is $ G'$-equivariant. This implies that $E's$, which is a reduction of structure group of $E_G$ to $ G'$, is isomorphic to $\beta(E)$. In other words, $E_G$ is the extension of structure group of $\beta(E)$ by the embedding $ G'\hookrightarrow G$.

It is left to show injectivity. Let $F$ and $F'$ be two ($\theta,c$)-twisted $\Gamma'$-equivariant $G_0$-bundles over $Y'$ and let $E$ and $E'$ be the corresponding $ G'$-bundles over $X$. Since $G_0$ and $Y'$ are connected, both $E$ and $E'$ are connected. Assume that they have the same extension of structure group $ E_G$ to $  G$. Note that $E_G$ has an explicit decomposition into connected components, namely
\begin{equation*}
    E_G=\bigsqcup_{\gamma\Gamma'\in \Gamma/\Gamma'}E(1,\gamma),
\end{equation*}
where each coset in $\Gamma/\Gamma'$ has one and only one representative component in the union. Thus $E'$ must be equal to one of these components, say $E(1,\gamma)$ for some $\gamma\in\Gamma$. If $s:=(1,\gamma)$ then the stabilizer of $E'=Es$ in $G$ is $G'$, hence the stabilizer of $Es/G_0$ is $\Gamma'$, which is identified with the Galois group of $Es$ over $X$. We may identify $Y'$ with the quotient $E/G_0\subset E_G/G_0$, thus fixing a copy of $Y'$ inside $E/G_0\cong Y$. The first observation is that, on the one hand, $Es/G_0=E'/G_0\cong Y'$ and, on the other, $Es/G_0=(E/G_0)\gamma\cong Y'\gamma\subset Y$. Thus we have an isomorphism of $\Gamma'$-bundles $Y'\cong Y'\gamma$. This is the composition of the map sending $y\in Y'$ to $y\gamma\in Y'\gamma$ and an automorphism of $Y'\gamma$ over the identity on $X$, i.e. an element of the Galois group of $Y'\gamma$ over $X$, which is equal to $\Gamma'$. At the end of the day the isomorphism $Y'\cong Y'\gamma$ is given by an element $z\in\Gamma$. Since it is an isomorphism of $\Gamma'$-bundles it must commute with the action of $\Gamma'$, i.e. it must lie in the centralizer $\cent$. Therefore $E'=E(1,z)$ for some element $z\in\cent$, which means that $E'\cong\Int_{(1,z)}^{-1}(E)$ as required.
\end{proof}

\section{Associated bundles}\label{section-associated}

This Section follows Section 3.2 in \cite{GGM}. Fix a (smooth, complex) Lie group $G$. Let $E$ be a principal $G$-bundle over $X$ and $M$ be
a smooth or complex manifold equipped with a ($\theta,c$)-twisted right ($G,\Gamma$)-action (see Definition \ref{def-twisted-action}). 
% In the (most common) case, when the $G$-action on
% $M$ is on the left, we convert it into a right action in the standard
% way, by defining
% \begin{displaymath}
%   m\cdot g = g^{-1}\cdot m.
% \end{displaymath}
We denote by $E(M)$ the orbit space $(E\times M)/G$. In the case
of a left action of $G$ on $M$, we thus have the twisted product
$E(M)=E\times_GM$, which can be viewed as the quotient of
$E\times M$ by the equivalence relation
\begin{equation}\label{equiv-associated}
  (eg, m)\sim(e,g\cdot m)
\end{equation}
for any
$e\in E$, $g\in G$ and $m\in M$.

\begin{remark}
  \label{rem:sections-equivariant-correspondence}
  The smooth (holomorphic) sections $s$ of $E(M)$ are in natural bijection with the
  smooth (holomorphic) maps $\tilde{s}:E\to M$ satisfying the $G$-equivariance
  condition
  \begin{equation}\label{antiequivariant}
    \tilde{s}(eg)=\tilde{s}(e)\cdot g \;\; \mbox{for any}\;\; e\in E\;\; \mbox{and} \;\; g\in G.
  \end{equation}
  % Furthermore, $\tilde{s}$ is smooth if and only if $s$ is smooth.
  To see this, let $\tilde{s}:E\to M$ satisfy
  \eqref{antiequivariant}. Then $(\Id,\tilde{s})\colon E\to E\times M$
  is $G$-equivariant and, therefore, descends to the quotient so that
  we have a commutative diagramme
  \begin{displaymath}
    \begin{CD}
      E @>{(\Id,\tilde{s})}>> E\times M \\
      @VVV @VVV \\
      E/G @>{s}>> (E\times M)/G\ ,
    \end{CD}
  \end{displaymath}
  and $s\colon E/G = X \to E(M) = (E\times M)/G$ is the section
  corresponding to $\tilde{s}$. Conversely, given a section $s$, we
  can recover $\tilde{s}$ by setting $\tilde{s}(e) = m$, where
  $s([e])=[(e,m)]$. This is well-defined because the fibres of
  $E\to X$ are $G$-torsors.
\end{remark}

We view $E$ as the $G$-frame bundle of $E(M)$ in the usual way: an
element $e\in E$ with $\pi(e)=x$ defines the frame
\begin{equation} \label{frame}
\begin{aligned}
  M &\xra{\cong} E(M)_x, \\
  m &\mapsto [e,m].
\end{aligned}
\end{equation}

\begin{proposition}
  \label{prop-associated-bundle-equivariant}
  Let $M$ be a manifold equipped with $(\theta,c)$-twisted right $(G,\Gamma)$-action and
  let $\pi\colon E\to X$ be a  $(\theta,c)$-twisted $\Gamma$-equivariant principal
  $G$-bundle. Then the associated fibre
  bundle with typical fibre $M$,
\begin{align*}
  E(M):=E\times_{G}M \to X;\,
  [e,m] &\mapsto \pi(e),
\end{align*}
defined by the above construction is a $\Gamma$-equivariant fibre bundle.
\end{proposition}

\begin{proof}
Consider the $\Gamma$-action on $E\times M$.
\begin{displaymath}
  [e,m]\cdot\gamma = [e\cdot\gamma,m\cdot\gamma].
\end{displaymath}
The $G$-equivariance of the $\Gamma$-actions on $E$ and $M$ implies that this action is $G$-equivariant, hence it induces an action on the quotient $E(M)=(E\times M)/G$. Moreover, we have
\begin{equation*}
    ((e,m)\cdot\gamma)\cdot\gamma'=(ec(\gamma,\gamma'),\rho(c(\gamma,\gamma'))^{-1}m)\cdot(\gamma\gamma')=(e,m)\cdot(\gamma\gamma'),
\end{equation*}
where $(e,m)$ is the class of $[e,m]$ in $E(M)$, so this is a genuine group action.
\end{proof}

% \begin{remark}
% If the subgroup of $Z$ generated by the image of  the cocycle $c$ acts trivially on $M$, then $E(M)$ acquires a true $\Gamma$-equivariant structure.
% \end{remark}

\begin{remark}\label{remark-associated-vector-bundle}
An important situation where Proposition \ref{prop-associated-bundle-equivariant} may be applied is the case when an embedding of $G\times_{\theta,c}\Gamma$ in $\GL(n,\C)$ is provided, where the twisted product $G\times_{\theta,c}\Gamma$ is given by Definition \ref{def-twisted-product}. In this case the right action of $\Gamma$ on $\C^n$ given for each $\gamma\in\Gamma$ by multiplying on the left by the image of $(1,\gamma)^{-1}$ in $\GL(n,\C)$ satisfies \ref{eq-action-gamma-V}. Hence a twisted $\Gamma$-equivariant $G$-bundle $E$ provides a $\Gamma$-equivariant vector bundle $E(\C^n)$ of rank $n$
, and conversely an equivariant vector bundle whose bundle of frames has a reduction of structure group to $G$ which is preserved by the $\Gamma$-action provides a twisted $\Gamma$-equivariant $G$-bundle
.
\end{remark}

\begin{remark}
  Note that the map (\ref{frame}) has no equivariance properties with
  respect to $\Gamma$.
\end{remark}

Let $C^\infty(E,M)^{G,\Gamma}$ ($\Hol(E,M)^{G,\Gamma}$) be the space of smooth (holomorphic) maps from $E$ to $M$ which are both $G$ and $\Gamma$-equivariant, and let
$C^{\infty}(X,E(M))^{\Gamma}$ ($H^0(X,E(M))^{\Gamma}$) be the space of $\Gamma$-equivariant smooth (holomorphic)
sections of $E(M)\to X$.
% \footnote{In case we have a left action on $M$ we convert it into a
%   right action in the standard way.}

\begin{proposition}
\label{prop:EM-Gamma-equivariant-sections}
With the above notation, there is a bijection
\begin{displaymath}
  C^\infty(E,M)^{G,\Gamma} \xra{\cong} C^{\infty}(X,E(M))^{\Gamma}\quad(\Hol(E,M)^{G,\Gamma} \xra{\cong} H^0(X,E(M))^{\Gamma})
\end{displaymath}
sending $\tilde{s}\colon E\to M$ to $s\colon X\to E(M)$ defined by
\begin{displaymath}
  s(x) = [e, \tilde{s} (e)]
\end{displaymath}
for any $e\in E$ with $\pi(e)=x\in X$.
The inverse takes a section $s\colon X\to E(M)$ to $\tilde{s}\colon
E\to M$ defined by
\begin{displaymath}
  \tilde{s}(e) = m,\quad\text{where $s(\pi{(e)}) = [e,m]$}.
\end{displaymath}
\end{proposition}

\begin{proof}
  The correspondence between $G$-equivariant maps $E\to M$ and
  sections of $E(M)\to X$ is given by Remark~\ref{rem:sections-equivariant-correspondence}
and the $\Gamma$-equivariance statement follows from the definition of the $\Gamma$-action on $E(M)$.
\end{proof}

\newpage

\chapter{Twisted equivariant structures and Higgs pairs}\label{chapter-twisted-equivariant-higgs-pairs}

We now apply the formalism of Chapter \ref{chapter-twisted-equivariant-bundles} to Higgs pairs. Throughout this chapter $X$ will be a compact Riemann surface with canonical bundle $K_X$.

\section{Twisted equivariant Higgs pairs}\label{section-twisted-equivariant-higgs-pairs}

This Section follows \cite[Section 3.2]{GGM} and \cite{PBI}. Let $G$ be a connected reductive complex Lie group with centre $Z$. Let $V$ be a complex vector space equipped with a holomorphic representation
$$\rho:G\to\GL(V).$$
This defines a left action of $G$ on $V$, thus a right action:
\begin{equation*}
    V\times G\to V;\,(v,g)\mapsto \rho(g)^{-1}v.
\end{equation*}
Recall that we have a notion of $(G,V)$-Higgs pair, see Definition \ref{def-higgs-pair}.

Consider a finite group $\Gamma$ equipped with a homomorphism $\eta:\Gamma\to \Aut(X)$ to the group of holomorphic automorphisms of $X$. We have a right action of $\Gamma$ on $X$ such that $\gamma\in\Gamma$ sends $x\in X$ to $\etag^{-1}(x)$. Take a homomorphism 
$$\theta:\Gamma\to\Aut(G);\;\gamma\mapsto\tg$$
and a 2-cocycle $c\in Z^2_{\theta}(\Gamma,Z)$ (see Definition \ref{def-2-cocycle}). Fix a map $\rhog:\Gamma\to\GL(V)$. We will sometimes write
$v\cdot\rhog(\gamma):=\rhog(\gamma)^{-1}v$.
Assume that $\rhog$ determines a $(\theta,c)$-twisted right $\Gamma$-action on $V$ (Definition \ref{def-twisted-action}), i.e. it satisfies
\begin{equation}\label{eq-action-gamma-V}
   (\rho(g)v)\cdotv\gamma=\rho(\theta_{\gamma^{-1}}(g))(v\cdotv\gamma) \andd (v\cdotv\gamma)\cdotv{\gamma'}=(\rho(c(\gamma,\gamma'))^{-1}v)\cdotv{\gamma\gamma'}
\end{equation}
for each $v\in V$, $g\in G$ and $\gamma,\gamma'\in\Gamma$.

Let $E$ be a $(\theta,c)$-twisted $\Gamma$-equivariant $G$-bundle over $X$. By Proposition \ref{prop-associated-bundle-equivariant} the vector bundle $E(V)$ inherits a right $\Gamma$-action descending to the action of $\Gamma$ on $X$, hence $E(V)\otimes K_X$ is also $\Gamma$-equivariant with $\gamma\in\Gamma$ sending
$(e,v)\otimes k\in E(V)\otimes K_X$ to $(e\cdot\gamma,v\cdotv\gamma)\otimes(\gamma^*k).$

\begin{definition}\label{mega}
A \textbf{$(\theta,c,\rhog)$-twisted $\Gamma$-equivariant $(G,V)$-Higgs pair} over $X$ is a $(G,V)$-Higgs pair $(E,\phi)$ over $X$ equipped with a holomorphic $(\theta,c)$-twisted $\Gamma$-equivariant action $\cdot$ on $E$ ---i.e. satisfying  (\ref{eq-twisted-equivariant-axioms})--- such that $\phi$ is a $\Gamma$-invariant section. To emphasize the action of $\Gamma$ we sometimes denote the object by $(E,\cdot,\phi)$. When $\theta,c,\rhog$ or $\Gamma$ are clear from the context we may omit them from the notation. If $V=\lie g$, $\rho=\Ad$ and the action of $\Gamma$ on $\lie g$ is given by $\theta^{-1}$, we call $(E,\phi)$ a \textbf{$(\theta,c)$-twisted $\Gamma$-equivariant $G$-Higgs bundle}.

A \textbf{homomorphism} between $(\theta,c,\rhog)$-twisted $\Gamma$-equivariant $(G,V)$-Higgs pairs over $X$ is a $\Gamma$-equivariant morphism of Higgs pairs. There is a notion of $Z(\Gamma)$ isomorphism as in Section \ref{section-twisted-action}, namely a $Z(\Gamma)$-isomorphism of twisted equivariant bundles preserving the Higgs field.
\end{definition}

\begin{remark}\label{remark-exact-sequence-aut}
Let $(E,\phi)$ be a $(G,V)$-Higgs pair. Let $\Aut(E,\phi)$ be the group of automorphisms of $(E,\phi)$ covering the 
identity of $X$, and let 
$\Aut_{\Gamma,\eta,\theta}(E,\phi)$ be the group of biholomorphic maps 
$f:E\to E$ preserving $\phi$ (according to the chosen action $\rhog$ of $\Gamma$ on $V$), so that $f$ covers the automorphism $\eta^{-1}_\gamma:X\to X$ for some $\gamma\in \Gamma$ and satisfies that $f(eg)=f(e)\tg^{-1}(g)$ for each $e\in E$. Given $f,f'\in \Aut_{\Gamma,\eta,\theta}(E)$, their group product is $ff'=f'\circ f$.
There is  an exact
sequence
\begin{equation}\label{exact-aut}
1\to \Aut(E,\phi)\to \Aut_{\Gamma,\eta,\theta}(E,\phi) \to \Gamma,
\end{equation}
where the group multiplication on $\Gamma$ is transposed, i.e. the product of $\gamma$ and $\gamma'\in\Gamma$ is $\gamma'\gamma$.
A $(\theta,c,\rhog)$-twisted $\Gamma$-equivariant structure on $E$ is simply
a $c$-twisted lift of (\ref{exact-aut}), i.e. a map $\Gamma\to \Aut_{\Gamma,\eta,\theta}(E,\phi)$ satisfying the second equation of (\ref{eq-twisted-equivariant-axioms}).
\end{remark}

We have analogous results to Propositions \ref{prop-different-theta} and \ref{cor-different-c} in this context. Let $\theta'=\Int_s\theta$, where $\Int_s\in Z^1_{\theta}(\Gamma,\Int(G))$, and define $c_s\in Z^2_{\theta}(\Gamma,Z)$ as in (\ref{eq-def-cs}). By Lemma \ref{lemma-lifts-vs-non-abelian-cohomology} $\theta'$ is a homomorphism. Define a map $\rhog^s:=\rho(s)\rhog: \Gamma\to\GL(V)$, so that we have a right action of $\Gamma$ on $V$:
\begin{equation}\label{eq-def-rhos}
    V\times\Gamma\to V;\,(v,\gamma)\mapsto v\cdot\rhog^s(\gamma):=(\rho(\sg)^{-1}v)\cdot\rhog(\gamma)=(\rho(\sg)\rhog(\gamma))^{-1}v.
\end{equation}
This is a ($\theta',cc_s$)-twisted right action:
\begin{align*}
    (\rho(g)v)\cdot\rhog^s(\gamma)&=(\rho(\sg)^{-1}\rho(g)v)\cdot\rhog(\gamma)\\
    &=(\rho(\Int_{\sg}^{-1}(g))\rho(\sg)^{-1}v)\cdot\rhog(\gamma)\\
    &=\rho(\tg^{-1}\Int_{\sg}^{-1}(g))((\rho(\sg)^{-1}v)\cdot\rhog(\gamma))\\
    &=\rho((\Int_{\sg}\tg)^{-1}(g))(v\cdot\rhog^s(\gamma))\\
    &=\rho(\tg'^{-1}(g))(v\cdot\rhog^s(\gamma))
\end{align*}
and
\begin{align*}
    (v\cdot\rhog^s(\gamma))\cdot\rhog^s(\gamma')&
    =(\rho(s_{\gamma'})^{-1}((\rho(\sg)^{-1}v)\cdot\rhog(\gamma)))\cdot\rhog(\gamma')\\&
    =((\rho(\tg(s_{\gamma'}))^{-1}\rho(\sg)^{-1}v)\cdot\rhog(\gamma))\cdot\rhog(\gamma')\\&
    =(\rho(c(\gamma,\gamma'))^{-1}\rho(\sg\tg(s_{\gamma'}))^{-1}v)\cdot\rhog(\gamma\gamma')\\&
    =(\rho(s_{\gamma\gamma'})^{-1}\rho(c(\gamma,\gamma')\sg\tg(s_{\gamma'})s_{\gamma\gamma'}^{-1})^{-1}v)\cdot\rhog(\gamma\gamma')\\&
    =(\rho(cc_s(\gamma,\gamma'))^{-1}v)\cdot\rhog^s(\gamma\gamma')
\end{align*}
for every $v\in V$, $\gamma,\gamma'\in\Gamma$ and $g\in G$, so that $\rhog^s$ satisfies (\ref{eq-action-gamma-V}) as required.

\begin{proposition}\label{prop-different-theta-higgs}
Let $s:\Gamma\to\Int(G)$ be a map such that $\Int_s\in Z^1_{\theta}(\Gamma,\Int(G))$. Let $\cc{}(\theta,c,\rhog)$ be the category of $(\theta,c,\rhog)$-twisted $\Gamma$-equivariant $(G,V)$-Higgs pairs. Then the categories $\cc{}(\theta,c,\rhog)$ and $\cc{}(\Int_s\theta,cc_s,\rhog^s)$ are equivalent.
\end{proposition}
\begin{proof}
    Let $(E,\cdot,\phi)$ be a \twistedr $(G,V)$-Higgs pair. By the proof of Proposition \ref{prop-different-theta} we may define a new $\Gamma$-action $*$ on $E$ making $(E,*)$ a $(\theta',cc_s)$-twisted $\Gamma$-equivariant $G$-bundle. It is left to show that $\phi$ is $\Gamma$-invariant with respect to $*$ and $\rhog^s$. If $\phi$ is locally equal to $(e,v)\otimes k\in E(V)\otimes K_X$, we have
    \begin{align*}
        (e*\gamma,v\cdot\rhog^s(\gamma))\otimes \gamma^*k&
        =((e\sg)\cdot\gamma,(\rho(\sg)^{-1}v)\cdot\rhog(\gamma))\otimes \gamma^*k\\&
        =((e\cdot\gamma)\tg^{-1}(\sg),\rho(\tg^{-1}(\sg))^{-1}(v\cdot\rhog(\gamma)))\otimes \gamma^*k\\&
        =(e\cdot\gamma,v\cdot\rhog(\gamma))\otimes\gamma^*k,
    \end{align*}
which, since $\phi$ is invariant with respect to $\cdot$ and $\rhog$, is also equal to $\phi$, as required.
\end{proof}

\begin{corollary}\label{cor-different-c-higgs}
    If $c$ and $c'$ are cohomologous, i.e. there is a map $s:\Gamma\to Z$ such that $c'(\gamma,\gamma')=c(\gamma,\gamma')s_{\gamma\gamma'}^{-1}s_{\gamma}\tg(s_{\gamma'})$, there is an equivalence of categories between $\cc{}(\theta,c,\rhog)$ and $\cc{}(\theta,c',\rhog^{s})$.
\end{corollary}
\begin{proof}
Follows directly from Proposition \ref{prop-different-theta-higgs} and (\ref{eq-def-cs}).
\end{proof}

By Proposition \ref{prop-different-theta-higgs} we may assume that the lift $\theta$ is chosen so that it preserves a maximal compact subgroup $K$ of $G_0$. We may define notions of (poly, semi)stability by changing slightly Definition \ref{def-stability-higgs-pair} \cite{PBI}: first note that if $s$ is in the $\Gamma$-invariant part $i\lie k^{\Gamma}$ of $i\lie k$ then $P_s$ is $\Gamma$-invariant and so is its Levi subgroup $L_s$, since
\begin{equation}\label{eq-def-levi}
    L_s=\{g\in G\suhthat \lim_{t\to\infty} e^{ts}ge^{-ts}=g\}.
\end{equation}
Hence $\theta$ induces $\Gamma$-actions on $G/P_s$ and $P_s/L_s$. If $(E,\cdot)$ is a $(\theta,c)$-twisted $\Gamma$-equivariant bundle then combining the actions $\cdot$ and $\theta$ makes the fibre bundle $E(G/P_s)$ into a $\Gamma$-equivariant bundle by Proposition \ref{prop-associated-bundle-equivariant}, as $Z$ is contained in $P_s$. This lets us introduce the corresponding space of $\Gamma$-invariant reductions $H^0(X,E(G/P_s))^{\Gamma}$. Given such a reduction $\sigma$, the corresponding $P_s$-bundle $E_{\sigma}$ is $(\theta,c)$-twisted $\Gamma$-equivariant. Moreover, since $Z$ is contained in $L_s$, we also have a $\Gamma$-equivariant associated bundle $E(P_s/L_s)$ with a corresponding space of $\Gamma$-invariant reductions $H^0(X,E_{\sigma}(P_s/L_s))^{\Gamma}$.

For each $s\in i\lie k$ we may define
$$V_s:=\{v\in V\suhthat \rho(e^{ts})v\,\text{remains bounded as}\;t\to\infty\}$$
and
\begin{equation*}
    V^0_s=\{v\in V\suhthat \lim_{t\to\infty} \rho(e^{ts})v=v\}.
\end{equation*}
Given a reduction $\sigma\in H^0(X,E(G/P_s))$, we may define a sub-bundle $E(V)_{\sigma,s}:=E_{\sigma}\times_{P_s}V_s\subseteq E(V)$. Given a further reduction $\sigma'\in H^0(X,E_{\sigma}(P_s/L_s))$, we also have a sub-bundle $E(V)^0_{\sigma',s}:=E_{\sigma'}\times_{L_s}V^0_s\subseteq E(V)_{\sigma,s}$.

\begin{definition}\label{definition-stable-twisted-equivariant}
Let $z\in i\zk^{\Gamma}$, where $\zk$ is the centre of $\lie k$. A $(\theta,c,\rhog)$-twisted $\Gamma$-equivariant $(G,V)$-Higgs pair $(E,\phi)$ over $X$ is:
\begin{itemize}
    \item \textbf{$z$-semistable} if $\deg E(\sigma,s)\ge z$ for any $s\in i\lie k^{\Gamma}$ and any $\Gamma$-invariant reduction of structure group $\sigma\in H^0(X,E(G/P_s))^{\Gamma}$ such that $\phi\in H^0(X,E(V)_{\sigma,s}\otimes K_X)$.
    
    \item \textbf{$z$-stable} if $\deg E(\sigma,s)> z$ for any $s\in i\lie k^{\Gamma}$ and any reduction of structure group $\sigma\in H^0(X,E(G/P_s))^{\Gamma}$ such that $\phi\in H^0(X,E(V)_{\sigma,s}\otimes K_X)$.
    
    \item \textbf{$z$-polystable} if it is $z$-semistable and, if $\deg E(\sigma,s)=z$ for some $s\in i\lie k^{\Gamma}$ and a reduction $\sigma\in H^0(X,E(G/P_s))^{\Gamma}$ such that $\phi\in H^0(X,E(V)_{\sigma,s}\otimes K_X)$, there is a further $\Gamma$-invariant holomorphic reduction of structure group $\sigma'\in H^0(X,E_{\sigma}(P_s/L_s))^{\Gamma}$ with $\phi\in H^0(X,E(V)^0_{\sigma',s}\otimes K_X)$.

If $z=0$ we omit it.
\end{itemize}
\end{definition}

\begin{remark}
    When $V=\lie g$, $\rho$ is the adjoint representation of $G$ on $\lie g$ and $\rhog=\theta$, we have $E(V)_{\sigma,s}=E(\lie p_s)$ and $E(V)^0_{\sigma',s}=E(\lie p_s)$ in Definition \ref{definition-stable-twisted-equivariant}.
\end{remark}

There is a moduli space $\cM_z(X,G,\Gamma,\theta,c,V,\rhog)$ which classifies isomorphism classes of $z$-polystable $(\theta,c,\rhog)$-twisted $\Gamma$-equivariant $(G,V)$-Higgs pairs. We denote the union of all these varieties corresponding to different $z\in i\zk^{\Gamma}$ by $\cM(X,G,\Gamma,\theta,c,V,\rhog)$. To our knowledge this has not been constructed in the literature, but we plan to address this in an upcoming paper. 

If $\Gamma$ is trivial we obtain the moduli space of $(G,V)$-Higgs pairs over $X$, which we denote $\cM(X,G,V)$. If $V=\lie g$, $\rho$ is the adjoint representation of $G$ on $\lie g$ and $\rhog=\theta$, we obtain the moduli space of \textbf{$(\theta,c)$-twisted $\Gamma$-equivariant $G$-Higgs bundles}, which we write $\cM(X,G,\Gamma,\theta,c)$. 
Given a character 
\begin{equation*}
    \mu:\Gamma\to\C^*;\,\gamma\mapsto\mug
\end{equation*}
of $\Gamma$, there is an alternative $(\theta,c)$-twisted action of $\Gamma$ on $\lie g$ given by $\rhomu:=\mug^{-1}\theta$. It is clear that the following categories are equivalent:
\begin{enumerate}
    \item The category of \textbf{$(\theta,c,\rhomu)$-twisted $\Gamma$-equivariant $G$-Higgs bundles}.
    \item The category of \textbf{$(\theta,c,\mu)$-twisted $\Gamma$-equivariant $G$-Higgs bundles}, whose objects are pairs $(E,\phi)$ consisting of a $G$-bundle $E$ equipped with a $(\theta,c)$-twisted $\Gamma$-equivariant action $\cdot$ and a Higgs field $\phi\in H^0(X,E(\lie g)\otimes K_X)$ fitting in the following diagramme:
    $$
    \begin{matrix}\label{higa}
    E(\mathfrak{g})\otimes K_X & \stackrel{(\cdot,\tg^{-1})\otimes\etag^{*}}{\longrightarrow} & 
    E(\mathfrak{g})\otimes K_X\\
    ~\Big\uparrow\varphi && ~\,\text{  }~\,\text{ }\Big\uparrow \mug\varphi\\
    X & \stackrel{\etag^{-1}}{\longrightarrow} & X
    \end{matrix},
    $$
    and whose morphisms are $\Gamma$-equivariant homomorphisms of $G$-bundles. By abuse of notation $(\cdot,\tg^{-1})$ denotes the group action of $\Gamma$ on $E(\lie g)$ induced by the corresponding $(\theta,c)$-twisted action on $E\times\lie g$.
\end{enumerate}
Via this equivalence of categories, we may construct a moduli space $\cM(X,G,\Gamma,\theta,c,\mu)$ of $(\theta,c,\mu)$-twisted $\Gamma$-equivariant Higgs bundles.

The notions of (poly,semi)stability for $(\theta,c,\rhog)$-twisted $\Gamma$-equivariant $(G,V)$-Higgs pairs also restrict to notions for $(\theta,c)$-twisted $\Gamma$-equivariant $G$-bundles.
There is a moduli space classifying isomorphism classes of $z$-polystable $(\theta,c)$-twisted $\Gamma$-equivariant $G$-bundles over $X$ for every $z\in i\zk^{\Gamma}$, which we denote by $M(X,G,\Gamma,\theta,c)$.

\section{Non-abelian Hodge correspondence for twisted equivariant Higgs bundles}\label{section-twisted-equivariant-higgs-pairs-and-hitchin-equations}

With notation as in Section \ref{section-twisted-equivariant-higgs-pairs} let $K\subset G$ be a $\Gamma$-invariant maximal compact subgroup
of $G$, where $\Gamma$ acts on $G$ via $\theta$. Let $(E,\phi)$ be a $(\theta,c,\rhog)$-twisted $\Gamma$-equivariant $(G,V)$-Higgs bundle on $X$. Let $h\in\Omega^0(X,E(G/K))^{\Gamma}$ be a $\Gamma$-invariant reduction of structure group of $E$ from $G$
to $K$, where $\Gamma$ acts equivariantly on $E(G/K)$ by Proposition \ref{prop-associated-bundle-equivariant}. Let $F_h$ be the curvature of  the corresponding Chern connection.
 Let $\sigma_h:\Omega^{1,0}(X,E(V))\to \Omega^{0,1}(X,E(V))$
be the $\C$-antilinear map defined by the reduction $h$ and the conjugation 
between $(1,0)$- and $(0,1)$-forms on $X$.
Consider the {\bf Hitchin equation} 
\begin{equation}\label{hitchin-equation-equivariant}
F_h+[\phi,\sigma_h(\phi)]=-2\pi iz\omega,
\end{equation}
where $z\in i\zk$ and $\omega$ is a Kähler form of $X$ with total volume 1.
One has 
the following (see \cite{hitchin1987,simpson,PBI}). 
\begin{theorem}\label{EH1-equivariant}
Let $(E,\phi)$ be a $(\theta,c,\rhog)$-twisted $\Gamma$-equivariant $(G,V)$-Higgs pair on $X$.
Then $(E,\phi)$ is $z$-polystable if and only if the $(G,V)$-Higgs pair $(E,\phi)$ 
admits a $\Gamma$-equivariant metric $h$ satisfying (\ref{hitchin-equation}). 
\end{theorem}

An equivariant base point is a $\Gamma$-equivariant map $x:\Gamma \to X$, 
where $\Gamma$ acts on itself by multiplication, and on $X$ via $\eta$.
Suppose that $(X,x)$ has a universal $\Gamma$-equivariant covering 
$(\widehat{X},\widehat{x})\to (X,x)$. By the universal property of such 
covering, the group of automorphisms of the equivariant covering $\widehat{X}$
over $X$ is uniquely determined by $X$, up to unique isomorphism. This group is
called  the {\bf equivariant fundamental group} of $X$ with respect to 
the action of $\Gamma$ and is denoted by $\pi_1(X,\Gamma, x)$ or simply
$\pi_1(X,\Gamma)$  (see \cite[Definition 3.1]{Huis}). 

Let $1\in \Gamma$ be the identity, set $x_1:=x(1)$ and let
 $\pi_1(X,x_1)$ be the fundamental group
of $X$ with base point $x_1$.
By \cite[Proposition 3.2]{Huis}, $\pi_1(X,\Gamma,x)$ fits into an exact 
sequence
 \begin{equation}\label{fundu} 
1\to \pi_1(X,x_1)\to \pi_1(X,\Gamma,x)\to \Gamma\to 1.
\end{equation}
 Note that if $\Gamma$ acts trivially on $X$ (i.e., $\eta$ is trivial) 
then $\pi_1(X,\Gamma,x)=\pi_1(X,x)$ and if $\Gamma$ acts
 freely on $X$ then $\pi_1(X,\Gamma,x)=\pi_1(X/\Gamma,\bar{x})$ where 
$\bar{x}$ is 
the  composite map of $x$ and the quotient $X\to X/\Gamma$ (see \cite[Proposition 3.4]{Huis}). 

Suppose that $\eta$ is non-trivial. Choose $x\in X$ so that it is not fixed by any  $\gamma\in \Gamma$ with $\gamma\neq 1$.
 By the description of the equivariant fundamental group in terms of equivariant loops (see \cite[Section 6]{Huis}) we can 
 identify the set with the set of all homotopy classes of maps $\sigma:[0,1]\to X$ such that $\sigma(0)=x$ and 
 $\sigma(1)\in \{\eta_{\gamma}(x)\mid \gamma\in \Gamma\}$. 
Under this identification the surjective map 
 $p: \pi_1(X,\Gamma,x)\to \Gamma$ can be identified with 
$p(\sigma)=\gamma$ if 
$\sigma(1)=\eta_{\gamma}(x)$.
 
Let $G\times_{\theta,c}\Gamma$ be the twisted product given by Definition \ref{def-twisted-product}.
A representation $\widehat{\rho}$ of $\pi_1(X,\Gamma,x)$ on $ G\times_{\theta,c}\Gamma$ 
is said to be  {\bf $(\theta,c)$-twisted $\Gamma$-equivariant} 
if it is an extension of a representation 
$\rho\,:\, \pi_1(X,\, x_1)\,\longrightarrow\, G$ 
fitting  in a commutative diagramme
of homomorphisms

\begin{equation}\label{compatible-rep}
\xymatrix{
0\ar[r] &\pi_1(X,\,x_1)\ar[r]^{}\ar[d]^{\rho} &\pi_1(X,\Gamma,x)\ar[r]^{}\ar[d]^{\widehat{\rho}} &\Gamma\ar[r]\ar[d]^{\Id}&0\\
0\ar[r] &G\ar[r]^{} &G\times_{\theta,c}\Gamma\ar[r]^{} &\Gamma\ar[r]&0.
}
\end{equation}

Denote by $\Hom(\pi_1(X,\Gamma,x), G\times_{\theta,c}\Gamma)$ the set of   
$(\theta,c)$-twisted $\Gamma$-equivariant representations 
$\widehat{\rho}\,:\,
\pi_1(X,\Gamma, x)\,\longrightarrow\, G\times_{\theta,c}\Gamma$. 

Let $\cM(X,G,\Gamma,\theta,c)$ be  the moduli space
of $(\theta,c)$-twisted $\Gamma$-equivariant $G$-Higgs bundles considered in Section 
\ref{section-twisted-equivariant-higgs-pairs}. 
Let 
$$\calR(X,G,\Gamma,\theta,c):=\Hom(\pi_1(X,\Gamma,x), G\times_{\theta,c}\Gamma)\sslash G$$ 
be the moduli space of  
 $G$-conjugacy classes of representations of  
$\pi_1(X,\Gamma,x)$ in $ G\times_{\theta,c}\Gamma$
whose restriction to $\pi_1(X,\,x_1)$ is reductive, where $G\times_{\theta,c}\Gamma$ is given by Definition \ref{def-twisted-product}.
We then have the following Theorem (see \cite{PBI}): 

\begin{theorem}[Non-abelian Hodge correspondence]\label{equivariant-nahc}
There is a homeomorphism between $\cM(X,G,\Gamma,\theta,c)$ 
and  $\calR(X,G,\Gamma,\theta,c)$.
\end{theorem}

\section{Local structure at isotropy points}\label{section-isotropy}

Let $x\in X$, and \[\Gamma_{x}:=\{\gamma\in \Gamma \mid \eta_{\gamma}(x)=x\}\] be the corresponding isotropy subgroup. 
Let $\PPP:=\{x\in X \mid \Gamma_{x}\neq \{1\}\}$.

It is well-known that, when $\Gamma$ acts on $X$ faithfuly and properly 
discontinuously,
$\PPP$ consists of a finite number of points $\{x_{1},\cdots,x_{r}\}$ 
and for each $x_{i}\in \PPP$ 
the isotropy subgroup $\Gamma_{x_i}\subset \Gamma$ is cyclic (see
\cite{miranda} for example). Each $x_i\in\PPP$ is called an \textbf{isotropy point} of $X$.

Let $(E,\cdot,\phi)$ be a $(\theta,c,\rhog)$-twisted $\Gamma$-equivariant $(G,V)$-Higgs pair with bundle projection $\pi:E\to X$. 
The underlying  $(\theta,c)$-twisted $\Gamma$-equivariant
structure on $E$ determines the following. For each $x\in \PPP$ and $e\in E$ such that $\pi(e)=x$, there is a map 
\[\sigma_{e}:\Gamma_{x}\to G\]
defined by \[e\cdot \gamma=e\tg^{-1}(\sigma_{e}(\gamma))^{-1}.\]

One has the following.
\begin{proposition}\label{local-data}

 \begin{enumerate}
  \item $\sigma_{e}(\gamma\gamma')=c(\gamma,\gamma')\sigma_{e}(\gamma)\theta_{\gamma}(\sigma_{e}(\gamma'))$, $\gamma,\gamma'\in \Gamma$.
  \item $\sigma_{eg}(\gamma)=g^{-1}\sigma_{e}(\gamma)\theta_{\gamma}(g)$ 
  for all $\gamma\in \Gamma$ and $g\in G$.
 \end{enumerate}
\end{proposition} 
 \begin{proof}
  For each $e\in E$ we have $e\cdot(\gamma\gamma')=e\theta_{\gamma\gamma'}^{-1}(\sigma_e(\gamma\gamma'))^{-1}$,
   and 
   \begin{equation*}
\begin{split}
(e\cdot\gamma)\cdot\gamma' & = (e\theta_{\gamma}^{-1}(\sigma_e(\gamma))^{-1})\cdot\gamma'\\
 & = (e\cdot\gamma')\theta_{\gamma'}^{-1}(\theta_{\gamma}^{-1}(\sigma_e(\gamma)))^{-1} \\
 & = e\theta_{\gamma'}^{-1}(\sigma_e(\gamma'))^{-1}\theta_{\gamma\gamma'}^{-1}(\sigma_e(\gamma))^{-1}\\
 & = e\theta_{\gamma\gamma'}^{-1}(\sigma_e(\gamma)\theta_{\gamma}(\sigma_e(\gamma')))^{-1}   \\
\end{split}
\end{equation*}
Since, by (\ref{eq-twisted-equivariant-axioms}), $(e\cdot(\gamma\gamma'))\theta_{\gamma\gamma'}^{-1}(c(\gamma,\gamma'))=(e\cdot\gamma)\cdot\gamma'$,
we have 
\[e\theta_{\gamma\gamma'}^{-1}(\sigma_e(\gamma\gamma'))^{-1}\theta_{\gamma\gamma'}^{-1}(c(\gamma,\gamma'))
=e\theta_{\gamma\gamma'}^{-1}(\sigma_e(\gamma)\theta_{\gamma}(\sigma_e(\gamma')))^{-1} \;\; \mbox{for every}\;\; e\in E.\]
Thus, 
$\sigma_e(\gamma\gamma')=c(\gamma,\gamma')\sigma_e(\gamma)\theta_{\gamma}(\sigma_e(\gamma'))$.

To check $(2)$ note that 
$$ 
eg\tg^{-1}(\sigma_{eg}(\gamma))^{-1}=e\tg^{-1}(\sigma_e(\gamma))^{-1}\theta_{\gamma}^{-1}(g),
$$
and so $\sigma_{eg}(\gamma)=g^{-1}\sigma_e(\gamma)\theta_{\gamma}(g)$.
  \end{proof}

For each isotropy point $x\in X$ , let us denote by ${Z^1}_{c_x}(\Gamma_x,G)$ the set of all $\beta:\Gamma_{x}\to G$ satisfying 
$\beta(\gamma\gamma')=c_x(\gamma,\gamma')\beta(\gamma)\theta_{\gamma}(\beta(\gamma'))$ 
where $c_x$ is the $2$-cocycle given by the restriction of $c$ to $\Gamma_x$. 
We call such a $\beta:\Gamma\to G$ a twisted $1$-cocycle for the action of $\Gamma$ on $G$ given by  $\theta$. 
Two twisted $1$-cocycles $\beta$ and $\beta'$ are related if there exists a $g\in G$ 
such that $\beta=g^{-1}\beta'\theta_{\gamma}(g)$. 
We denote the set of all twisted $1$-cocycles modulo the above defined relation by ${H^1}_{c_x}(\Gamma,G)$.

%%%%%%%%%%%%%%%%%%%%%%%%%%%%%%%%%%%%%%%%%%%%%%%%%%%%

%{\bf Question:} {\bf Is ${H^1}_{c_x}(\Gamma_x,G)$ finite?}
%%%%%%%%%%%%%%%%%%%%%%%%%%%%%%%%%%%%%%%%%%%%%%%%%%%

Let $(E,\phi)$ be a $(\theta,c,\rhog)$-twisted $\Gamma$-equivariant $(G,V)$-Higgs pair. 
From Proposition \ref{local-data}, we have 
that, for each  $x\in \PPP$, there is a unique equivalence class $\sigma_x$ of a twisted $1$-cocycle, which only depends on the $\Gamma$-equivariant isomorphism class of 
$E$.

We fix a $\sigma_{x_i}\in {H^1}_{c_{x_i}}(\Gamma_{x_i},G)$ for each  $x_{i}\in \PPP$. We say that a 
$(\theta,c,\rhog)$-twisted $\Gamma$-equivariant 
$(G,V)$-Higgs pair $(E,\phi)$ has 
{\bf local type} $\sigma_{x_i}$ at a fixed point $x_{i}\in X$ 
if the twisted $1$-cocycle induced by the 
$(\theta,c)$-twisted $\Gamma$-equivariant structure on $E$ 
is $\sigma_{x_i}$.

We define  $\cM(X,G,\Gamma,\theta,c,V,\rhog,\sigma)$ as 
the subvariety of the moduli space 
of twisted equivariant Higgs pairs $\cM(X,G,\Gamma,\theta,c,V,\rhog)$
with fixed local types $\sigma_{x_i}$, $i=1,\cdots, r$. When $V=\lie g$, $\rho$ is the adjoint representation and $\rhog=\theta$ (i.e., when we consider twisted equivariant Higgs bundles) we omit $V$ and $\rhog$. If we are considering the moduli space of $(\theta,c,\mu)$-twisted $\Gamma$-equivariant $G$-Higgs bundles we write $\cM(X,G,\Gamma,\theta,c,\mu,\sigma)$. Similarly, we define  $M(X,G,\Gamma,\theta,c,\sigma)$ as 
the subvariety of the moduli space 
of twisted equivariant $G$-bundles $M(X,G,\Gamma,\theta,c)$
with fixed local types $\sigma_{x_i}$, $i=1,\cdots, r$.

Let $c,c'\in Z^2_\theta(\Gamma,Z)$ be two cohomologous $2$-cocycles, that is, there exists a map $s:\Gamma\to Z$ such that 
\begin{equation}\label{com}
c'(\gamma',\gamma'')=c(\gamma',\gamma'')s_{\gamma'}\theta_{\gamma'}(s_{\gamma''})s_{\gamma'\gamma''}^{-1}, ~ \gamma',\gamma''\in \Gamma. \end{equation}
Let $\sigma_{x_i}\in {H^1}_{c_{x_i}}(\Gamma_{x_i},G)$. We have $1$-cocycles $\sigma^s_{x_i}:=s^{-1}\sigma_{x_i}\in {H^1}_{c'_{x_i}}(\Gamma_{x_i},G)$ induced by $\sigma_{x_i}$, where we are calling $s$ to its restriction to $\Gamma_{x_i}$ by abuse of notation. Indeed,
\begin{align*}
    s^{-1}_{\gamma\gamma'}\sigma_{x_i}(\gamma\gamma')
    &=s_{\gamma}\tg(s_{\gamma'})s^{-1}_{\gamma\gamma'}c(\gamma,\gamma')s_{\gamma}^{-1}\sigma_e(\gamma)\tg(s_{\gamma'})^{-1}\tg(\sigma_e(\gamma'))\\&
    =c'(\gamma,\gamma')\sigma^s_e(\gamma)\tg(\sigma^s_e(\gamma')).
\end{align*}

One  has the following.
\begin{proposition}\label{cohom}
 If $c$ and $c'$ are
cohomologous cocycles in $Z^2_\theta(\Gamma,Z)$ with $c'(\gamma,\gamma')=c(\gamma,\gamma')s_{\gamma\gamma'}^{-1}s_{\gamma}\tg(s_{\gamma'})$, $\rhog^{s}:=\rho(s)\rhog$ and $\sigma^s=s^{-1}\sigma$,
$$
\widetilde{\cM}(X,G,\Gamma,\theta,c,V,\rhog,\sigma)
\cong \widetilde{\cM}(X,G,\Gamma,\theta,c',V,\rhog^{s},\sigma^s).
$$
\begin{proof}
 Let $(E,\cdot,\phi)\in \widetilde{\cM}(X,G,\Gamma,\theta,c,V,\rhog,\sigma)$.
We can show that $(E,\phi)$ admits a $(\theta,c',\rhog^s)$-twisted $\Gamma$-equivariant structure 
with local 
types $\{\sigma^s_{x_i}\}$ namely define
 \[e*\gamma:=(es_{\gamma})\cdot\gamma\]
 as in the proof of Corollary \ref{cor-different-c}. By Corollary \ref{cor-different-c-higgs} this provides a $(\theta,c',\rhog^{s})$-twisted $\Gamma$-equivariant action on $(E,\phi)$. Moreover, we have
\begin{align*}
    e*\gamma
    =(e\cdot\gamma)\tg^{-1}(s_{\gamma})
    =e\tg^{-1}(\sg^{-1}\sigma_{e}(\gamma))^{-1}
    =e\tg^{-1}(\sigma^s_{e}(\gamma))^{-1},
\end{align*}
hence the local type of $(E,*,\phi)$ is $\{\sigma^s_{x_i}\}$.
Similarly we can construct an inverse.
\end{proof}
\end{proposition}

\section{Non-connected groups and twisted equivariant Higgs pairs}\label{section-twisted-and-non-connected-higgs}

 Let $G$ be a (non-connected) reductive complex Lie group with connected component $G_0$ and group of connected components $\Gamma:=G/G_0$. Let $Z$ be the centre of $G_0$. We have a short exact sequence (\ref{eq-general-extension}) with characteristic homomorphism $a:\Gamma\to\Out(G_0)$. Let $\theta:\Gamma\to\Aut(G_0)$ be a lift of $a$ and $c\in Z^2_a(\Gamma,Z)$ the 2-cocycle given by Proposition \ref{prop-extensions-isomorphic-twisted-group}, so that $G\cong G_0\times_{\theta,c}\Gamma$ as extensions of $G_0$. Recall that this isomorphism is given by a section $\Gamma\to G$. Let $\rhog:\Gamma\to\GL(V)$ be the pullback of $\rho$ by this section. The right action of $G$ on $V$ given by 
 \begin{equation*}
     V\times G\to V;\,(v,g)\mapsto \rho(g)^{-1}v
 \end{equation*}
yields by Proposition \ref{prop-bijection-G,Gamma-action} a right $G_0$-action given by the restriction to ${G_0}$ and a $(\theta,c)$-twisted $\Gamma$-action
 \begin{equation*}
     V\times \Gamma\to V;\,(v,\gamma)\mapsto v\cdot\rhog(\gamma):= \rhog(\gamma)^{-1}v=\rho(1,\gamma)^{-1}v.
 \end{equation*}
 
 Take a connected $\Gamma$-bundle $Y\to X$ with canonical bundle $K_Y$.

\begin{definition}\label{def-categories-higgs}
  We denote by $\hat{\mathcal{C}}_1$ the category whose
  \begin{itemize}
  \item \textbf{objects} are $(\theta,c,\rhog)$-twisted $\Gamma$-equivariant
    $(G_0,V)$-Higgs pairs $(E,\phi)$ over $Y$ such that the twisted
    $\Gamma$-action descends to the action of $\Gamma$ as covering
    transformations of the fixed $\Gamma$-covering $E/G_0\cong Y\to X$,
    and whose
  \item \textbf{morphisms} are $Z(\Gamma)$-isomorphisms of twisted equivariant Higgs pairs, i.e. holomorphic
    maps $f\colon E\to E'$ both $G$ and $\Gamma$-equivariant such that the diagramme
    \begin{displaymath}
      \begin{CD}
      E @>{f}>> E'\\
      @VVV @VVV \\
      Y @>{\bar{f}}>> Y
      \end{CD}
    \end{displaymath}
    commutes and the induced map $\bar{f}\colon Y\to Y$  is
    a covering transformation which belongs to $Z(\Gamma)$.
  \end{itemize}

    We denote by $\hat{\mathcal{C}}_2$ the category whose
  \begin{itemize}
  \item \textbf{objects} are $(G,V)$-Higgs pairs on $X$ such
    that there is an isomorphism of $\Gamma$-bundles $E/G_0\cong Y$
    and whose
  \item \textbf{morphisms} are morphisms of principal
    $G$-bundles.
  \end{itemize}
\end{definition}

\begin{proposition}[Proposition 4.5 in \cite{GGM}]\label{prop-twisted-equivariant-higgs pairs-one-to-one}
The categories $\hat{\mathcal{C}}_1$ and $\hat{\mathcal{C}}_2$ are equivalent. 
\end{proposition}
\begin{proof}
    We have forgetful functors $\hcc1\to\cc1$ and $\hcc2\to\cc2$ to the category of $(\theta,c)$-twisted $\Gamma$-equivariant
    $G_0$-bundles over $Y$ and the category of $G$-bundles $E$ over $X$ such that $E/G_0\cong Y$ respectively, see Definition \ref{def-categories}. Proposition \ref{prop-twisted-equivariant-bundles-one-to-one} provides an equivalence of categories between $\cc1$ and $\cc2$. We show that it lifts to an equivalence of categories between $\hcc1$ and $\hcc2$: given a $G$-bundle $E$ over $X$, the corresponding twisted equivariant $G_0$-bundle $F\to Y$ is given by the factorization $E\to E/G_0\cong Y\to X$. Let $p:E\to Y$ be the middle map. Now take a Higgs field $\phi\in H^0(X,E(V)\otimes K_X)$. We define a $\Gamma$-invariant Higgs field $\psi\in H^0(Y,F(V)\otimes K_Y)$ as follows: 
    for each $y\in Y$ let $x:=p_Y(y)$, choose $e\in p^{-1}(y)$, let $\phi_x=(e,v)\otimes k\in E(V)\otimes K_X\vert_x$ and set $\psi:=(e,v)\otimes p_Y^*k\in E(V)\otimes p_Y^*K_X\vert_y\cong F(V)\otimes K_Y\vert_y$. This is independent of the choice of $e$ since, given $g\in G_0$, the Higgs field $\phi$ at $x$ is also represented by $(e\cdot g,\rho(g)^{-1}v)\otimes k$, and $(e\cdot g,\rho(g)^{-1}v)=(e,v)$ as elements of $F(V)$. Moreover, we have 
    \begin{align*}
        (e\cdot\gamma,v\cdot\rhog(\gamma))\otimes k=(e(1,\gamma),\rho(1,\gamma)^{-1}v)\otimes k=(e,v)\otimes k
    \end{align*}
    for every $\gamma\in \Gamma$ as elements of $E(V)\otimes K_X$, so that $\psi$ is $\Gamma$-invariant. Conversely, a $\Gamma$-invariant section $\psi\in H^0(Y,F(V)\otimes K_Y)$ provides a section $\phi\in H^0(X,E(V)\otimes K_X)$ defined by $\phi_x:=(e,v)\otimes K$, where $(e,v)\otimes p_Y(k)\in F(V)\otimes K_Y\vert_y$ is a local expression of $\psi$ for some $y\in p_Y^{-1}(x)$ and $p_Y(k)$ denotes the image of $k$ by the natural homomorphism $K_Y\cong p_Y^*K_X\to K_X$. The $\Gamma$-invariance of $\psi$ implies that this is independent of the choice of representative.

    The definition of the functor at the level of morphisms is the same as in Section \ref{section-twisted-and-non-connected-principal}, which is well defined because a morphism of $(G,V)$-Higgs pairs over $X$ preserves the Higgs fields and so does the corresponding map between twisted equivariant $(G_0,V)$-Higgs pairs over $Y$.
\end{proof}

Proposition \ref{prop-twisted-equivariant-higgs pairs-one-to-one} and Definition \ref{definition-stable-twisted-equivariant} induce notions of (semi,poly)stability for Higgs pairs with non-connected structure group, which we introduce next. Choose the lift $\theta$ of the characteristic homomorphism of \ref{eq-general-extension} so that $\theta(\Gamma)$ preserves a maximal compact subgroup $K$ of $G_0$. In particular we have an action of $\Gamma$ on $\lie k$, the Lie algebra of $K$.

\begin{definition}\label{def-stability-higgs-pair-non-connected}
Let $z\in i\zk^{\Gamma}$. A $(G,V)$-Higgs pair $(E,\phi)$ over $X$ is:

\begin{itemize}
    \item \textbf{$z$-semistable} if $\deg E(\tau,s)\ge \pair{z}s$ for any $s\in i\lie k^{\Gamma}$ and any reduction of structure group $\tau\in H^0(X,E(G/P_s))$ such that $\phi\in H^0(X,E(V)_{\tau,s}\otimes K_X)$.
    \item \textbf{$z$-stable} if $\deg E(\tau,s)> \pair{z}s$ for any $s\in i\lie k^{\Gamma}$ and any reduction of structure group $\tau\in H^0(X,E(G/P_s))$ such that $\phi\in H^0(X,E(V)_{\tau,s}\otimes K_X)$.
    \item \textbf{$z$-polystable} if it is $z$-semistable and, if $\deg E(\tau,s)=\pair{z}s$ for some $s\in i\lie k^{\Gamma}$ and a reduction $\tau\in H^0(X,E(G/P_s))$ such that $\phi\in H^0(X,E(V)_{\tau,s}\otimes K_X)$, there is a further holomorphic reduction of structure group $\tau'\in H^0(X,E_{\tau}(P_s/L_s))$ with $\phi\in H^0(X,E(V)^0_{\tau',s}\otimes K_X)$.
\end{itemize}
\end{definition}

 A moduli space $\mdl(X,G,V)$ classifying isomorphism classes of $z$-polystable $(G,V)$-Higgs pairs over $X$ for every $z\in i\zk$ may be constructed using GIT, see \cite[Section 2.6]{schmitt:2008}. Theorem \ref{EH1} also holds when $G$ is non-connected, see \cite{PBI}.

  \begin{definition}
     Let $H\le G$ be a reductive subgroup. Given a $(H,V)$-Higgs pair $(E,\phi)$ over $X$, its \textbf{extension of structure group} to $G$ is the $(G,V)$-Higgs pair $(E_G,\phi_G)$ over $X$ given by:
     \begin{itemize}
         \item The extension of structure group $E_G$ of $E$ to $G$.
         \item If $\phi$ be locally equal to $(e,v)\otimes k\in E(V)\otimes K_X$, then $\phi_G$ is also locally equal to $(e,v)\otimes k\in E_G(V)\otimes K_X$.
     \end{itemize}
     Conversely, $(E,\phi)$ is a \textbf{reduction of structure group} of $(E_G,\phi_G)$ to $H$. 
 \end{definition}

 \begin{lemma}\label{lemma-exists-higgs-field-reduction}
     Let $H\le G$ be a subgroup. Given a $(G,V)$-Higgs pair $(E_G,\phi_G)$ and a reduction of structure group $E$ of $E_G$ to $H$, there exists a Higgs field $\phi\in H^0(X,E(V)\otimes K_X)$ making $(E,\phi)$ a reduction of structure group of $(E_G,\phi_G)$.
 \end{lemma}
 \begin{proof}
     If $\phi_G$ is locally of the form $(e,v)\otimes k$ for some $e\in E_G$, $v\in V$ and $k\in K_X$, by transitivity of the action of $G$ on $E_G$ we may always find $g\in G$ such that $eg\in E$. Thus $(eg,\rho(g)^{-1}v)\otimes k\in E(V)\otimes K_X$ is also a local representation of $\phi_G$, so we may define $\phi$ to have this same local form.
 \end{proof}

  Now let $Y\to X$ be a (not necessarily connected) $\Gamma$-bundle with monodromy group $\Gamma'\le\Gamma$ and connected component $p_{Y'}:Y'\to X$ (see Section \ref{section-monodromy}). Let $G'\cong G\times_{\theta,c}\Gamma'$ be the preimage of $\Gamma'$ under the natural surjection $G\to\Gamma$.

 \begin{proposition}\label{prop-polystability-extension-of-structure-group-non-connected}
     Fix $z\in i\zk^{\Gamma}$. Let $(E,\phi)$ be a $(G',V)$-Higgs pair over $X$ with extension of structure group $(E_G,\phi_G)$ to $G$. Then  $(E,\phi)$ is $z$-polystable if and only if $(E_G,\phi_G)$ is $z$-polystable. In particular there exists an extension of structure group morphism $\mdl(X,G')\to\mdl(X,G)$.
 \end{proposition}
 \begin{proof}
     Assume that $(E,\phi)$ is $z$-polystable. First we prove that $(E_G,\phi_G)$ is $z$-semistable: consider a reduction of structure group $\tau_G\in H^0(X,E_G(G/P_s))$ such that $\phi_G$ lies in $H^0(X,E_G(V)_{\tau_G,s}\otimes K_X)$, where $s\in\lie k^{\Gamma}$. Since $s$ is $\Gamma$-invariant the parabolic subgroup $P_s\le G$ intersects every connected component in $G$, hence the fibre of the total space of the reduction $(E_G)_{\tau_G}$ over $x\in X$ intersects every connected component of $E_G\vert_x$. Thus, since $E\vert_x$ is a union of connected components of $E_G\vert_x$, its intersection with $P_s$ must be isomorphic to $G'\cap P_s=P'_s$, where $P'_s:=\{g\in G'\suhthat e^{ts}ge^{-ts}\,\text{remains bounded as}\;t\to\infty\}$ is the parabolic subgroup of $G'$ defined by (\ref{eq-def-Ps}). Thus the intersection of $(E_G)_{\tau}$ with $E$ is a reduction of structure group of $E$ to $P'_s\le P_s$, which we call $\tau\in H^0(X,E(G'/P'_s))$. If the Higgs field $\phi_G$ can be locally expressed in the form $(e,v)\otimes k$, where $e\in E_{\tau}$, $v\in V_s$ and $k\in K_X$, and so $\phi$, which has the same form, is in $H^0(X,(E_{\tau}\times_{\tau,s}V_s)\otimes K_X)=H^0(X,E(V)_{\tau,s}\otimes K_X)$. We also have
     \begin{equation*}
         \deg E_G(\tau_G,s)=\deg E(\tau,s)\ge \pair{z}s,
     \end{equation*}
     where the second equation follows from $z$-semistability of $E$ and the first equation follows from definition (\ref{eq-def-deg}) and the fact that the connection of $E_G$ is induced by a connection of $E$ (so that the corresponding curvatures are equal). 

     Now assume that $\deg E_G(\tau_G,s)=\deg E(\tau,s)= \pair{z}s$. By $z$-polystability of $(E,\phi)$ there is a further holomorphic reduction of structure group $\tau'\in H^0(X,E_{\tau}(P'_s/L'_s))$ with $\phi\in H^0(X,E(V)^0_{\tau',s})$, where $L'_s$ is the Levi subgroup of $P'_s$. Thus the extension of structure group of $E_{\tau'}$ to $L_s$ yields a reduction $\tau'_G\in H^0(X,(E_{G})_{\tau_G}(P_s/L_s))$ such that $\phi_G\in H^0(X,E_G(V)^0_{\tau'_G,s})$, as required.

     The converse can be shown using similar arguments: assume that the $(G,V)$-Higgs pair $(E_G,\phi_G)$ is $z$-polystable. For $z$-semistability consider a reduction of structure group $\tau\in H^0(X,E(G'/P'_s))$ such that $\phi\in H^0(X,E(V)_{\tau,s}\otimes K_X)$, where $s\in\lie k^{\Gamma}$. Then, extending structure group, we get a reduction $\tau_G\in H^0(X,E_G(G/P_s))$ and so we conclude that 
     \begin{equation*}
         \deg E(\tau,s)=\deg E_G(\tau_G,s)\ge \pair{z}s.
     \end{equation*}
     If equality holds, we may find a further reduction $\tau'_G\in H^0(X,(E_G)_{\tau_G}(P_s/L_s))$ with $\phi_G\in H^0(X,E_G(V)^0_{\tau_G',s})$, and the intersection of $(E_G)_{\tau_G}$ with $E$ yields a reduction $\tau'\in H^0(X,E_{\tau}(P'_s/L'_s))$ with $\phi\in H^0(X,E(V)^0_{\tau',s})$.
 \end{proof}

\begin{proposition}\label{prop-action-centralizer-higgs}
    Let $Z_{\Gamma}(\Gamma')$ be the centralizer of $\Gamma'$ in $\Gamma$. There is a natural action of $Z_{\Gamma}(\Gamma')$ on $\mdl(Y',G_0,\Gamma',\theta,c,V,\rhog)$ on the left given as follows: take a $(\theta,c,\rhog)$-twisted $\Gamma'$-equivariant $(G_0,V)$-Higgs pair $(E,\cdot,\phi)$ over $Y'$ and an element $z\in Z_{\Gamma}(\Gamma')$. Then $z$ sends $E$ to $(\theta_z(E),*,\rhog(z)\phi)$, where $*$ is defined as in Proposition \ref{prop-action-centralizer}, $\theta_z(E)$ is the extension of structure group by $\theta_z$ and $\rhog(z)\phi$ is defined using the action $\rhog$ on $V$.
    This induces an action of $\cent$ on the moduli space $\mdl(X,G',V)_{Y'}$ of $(G',V)$-Higgs pairs $(E',\phi')$ such that $E'/G_0\cong Y'$ via Proposition \ref{prop-twisted-equivariant-higgs pairs-one-to-one}, where $z\in \cent$ sends $(E',\phi')$ to $(\Int_{(1,z)}(E'),\rho(1,z)\phi')$.
\end{proposition}
\begin{proof}
    The fact that this is a left group action follows directly from its definition. It is left to show that $\phi^z:=\rhog(z)\phi$ is well defined and $\Gamma$-invariant. Assume that $\phi$ is locally equal to $(e,v)\otimes k\in H^0(Y',E(V)\otimes K_{Y'})$. Then $\phi^z$ is locally equal to $(e,\rhog(z)v)\otimes k\in (\theta_z(E)\times_{\rho}V)\otimes K_X$. Another representative is of the form $(e\cdot g,\rho(g)^{-1}v)\otimes k$ for some $g\in G_0$, which yields
    \begin{align*}
        (e\cdot g,\rhog(z)\rho(g)^{-1}v)\otimes k&
        =(e*\theta_z(g),\rho(\theta_z(g))^{-1}\rhog(z)v)\otimes k
        =(e,\rhog(z)v)\otimes k
    \end{align*}
    where $*$ also denotes the $G_0$-action on $\theta_z(E)$, so that $\rhog(z)\phi\in (\theta_z(E)\times_{\rho}V)\otimes K_X$ is well defined.
    
    % Now let $*$ denote the $\Gamma$-equivariant action on $(\theta_z(E)\times_{\rho}V)\otimes K_X$ determined by $*$ and $\rhog$. 
    In what follows we denote by $\cdot$ the $G'\cong G_0\times_{\theta,c}\Gamma'$-action on $E$ given by Proposition \ref{prop-bijection-G,Gamma-action}, or equivalently the $G'$-action making the total space of $E$ a $G'$-bundle over $X$ as in Proposition \ref{prop-twisted-equivariant-bundles-one-to-one}. Let $W$ be an open subset of $X$ where $K_X$ is trivial, and let $U:=p_{Y'}^{-1}(W)$. The open set $U$ is $\Gamma$-invariant and trivializes $K_{Y'}\cong p_{Y'}^*K_X$. Hence the local sections $\phi\vert_U$ and $\phi^z\vert_{U}$ may be regarded as $G_0$-equivariant maps $\phi_U:E_U\to V$ and $\phi^z_U:\theta_z(E)_U\to V$ respectively, and moreover by Proposition \ref{prop:EM-Gamma-equivariant-sections} $\phi_U$ is $\Gamma$-equivariant. They are related by
    $
        \phi^z(e)=\rhog(z)\phi(e)
    $
    for each $e\in E_U$ or, using the $\Gamma$-equivariance of $\phi_U$, by $\phi^z(e\cdot z)=\phi(e)$. We have
    \begin{align*}
        \phi^z((e\cdot z)*\gamma)&
        =\phi^z(e\cdot(1,z)\cdot(1,z)^{-1}\cdot(1,\gamma)\cdot(1,z))\\&
        =\phi^z((e\cdot\gamma)\cdot z)\\&
        =\phi(e\cdot\gamma)\\&
        =
        \phi(e)\cdot\rhog(\gamma)\\&
        =\phi^z(e\cdot z)\cdot\rhog(\gamma)
    \end{align*}
    for each $e\in E_U$ and $\gamma\in\Gamma'$. This implies that $\phi^z_U$ is $\Gamma$-equivariant as a map and so, by Proposition \ref{prop:EM-Gamma-equivariant-sections}, the corresponding section of $(\theta_z(E)\times_{\rho}V)\otimes K_{Y'}\vert_U$ is $\Gamma$-invariant. Since $U$ is arbitrary, the global section $\phi^z\in H^0(Y',(\theta_z(E)\times_{\rho}V)\otimes K_{Y'})$ is $\Gamma$-invariant as required.
\end{proof}

\begin{theorem}\label{th-prym-narasimhan-ramanan-higgs}
Let $Y\in H^1(X,\Gamma)$ with monodromy group $\Gamma'$ and corresponding connected component $Y'$, and denote by $\mdl_Y(X,G,V)$ the moduli space of $(G,V)$-Higgs pairs $E$ over $X$ such that $E/G_0\cong Y$. We have a bijection
\begin{equation}\label{eq-narasimhan-ramanan-higgs}
    \mdl(Y',{G_0},\Gamma',\theta,c,V,\rhog)/Z_{\Gamma}(\Gamma')\xrightarrow{\sim} \mdl_{Y}(X,G,V),
\end{equation}
where $Z_{\Gamma}(\Gamma')$ is the centralizer of $\Gamma'$ in $\Gamma$, which acts on $\mdl(Y',{G_0},\Gamma',\theta,c,V,\rhog)$ as in Proposition \ref{prop-action-centralizer-higgs}. The bijection is given by Proposition \ref{prop-twisted-equivariant-higgs pairs-one-to-one} and extension of structure group from $G'$ to $G$.
\end{theorem}

\begin{proof}
    Given a $(G,V)$-Higgs pair $(E_G,\phi_G)$ over $X$ such that $E_G/G_0\cong Y$, take a connected component $E\subset E_G$ such that $E/G_0\cong Y'$. By Lemma \ref{lemma-exists-higgs-field-reduction} there exists $\phi\in H^0(X,E(V)\otimes K_X)$ such that $(E,\phi)$ is a reduction of structure group of $(E_G,\phi_G)$ to $G'$, and so surjectivity follows by Proposition \ref{prop-twisted-equivariant-higgs pairs-one-to-one}. 

To show that the morphism is well-defined, consider a $(G',V)$-Higgs pair $(E,\phi)$ with extension of structure group $(E_G,\phi_G)$ to $G$. Let $z\in\cent$, $s:=(1,z)\in G$ and $\beta:=\Int_{s^{-1}}\in \Aut(G')$. In the proof of Theorem \ref{th-prym-narasimhan-ramanan-principal} we saw that $Es\subset E_G$ is a reduction of structure group to $G'$ isomorphic to $\beta(E)$. If $\phi$ is locally equal to $(e,v)\otimes k\in E(V)\otimes K_X$ then the Higgs field of $Es$ given by Lemma \ref{lemma-exists-higgs-field-reduction} is locally equal to $(es,\rho(s)^{-1}v)\otimes k=(es,\rhog(z)^{-1}v)\otimes k$. Recall that the isomorphism $Es\cong \beta(E)$ is multiplication by $s^{-1}$, so that the induced Higgs field is locally equal to $(ess^{-1},\rhog(z)^{-1}v)\otimes k=(e,\rhog(z)^{-1}v)\otimes k$, which is a local representation of $\rhog(z)^{-1}\phi$. Hence the extension of structure group of $(\Int_{s}^{-1}(E),\rhog(z)^{-1}\phi)$ is also $(E_G,\phi_G)$, as required.

It is left to show injectivity. Let $(F,\psi)$ and $(F',\psi')$ be ($\theta,c,\rhog$)-twisted $\Gamma'$-equivariant $(G_0,V)$-Higgs pairs over $Y'$ and let $(E,\phi)$ and $(E',\phi')$ be the corresponding $(G',V)$-Higgs pairs over $X$ given by Proposition \ref{prop-twisted-equivariant-higgs pairs-one-to-one}. Assume that they have the same extension of structure group $(E_G,\phi_G)$ to $G$. By the proof of Theorem \ref{th-prym-narasimhan-ramanan-principal} $E'=Es\subset E_G$, where $s=(1,z)$ for some $z\in\cent$, thus $E'\cong\Int_{s}^{-1}(E)$. By the previous paragraph the Higgs field on $\Int_{s}^{-1}(E)$ induced by the Higgs field on $Es$ given by Lemma \ref{lemma-exists-higgs-field-reduction} is equal to $\rhog(z)\phi$. Thus $(E',\phi')\cong (\Int_{s}^{-1}(E),\rhog(z^{-1})\phi)$, as required.
\end{proof}

\begin{corollary}\label{cor-prym-narasimhan-ramanan-principal}
Let $Y\in H^1(X,\Gamma)$ with monodromy group $\Gamma'$ and corresponding connected component $Y'$, and denote by $M_Y(X,G)$ the moduli space of $G$-bundles $E$ over $X$ such that $E/G_0\cong Y$. We have a bijection
\begin{equation}\label{eq-narasimhan-ramanan-principal}
    M(Y',{G_0},\Gamma',\theta,c)/Z_{\Gamma}(\Gamma')\xrightarrow{\sim} M_{Y}(X,G),
\end{equation}
where $Z_{\Gamma}(\Gamma')$ is the centralizer of $\Gamma'$ in $\Gamma$, which acts on $M(Y',{G_0},\Gamma',\theta,c)$ as in Proposition \ref{prop-action-centralizer-higgs}. The bijection is given by Proposition \ref{prop-twisted-equivariant-bundles-one-to-one} and extension of structure group from $G'$ to $G$.
\end{corollary}

\newpage

\chapter{The Prym--Narasimhan--Ramanan construction}\label{chapter-prym-narasimhan-ramanan}

Let $X$ be a compact Riemann surface with canonical bundle $K_X$ and $G$ a connected reductive complex Lie group with centre $Z$. We consider the problem of finding the fixed points of the action of a finite subgroup $\Gamma$ of $H^1(X,Z)\rtimes\Out(G)\times\C^*$ on the moduli space $\mdl(X,G)$ of $G$-Higgs bundles over $X$ (see Section \ref{section-action}). Projections on the second, third and first factors provide homomorphisms $a:\Gamma\to \Out(G)$, $\mu:\Gamma\to\C^*$ and a 1-cocycle $\alpha\in Z^1_a(\Gamma,H^1(X,Z))$ respectively (see Definition \ref{def-1-cocycle}). Note the absence of $\eta:\Gamma\to\Aut(X)$, which is trivial in this Chapter.

An answer to this problem is already given in \cite{PR} when $\Gamma$ is cyclic, where Theorem 6.10 is an special case of Theorem \ref{th-fixed-points-oscar-ramanan-higgs}. However, even when $\Gamma$ is cyclic the Prym--Narasimhan--Ramanan construction given by Theorem \ref{th-prym-narasimhan-ramanan-oscar-ramanan-higgs} is new. 

\section{Automorphisms of a reductive complex Lie group}\label{section-Gtheta}
Take a homomorphism 
$$\theta:\Gamma\to\autg;\,\gamma\mapsto\tg$$
lifting $a$.
We study the fixed points the action of $\Gamma$ on $G$. We assume that $\zt$ is finite. This is true, for example, if $Z$ is finite (i.e. $G$ is semisimple), since in this case the set of maps $\Fun(\Gamma,H^1(X,Z))$ from $\Gamma$ to $H^1(X,Z)$ is finite, or if $a$ is trivial, since then 
$Z^1_a(\Gamma,Z)=\Hom(\Gamma,Z).$

We define
$$\gt:=\{g\in G\suhthat\tg(g)=g\forevery\gamma\in\Gamma\}\le G$$
and
$$\gs:=\{g\in G\suhthat\tg(g)=z(\gamma,g)g,\,z(\gamma,g)\in Z\forevery\gamma\in\Gamma\}\le G.$$
The group $\Gamma$ acts (trivially) on $\gt$ and, since it acts on $Z$, it also acts on $\gs$ (note that $Z$ is a subgroup of $\gs$). The group $\gt$ is a normal subgroup of $\gs$: for every $\gamma\in\Gamma$, $g\in\gt$ and $s\in\gs$, we have
$$\tg(sgs^{-1})=z(\gamma,g)sgs^{-1}z(\gamma,g)^{-1}=sgs^{-1}.$$

To understand better how $\gt$ lies inside $\gs$ we use 1-cocycles ---see Definition \ref{def-1-cocycle}. Note that the set of 1-cocycles $\zt$ has a group structure induced by $Z$. There is an exact sequence of groups
\begin{equation}\label{eq-exact-seq-groups}
    1\to \gt\to\gs\xrightarrow{\cct} \zt,
\end{equation}
where the last homomorphism sends $g\in\gs$ to the map
$$\Gamma\to Z;\,\gamma\mapsto g^{-1}\tg(g)=\tg(g)g^{-1}=z(\gamma,g)\in Z.$$
To see why $\cct$ is well defined note that, if $\gamma$ and $\gamma'$ are elements of $\Gamma$ and $g\in \gs$, we have:
\begin{align*}
    \cct(\gamma\gamma')g=\theta_{\gamma\gamma'}(g)=\tg\left(\theta_{\gamma'}(g)\right)=
\tg(\cct(\gamma')g)=\tg(\cct(\gamma'))\tg(g)&=\tg(\cct(\gamma'))\cct(\gamma)g\\&=\cct(\gamma)\tg(\cct(\gamma'))g,
\end{align*}
where $\cct$ is evaluated at $g$. The exactness of (\ref{eq-exact-seq-groups}) implies that $\cct$ factors through the quotient 
\begin{equation*}
    \gamt:=\gs/\gt
\end{equation*}
 via a group embedding
$\gamt\hookrightarrow\zt.$
In particular, the finiteness of $\zt$ implies that $\gs$ is a finite extension of $\gt$, and the reductiveness of $\gt$ (see Proposition 3.6 in chapter 3 of \cite{onishchik3}) is inherited by $\gs$.

On the other hand, if $\fun{A}{B}$ denotes the set of maps from a set $A$ to a set $B$, we have a natural group homomorphism
\begin{equation}\label{eq-iso-cohomology}
    H^1(X,\fun{\Gamma}{Z})\to\fun{\Gamma}{H^1(X,Z)},
\end{equation}
where the group structures on both sides are induced by $Z$ ---this is an isomorphism if $Z$ is finite (i.e. $G$ is semisimple).
To define it we use the homomorphism 
\begin{equation*}
    \homm{\pi_1(X)}{A}\to H^1(X,A);\,\rho\mapsto \tilde X\times_{\rho}(A),
\end{equation*}
that exists for any abelian group $A$, where $\tilde X$ is the universal cover of $X$. This has an inverse
\begin{equation}\label{eq-iso-rep-fundamental-cohomology}
    H^1(X,A)\xrightarrow{\sim}\homm{\pi_1(X)}{A}
\end{equation}
if $A$ is also finite. Then (\ref{eq-iso-cohomology}) is the composition
\begin{align*}
    H^1(X,\fun{\Gamma}{Z})\xrightarrow{\sim} \homm{\pi_1(X)}{\fun{\Gamma}{Z}}&\xrightarrow{\sim}\fun{\Gamma}{\homm{\pi_1(X)}{Z}}\\&
    \to\fun{\Gamma}{H^1(X,Z)}.
\end{align*}
One can see that (\ref{eq-iso-cohomology}) restricts to a homomorphism 
\begin{equation}\label{eq-cohomology-Z-iso-Z-cohomology}
H^1(X,\zt)\xrightarrow{\sim} Z^1_{a}(\Gamma,H^1(X,Z)),    
\end{equation}
which is an isomorphism if $Z$ is finite.

Recall that, given a complex Lie group $H$, the set of equivalence classes of \v{C}ech 1-cocycles on $X$ with values on the sheaf $\underline{H}$ of holomorphic functions to $H$ is in natural bijection with the set of isomorphism classes of holomorphic $H$-bundles. We call it $H^1(X,\underline{H})$. It has a distinguished element (representing the trivial bundle), which gives $H^1(X,\underline{H})$ the structure of a pointed set. Using this notation, the homomorphism $\cct$ induces a morphism of pointed sets
$$H^1(X,\underline{\gs})\to H^1(X,\zt)$$
by extension of structure group.
Composing with (\ref{eq-cohomology-Z-iso-Z-cohomology}), we get a morphism of pointed sets
$$\ctt:H^1(X,\underline{\gs})\to Z^1_{a}(\Gamma,H^1(X,Z)).$$
Note that the factorization of $\ctt$ through the cohomology of the quotient
$$H^1(X,\gamt)\to Z^1_{a}(\Gamma,H^1(X,Z))$$
is injective, since it is a composition of an isomorphism and a homomorphism induced by an embedding in an abelian group. The last one is injective because of (\ref{eq-iso-rep-fundamental-cohomology}).

The group $\gt$ is connected when $G$ is simply connected and the image of the homomorphism $\theta$ is cyclic (see chapter 8 in \cite{steinberg-endomorphisms-of-linear-algebraic-groups}). However, it is not connected in general even when $G$ is simply connected. For this reason, we will need to "refine" (\ref{eq-exact-seq-groups}) using the connected component $\gt_0$ of $\gt$. Due to the fact that $\gs$ is an extension of $\gt$ by a finite group, $\gt_0$ is also the connected component of $\gs$. Thus we have an extension
\begin{equation}\label{eq-extension-connected-component}
    1\to\gt_0\to\gs\xrightarrow{\pt}\gamtt\to1,
\end{equation}
where $\gamtt$ is a finite group because $\gs$ is reductive. Of course there is a natural surjective homomorphism $\gamtt\to\gamt,$
which induces a morphism $H^1(X,\gamtt)\to H^1(X,\gamt)$. We call $\qt$ to the composition 
\begin{equation}\label{eq-def-qt}
    \qt:H^1(X,\gamtt)\to H^1(X,\gamt)\hookrightarrow H^1(X,\zt)\to Z^1_{a}(\Gamma,H^1(X,Z)).
\end{equation}

\section{Restrictions of the adjoint representation}\label{section-adjoint-representation} 
Let $\lieg$ be the Lie algebra of $G$. We consider the restriction of the adjoint representation
$$\Ad:G\to\GL(\lieg)$$
to the subgroups $\gt$ and $\gs$, defined in Section \ref{section-Gtheta}. Given a character
$$\mu:\Gamma\to \C^*;\,\gamma\mapsto \mu_{\gamma},$$
we may consider the $\mu$-weight subspace of $\lie g$, given by
$$\liegm:=\{v\in\lieg\suhthat \tg(v)=\mu_{\gamma} v\forevery\gamma\in\Gamma\}.$$
One can see that it is preserved by the adjoint action of $\gs$: for every $g\in\gs$, $\gamma\in\Gamma$ and $v\in\lieg^{\theta}_{\mu}$, we have
$$\tg\Ad_g(v)=\Ad_{\tg(g)}(\tg(v))=\Ad_{z(\gamma,g) g}(\mug v)=\mug\Ad_{g}(v).$$
Abusing notation, we call $\Ad:\gs\to\GL(\liegm)$ to the restriction of the adjoint representation.

The following lemma is a crucial ingredient to prove Proposition \ref{prop-polystability-extension-structure-group}:

\begin{lemma}\label{lemma-compact-involution}
Given a homomorphism $\theta:\Gamma\to\Aut(G)$, there exists a compact involution $\sigma$ of $G$ preserving $\gs$ such that, for every character $\mu:\Gamma\to\C^*$,
$$d\sigma(\lieg^{\theta}_{\mu})=\lieg^{\theta}_{\mu^{-1}}.$$
\end{lemma}
\begin{proof}
Let $\liegt:=\sum_{\mu\in \Hom(\Gamma,\C^*)}\liegm$. This is a subalgebra of $\lieg$, since $[\liegm,\lieg^{\theta}_{\mu'}]=\lieg^{\theta}_{\mu\mu'}$ for every pair of homomorphisms $\mu$ and $\mu'\in\Hom(\Gamma,\C^*)$. In fact $\liegt$ is the subalgebra $\lieg^{C}$ of fixed points of $\lieg$ under the action of the commutator $C:=[\theta(\Gamma),\theta(\Gamma)]=\langle\{\theta_{\gamma\gamma'\gamma^{-1}\gamma'^{-1}}\}_{\gamma,\gamma'\in\Gamma}\rangle$: note that, for each triple $\gamma,\gamma'$ and $\gamma''\in\Gamma$ and an element $v\in\lieg^C$, we have
$$\theta_{\gamma\gamma'\gamma^{-1}\gamma'^{-1}}\theta_{\gamma''}(v)=\theta_{\gamma''}\theta_{\gamma\gamma'\gamma^{-1}\gamma'^{-1}}(v)=\theta_{\gamma''}(v),$$
so that $\theta(\Gamma)$ acts on $\lieg^C$. The automorphisms of this subalgebra induced by the elements of $\theta(\Gamma)$ can be simultaneously diagonalizable (note that they are semisimple, since they have finite order), thus giving a decomposition as in the definition of $\liegt$. This shows $\lieg^C\subseteq\liegt$, and the reverse inclusion is clear. 

The subalgebra of fixed points under a family of semisimple automorphisms is reductive (see \cite{onishchik3}). The restriction of $\theta(\Gamma)$ to $\lie g^C$ is now an abelian group of automorphisms of a reductive Lie algebra, in particular it has a decomposition $\{1\}=\Gamma_0\subset\Gamma_1\subset\dots\subset\Gamma_{k-1}\subset\Gamma_k=\theta(\Gamma)\vert_{\lie g^C}$ by subgroups such that $\Gamma_{i+1}/\Gamma_i$ is cyclic (just take a set of generators for $\theta(\Gamma)$, order them and let $\Gamma_i$ be the subgroup generated by the first $i$ of them). Hence, by \cite[Theorem 7.6]{borel-invariant-cartan-commuting-automorphisms}, there exists a Cartan subalgebra $\lie t$ of $\lie g^C$ that is preserved by the action of $\theta(\Gamma)$. Since the proof is inductive on the dimension of the Lie algebra, we may assume that $\lie t^{\theta}=\lie t\cap\lie g^{\theta}$ is a Cartan subalgebra of $\lie g^{\theta}$.

Let $\Lambda$ be the set of roots of $\lie g^C$, given by the adjoint action of $\lie t$, and choose a system of positive roots $\Lambda^+$. Consider the root space decomposition
\begin{equation}\label{eq-root-space-decomposition-g^C}
    \lie g^C=\lie t\oplus\bigoplus_{\lambda\in\Lambda^+}\lie g_{\lambda}^C\oplus \lie g_{-\lambda}^C.
\end{equation}
Since $\lie g^{\theta}$ is a reductive subalgebra of $\lie g^C$, its Cartan subalgebra $\lie t^{\theta}$ has a root space decomposition too, say
\begin{equation}\label{eq-root-space-decomposition-g^theta}
    \lie g^{\theta}=\lie t^{\theta}\oplus\bigoplus_{\lambda\in\Lambda_0^+}\lie g_{\lambda}^{\theta}\oplus \lie g_{-\lambda}^{\theta},
\end{equation}
where $\Lambda_0$ is the set of roots of $\lie g^{\theta}$. Since root spaces are one-dimensional, (\ref{eq-root-space-decomposition-g^theta}) appears as a summand in (\ref{eq-root-space-decomposition-g^C}). In particular the space of roots of $\lie g^C$ ``contains" the space of roots of $\lie g^{\theta}$ and we may assume that $\Lambda^+$ ``contains" a system of positive roots of $\lie g^{\theta}$, in the sense that they are represented by elements of $\lie t^{\theta}$. 

Let $G^C$ be the connected reductive subgroup of $G$ with Lie algebra $\lie g^C$. Recall that there is a family of compact involutions $\sigma$ of $G^C$ associated with $\lie t$. We may first define a holomorphic involution as follows: the reductive Lie algebra $\lie g^C$ can be described using $\mathfrak{sl}_2$-triples $(x_{\lambda},t_{\lambda},x_{-\lambda})$, where $\lambda$ is a positive root, $t_{\lambda}\in\lie t$ and $x_{\lambda}\in\lie g^C_{\lambda}$. The pair $(x_{\lambda},x_{-\lambda})$ is well defined up to the action of $\C^*$ such that $c\in\C^*$ sends it to $(cx_{\lambda},c^{-1}x_{-\lambda})$. Then the involution sends $(x_{\lambda},t_{\lambda},x_{-\lambda})$ to $(-x_{-\lambda},-t_{\lambda},-x_{\lambda})$. Composing with the antiholomorphic involution which fixes $x_{\lambda}$ and $t_{\lambda}$ for every root $\lambda$ we get a compact involution, which we call $\sigma$. We now claim that, for every $\gamma\in\Gamma$ and every eigenspace $\lieg_{\nu}$ of $\tg$, we may choose the elements $x_{\lambda}$ so that we have $d\sigma(\lieg_{\nu})=\lieg_{\nu^{-1}}.$

Fix $\gamma\in\Gamma$ with order $n$. Since $\tg$ preserves $\lie t$, it must commute the roots. Thus, the eigenvectors of $\tg$ in $\lie t$ with eigenvalue $\nu\in \C^*$ (an $n$-th root of unity) are linear combinations of elements of the form $t_{\lambda}+\overline{\nu} t_{\tg(\lambda)}+\dots+\overline{\nu}^{n-1}t_{\tg^{n-1}(\lambda)}$, where $\lambda$ is a root. Each of these elements satisfies
$$d\sigma(t_{\lambda}+\overline{\nu} t_{\tg(\lambda)}+\dots+\overline{\nu}^{n-1}t_{\tg^{n-1}(\lambda)})=-t_{\lambda}-\nu t_{\tg(\lambda)}-\dots-\nu^{n-1}t_{\tg^{n-1}(\lambda)},$$
which is a $\nu^{-1}$-eigenvector. On the other hand, given a root $\lambda$, the automorphism $\tg$ sends the element $x_{\lambda}$ to $c_{\lambda}x_{\tg(\lambda)}$, where $c_{\lambda}\in\C^*$. If $\tg(\lambda)=\lambda$ then $c_{\lambda}$ must be an $n$-th root of unity. Thus, by choosing a representative root $\lambda$ in each orbit of the action of $\Gamma$ on the set of roots and then choosing a representative in the intersection of the orbit of $x_{\lambda}$ with the $\theta_{\gamma'}(\lambda)$-root space for every $\gamma'\in\Gamma$, we may assume that $c_{\lambda}$ is a root of unity. Moreover,
$$t_{\tg(\lambda)}=\tg([x_{\lambda},x_{-\lambda}])=c_{-\lambda}c_{\lambda}[x_{\tg(\lambda)},x_{-\tg(\lambda)}]=c_{-\lambda}c_{\lambda}t_{\tg(\lambda)},$$
so we must have $c_{-\lambda}=c_{\lambda}^{-1}=\overline{c_{\lambda}}$. A $\nu$-eigenvector in the sum of root spaces is a linear combination of elements of the form $x_{\lambda}+\overline{\nu}c_{\lambda} x_{\tg(\lambda)}+\dots+\overline{\nu}^{n-1}c_{\tg^{n-2}(\lambda)}x_{\tg^{n-1}(\lambda)}$. These satisfy
$$d\sigma(x_{\lambda}+\overline{\nu}c_{\lambda} x_{\tg(\lambda)}+\dots+\overline{\nu}^{n-1}c_{\tg^{n-2}(\lambda)}x_{\tg^{n-1}(\lambda)})=$$$$x_{-\lambda}+\nu\overline{c_{\lambda}} x_{-\tg(\lambda)}+\dots+\nu^{n-1}\overline{c_{\tg^{n-2}(\lambda)}}x_{-\tg^{n-1}(\lambda)}=$$$$x_{-\lambda}+\nu c_{-\lambda} x_{-\tg(\lambda)}+\dots+\nu^{n-1}c_{-\tg^{n-2}(\lambda)}x_{-\tg^{n-1}(\lambda)},$$
which is a $\nu^{-1}$-eigenvector. This proves that $d\sigma(\lieg_{\nu})\subset\lieg_{\nu^{-1}}$, but reversing $\nu^{-1}$ and $\nu$ shows that $d\sigma(\lieg_{\nu^{-1}})\subset\lieg_{\nu}$, which implies that $\lieg_{\nu^{-1}}\subset d\sigma(\lieg_{\nu})$ too.

To finish the argument pick a lift $\tau:\gamtt\to\Aut(\gt)$ of the characteristic homomorphism of (\ref{eq-extension-connected-component}) that leaves $\lie t^{\theta}$ invariant. We may further assume that $d\taug(x_{\lambda})=x_{\taug(\lambda)}$ for each $\lambda\in\Lambda_0$ and $\gamma\in\gamtt$, where $\taug(\lambda)$ denotes the image of $\lambda$ under the automorphism induced by $\taug$ on the space of roots of $\lie g^{\theta}$. Then $d\sigma\vert_{\lie g^{\theta}}$ commutes with $d\taug$ for each $\gamma\in\gamtt$. Thus $\sigma\vert_{\gt}$ commutes with $\tau(\gamtt)$, i.e. the maximal compact subgroup $K_0:=(\gt_0)^{\sigma}$ is $\tau(\gamtt)$-invariant.

By Proposition (\ref{prop-extensions-isomorphic-twisted-group}) there exists a 2-cocycle $c\in Z^2_{\tau}(\gamtt,K_Z)$ with values in $K_Z:=K_0\cap Z(\gt_0)$ such that the extensions $\gs$ and $\gt_0\times_{\tau,c}\gamtt$ of $\gt$ are isomorphic. Conjugation by $(1,\gamma)\in \gt_0\times_{\tau,c}\gamtt$ on $\gt$ is equal to $\taug$, which implies that we have a maximal compact subgroup 
\begin{equation*}
    K:=\bigcup_{\gamma\in\gamtt} K_0(1,\gamma)
\end{equation*}
of $\gs$. Extending this to a maximal compact subgroup of $G$ provides the required compact form $\sigma$ on $G$, which restricts to the original $\sigma$ on $G^C$.
\end{proof}

\section{Galois cohomology and GIT}\label{section-representation-varieties}

Let $G$ be a connected reductive complex Lie group, $\Gamma$ a finite group and
$\theta:\Gamma\to\Aut(G)$
a homomorphism.
Consider the action of $G$ on the set of 1-cocycles given by
$$G\times Z^1_{\theta}(\Gamma,G)\to Z^1_{\theta}(\Gamma,G);\,(g,\beta)\mapsto g\beta\theta(g)^{-1},$$
where $g$ is regarded as a constant map. The quotient of $Z^1_{\theta}(\Gamma,G)$ by this action is  precisely the \textbf{first non-abelian cohomology set} $H^1_{\theta}(\Gamma,G)$.

There is a closed embedding
$$Z^1_{\theta}(\Gamma,G)\hookrightarrow\Hom(\Gamma,G\rtimes_{\theta}\Gamma);\,\beta\mapsto (\gamma\mapsto (\beta(\gamma),\gamma)),$$
where the semidirect product is defined using $\theta$.
This is $G$-equivariant for the conjugation action of $G$ on $\Hom(\Gamma,G\rtimes_{\theta}\Gamma)$, since
\begin{align*}
g(\beta(\gamma),\gamma)g^{-1}=(g\beta(\gamma)\tg(g)^{-1},\gamma)
\end{align*}
for each $g\in G$, $\gamma\in\Gamma$ and $\beta\in Z^1_{\theta}(\Gamma,G)$.

Recall that a homomorphism $\Hom(\Gamma,G\rtimes_{\theta}\Gamma)$ is \textbf{reductive} if its composition with the adjoint representation is completely reducible. On the other hand, $G\rtimes_{\theta}\Gamma$ is a reductive group (it is a finite extension of a reductive group) and so reductive representations have closed orbits \cite{representations-finitely-generated-groups}. But $\Gamma$ is finite, hence all its representations are reductive. We have shown:

\begin{proposition}\label{prop-representation-varieties}
All the elements of $Z^1_{\theta}(\Gamma,G)$ have closed orbits under the action of $G$.
\end{proposition}

\section{Simple \texorpdfstring{$G$}{G}-bundles and Galois cohomology}\label{section-simple-and-H}
Let $X$ be a compact Riemann surface, and $G$ a connected reductive complex Lie group. Let $\Gamma$ be a finite subgroup of $H^1(X,Z)\rtimes\Out(G)\times\C^*$. Projections on the second, third and first factors provide homomorphisms $a:\Gamma\to \Out(G,V)$, $\mu:\Gamma\to\C^*$ and a 1-cocycle $\alpha\in Z^1_a(\Gamma,H^1(X,Z))$ respectively. Recall that this is a map $\alpha:\Gamma\to H^1(X,Z)$ satisfying
\begin{equation*}
    \alpha_{\gamma\gamma'}=\alg\ag(\alpha_{\gamma'})
\end{equation*}
for each $\gamma$ and $\gamma'\in\Gamma$.

Fix a lift $\theta:\Gamma\to\Aut(G)$ of $a$. Abusing notation we also call $\theta$ to the induced homomorphism $\Gamma\to\Aut(G/Z)$. 
The first step to describe the fixed points is to construct a map
\begin{equation*}
    \wf:H^1(X,\ug)_s^{\llambda}\to 
    % \x(\llambda,\Int(G));
    % new
    H^1_{\theta}(\Gamma,G/Z);
\end{equation*}
here $H^1(X,\ug)_s^{\llambda}$ is the set of isomorphism classes of simple $G$-bundles which are fixed under the action of $\llambda$ and $H^1_{\theta}(\Gamma,G/Z)$ is the first Galois cohomology set of $\Gamma$ with values in $G/Z$, consisting of equivalence classes of 1-cocycles as given in Definition \ref{def-1-cocycle}. 

Let $E$ be a simple $G$-bundle over $X$ such that $E\cong\tg^{-1}(E\otimes\alg)$ for every $\gamma$ in $\Gamma$. In other words, for each $\gamma\in\Gamma$ we have an isomorphism
$$\hg:E\xrightarrow{\sim}\tg^{-1}(E\otimes\alg).$$
This in turn induces an isomorphism
$$\ohg:E/Z\to\tg^{-1}(E\otimes\alg)/Z.$$
Note that the simplicity of $E$ implies that this is independent of the choice of $\hg$. Since $\tg(Z)=Z$, there are natural isomorphisms
$$\tg^{-1}(E\otimes\alg)/Z\cong\tg^{-1}(E)/Z\cong\tg^{-1}(E/Z).$$
According to Section \ref{section-action}, the total spaces of $\tg^{-1}(E/Z)$ and $E/Z$ are naturally biholomorphic. After composing we may regard the isomorphisms $\ohg$ as biholomorphisms
$$\ohg:E/Z\to E/Z$$
satisfying
$$\ohg(eg)=\ohg(e)\tg(g).$$

Define a holomorphic map
$$f:E\to \fun{\Gamma}{G/Z}$$
in such a way that $\ohg(e)=e \fg(e)$ for each $e\in E/Z$. A straightforward calculation shows that 
\begin{equation}\label{eq-conjugacy-f}
    f(e g)=g^{-1}f(e)\theta(g)
\end{equation}
for every $g\in G/Z$ and $e\in E/Z$, where we are identifying elements in $G/Z$ with constant functions. 

\begin{lemma}\label{lemma-f-Z}
For every $e\in E$, $f(e)\in Z^1_{\theta}(\Gamma,G/Z)$.
\end{lemma}
\begin{proof}
By Remark \ref{remark-induced-iso}, if $\gamma$ and $\gamma'$ are elements of $\Gamma$, the isomorphism $\hg$ induces an isomorphism 
$$(E,\phi)\gamma'\to (E,\phi)\gamma\gamma'$$
which we also call $\hg$. Since $h_{\gamma}h_{\gamma'}h_{\gamma\gamma'}^{-1}$ is in the gauge group of $(E,\phi)$, which is $Z$, we have $\ohg \oh_{\gamma'}=\oh_{\gamma\gamma'}$ and so
$$ef_{\gamma\gamma'}(e)=\oh_{\gamma\gamma'}(e)=\oh_{\gamma}\oh_{\gamma'}(e)=\oh_{\gamma}(ef_{\gamma'}(e))=\oh_{\gamma'}(e) \theta_{\gamma}(f_{\gamma'}(e))=ef_{\gamma}(e) \theta_{\gamma}(f_{\gamma'}(e))$$
for each $e\in E/Z$, as required.
\end{proof}

Note that the fact that $G/Z$ is an affine algebraic variety implies that $\fung$ is affine. Since $Z^1_{\theta}(\Gamma,G/Z)$ is a closed subvariety of $\fung$, it is itself affine. Consider the action of $G/Z$ on $Z^1_{\theta}(\Gamma,G/Z)$, such that $g\in G/Z$ sends $\beta\in Z^1_{\theta}(\Gamma,G/Z)$ to $g\beta\theta(g)^{-1}$. Because of Lemma \ref{lemma-f-Z} and (\ref{eq-conjugacy-f}) we have a morphism
$$X\to Z^1_{\theta}(\Gamma,G/Z)\sslash G/Z,$$
where the right hand side is the corresponding GIT quotient \cite{mumford}. Since $Z^1_{\theta}(\Gamma,G/Z)$ is affine, so is $Z^1_{\theta}(\Gamma,G/Z)\sslash G/Z$. Since this is an algebraic morphism from a projective variety to an affine variety, it must be constant. Using Proposition \ref{prop-representation-varieties} we conclude that, for each $e\in E$, the orbit of $f(e)$ under the action of $G/Z$ is closed. Since the closures of all these orbits intersect (because their images in the GIT quotient are equal), all the orbits must coincide. Thus, we get a map
$$\wf:H^1(X,\ug)_s^{\llambda}\to H^1_{\theta}(\Gamma,G/Z)$$
sending $E$ to the class of $f(e)$ for any $e\in E$. Moreover, we have:

\begin{lemma}\label{lemma-orbit-fibre-principal}
An element $\beta\in Z^1_{\theta}(\Gamma,G/Z)$ is in the class $\wf(E,\phi)$ if and only if, for every (or for some) $x\in X$, there exists $e$ in the fibre of $E$ over $x$ such that $f(e)=\beta$.
\end{lemma}
\begin{proof}
The {\sl if} direction follows immediately from the definition of $\wf$. For the {\sl only if} direction, fix $x\in X$ and let $\beta\in \wf(E,\phi)$. From the previous paragraphs we know that there exist $e'$ in the fibre of $x$ and $g\in G/Z$ such that $f(e')=g^{-1}\beta\theta(g).$ By (\ref{eq-conjugacy-f}) we get
$$f(e'g^{-1})=gf(e')\theta(g^{-1})=\beta,$$
so we set $e:=e'g^{-1}$.
\end{proof}

\begin{remark}\label{remark-G/Z-vs-Int(G)}
    Using the natural isomorphism $G/Z\cong\Int(G)$, we will sometimes identify the image of $\wf$ with $H^1_{\theta}(\Gamma,\Int(G))$, where the action of $\Gamma$ on $\Int(G)$ is conjugation by $\theta$.
\end{remark}

Now let $S$ be the set of isomorphism classes of simple $G$-Higgs bundles, and let $S^{\Gamma}$ be the subset of fixed points. Using the same arguments above we may define a map $\wf:S^{\Gamma}\to Z^1_{\theta}(\Gamma,G/Z)\cong Z^1_{\theta}(\Gamma,\Int(G))$. Moreover, we have

\begin{lemma}\label{lemma-orbit-fibre-higgs}
Let $(E,\phi)$ be a simple $G$-Higgs bundle preserved by the action of $\Gamma$. An element $\beta\in Z^1_{\theta}(\Gamma,G/Z)$ is in the class $\wf(E,\phi)$ if and only if, for every (or for some) $x\in X$, there exists $e$ in the fibre of $E$ over $x$ such that $f(e)=\beta$.
\end{lemma}

% \section{Fixed points and \texorpdfstring{$Z$}{Z}-bundles of finite order}

% The group structure of $Z$ determines a group multiplication on $H^1(X,\uz)$, which in turn makes $Z^1_a(\Gamma,H^1(X,\uz))$ an abelian group. In this section we prove the following:

% \begin{proposition}
%     If $H^1(X,\ug)^{\Gamma}$ is non-empty then $\alpha$ has finite order.
% \end{proposition}

% \begin{proof}
%     Recall that there is an embedding $G\hookrightarrow\GL(N,\C)$ for some number $N$. Since $G$ is reductive, this can be decomposed into irreducible representations. On the other hand the centre $Z\le G$ is isomorphic to $(\C^*)^m\times \Lambda$, where $\Lambda$ is a finite abelian group. For each copy of $\C^*$ in $Z$ we may find an irreducible linear representation $\sigma:G\to\GL(n,\C)$ such that $\sigma(\C^*)$ is non-zero. Since $\C^*$ is one-dimensional, the kernel must be finite dimensional. Moreover, $\sigma(\C^*)\subset \C^*\le \GL(n,\C)$ by irreducibility, which implies equality. Thus the composition $\det\circ\sigma:\C^*\to\C^*$ with the determinant homomorphism $\det:\GL(n,\C)\to\C^*$ also has finite kernel. Now consider a $Z$-bundle $L$ and a $G$-bundle $E$ over $X$, together with an automorphism  such that $E\cong \tg(E)$
% \end{proof}

\section{Simple \texorpdfstring{$G$}{G}-Higgs bundles and fixed points}\label{section-fixed-isomorphism-classes}
We are now ready to describe the fixed points in the set of isomorphism classes of $G$-bundles under the $\llambda$-action, following \cite{oscar-barajas-higgs}. 
% We assume hereafter that $\alpha$ has finite order as an element of $Z^1_a(\Gamma,H^1(X,Z))$. This is true, for example, if $Z$ is finite (i.e. $G$ is semisimple), since in this case $\Fun(\Gamma,H^1(X,Z))$ is finite, or if $a$ is trivial, since then $Z^1_a(\Gamma,H^1(X,Z))=\Hom(\Gamma,H^1(X,Z))$. 
% Note that $\alpha$ can be thought of as an element of $H^1(X,Z^1_{a}(\llambda,Z))$ via the composition
% \begin{align*}
%     Z^1_{a}(\llambda,H^1(X,Z)_0)\cong Z^1_{a}(\llambda,\Hom(\pi_1(X),Z))\cong \\\Hom(\pi_1(X),Z^1_{a}(\llambda,Z))\hookrightarrow H^1(X,Z^1_{a}(\llambda,Z)),
% \end{align*}
% where $H^1(X,Z)_0$ is the group of flat $Z$-bundles over $X$. Indeed, $\alg$ is flat for each $\gamma\in\Gamma$ because it has finite order.

% First note that, given a Lie group $G$, an automorphism $\theta:G\to G$ and a $G$-bundle $E$, we may define the extension of structure group $\theta(E)$ induced by $\theta$. Alternatively, this is the $G$-bundle with total space equal to $E$ and $G$-action given by $G\times E\ni(s,e)\mapsto e\theta^{-1}(s)$, where we have written the original $G$-action on $E$ by attaching elements of $G$ on the right. If $\theta=\Int_s$ for some $s\in G$ then the map 
% $E\ni e\mapsto es$
% induces an isomorphism of $G$-bundles $E\xrightarrow{\sim}\theta(E)$.

\begin{proposition}\label{prop-reduction-principal}
Let $ E $ be a $G$-bundle and $\theta\in\Hom(\llambda,\Aut(G))$ a lift of $a$. With notation as in Section \ref{section-Gtheta}, assume that there is a $\gs$-bundle $F$ which is a reduction of structure group of $ E $ such that 
\begin{equation}\label{eq-c-theta}
    \ctt(F)\cong\alpha.
\end{equation}
Then $ E $ is isomorphic to $E\gamma= \tg^{-1}(E\otimes\alg)$ for every $\ambda\in\llambda$.
\end{proposition}

\begin{proof}
It is enough to get an isomorphism
$$\hg:F\xrightarrow{\sim}\tg^{-1}(F\otimes\alg).$$
This in turn would induce an isomorphism from $E$ to $\tg^{-1}(E\otimes \alg)$.

Fix $\ambda\in\llambda$ and choose an open cover $\{U_i\}_{i\in I}$ of $X$ which trivializes both $F$ and $\alg$. Let $e_i$ and $z_i$ define local sections of $F$ and $\alg$ resp. on $U_i$. We get transition functions
$$g_{ij}:U_i\cap U_j\to G,\,z_{ij}:U_i\cap U_j\to Z$$
satisfying $e_j=e_ig_{ij}$ and $z_j=z_iz_{ij}$. A set of local trivializations for $\tg^{-1}(F\otimes\alg)$ is then $\{U_i,e_i\otimes z_i\}$ (with the $e_i\otimes z_i$'s regarded as local sections of $\tg^{-1}(F\otimes \alg)$), and the corresponding transition functions are $\tg^{-1}(g_{ij}z_{ij}):U_i\cap U_j\to G.$
But, using (\ref{eq-c-theta}), we may assume that $z_{ij}=g_{ij}^{-1}\tg(g_{ij})$ and so
$\tg^{-1}(g_{ij}z_{ij})=\tg^{-1}(g_{ij}g_{ij}^{-1}\tg(g_{ij}))=g_{ij}.$
Hence, we may set $\hg(e_i):=e_i\otimes z_i$ and extend the isomorphism to the whole $F$ imposing that it respects the $\gs$-actions.
\end{proof}

\begin{proposition}\label{prop-reduction-higgs}
Let $(E,\phi) $ be a $G$-Higgs bundle and $\theta\in\Hom(\llambda,\Aut(G))$ a lift of $a$. With notation as in Section \ref{section-Gtheta}, assume that there is a $(\gs,\liegm)$-Higgs pair $(F,\psi)$ which is a reduction of structure group of $(E,\phi)$ satisfying (\ref{eq-c-theta}).
Then $(E,\phi)$ is isomorphic to $(E,\phi)\gamma= (\tg^{-1}(E\otimes\alg),\mug\tg^{-1}(\phi))$ for every $\ambda\in\llambda$.
\end{proposition}

\begin{proof}
    Mimicking the proof of Proposition \ref{prop-reduction-principal} we get an isomorphism $\hg:E\to \tg^{-1}(E\otimes\alg)$ which is induced by the identity on $F$. It is left to check that 
$$\hg(\psi)=\mu_{\gamma}\tg^{-1}(\psi).$$
Note that, if $\psi$ is locally of the form $(e_i,r)\otimes k$, where $e_i$ is the local section defined above, then the local form of $\hg(\psi)$ is also $(e_i,r)\otimes k$ with $e_i$ considered as a local section of $\tg^{-1}(E)$. But $\mug\tg^{-1}(\psi)$ is locally equal to
$$\mug(e_i,\tg^{-1}(r))\otimes k=\mug\mu_{\gamma}^{-1}(e_i,r)\otimes k=(e_i,r)\otimes k,$$
as required.
\end{proof}

\begin{proposition}\label{prop-simple-fixed-points-principal}
Let $ E $ be a simple $G$-bundle over $X$ which is isomorphic to $ E \ambda$ for every $\ambda$ in $\llambda$. Then a 1-cocycle $\beta\in Z^1_{\theta}(\Gamma,G/Z)$ is in $\wf(E)\in H^1_{\theta}(X,G/Z)$ (definitions as in Section \ref{section-simple-and-H}) if and only if there exists a $G_{\beta\theta}$-bundle $F$ which is a reduction of structure group of $ E $ and satisfies 
\begin{equation}\label{eq-c-betatheta}
    \tilde c_{\beta\theta}(F)\cong\alpha,
\end{equation}
where we are identifying $G/Z$ with $\Int(G)$. For each $\beta\in Z^1_{\theta}(\Gamma,G/Z)$ such a reduction is unique.
\end{proposition}

\begin{proof}
Choose an element $\beta\in \wf(E)$ and let $s:\llambda\to G$ be a map such that $\beta=\Int_s$. For each $\ambda\in\llambda$ take an isomorphism 
$\hg:E\to \tg^{-1}(E\otimes\alg).$
By Lemma \ref{lemma-orbit-fibre-principal}, if we define $f$ as in Section \ref{section-simple-and-H}, $f^{-1}(\beta)$ has non-empty intersection with every fibre of $E$ over $X$. Define
\begin{equation}\label{eq-definition-reduction-from-iso}
   F:=\{e\in E:\hg(e)=e\sg\otimes\zg(e),\,\zg(e)\in Z(G)\}.
\end{equation}
This is the preimage of $f^{-1}(\beta)$ under the natural projection $E\to E/Z.$
Then, given $e\in F$ and $g\in G$, the element $eg$ is also in $F$ if and only if
\begin{align*}
    eg\sg\otimes\zg(e)=(e\sg\otimes\zg(e)) \beta^{-1}(g)=\hg(e)\beta^{-1}(g)&=\hg(e\tg^{-1}\beta^{-1}(g))\\&=e(\beta\tg)^{-1}(g)\sg\otimes\zg(eg)
\end{align*}
for every $\ambda\in\llambda$, or equivalently $\beta\tg(g)g^{-1}\in Z$. This shows that $F$ provides a reduction of structure group to $G_{\beta\theta}$. Moreover, the isomorphisms
\begin{equation*}
    E\xrightarrow{\hg} \tg^{-1}(E\otimes\alg)\xrightarrow{\sg^{-1}}\tg^{-1}\bg^{-1}(E\otimes\alg)=(\bg\tg)^{-1}(E\otimes\alg),
\end{equation*}
where the second map is multiplication by $\sg^{-1}$ on the right, restricts to isomorphisms
$$F\to(\bg\tg)^{-1}(F\otimes\alg);\,e\mapsto e\otimes\zg(e)$$
for every $\ambda\in\llambda$. To see why (\ref{eq-c-betatheta}) is true we fix $\ambda\in\llambda$ and use an open cover $\{U_i\}_{i\in I}$ of $X$ trivializing $F$ equipped with local sections $e_i$ on each $U_i$, so that $z_i:=\zg(e_i)$ is a set of local sections for $\alg$. Setting $z_j=z_iz_{ij}$ and $e_j=e_ig_{ij}$, we have
\begin{align*}
    (e_i\sg\otimes z_i)\bg^{-1}(g_{ij}) z_{ij}=(e_ig_{ij}\sg\otimes z_i) z_{ij}=e_j\sg\otimes z_j=\hg(e_j)&=\hg(e_ig_{ij})\\&=(e_i\sg\otimes z_i)\tg(g_{ij}).
\end{align*}
Hence
$z_{ij}=\bg(z_{ij})=\bg\tg(g_{ij})g_{ij}^{-1}$,
which implies that $\alg\cong c_{\beta\theta}(E)$.
The uniqueness of the reduction follows from Proposition \ref{prop-reduction-principal}, the simplicity of $E$ and the fact that the resulting isomorphisms completely determine the reduction by (\ref{eq-definition-reduction-from-iso}).

Now let $\beta'\in Z^1_a(\Gamma,G/Z)$ be another 1-cocycle and assume that there is a reduction of structure group $F'$ of $ E $ to $G_{\beta'\theta}$ satisfying (\ref{eq-c-betatheta}) with $\beta'$ instead of $\beta$. Let $s':\Gamma\to G$ be a map such that $\beta'=\Int_{s'}$. Proposition \ref{prop-reduction-principal} provides isomorphisms
$$\hg': E \xrightarrow{\sim}(\beta'\tg)^{-1}(E)\otimes\ambda\xrightarrow{\sg'}\tg^{-1}(E)\otimes\ambda,$$
where the second morphism is multiplication by $\sg'\in G$, and these induce isomorphisms
$$\ohg':E/Z\xrightarrow{\sim}\tg^{-1}(E)/Z.$$
Note from the proof of Proposition \ref{prop-reduction-principal} that there exists $e\in E/Z$ such that $\ohg'(e)=e\bg'$, where we regard $\beta'$ as the image of $s'$ in $\Fun(\Gamma,G/Z)$. Thus, by Lemma \ref{lemma-orbit-fibre-principal}, $\beta'\in\wf(E)$. 
\end{proof}

\begin{proposition}\label{prop-simple-fixed-points-higgs}
Let $(E,\phi)$ be a simple $G$-Higgs bundle over $X$ which is isomorphic to $(E,\phi) \ambda$ for every $\ambda$ in $\llambda$. Then a 1-cocycle $\beta\in Z^1_{\theta}(\Gamma,G/Z)$ is in $\wf(E,\phi)\in H^1_{\theta}(X,\Int(G))$ (definitions as in Section \ref{section-simple-and-H}) if and only if there exists a $G_{\beta\theta}$-bundle $F$ which is a reduction of structure group of $ E $ and satisfies (\ref{eq-c-betatheta}).
For each $\beta\in Z^1_{\theta}(\Gamma,G/Z)$ such a reduction is unique.
\end{proposition}
\begin{proof}
Take isomorphisms $\hg:(E,\phi)\to\theta^{-1}(E\otimes\alg,\mug\phi)$ for each $\gamma\in\Gamma$ and
assume that a 1-cocycle $\beta\in Z^1_{\theta}(\Gamma,G/Z)$ is in $\wf(E,\phi)\in H^1_{\theta}(X,\Int(G))$. Mimicking the proof of Proposition \ref{prop-simple-fixed-points-principal} and replacing Proposition \ref{prop-reduction-principal} with Proposition \ref{prop-reduction-higgs} we get a reduction of structure group of $E$ to a $G_{\beta\theta}$-bundle $F$ given by (\ref{eq-definition-reduction-from-iso}). Let $\theta':=\beta\theta$ and let $\hg'$ be the composition
\begin{equation*}
    \hg':(E,\phi)\xrightarrow{\hg}\theta^{-1}(E\otimes\alg,\mug\phi)\xrightarrow{\sg^{-1}}\theta'^{-1}(E\otimes\alg,\mug\phi),
\end{equation*}
where $s:\Gamma\to G$ is any map such that $\beta=\Int_s$. 

The Higgs field is in $H^0(X,F(\lie g^{\theta'}_{\mu})\otimes K_X)$: if $\phi$ is locally of the form $(e,v)\otimes k$, we may assume by Lemma \ref{lemma-exists-higgs-field-reduction} that $e\in F$, hence $\hg(e)=e\sg\otimes \zg(e)$ by (\ref{eq-definition-reduction-from-iso}) and so $\hg'(e)=e\otimes \zg(e)$. Therefore, we have
\begin{equation}\label{eq-Phi-induces-local}
    \tg'^{-1}(E)(\lieg)\otimes k\ni(e,v)\otimes k=(\hg'(e),v)\otimes k=\mu_{\gamma}(e,\tg'^{-1}(v))\otimes k
\end{equation}
for every $\gamma\in \Gamma$ (here we are using the identification between $\tg'^{-1}(E\otimes\alg)\times_{\Ad}\lieg$ and $\tg'^{-1}(E)\times_{\Ad}\lieg$).
That is to say,
$v=\mu_{\gamma}\tg'^{-1}(v)$
whenever $k$ does not vanish, and so $\phi$ is induced by a section of $F(\lieg^{\beta\theta}_{\mu})\otimes K_X$ as required. Thus, we have a reduction to a $(G_{\beta\theta},\lieg^{\beta\theta}_{\mu})$-Higgs pair $(F,\psi)$. 

The uniqueness of the reduction and the converse statement (i.e. the \textit{if} statement) follow as in the proof of Proposition \ref{prop-simple-fixed-points-principal} after replacing Proposition \ref{prop-reduction-principal} with Proposition \ref{prop-reduction-higgs}.
\end{proof}

\section{Fixed points in the moduli space of \texorpdfstring{$G$}{G}-Higgs bundles}\label{section-fixed-moduli}

We use the results of Section \ref{section-fixed-isomorphism-classes} to give a description of the fixed points in the moduli space of $G$-Higgs bundles. As in Section \ref{section-Gtheta} we assume that $Z^1_a(\Gamma,Z)$ is finite. 
We start with:

\begin{proposition}\label{prop-polystability-extension-structure-group}
Let $\theta\in\homgh$ be a lift of $a$. Take a lift $\gamtt\to\Aut(\gt_0)$ of the characteristic homomorphism of \ref{eq-extension-connected-component} preserving a maximal compact subgroup $K^{\theta}\le\gt_0$ and extend $K^{\theta}$ to a maximal compact subgroup $K$ of $G$. Fix $z\in i(\zk^{\theta})^{\Gamma}$, where $\zk^{\theta}$ is the centre of $\lie k^{\theta}$. Then:
\begin{enumerate}
    \item If a $(\gs,\liegm)$-Higgs pair $(F,\psi)$ is $z$-polystable, the $G$-Higgs bundle $(E,\phi)$ obtained by extension of structure group is also $z$-polystable.
    \item If $(E,\phi)$ is a $z$-(semi,poly)stable $G$-Higgs bundle with a reduction of structure group to a $(\gs,\liegm)$-pair $(F,\psi)$, then $(F,\psi)$ is $z$-(semi,poly)stable.
    \item Given $g\in G$ and $\theta':=\Int_g\theta\Int_{g^{-1}}$,
    % we have an isomorphism $\gamtt\cong \widetilde\Gamma_{\theta'}$ induced by the equality $G_{\theta'}=g\gs g^{-1}$, and so we may regard $\Lambda$ as a subgroup of $\widetilde\Gamma_{\theta'}$. Moreover, 
    there is an isomorphism between $\cM(X,\gs,\liegm)$ and $\cM(X,G_{\theta'},\lieg^{\theta'}_{\mu})$ making the following diagramme commute:
\begin{equation}\label{eq-diagramme-extension-theta-theta'}
    \begin{tikzcd}
\cM(X,\gs,\liegm)\arrow{r}\arrow{d} &
\cM(X,G)\\
\cM(X,G_{\theta'},\lieg^{\theta'}_{\mu})\arrow{ur}
\end{tikzcd},
\end{equation}
where the morphisms to $\cM(X,G)$ are given by extension of structure group. For each $\alpha\in Z^1_{a}(\Gamma,H^1(X,Z))$, it restricts to a diagramme
\[\begin{tikzcd}
\cM_{\alpha}(X,\gs,\liegm)\arrow{r}\arrow{d} &
\cM(X,G)\\
\cM_{\alpha}(X,G_{\theta'},\lieg^{\theta'}_{\mu})\arrow{ur}
\end{tikzcd},
\]
where $\cM_{\alpha}(X,\gs,\liegm)$ is the moduli space of $(\gs,\liegm)$-Higgs pairs $(F,\psi)$ such that $\ctt(F)\cong\alpha$. 
\end{enumerate}
(1), (2) and (\ref{eq-diagramme-extension-theta-theta'}) are also true after replacing $\gs$ and $G_{\theta'}$ by $\gt$ and $G^{\theta'}$ respectively, and for any $z\in i\zk^{\theta}$. 
\end{proposition}
\begin{proof}
The proof of (2) is precisely the same as the proof of Proposition 5.7 (2) in \cite{PR}: suppose that  $(F,\psi)$ is not $z$-semistable.
Following Definition \ref{def-stability-higgs-pair-non-connected}, there is an 
$s\in i(\lie k^{\theta})^{\Gamma}$ defining a parabolic subgroup $P_s\in \gs$, 
and a reduction $\tau$ of
$F$ to a $P_s$-bundle such that $\deg F(s,\tau)< \langle z,s\rangle$ and $\psi\in H^0(X,F(\liegm)_{\tau,s}\otimes K_X)$. 
But $s$ also defines a parabolic subgroup $\widetilde{P}_s$ of $G$, and 
the reduction $\tau$ defines a reduction  $\widetilde{\tau}$ of $E$ to
$\widetilde{P}_s$ such that  
$\deg E(s,\widetilde{\tau})=\deg F(s,\tau)$. Moreover, it is straightforward to check that $(\liegm)_{s}=\widetilde{P}_s\cap\liegm$ and so $\phi\in H^0(X,E_{\widetilde\tau}(\lie p_s)\otimes K_X)$. This contradicts the $z$-semistability of $(E,\phi)$. The same argument applies to stability and polystability.

Now we prove (1), which also follows \cite{PR}: fix a maximal compact subgroup $K_{\theta}$ of $\gs$ and consider a maximal compact subgroup $K$ of $G$ containing it, so that $K_{\theta}=K\cap\gs$. By Lemma \ref{lemma-compact-involution} we may assume that $K$ is defined by an antiholomorphic involution $\sigma$ of $G$ satisfying 
\begin{equation}\label{eq-tau}
    d\sigma(\liegm)=\lieg^{\theta}_{\mu^{-1}}
\end{equation}
for every homomorphism $\mu:\Gamma\to\C^*$. Then, given a $z$-polystable $(\gs,\liegm)$-Higgs pair $(F,\psi)$, by Theorem \ref{EH1-equivariant} and Proposition \ref{prop-twisted-equivariant-bundles-one-to-one} there exists a $\gamtt$-invariant reduction $h_{\theta}\in \Omega^0(F/K_{\theta})$ satisfying the Hitchin equation (\ref{hitchin-equation}). Let $(E,\phi)$ be the extension of structure group of $(F,\psi)$ to $G$. Using the inclusion $F/K_{\theta}\subset E/K$, we get a reduction of structure group $h\in \Omega^0(E/K)$. 

On the other hand, (\ref{hitchin-equation}) is an equation setting a moment map equal to $-i2\pi z$: if we consider the topological bundle underlying $E$ and the space $\mathcal A$ of pairs $(A,\phi)$, where $A$ is a hermitian $G$-connection on $E$ and $\phi\in\Omega^1(E(\lie g))$, there is an action of the gauge group preserving the metric $h$ and this provides a moment map
$m:X\to E_h(\lie k)^*,$
where $E_h$ is the reduction of $E$ to $K$ given by $h$ and $\lie{k}$ is the Lie algebra of $K$. Using the Killing form, $m$ may be regarded as a map $X\to E_h(\lie{k})$. The space $\mathcal B$ of $(\gs,\liegm)$-Higgs pairs $(B,\psi)$, where $B$ is a hermitian $\gs$-connection and $\psi\in\Omega^1(F(\liegm))$ (here $F$ is the reduction of structure group to $\gs$ determined by the $\gs$-connection) is then embedded in $\mathcal A$, and the corresponding moment map $m_{\theta}$ is the restriction of
$$X\xrightarrow{m}E_h(\lie{k})\to F_{h_{\theta}}(\lie k^{\theta})$$
to $\mathcal B$, where the second homomorphism is given by orthogonal projection and we use the obvious notations. Given a $z$-polystable $(\gs,\liegm)$-Higgs pair $(F,\psi)$ as in the previous paragraph, the moment map at $(F,\psi)$ is given by:
$$m(E,\phi)=\Lambda(F_{h}+[\psi,\sigma_{h}(\psi)]),$$
where $\Lambda$ is the adjoint operator of wedging by the Kähler form $\omega$ on $X$.
But $F_h=F_{h_{\theta}}$ (the curvature of the Chern connection for the metric $h_{\theta}$), and (\ref{eq-tau}) implies that $[v,\sigma(v)]\in [\liegm,\lie g^{\theta}_{\mu^{-1}}]\subset\lie g^{\theta}$ for each $v\in\liegm$, hence $[\psi,\sigma_{h}(\psi)]\in \Omega^0(X,F_{h_{\theta}}(\lie k\cap\lie g^{\theta}))= \Omega^0(X,F_{h_{\theta}}(\lie k^{\theta}))$. This means that the moment map $m_{\theta}$ is just the restriction of $m$ and so $m(E,\phi)=m_{\theta}(F,\psi)$ as required. Thus, $(E,\phi)$ satisfies (\ref{hitchin-equation}) and so, by Theorem \ref{EH1}, it is $z$-polystable.

The isomorphism in (3) is given as follows: consider a $G$-Higgs bundle $(E,\phi)$, a homomorphism $\theta\in\homgh$ and a reduction of structure group $F$ to $\gs$ such that $\phi$ is induced by a section $\psi$ of $F(\liegm)\otimes K_X$. Let $g\in G$ and $\theta':=\Int_{g}\theta \Int_{g^{-1}}$. We show that there is also a reduction $F'$ to $G_{\theta'}$ such that $\phi$ is induced by a section $\psi'$ of $F'(\lieg^{\theta'}_{\mu})\otimes K_X$. We set 
$$F':=Fg^{-1},$$
so that $F'$ is a reduction to $\Int_{g}(\gs)$. But $u\in G$ is contained in $\Int_{g}(\gs)$ if and only if
$$\Int_g\theta\Int_{g^{-1}}(u)=z(\gamma,u)u$$
for each $\gamma\in\Gamma$, which yields $\Int_{g}(\gs)=G_{\theta'}$. Similarly $\Int_{g}(\gt)=G^{\theta'}$, therefore $g^{-1}$ induces an isomorphism from $F/\gt$ to $F'/G^{\theta'}$ and so $\ctt(F)=\ctt(F')$. On the other hand, if $\psi$ is locally equal to $(e,v)\otimes k$ for some $e\in F$, $v\in\liegm$ and $k\in K$, we may define $\psi'$ locally as $(eg^{-1},\Ad_gv)\otimes k$. It is straightforward to see that this is a section of $F'(\lieg^{\theta'}_{\mu})\otimes K_X$.

The proof for $\gt$ is precisely the same.
\end{proof}

\begin{corollary}
    \label{cor-polystability-extension-structure-group-principal}
Let $\theta\in\homgh$ be a lift of $a$ and fix $z\in i(\zk^{\theta})^{\Gamma}$. Then:
\begin{enumerate}
    \item If a $\gs$-bundle $F$ is $z$-polystable, the $G$-bundle obtained by extension of structure group is also $z$-polystable.
    \item If $E$ is a $z$-(semi,poly)stable $G$-bundle with a reduction of structure group to a $\gs$-bundle $F$, then $F$ is $z$-(semi,poly)stable.
    \item Given $g\in G$ and $\theta':=\Int_g\theta\Int_{g^{-1}}$,
    % we have an isomorphism $\gamtt\cong \widetilde\Gamma_{\theta'}$ induced by the equality $G_{\theta'}=g\gs g^{-1}$, and so we may regard $\Lambda$ as a subgroup of $\widetilde\Gamma_{\theta'}$. Moreover, 
    there is a canonical isomorphism between $M(X,\gs)$ and $M(X,G_{\theta'})$ making the following diagramme commute:
    \[\begin{tikzcd}
M(X,\gs)\arrow{r}\arrow{d} &
M(X,G)\\
M(X,G_{\theta'})\arrow{ur}
\end{tikzcd},
\]
where the morphisms to $M(X,G)$ are given by extension of structure group. For each $\alpha\in Z^1_{a}(\Gamma,H^1(X,Z))$, it restricts to a diagramme
\[\begin{tikzcd}
M_{\alpha}(X,\gs)\arrow{r}\arrow{d} &
M(X,G)\\
M_{\alpha}(X,G_{\theta'})\arrow{ur}
\end{tikzcd},
\]
\end{enumerate}
where $M_{\alpha}(X,\gs)$ is the moduli space of $\gs$-bundles $F$ such that $\ctt(F)\cong\alpha$.
\end{corollary}

We call $\widetilde{\mdl}_{\alpha}(X,\gs,\liegm)$ and $\widetilde M_{\alpha}(X,\gs)$ to the images of ${\mdl}_{\alpha}(X,\gs,\liegm)$ and $M_{\alpha}(X,\gs)$ in $\mdl(X,G)$ and $M(X,G)$ respectively, and similarly replacing $\gs$ with $\gt$. By Proposition \ref{prop-polystability-extension-structure-group}, if $\theta'=\Int_g\theta\Int_{g^{-1}}$ for some $g\in G$, we have 
$\widetilde{\mdl}_{\alpha}(X,\gs,\liegm)=\widetilde{\mdl}_{\alpha}(X,G_{\theta'},\lie g^{\theta'}_{\mu})$.

Let $\mdl(X,G)^{\llambda}$ and $M(X,G)^{\llambda}$ be the fixed point locus of $\mdl(X,G)$ and $M(X,G)$ respectively under the action of $\llambda$, and let $\mdl_{ss}(X,G)^{\llambda}$ and $M_{ss}(X,G)^{\llambda}$ be the intersections with the stable and simple loci. Given a 1-cocycle $\bg\in Z^1_{\theta}(\Gamma,G/Z)$, we call $[\beta]$ to its cohomology class in $H^1_{\theta}(\Gamma,G/Z)$ (see Definition \ref{def-1-cocycle}). Combining Propositions \ref{prop-reduction-principal} and \ref{prop-simple-fixed-points-principal} we get: 

\begin{theorem}\label{th-fixed-points-oscar-ramanan-principal}
Fix a homomorphism $\theta:\Gamma\to\Aut(G)$ lifting $a$. We have the inclusions
    $$\bigcup_{[\beta]\in H^1_{\theta}(\Gamma,\Int(G))}\widetilde{{M}}_{\alpha}(X,G_{\beta\theta})\subset{M}(X,G)^{\llambda}$$
and    
    $${M}_{ss}(X,G)^{\llambda}\subset\bigcup_{[\beta]\in H^1_{\theta}(\Gamma,\Int(G))}\widetilde{{M}}_{\alpha}(X,G_{\beta\theta}).$$
Moreover, the intersections 
$${M}_{ss}(X,G)\cap\widetilde{{M}}_{\alpha}(X,G_{\beta\theta})={M}_{ss}(X,G)^{\llambda}\cap\widetilde{{M}}_{\alpha}(X,G_{\beta\theta})$$
are disjoint for different $[\beta]\in H^1_{\theta}(\Gamma,\Int(G))$.
\end{theorem}

We also have the corresponding result for Higgs bundles, which follows from Propositions \ref{prop-reduction-higgs} and \ref{prop-simple-fixed-points-higgs}:

\begin{theorem}\label{th-fixed-points-oscar-ramanan-higgs}
Fix a homomorphism $\theta:\Gamma\to\Aut(G)$ lifting $a$. We have the following relations between moduli spaces:
\begin{enumerate}
    \item $$\bigcup_{[\beta]\in H^1_{\theta}(\Gamma,\Int(G))}\widetilde{\cM}_{\alpha}(X,G_{\beta\theta},\lieg^{\beta\theta}_{\mu})\subset\cM(X,G)^{\Gamma}. $$
    
    \item $$\cM_{ss}(X,G)^{\Gamma}\subset\bigcup_{[\beta]\in H^1_{\theta}(\Gamma,\Int(G))}\widetilde{\cM}_{\alpha}(X,G_{\beta\theta},\lieg^{\beta\theta}_{\mu}).$$
    
\end{enumerate}

As $[\beta]$ runs over $H^1_{\theta}(\Gamma,\Int(G))$, $\beta\theta$ runs over all the conjugacy classes (under conjugation by $\Int(G)$) of elements of $\homgh$ lifting $a$. Moreover, the intersections 
$$\cM_{ss}(X,G)\cap\widetilde{\cM}_{\alpha}(X,G_{\beta\theta},\lieg^{\beta\theta}_{\mu})=\cM_{ss}(X,G)^{\Gamma}\cap\widetilde{\cM}_{\alpha}(X,G_{\beta\theta},\lieg^{\beta\theta}_{\mu})$$
are disjoint for different $[\beta]\in H^1_{\theta}(\Gamma,\Int(G))$.
\end{theorem}

\begin{corollary}
    If the character $\mu:\Gamma\to\C^*$ is trivial, we have the following inclusions:
    \begin{enumerate}
    \item $$\bigcup_{[\beta]\in H^1_{\theta}(\Gamma,\Int(G))}\widetilde{\cM}_{\alpha}(X,G_{\beta\theta})\subset\cM(X,G)^{\Gamma}. $$
    
    \item $$\cM_{ss}(X,G)^{\Gamma}\subset\bigcup_{[\beta]\in H^1_{\theta}(\Gamma,\Int(G))}\widetilde{\cM}_{\alpha}(X,G_{\beta\theta}).$$
    
\end{enumerate}
Moreover, the intersections 
$$\cM_{ss}(X,G)\cap\widetilde{\cM}(X,G_{\beta\theta})=\cM_{ss}(X,G)^{\Gamma}\cap\widetilde{\cM}(X,G_{\beta\theta})$$
are disjoint for different $[\beta]\in H^1_{\theta}(\Gamma,\Int(G))$.
\end{corollary}

\begin{corollary}
    If $\widetilde{\mdl}(X,\gt)$ is the image of the extension of structure group morphism $\mdl(X,\gt)\to\mdl(X,G)$ and $\alpha$ is trivial, we have the following inclusions:
    \begin{enumerate}
    \item $$\bigcup_{[\beta]\in H^1_{\theta}(\Gamma,\Int(G))}\widetilde{\cM}(X,G^{\beta\theta},\lieg^{\beta\theta}_{\mu})\subset\cM(X,G)^{\Gamma}. $$
    
    \item $$\cM_{ss}(X,G)^{\Gamma}\subset\bigcup_{[\beta]\in H^1_{\theta}(\Gamma,\Int(G))}\widetilde{\cM}(X,G^{\beta\theta},\lieg^{\beta\theta}_{\mu}).$$
    
\end{enumerate}

Moreover, the intersections 
$$\cM_{ss}(X,G)\cap\widetilde{\cM}_{\alpha}(X,G^{\beta\theta},\lieg^{\beta\theta}_{\mu})=\cM_{ss}(X,G)^{\Gamma}\cap\widetilde{\cM}_{\alpha}(X,G^{\beta\theta},\lieg^{\beta\theta}_{\mu})$$
are disjoint for different $[\beta]\in H^1_{\theta}(\Gamma,\Int(G))$.
\end{corollary}
\begin{proof}
    Follows from Theorem \ref{th-fixed-points-oscar-ramanan-higgs} using $\widetilde{\cM}_{1}(X,G_{\beta\theta},\lieg^{\beta\theta}_{\mu})=\widetilde{\cM}(X,G^{\beta\theta},\lieg^{\beta\theta}_{\mu})$, where we denote the trivial 1-cocycle in $Z_a^1(\Gamma,H^1(X,Z))$ by $1$.
\end{proof}

% \begin{remark}
% If $\mu$ is trivial then the moduli spaces involved in the statement of Theorem \ref{th-fixed-points-oscar-ramanan-higgs} are moduli spaces of honest Higgs bundles with suitable structure group.
% \end{remark}

\section{The Prym--Narasimhan--Ramanan construction of fixed points}\label{section-prym-narasimhan-ramanan}

The Prym--Narasimhan--Ramanan construction can be now given as a corollary of Theorem \ref{th-prym-narasimhan-ramanan-higgs} which, combined with Theorem \ref{th-fixed-points-oscar-ramanan-higgs}, provides a characterization of the subvariety of fixed points.

Let $X$ be a compact Riemann surface and let $G$ be a connected reductive complex Lie group with centre $Z$. Let $\Gamma$ be a finite subgroup of $H^1(X,Z)\rtimes\Out(G)\times\C^*$ and consider a lift $\theta$ of $a$. In this section we assume that the order of $\alpha\in Z^1_a(\Gamma,H^1(X,Z))$ and $\zt$ are both finite. 

\begin{remark}\label{remark-finiteness-zt}
    The second assumption is not necessary, since the fact that the structure group of $\alpha$ is a finite subgroup $\gam\le\zt$ implies that any $\gs$-bundle $F$ such that $\ctt(F)\cong\alpha$ has a reduction of structure group to $c_{\theta}^{-1}(\gam)$. Thus in what follows we could replace $\gs$ by $c_{\theta}^{-1}(\gam)$ and eliminate the finiteness assumption on $\zt$. However, for simplicity we keep this assumption.
\end{remark}

We keep the notation of Section \ref{section-Gtheta}. Let $Z( \gt_0)$ be the centre of $ \gt_0$ and $\hat Y\in H^1(X,\gamtt)$, a $\gamtt$-bundle over $X$. The connected component $Y$ of $\hat Y$ may be regarded as a connected étale cover of $X$ with Galois group $\gam\le\gamtt$. Consider the subgroup $G_{Y}\le\gs$ which is the preimage of $\gam$ under the quotient $\gs\to\gamtt$. Choose a lift
$\taut:\gamtt\to\Aut( \gt_0)$
of the characteristic homomorphism of the extension (\ref{eq-extension-connected-component}).
By Proposition \ref{prop-extensions-isomorphic-twisted-group}, there exists a 2-cocycle $\ct\in Z^2_{\taut}(\gamtt,Z( \gt_0))$ such that
$\gs\cong  \gt_0\times_{(\taut,\ct)}\gamtt.$
On the other hand, a $\gs$-bundle $E$ over $X$ satisfies $\ctt(E)\cong{\alpha}$ if and only if $E/ \gt_0\cong\hat Y$ for some element $\hat Y\in H^1(X,\gamtt)$ such that $\qqt(Y)\cong\alpha$, where by abuse of notation $\qqt$ is given by applying extension of structure group from $\gam$ to $\gamtt$ and then using the homomorphism $\gamtt\to\gamt$ defined in Section \ref{section-Gtheta}.
The application of Theorem (\ref{th-prym-narasimhan-ramanan-higgs}) yields:

\begin{theorem}\label{th-prym-narasimhan-ramanan-oscar-ramanan-higgs}
For each homomorphism $\theta:\Gamma\to\Aut(G)$ lifting $a$ we have an isomorphism
\begin{equation}
    \bigsqcup_{\qqt(Y)\cong \alpha}\mdl(Y,\gt_0,\gam,\tau^{\theta},c^{\theta}, \lieg^{\theta}_{\mu})/Z_{\gamtt}(\gam)\cong \mdl_{\alpha}(X,\gs,\liegm),
\end{equation}
where $Z_{\gamtt}(\gam)$ is the centralizer of $\gam$ in $\gamtt$, which acts on  by Proposition \ref{prop-action-centralizer-higgs}.

Fix such a lift $\theta$. Let $\widetilde{\mdl}(Y,\gt_0,\tau^{\beta\theta},\gam,c^{\beta\theta}, \lieg^{\beta\theta}_{\mu},\mu)/Z_{\gamtt}(\gam)$ be the image of the moduli space $\mdl(Y,\gt_0,\tau^{\beta\theta},\gam,c^{\beta\theta}, \lieg^{\beta\theta}_{\mu},\mu)$ in $\mdl(X,G)$ via the composition of the isomorphism given in Theorem \ref{th-prym-narasimhan-ramanan-higgs} and extension of structure group from $\gs$ to $G$. Then we have the following inclusions:

\begin{enumerate}
    \item $$\bigcup_{[\beta]\in H^1_{\theta}(\Gamma,\Int(G)),\qqt(Y)\cong \alpha}\widetilde{\mdl}(Y,\gt_0,\gam,\tau^{\beta\theta},c^{\beta\theta}, \lieg^{\beta\theta}_{\mu})/Z_{\gamtt}(\gam)\subset\cM(X,G)^{\Gamma}.$$
    
    \item $$\cM_{ss}(X,G)^{\Gamma}\subset\bigcup_{[\beta]\in H^1_{\theta}(\Gamma,\Int(G)),\qqt(Y)\cong \alpha}\widetilde{\mdl}(Y,\gt_0,\gam,\tau^{\beta\theta},c^{\beta\theta}, \lieg^{\beta\theta}_{\mu})/Z_{\gamtt}(\gam).$$
    
\end{enumerate}

As $[\beta]$ runs over $H^1_{\theta}(\Gamma,\Int(G))$, $\beta\theta$ runs over all the conjugacy classes (under conjugation by $\Int(G)$) of elements of $\homgh$ lifting $a$. Moreover, the intersections 
$$\cM_{ss}(X,G)\cap\widetilde{\mdl}(Y,\gt_0,\gam,\tau^{\beta\theta},c^{\beta\theta}, \lieg^{\beta\theta}_{\mu})/Z_{\gamtt}(\gam)
% =\cM_{ss}(X,G)^{\Gamma}\cap\widetilde{\mdl}(Y,\tau^{\beta\theta},c^{\beta\theta}, \gt_0,\gam,\lieg^{\beta\theta}_{\mu},\mu)/Z_{\gamtt}(\gam)
$$
are disjoint for different $[\beta]\in H^1_{\theta}(\Gamma,\Int(G))$ and $Y$.

\end{theorem}

\begin{theorem}\label{th-prym-narasimhan-ramanan-oscar-ramanan-principal}
For each homomorphism $\theta:\Gamma\to\Aut(G)$ lifting $a$ we have an isomorphism
\begin{equation}
    \bigsqcup_{\qqt(Y)\cong \alpha}M(Y, \gt_0,\gam,\tau^{\beta\theta},c^{\beta\theta})/Z_{\gamtt}(\gam)\cong M_{\alpha}(X,\gs),
\end{equation}
where $Z_{\gamtt}(\gam)$ is the centralizer of $\gam$ in $\gamtt$, which acts on  by Proposition \ref{prop-action-centralizer-higgs}.

Fix such a lift $\theta$. Let $\widetilde{M}(Y, \gt_0,\gam,\tau^{\beta\theta},c^{\beta\theta})/Z_{\gamtt}(\gam)$ be the image of the moduli space $M(Y, \gt_0,\gam,\tau^{\beta\theta},c^{\beta\theta})/Z_{\gamtt}(\gam)$ in $M(X,G)$ via the composition of the isomorphism given in Theorem \ref{th-prym-narasimhan-ramanan-higgs} and extension of structure group from $\gs$ to $G$. Then we have the following inclusions:

\begin{enumerate}
    \item $$\bigcup_{[\beta]\in H^1_{\theta}(\Gamma,\Int(G)),\qqt(Y)\cong \alpha}\widetilde{M}(Y, \gt_0,\gam,\tau^{\beta\theta},c^{\beta\theta})/Z_{\gamtt}(\gam)\subset M(X,G)^{\Gamma}.$$
    
    \item $$ M_{ss}(X,G)^{\Gamma}\subset\bigcup_{[\beta]\in H^1_{\theta}(\Gamma,\Int(G)),\qqt(Y)\cong \alpha}\widetilde{M}(Y, \gt_0,\gam,\tau^{\beta\theta},c^{\beta\theta})/Z_{\gamtt}(\gam).$$
    
\end{enumerate}

As $[\beta]$ runs over $H^1_{\theta}(\Gamma,\Int(G))$, $\beta\theta$ runs over all the conjugacy classes (under conjugation by $\Int(G)$) of elements of $\homgh$ lifting $a$. Moreover, the intersections 
$$M_{ss}(X,G)\cap\widetilde{M}(Y, \gt_0,\gam,\tau^{\beta\theta},c^{\beta\theta})/Z_{\gamtt}(\gam)
% =\cM_{ss}(X,G)^{\Gamma}\cap\widetilde{M}(Y,\tau^{\beta\theta},c^{\beta\theta}, \gt_0,\gam,\lieg^{\beta\theta}_{\mu},\mu)/Z_{\gamtt}(\gam)
$$
are disjoint for different $[\beta]\in H^1_{\theta}(\Gamma,\Int(G))$ and $Y$.

\end{theorem}

\section{The Prym--Narasimhan--Ramanan construction for character varieties}\label{section-prym-narasimhan-ramanan-character-varieties}

Keep the assumptions of Section \ref{section-prym-narasimhan-ramanan}, namely that the order of $\alpha\in Z^1_a(\Gamma,H^1(X,Z))$ and $\zt$ are both finite. If we suppose that $\mu$ is trivial, Theorem \ref{th-prym-narasimhan-ramanan-oscar-ramanan-higgs} yields a description of the fixed points of a $\Gamma$-action on the character variety $\calR(X,G)$. We keep the definitions of Section \ref{section-twisted-equivariant-higgs-pairs-and-hitchin-equations}. The action of $\Gamma$ on $\calR(X,G)$ induced by its action on $\cM(X,G)$ via Theorem \ref{rep1} is given as follows: for each $\gamma\in\Gamma$, the corresponding $Z$-bundle $\alg$ is flat ---it has finite order by assumption--- and so it is given by a representation $\rho_{\gamma}:\pi_1(X)\to Z$. Given a representation $\rho:\pi_1(X)\to G$, we may multiply it by $\rho_{\gamma}$ to get a new representation $\rho\otimes\rho_{\gamma}$. Now, given the automorphism $\theta_\gamma$  of $G$ and  
$\rho\in \Hom(\pi_1(X,x),G)$, there is another representation of $\pi_1(X,x)$ in $G$ given by $\theta_\gamma\circ \rho$. This defines a left action of $\Aut(G)$ on 
$\calR(X,G)$ that clearly descends to an action of $\Out(G)$.
So for every $\gamma\in \Gamma$ and  $\rho\in \Hom(\pi_1(X),G)$ we have 
$\rho\cdot\gamma \in \Hom(\pi_1(X),G)$ given by 
$$
\rho\cdot\gamma=\theta_\gamma^{-1}\circ (\rho\otimes\alg).
$$
It is straightforward to show (see 
\cite{biswas-calvo-García-prada,PR,ow} 
for a similar computation) that the right action of $\Gamma$ on  $\calR(X,G)$ 
given by this coincides with the action of $\Gamma$ on $\cM(X,G)$ defined in 
Section \ref{section-action} via the non-abelian Hodge correspondence (recall
that here we are taking the character $\mu:\Gamma\to \C^*$ to be trivial).

For each lift $\theta$ of $a$, let $\taut$ and $\ct$ satisfy $\gs\cong\gt_0\times_{\taut,\ct}\gamtt$. Let $Y$ be the connected component of a $\gamtt$-bundle, with Galois group $\Gamma_Y\le\gamtt$. Then Proposition \ref{prop-polystability-extension-structure-group} and Theorem \ref{th-prym-narasimhan-ramanan-higgs} provide, using the non-abelian Hodge correspondence \ref{equivariant-nahc}, a morphism $\calR(Y,\gt_0,\Gamma_Y,\taut,\ct)/Z_{\gamtt}(\Gamma_Y)\to\calR(X,G)$. Let $\wcalR(Y,\gt_0,\Gamma_Y,\taut,\ct)/Z_{\gamtt}(\Gamma_Y)$ be its image. Theorem \ref{th-prym-narasimhan-ramanan-oscar-ramanan-higgs} provides the following result:

\begin{theorem}\label{th-prym-narasimhan-ramanan-character-varieties}
Let $\mu$ be trivial and fix $\theta\in\homgh$ lifting $a$. We have the following relations between character varieties:
\begin{enumerate}
    \item $$\bigcup_{[\beta]\in H^1_{\theta}(\Gamma,\Int(G)),q_{\beta\theta}(Y)=\alpha}\wcalR(Y,G^{\beta\theta}_0,\Gamma_Y,\tau^{\beta\theta},c^{\beta\theta})/Z_{\gamtt}(\Gamma_Y)\subset\calR(X,G)^{\Gamma}. $$
    
    \item $$\calR_{\irr}(X,G)^{\Gamma}\subset\bigcup_{[\beta]\in H^1_{\theta}(\Gamma,\Int(G)),q_{\beta\theta}(Y)=\alpha}\wcalR(Y,G^{\beta\theta}_0,\Gamma_Y,\tau^{\beta\theta},c^{\beta\theta})/Z_{\gamtt}(\Gamma_Y).$$
    
\end{enumerate}
\end{theorem}

\newpage

\chapter{Prym--Narasimhan--Ramanan construction for a finite cyclic group action}\label{chapter-cyclic}

\section{Action of a finite cyclic subgroup of the Jacobian on the moduli space of Higgs bundles}\label{section-example-generalize-narasimhan}
In this section we apply Theorem \ref{th-prym-narasimhan-ramanan-oscar-ramanan-higgs} to the case where $G=\GL(n,\C)$ and $a$ is trivial. In particular this shows that Theorem \ref{th-prym-narasimhan-ramanan-oscar-ramanan-higgs} gives the construction in \cite{narasimhan-ramanan} when applied to $\GL(n,\C)$-bundles, or vector bundles of rank $n$.

Let $L$ be a line bundle over $X$ with order $r$, which may be seen as an element of $H^1(X,\Z/r\Z)$, and assume that $n=rm$ is divisible by $r$ (otherwise there would be no fixed points). Let $\Gamma<H^1(X,\Z/r\Z)$ be the subgroup generated by $L$ and consider a homomorphism
$$\mu:\Gamma\to\C^*;\,\gamma\mapsto\mug.$$
We want to calculate the fixed points of the action of $\Gamma$ on $\cM(X,\GL(n,\C))$. 

Every element in $\Hom(\Z/r\Z,\Int(\GL(n,\C)))$ is determined by the image of the generator of $\Z/r\Z$. Moreover, every element in $\GL(n,\C)$ with finite order is diagonalizable and so every class in $H^1(\Z/r\Z,\Int(\GL(n,\C)))$, which is equal to the character variety $\mathfrak{X}(\Z/r\Z,\Int(\GL(n,\C))):=\Hom(\Z/r\Z,\Int(\GL(n,\C)))/\Int(\GL(n,\C))$, is represented by an automorphism of the form $\theta:=\Int_M$, where
$$M:=\begin{pmatrix}
    I_{p_1}&0&\ldots&0\\
    0&\zeta I_{p_2}&\ldots&0\\
    \vdots&\vdots&\ddots&\vdots\\
    0&0&\ldots&\zeta^{r-1}I_{p_r}
    \end{pmatrix},$$
$\zeta$ is a primitive $r$-th root of unity and $p_1+\ldots+p_r=n$. The group of fixed points $\GL(n,\C)^{\theta}$ under this automorphism is equal to 
$\GL(p_1,\C)\times\GL(p_2,\C)\times\ldots\times\GL(p_r,\C).$

In order to define the group $\GL(n,\C)_{\theta}$, as defined in Section \ref{section-Gtheta}, note that the group multiplication $\Z/r\Z\times\Z/r\Z\to\Z/r\Z$ induces an inclusion of $\Z/r\Z$ in the group of permutations of the eigenvalues of $M$. Let $A$ be a matrix permuting the eigenspaces of $M$, in the sense that it sends each eigenspace to its permutation. We denote by $p(A)$ the corresponding permutation of the eigenvalues of $M$. We then call $A$ a \textbf{permutation matrix} if $p(A)\in\Z/r\Z$. The group $\GL(n,\C)_{\theta}$ is generated by $\GL(n,\C)^{\theta}$ and a set of permutation matrices, one for each element $k$ in $\Z/r\Z$ such that the eigenspaces corresponding to the same orbit of eigenvalues under the action of $\langle k\rangle$ have the same dimension. On the other hand, we know from Theorem \ref{th-fixed-points-oscar-ramanan-higgs} that the moduli of fixed points $\cM^{\Gamma}(X,\GL(n,\C))$ is empty unless the homomorphism
$$\cct:\GL(n,\C)_{\theta}\to\Z/r\Z$$
is surjective. This happens if and only if all the eigenspaces have the same dimension, i.e. $p_1=\dots=p_r=m$. We will assume this from now on. 

Thus, $\GL(n,\C)_{\theta}$ is a $\Z/r\Z$-extension of $\GL(n,\C)^{\theta}$ generated by $\GL(n,\C)^{\theta}$ and the permutation matrix
$$S:=\begin{pmatrix}
    0&I_m&0&\ldots&0\\
    0&0&I_m&\ldots&0\\
    \vdots&\vdots&\ddots&\ddots&\vdots\\
    0&0&0&\ldots&I_m\\
    I_m&0&0&\ldots&0\\
    \end{pmatrix}.$$
Since $S^r=1$, the 2-cocycle involved is the trivial map
$$\Z/r\Z\times\Z/r\Z\to Z(\GL(n,\C)^{\theta})\cong (\C^*)^{r}.$$
This provides an isomorphism
$$\gls\cong \glt \times_{(\Int_S,1)}\Z/r\Z=\glt\rtimes_{\Int_S}\Z/r\Z,$$
where by abuse of notation we are writing the image $\Int_S$ of the generator instead of the whole homomorphism
$$\Z/2\Z\to\Aut(\glt).$$ 

Let 
$$p_L:X_L\to X$$ 
be the Galois $r$-cover of $X$ defined by $L$ and let $\lambda$ be a generator of its Galois group $\Gamma_L$, which is isomorphic to $\Z/r\Z$. According to Theorem \ref{th-fixed-points-oscar-ramanan-higgs}, simple fixed points come from $(\Int_S,c)$-twisted $\Z/r\Z$-equivariant $\glt$-bundles over $X_L$. The associated vector bundles are direct sums of vector bundles of rank $m$ with a $\Gamma$-equivariant action. Take one of them, and call it $V=F\oplus F_1\oplus\ldots\oplus F_{d-1}$. The automorphism induced by $\lambda$ via Remark \ref{remark-associated-vector-bundle} and the natural embedding of $\GL(n,\C)_{\theta}$ in $\GL(n,\C)$ exchanges the action of the copies of $\GL(m,\C)$ (as does the conjugation by $S$), therefore its effect on $V$ is permuting the summands. This provides an isomorphism
$$V\cong \lambda^*V$$
permuting the summands or, in other words,
$$V\cong F\oplus\lambda ^*F\oplus\lambda^{2*}F\oplus\ldots\oplus\lambda^{(r-1)*}F.$$
The action of the Galois group is the permutation action on $V$.

To see what happens to the Higgs field, we first note that $\gl_{\mu}^{\theta}$ is the vector space of permutation matrices $A$ such that $p(A)=\mu(L)^{-1}$. The homomorphism $\pi_1(X)\to\C^*$ determined by $L$ induces a natural isomorphism $\Gamma_L\cong \Gamma^*:=\Hom(\Gamma,\C^*)$. Regarding $\mu$ as an element of $\Gamma_L$ and using the $\Gamma_L$-equivariance, we find that the Higgs field on $V$ must be induced by homomorphisms $\lambda^{k*}F\to\mu^{-1*}\lambda^{k*}F\otimes K_L$ obtained by pulling back a homomorphism 
$$\psi:F\to\mu^{-1*}F\otimes K_L$$
under the elements of $\gal(X_L/X)$ where $K_L$ is the canonical bundle of $X_L$. 

In summary,
\begin{proposition}
A simple $\GL(n,\C)$-Higgs bundle $(E,\phi)$ over $X$ is isomorphic to $(E\otimes L,\mu(L)\phi)$ if and only if it is the pushforward of a vector bundle $F$ of rank $m$ over $X_L$ and $\phi$ is induced by a homomorphism
$$\psi:F\to\mu^{-1*}F\otimes K_X.$$
\end{proposition}

\begin{corollary}\label{cor-U(m,m)}
Let $L$ be a line bundle of order 2 over $X$ with corresponding \'Etale cover $X_L\to X$. Consider the involution of $\cM(X,\GL(2m,\C))$ sending each Higgs bundle $(E,\phi)$ to $(E\otimes L,-\phi)$. Then a smooth point is fixed under this action if and only if it is the pushforward of a vector bundle $F$ of rank $m$ on $X_L$ equipped with a homomorphism $\psi:F\to\lambda^*F\otimes K_L$, where $\lambda$ is the generator of $\gal(X_L/X)\cong \Z/2\Z$ and $K_L$ is the canonical bundle of $X_L$.
\end{corollary}

\begin{remark}
With data $(F,\psi)$ as in Corollary \ref{cor-U(m,m)}, the higgs bundle $(F\oplus\lambda^*F,\psi\oplus\lambda^*\psi)$ ---whose quotient by the permutation action of $\gal(X_L/X)$ is the pushforward of $F$--- is a $\U(m,m)$-Higgs bundle on $X_L$.
\end{remark}

\section{Dualization of Higgs bundles}\label{section-dualization}

Throughout this section, attaching an ``N" to the left of a subgroup of $\SL(n,\C)$ denotes its normalizer in $\SL(n,\C)$. If $n\ge 3$, the group of outer automorphisms of $\SL(n,\C)$ is isomorphic to $\Z/2\Z$. A lift of its generator $a$ is given by the automorphism $\theta$ sending $A$ to $A^{t-1}$. Given a line bundle $L$, 
$$L\theta(L)=LL^*\cong\oo,$$
the trivial bundle. Thus we have a homomorphism
$$\Z/2\Z\to H^1(X,\C^*)\rtimes\Out(\SL(n,\C))$$
sending the generator $-1\in\Z/2\Z$ to $(L,a)$. If $E$ is a simple fixed point,
$$\oo\cong\det(E)\cong\det(E^*)\otimes L^n\cong L^n,$$
so that necessarily $L$ has finite order $r$ dividing $n$ and so $L$ reduces to a a $\Z/n\Z$-bundle. From \cite{helgason} we also know that every class of $H^1_{\theta}(\Z/2\Z,\Int(\SL(n,\C)))$ is represented either by $\theta$ or, in case that $n=2m$ is even, $\nu:=\Int_J\theta$, where
\begin{equation}\label{eq-S-symplectic-form}
    J:=
\begin{pmatrix}
0&I_m\\-I_m&0
\end{pmatrix}.
\end{equation}

We need to calculate the extensions $\NSO(n,\C)=\SL(n,\C)_{\theta}$ and $\NSp(2m,\C)=\SL(n,\C)_{\nu}$ (see \cite{PR}) of $\SO(n,C)$ and $\Sp(2m,\C)$, respectively. To do so, we first calculate $c_{\theta}(\NSO(n,\C))$ and $c_{\nu}(\NSp(2m,\C))$. Recall that $c_{\theta}$ is the homomorphism which sends $g\in \NSO(n,\C)$ to $\theta(g)g^{-1}$, which is an element of the centre of $\SO(n,\C)$, and we may define $c_{\nu}$ similarly. For the symplectic group consider the group of invertible matrices
\begin{equation}\label{eq-def-Mc}
    M_c:=\begin{pmatrix}
I_m&0\\0&c^{-1}I_m
\end{pmatrix},
\end{equation}
where $c\in\C^*$. They satisfy
$$c_{\nu}(M_c)=cI_n,$$
so that $c_{\nu}$ is an isomorphism from the group to $\C^*$. Since $\C^*$ coincides with the centre of $\GL(2m,\C)$, every coset of the normalizer of $\Sp(2m,\C)$ in $\GL(2m,\C)$ must be represented by a matrix $M_c$. Classes in the quotient $\NSp(2m,\C)/\Sp(2m,\C)$ must be in bijection with elements of the group with determinant 1, which are exactly the elements $M_c$ such that $c$ is an $m$-th root of unity. Choose a primitive $m$-th root of unity $\zeta$ and set
$$\Int_M:\Z/m\Z\to\Int(\Sp(2m,\C));\,k\mapsto \Int_{M_{\zeta^k}}.$$
Then we get an isomorphism
\begin{equation}\label{eq-normalizer-sp}
    \NSp(2m,\C)\cong \Sp(2m,\C)\times_{(\Int_{M},1)}\Z/m\Z,
\end{equation}
i.e. $\NSp(2m,\C)$ is a semidirect product of $\Sp(2m,\C)$ by $\Z/m\Z$.

For the orthogonal group, given $c\in \C^*$, 
$$c_{\theta}(cI_n)=c^{-2}I_n,$$
and varying $c$ we get the whole centre of $\SL(n,\C)$. Again, classes in $\NSO(n,\C)/\SO(n,\C)$ are represented by scalar matrices with determinant 1, which are just $n$-th roots of unity. When $n$ is odd, taking a power $-2$ is an automorphism of the group of $n$-th roots $\Z/n\Z$, so that $\NSO(n,\C)$ is a semidirect product of $\SO(n,\C)$ with $\Z/n\Z$ given by conjugation by an element of the centre. In other words, a direct product.
However, when $n=2m$ is even the restriction of $c_{\theta}$ to $\Z/n\Z$ is not an isomorphism and so a 2-cocycle appears. In this case we may restrict to $n$-th roots of unity with argument less than $\pi$ to give a bijection between $\NSO(2m,\C)$ and $\SO(2m,\C)\times\Z/m\Z$. Since conjugation by elements in the centre is trivial, so is the characteristic homomorphism. The cocycle $c$ is defined as follows: define a map
$$\sigma:\Z/n\Z\to\Z/2\Z$$
which assigns an $n$-th root to $1$ if its argument is less than $\pi$ and $-1$ otherwise. Then set
\begin{equation}\label{eq-def-c-dualizing}
    c:\Z/n\Z\times\Z/n\Z\to \Z/2\Z;\, (\delta,\delta')\mapsto \sigma(\delta^{\frac 12}\delta'^{\frac 12}),
\end{equation}
where $\delta^{\frac 12}$ is the unique square root of $\delta$ with argument less than $\pi$.

\begin{remark}
    We may explicitly check that this is a 2-cocycle since, if $\delta,\delta'$ and $\delta''$ are $n$-th roots of unity,
$$c(\delta',\delta'')c(\delta,\delta'\delta'')=\sigma(\delta^{\frac 12}\delta'^{\frac 12})\sigma(\delta^{\frac 12}\delta'^{\frac 12}\delta''^{\frac 12})\sigma(\delta'^{\frac 12}\delta''^{\frac 12})=
c(\delta\delta',\delta'')c(\delta,\delta'),$$
where we have used that
$$(\delta\delta')^{\frac 12}=\sigma(\delta^{\frac 12}\delta'^{\frac 12})\delta^{\frac 12}\delta'^{\frac 12}.$$
\end{remark}

Applying the proof of Proposition \ref{prop-extensions-isomorphic-twisted-group} we have isomorphisms
\begin{equation}\label{eq-normalizer-so}
    \NSO(n,\C)\cong\begin{cases}  \SO(n,\C)\times \Z/n\Z    &   \text{if $n$ is odd}\\
                                \SO(n,\C)\times_{(1,c)} \Z/m\Z   &   \text{if $n=2m$.}
\end{cases}
\end{equation}

\begin{remark}
Note that the same reasoning that gives the even case in (\ref{eq-normalizer-so}) also provides an isomorphism
$$\NSp(2m,\C)\cong \Sp(2m,\C)\times_{(1,c)}\Z/m\Z.$$
Thus we have two descriptions of $\NSp(2m,\C)$, one with trivial 2-cocycle and one with a trivial lift of the characteristic homomorphism.
\end{remark}

We are now ready to describe the simple fixed points. Let $r$ be the order of $L$. Using Theorem \ref{th-fixed-points-oscar-ramanan-higgs}, we first conclude that there are no simple fixed points if $r$ does not divide $n$, or if it does not divide $m$ when $n=2m$. Let $X_L$ be the étale cover of $X$ associated to $L$, which has degree equal to $r$. If $r$ divides $n$ [or $m$], we have:

\begin{proposition}\label{prop-fixed-points-symplectic-cyclic}
    Let $\Gamma\cong\Z/2\Z$ equipped with the homomorphism $\Gamma\to J(X)\rtimes\Aut(\SL(n,\C))$ which assigns $-1$ to $(L,\theta)$, where $\theta$ is taking transpose and inverse and $L$ is a line bundle over $X$ of order $r$ dividing $n$, if $n$ is odd, or $m=n/2$ if $n$ is even. Let $X_L\to X$ be the étale cover of degree $r$ determined by $L$. Then we have:
    \begin{enumerate}
        \item If $n$ is odd:
        \begin{align*}
            &\wcM(X_L,\SO(n,\C),\Z/r\Z,1,1)\subset\mdl(X,\SL(n,\C))^{\Gamma}\andd\\
            &\mdl_{ss}(X,\SL(n,\C))^{\Gamma}\subset\wcM(X_L,\SO(n,\C),\Z/r\Z,1,1).
        \end{align*}
        \item If $n=2m$ is even:
                \begin{align*}
            \wcM(X_L,\SO(n,\C),\Z/r\Z,1,c)\cup\\\wcM(X_L,\Sp(2m,\C),\Z/r\Z,\Int_M,1)
            \subset &\mdl(X,\SL(n,\C))^{\Gamma}\andd\\
            \mdl_{ss}(X,\SL(n,\C))^{\Gamma}
            \subset  &\wcM(X_L,\SO(n,\C),\Z/r\Z,1,c)\cup\\&\wcM(X_L,\Sp(2m,\C),\Z/r\Z,\Int_M,1).
        \end{align*},
    \end{enumerate}
    where $c$ and $M$ are given by (\ref{eq-def-c-dualizing}) and (\ref{eq-def-Mc}) respectively.
\end{proposition}

\section{Tensorization of symplectic Higgs bundles by line bundles of order 2}\label{section-example-sp-cyclic}
Let $L\in H^1(X,\Z/2\Z)$ be a line bundle of order 2 and let $X_L$ be the corresponding étale cover of degree two over $X$. Consider the involution of the moduli space of $\Sp(2n,\C)$-Higgs bundles which sends $(E,\phi)$ to $(E\otimes L,\pm\phi)$. In other words, the generator of $\Gamma$ is sent to $(L,1,\pm 1)\in H^1(X,\Z/2\Z)\rtimes\Out(\Sp(2n,\C))\times\C^*$ ---general actions of finite subgroups of $H^1(X,\Z/2\Z)$ are considered in Chapter \ref{chapter-symplectic}. We note on passing that $\Out(\Sp(2n,\C))$ is trivial and the centre of $\Sp(2n,\C)$ is $\Z/2\Z$, so that this covers all the possible subgroups of order 2 of $H^1(X,\Z/2\Z)\rtimes\Out(\Sp(2n,\C))\times\C^*$.

Consider the embedding of $\Sp(2n,\C)$ in $\GL(2n,\C)$ associated to the symplectic form $J$ defined by (\ref{eq-S-symplectic-form}). According to \cite{helgason} the conjugacy classes of involutions in the quotient $\Aut(\Sp(2n,\C))/\Int(\Sp(2n,\C))$ are represented by $\Int_J$ and $\Int_{K_{p,q}}$, where 
$$
    J:=\begin{pmatrix}
    0 & I_{n}\\
    -I_{n} & 0
    \end{pmatrix}
    \andd
    K_{p,q}:=i\begin{pmatrix}
    -I_p & 0 & 0 & 0\\
    0 & I_q & 0 & 0\\
    0 & 0 & -I_p & 0\\
    0 & 0 & 0 & I_q 
    \end{pmatrix},$$
    and $p+q=n$.
The matrix $J$ is conjugate to
$$
J'=\begin{pmatrix}
    i & 0\\
    0 & -i
    \end{pmatrix}.$$
Set $\theta:=\Int_{J'}$ and $\tau_{p,q}:=\Int_{K_{p,q}}$. By Theorem \ref{th-fixed-points-oscar-ramanan-higgs} we have
\begin{align*}
    \cM_{ss}(X,\Sp(2n,\C))^{\Gamma}=\\\cM_{ss}(X,\Sp(2n,\C))\cap\left(\widetilde{\cM}_{L}(X,\Sp(2n,\C)_{\theta},\lieg^{\theta}_{\pm1})\cup\bigcup_{p+q=n}\widetilde{\cM}_{L}(X,\Sp(2n,\C)_{\tau_{p,q}},\lieg^{\tau_{p,q}}_{\pm1})\right),
\end{align*}
where the subscript $L$ indicates that the quotient by the action of the fixed-point subgroups $\Sp(2n,\C)^{\theta}$ and $\Sp(2n,\C)^{\tau_{p,q}}$ respectively, is a $\Z/2\Z$-bundle isomorphic to $L$. Since 
$$\Sp(2n,\C)_{\tau_{p,q}}=\Sp(2n,\C)^{\tau_{p,q}}$$
unless $p=q$, we actually have 
\begin{align}\label{eq-sp-fixed-points}
    \cM_{ss}(X,\Sp(2n,\C))^{\Gamma}=\\
    \cM_{ss}(X,\Sp(2n,\C))\cap\left(\widetilde{\cM}_{L}(X,\Sp(2n,\C)_{\theta},\lieg^{\theta}_{\pm1})\cup\widetilde{\cM}_{L}(X,\Sp(2n,\C)_{\tau},\lieg^{\tau}_{\pm1})\right),\nonumber
\end{align}
where $\tau:=\tau_{n/2,n/2}$ and the second component is only present if $n$ is even. From now on, whenever $n/2$ appears it will be implicit that $n$ is even.

We find that $\Sp(2n,\C)^{\theta}\cong\GL(n,\C)$ is the subgroup of matrices of the form
$$
\begin{pmatrix}
    A & 0\\
    0 & A^{t-1}
\end{pmatrix},
$$
and $\Sp(2n,\C)_{\theta}$ is the subgroup of $\Sp(2n,\C)$ generated by $\Sp(2n,\C)^{\theta}$ and $S$. Since $S^2=-I$, we get an isomorphism
$$\Sp(2n,\C)_{\theta}\cong\Sp(2n,\C)^{\theta}\times_{\Int_J,-1}\Z/2\Z,$$
where we call $-1$ to the 2-cocycle $c\in Z_{\Int_J}^2(\Z/2\Z,Z(\Sp(2n,\C)^{\theta}))$ such that $c(-1,-1)=-I_{2n}$.

On the other hand, $\Sp(2n,\C)^{\tau}\cong\Sp(n/2,\C)\times\Sp(n/2,\C)$ is the subgroup of matrices of the form
$$\begin{pmatrix}
    A & 0 & B & 0\\
    0 & O & 0 & P\\
    C & 0 & D & 0\\
    0 & Q & 0 & R
    \end{pmatrix},$$
where
$$
\begin{pmatrix}
    A & B\\
    C & D
\end{pmatrix}\andd
\begin{pmatrix}
    O & P\\
    Q & R
\end{pmatrix}
$$
are in $\Sp(n/2,\C)$. The group $\Sp(2n,\C)_{\tau}$ is generated by $\Sp(2n,\C)^{\tau}\cong\Sp(n/2,\C)\times\Sp(n/2,\C)$ and
$$T:=\begin{pmatrix}
    0 & I_{n/2} & 0 & 0\\
    I_{n/2} & 0 & 0 & 0\\
    0 & 0 & 0 & I_{n/2}\\
    0 & 0 & I_{n/2} & 0
    \end{pmatrix},$$
which has order two. Thus, $\Sp(2n,\C)_{\tau}\cong\Sp(2n,\C)^{\tau}\times_{\Int_T,1}\Z/2\Z$.

Let $X_L$ be the étale cover of $X$ associated to $L$. Consider the action of $\Gamma$ on $\C^{2n}$ such that the generator of $\Gamma$ multiplies vectors by $S$ on the left. This is the action provided by Remark \ref{remark-associated-vector-bundle} and the embedding of $\Sp(2n,\C)_{\theta}$ in $\GL(2n,\C)$. By Proposition \ref{prop-associated-bundle-equivariant} we may describe $(\Int_J,-1)$-twisted $\Z/2\Z$-equivariant $\Sp(2n,\C)$-bundles over $X_L$ in terms of vector bundles of rank $2n$ over $X_L$ which have the form $E\oplus E^*$, where $E$ is a vector bundle of rank $n$ over $X_L$ and the automorphism corresponding to the generator $\lambda$ of the Galois group exchanges the two summands. Let $f:E\to \lambda^*E^*$ and $f':E^*\to \lambda^*E$ be the restrictions of the homomorphism induced by the action of $\gal(X_L/X)$ on $E\oplus E^*$. On the one hand the equivariance of the $\Gamma$-action implies that $\lambda^*(f)f'=1$. On the other, the fact that the action preserves the symplectic form implies
\begin{equation*}
    \langle e,e^*\rangle=\omega(e,e^*)=\omega(f(e),f'(e^*))=-\langle f'(e^*),f(e)\rangle=-\langle f'^{*}f(e),e^*\rangle
\end{equation*}
for each $e\in E$ and $e^*\in E^*$, where $\langle \bullet,\bullet\rangle$ denotes the natural pairing between $E$ and $E^*$ and $\omega$ is the corresponding symplectic form. Thus $f'=-f^{*-1}$, and we conclude that $\lambda^*(f^*)=-f$.
The action of the Galois group on $E\oplus E^*$ given by $f\oplus -f^*$ is induced by the natural permutation action on $E\oplus\lambda^*E$ using the isomorphism $\lambda^*(f)$ between $\lambda^*E$ and $E^*$.

Similarly, when $n$ is even, by Remark \ref{remark-associated-vector-bundle} we know that $(\Int_T,1)$-twisted $\Z/2\Z$-equivariant $\Sp(n/2,\C)\times\Sp(n/2,\C)$-bundles over $X_L$ have associated vector bundles of the form $E\oplus E'$, where both $E$ and $E'$ are symplectic vector bundles. The action of the Galois group exchanges $E$ and $E'$, so that $E'\cong \lambda^*E$ and the symplectic form on $E'$ is just the pullback of the symplectic form on $E$.

Thus the relevant twisted equivariant principal bundles are in correspondence with the following objects:
\begin{enumerate}
    \item For $\widetilde{\cM}_{L}(X,\Sp(2n,\C)_{\theta})$: vector bundles $E$ of rank $n$ on $X_L$ equipped with an isomorphism $f:E\xrightarrow{\sim}\lambda^*E^*$ satisfying $\lambda^*f^*=-f$, where $\lambda\in\Gamma^*=\gal(X_L/X)$ is the generator. The corresponding vector bundle on $X$ is the pushforward of $E$, and the symplectic form is given by the standard one on $E\oplus E^*$ and $f$.
    \item For $\widetilde{\cM}_{L}(X,\Sp(2n,\C)_{\tau})$, assuming $n$ is even: symplectic vector bundles $E$ of rank $n$ on $X_L$. Again, the corresponding vector bundles on $X$ are the pushforwards.
\end{enumerate}
Let $K_L$ is the canonical bundle of $X_L$. The Higgs field has values in the $\pm 1$-weight space of the corresponding involution. In the $+1$ case this means that it preserves $E$, in the sense that it is given by a homomorphism $E\to E\times K_L$, whereas in the $-1$ case it is determined by a a $\lambda^*$-invariant homomorphism $E\to\lambda^*E\otimes K_L$, i.e. the Higgs field exchanges $E$ and $\lambda^*E$. Summing up and plugging in the stability conditions, which for symplectic bundles mean that no isotropic sub-bundle has non-negative degree, yields:

\begin{proposition}
Let $L$ be a line bundle of order 2 over $X$ with corresponding \'Etale cover $p_L:X_L\to X$ and Galois group generated by $\lambda\in\Aut(X_L)$. Let $K_L$ be the canonical bundle of $X_L$. Consider the automorphism of $\cM(X,\Sp(2n,\C))$ sending $(E,\phi)$ to $(E\otimes L,\phi)$. Then the smooth fixed points in $\cM(X,\Sp(2n,\C))$ are pushforwards (under $p_L$) of the following two types of Higgs bundles with extra structure:
\begin{enumerate}
    \item A stable Higgs bundle $(F,\psi)$ of rank $n$ on $X_L$ equipped with an isomorphism $f:(F,\psi)\xrightarrow{\sim}\lambda^*(F^*,\psi^*)$ such that $\lambda^*f^*=-f$.
    \item Only if $n$ is even: a stable symplectic Higgs bundle of rank $n/2$ over $X$.
\end{enumerate}
Conversely, such pushforwards are fixed points.

Now consider the automorphism of $\cM(X,\Sp(2n,\C))$ sending $(E,\phi)$ to $(E\otimes L,-\phi)$. Then the smooth fixed points in $\cM(X,\Sp(2n,\C))$ are pushforwards (under $p_L$) of the following two types of vector bundles with extra structure:
\begin{enumerate}
    \item A vector bundle bundle $F$ of rank $n$ on $X_L$ equipped with a homomomorphism $\psi:F\to\lambda^*F\otimes K_L$ and an isomorphism $f:(F,\psi)\xrightarrow{\sim}\lambda^*(F^*,\psi^*)$ such that $\lambda^*f^*=-f$.
    \item Only if $n$ is even: a symplectic vector bundle of rank $n/2$ over $X$ with a homomorphism $\psi:F\to\lambda^*F\otimes K_L$ which preserves the symplectic form.
\end{enumerate}
Conversely, such pushforwards are fixed points.
\end{proposition}

\begin{remark}
Let $\eta$ be the compact involution of $\Sp(2n,\C)$ yielding the maximal compact subgroup $\Sp(2n,\C)\cap U(n)$. Under the given embedding in $\GL(2n,\C)$ this is just conjugating, transposing and inverting, and it can be seen with a simple computation that it commutes with $\theta$ and $\tau$. Then $\eta\theta\in\Aut(\Sp(2n,\C))$ is a conjugate of $\eta\Int_J$ (see \ref{eq-S-symplectic-form}) which, since $\Int_J$ is just taking transpose and inverse on $\Sp(2n,\C)$, is the antiholomorphic involution given by conjugation of matrices. Thus the fixed point real subgroup $\Sp(2n,\C)^{\eta\theta}$ is isomorphic to $\Sp(2n,\R)$. We may also see that $\Sp(2n,\C)^{\eta\tau}$ is isomorphic to $\Sp(n/2,n/2)$. Let $\NSp(2n,\R)$ and $\NSp(n/2,n/2)$ denote the respective normalizers in $\Sp(2n,\C)$. Their quotients by their connected components, which are $\Sp(2n,\R)$ and $\Sp(n/2,n/2)$ respectively, are isomorphic to $\Z/2\Z$. According to Section 8.2 in \cite{PR} the non-abelian Hodge correspondence provides isomorphisms 
\begin{align*}
    \cM(X_L,\Sp(2n,\C)^{\theta},\lieg^{\theta}_{-1})\cong \calR(X_L,\Sp(2n,\R)),\\ \cM(X_L,\Sp(2n,\C)^{\tau},\lieg^{\tau}_{-1})\cong \calR(X_L,\Sp(n/2,n/2)),
    % \\
    % \cM(X_L,\Sp(2n,\C)_{\theta},\lieg^{\theta}_{-1})\cong \calR_L(X,\NSp(2n,\R))\andd\\
    % \cM(X_L,\Sp(2n,\C)_{\tau},\lieg^{\tau}_{-1})\cong \calR_L(X,\NSp(n/2,n/2)),
\end{align*}
where $\calR(X_L,\bullet)$ denotes the character variety of $\pi_1(X_L)$ with values in $\bullet$ 
% and the subscript $L$ in $\calR_L(X,\bullet)$ denotes the subvariety consisting of equivalence classes of representations whose composition with the quotient by the connected component of $\bullet$ yields the holonomy representation of $L$
. 
Thus our statements may be regarded as a correspondence between fixed points in $\mdl(X,\Sp(2n,\C))$ and twisted equivariant $\Sp(2n,\R)$ and $\Sp(n/2.n/2)$-Higgs bundles on $X_L$ respectively.
\end{remark}

\section{Action of the group generated by a \texorpdfstring{$Z(\Spin(n,\C))$-}{Z-}bundle of order two on the moduli space of \texorpdfstring{$\Spin(n,\C)$}{Spin(n,C)}-Higgs bundles.}\label{section-spin}
First we briefly recall the construction of the group $\Spin(n,\C))$. For more details see \cite{lawson_spin_1989}. Consider a complex vector space $V$ of dimension $n$ equipped with an orthogonal form $\omega$. The tensor algebra $T(V):=\bigoplus_kV^{\otimes k}$ has an ideal $I(V,\omega)$ generated by elements of the form $v\otimes v+\omega(v)$, where $v\in V$ and $\omega(v):=\omega(v,v)$, and the quotient is the Clifford algebra $\Cl(V,\omega):=T(V)/I(V,\omega)$. We define $\Pin(n,\C)\subset\Cl(V,\omega)$ to be the multiplicative group generated by elements $v\in V$ such that $\omega(v)=\pm 1$, and $\Spin(n,\C)<\Pin(n,\C)$ to be the subgroup of elements of even length. We have a degree two covering 
$$f:\Spin(n,\C)\to\SO(n,\C)$$
given by associating an element of $\Spin(n,\C)$ to its adjoint action on $V$. The fibre of the automorphism $\Ad_{v_1\dots v_{2k}}$ is $\pm v_1\dots v_{2k}$. We assume that $\omega$ is the standard orthogonal form on $V=\C^n$.

According to \cite{helgason}, the involutions of $\SO(n,\C)$ are given, up to conjugation by $\Int(G)$, by $\Int_{I_{p,q}}$ (for any $p,q$ such that $p+q=n$) and, when $n$ is even, $\Int_{J}$, where 
\begin{equation}\label{eq-def-J-Ipq-Spin}
    J:=\begin{pmatrix}
    0 & I_{n/2}\\
    -I_{n/2} & 0
    \end{pmatrix}\quad\text{and}\quad
I_{p,q}:=
\begin{pmatrix}
    -I_p & 0\\
    0 & I_q
\end{pmatrix}.
\end{equation}
The only inner automorphisms of order two up to conjugation are $\Int_{I_{p,q}}$ for one of $p$ or $q$ even (and $p+q=n$) and $\Int_J$ for $n=4m$. This is of course also true for $\Spin(n,\C)$, since $\Int(\SO(n,\C))\cong \Int(\Spin(n,\C))$. If we assume without loss that $p$ is even, we find 
$$\SO(n,\C)^{\Int_{I_{p,q}}}=\text{S}(\OO(p,\C)\times\OO(q,\C))=(\SO(p,\C)\times\SO(q,\C))\sqcup(\OO^-(p,\C)\times\OO^-(q,\C)),$$
where $\OO^-(p,\C)$ is the non-trivial coset of $\SO(p,\C)$ in $\OO(p,\C)$.

We split our analysis for $\Int_{I_{p,q}}$ in two cases: first consider that $p\ne q$. Then the group $\SO(n,\C)_{\Int_{I_{p,q}}}$ is equal to $\SO(n,\C)^{\Int_{I_{p,q}}}$. Therefore $\Spin(n,\C)_{\Int_{I_{p,q}}}$ is equal to $f^{-1}(\SO(n,\C)^{\Int_{I_{p,q}}})$, which is an extension of $\Spin(n,\C)^{\Int_{I_{p,q}}}$ by either $\Z/2\Z$ or 1. But $\Spin(n,\C)^{\Int_{I_{p,q}}}$ is connected, since $\Spin(n,\C)$ is simply connected \cite{onishchik3}, so that necessarily 
$$\Spin(n,\C)^{\Int_{I_{p,q}}}=f^{-1}(\SO(p,\C)\times\SO(q,\C))=\Spin(p,\C)\times\Spin(q,\C).$$ 
An element in $f^{-1}(\OO^-(p,\C)\times\OO^-(q,\C))$ with order two is $s_p:=iv_1v_{p+1}$, where $v_1\in\C^p$ and $v_{p+1}\in\C^q$ have norm 1. The adjoint action on $\Spin(p,\C)\times\Spin(q,\C)$ is determined by reflections on $v_1$ and $v_{p+1}$, and
$$s_p^2=-v_1v_{p+1}v_1v_{p+1}=v_1v_1v_{p+1}v_{p+1}=(-1)(-1)=1,$$
so that the 2-cocycle $c$ given by Proposition \ref{prop-extensions-isomorphic-twisted-group} is trivial.
Thus we get an isomorphism
$$\Spin(n,\C)_{\Int_{I_{p,q}}}\cong (\Spin(p,\C)\times\Spin(q,\C))\times_{\Int_{s_{p}},1}\Z/2\Z,$$
a semidirect product, where by abuse of notation we are writing the image of the generator instead of the whole homomorphism 
$$\Z/2\Z\to\Aut(\Spin(p,\C)\times\Spin(q,\C)).$$ 

When $n=4m$ we have the remaining case $p=q=2m$, where $\SO(n,\C)_{\Int_{I_{2m,2m}}}$ is generated by $\SO(2m,\C)\times\SO(2m,\C)$ and $J$. In this case $\Spin(4m,\C)_{\Int_{I_{2m,2m}}}$ is generated by $\Spin(2m,\C)\times\Spin(2m,\C)$, one of the elements $J'\in f^{-1}(J)$ and $s_{2m}$. Since $J'^2=J^2=-1$, we have an isomorphism 
$$\Spin(4m,\C)_{\Int_{I_{2m,2m}}}\cong(\Spin(2m,\C)\times\Spin(2m,\C))\times_{\tau,c}(\Z/2\Z)^2.$$
If the generators of $(\Z/2\Z)^2$ are $a$ and $b$ then $\tau_a=\Int_{J'}$, which exchanges the two factors, and $\tau_b=\Int_{s_{2m}}$. The 2-cocycle $c$ is trivial except for the pairs $(b,a)$, $(a,a)$, $(ba,b)$ and $(b,ab)$, which are mapped to $-1$.

Finally, when $n=4m$ we have that $\SO(4m,\C)^{J}\cong\GL(2m,\C)$ is the subgroup of matrices of the form
\begin{equation}\label{eq-def-A-B-matrix}
    \begin{pmatrix}
    A & B\\
    -B & A
\end{pmatrix},
\end{equation}
where the isomorphism sends this matrix to $A+iB$. The extension $\SO(4m,\C)_{J}$ is generated by $\SO(4m,\C)^{J}$ and $I_{2m,2m}$, so that
$$\SO(4m,\C)_{J}\cong\GL(2m,\C)\times_{\theta,1}\Z/2\Z,$$
where $\theta\in\Aut(\GL(2m,\C))$ consists of taking transpose and inverse. Indeed, note that the Lie algebra of $\SO(4m,\C)^{J}$ consists of matrices of the form (\ref{eq-def-A-B-matrix}) with $A$ antisymmetric and $B$ symmetric. The automorphism $\Ad_{I_{2m,2m}}$ sends this matrix to
\begin{equation*}
    \begin{pmatrix}
    A & -B\\
    B & A
\end{pmatrix},
\end{equation*}
which is mapped to $A-iB\in \gl(2m,\C)$, thus the induced automorphism of $\gl(2m,\C)$ sends $M$ to $-M^t$. Since this coincides with the differential of $\theta$, the automorphism $\Int_{I_{2m,2m}}$ must induce $\theta$.

The preimage of each component in $\Spin(n,\C)$ is connected: we may define a path from $1\in\Spin(n,\C)$ to $-1$ in $f^{-1}(\SO(4m,\C)^{J})$ by choosing a vector $v$ of norm $-1$, so that $vv=1$, and rotating one of the factors inside a two-dimensional subspace of $V$ until it gets to $-v$. Since $f$ is a degree two covering given by quotienting by $\{\pm 1\}$, this proves that $f^{-1}(\SO(4m,\C)^{J})$ is connected. Since this is the preimage of the connected component of $\SO(4m,\C)_{J}$, the preimage of the remaining component is also connected. Note that $f^{-1}(\SO(4m,\C)^{J})$ is a 2-cover of $\GL(2m,\C)$, hence it is isomorphic to $\GL(2m,\C)$ itself. 

Summing up, we have:

\begin{proposition}\label{prop-spin}
Let $L$ be a $Z(\Spin(n,\C))$-bundle of order two over $X$, and let $X_L$ be the (connected) étale cover associated to it. Let $f:\Spin(n,\C)\to\SO(n,\C)$ be the natural 2-cover, $J$ the matrix given in (\ref{eq-def-J-Ipq-Spin}) and $J'\in \Spin(n,\C)$ one of the elements in $f^{-1}(J)$.
\begin{enumerate}
    \item If $n\ne 4m$ then, with definitions as in Section \ref{section-prym-narasimhan-ramanan},
    \begin{equation*}
        \bigcup_{p\;\text{even},\, p+q=n}\widetilde{\mdl}(X_L,\Spin(p,\C)\times\Spin(q,\C),\Z/2\Z,\Int_{s_p},1)\subset \mdl(X,\Spin(n,\C))^L
    \end{equation*}
    and
    \begin{equation*}
        \mdl_{ss}(X,\Spin(n,\C))^L\subset\bigcup_{p\;\text{even},\, p+q=n}\widetilde{\mdl}(X_L,\Spin(p,\C)\times\Spin(q,\C),\Z/2\Z,\Int_{s_p},1),
    \end{equation*}
    where $s_p=iv_1v_{p+1}$, 1 is the trivial 2-cocycle and we are calling $\Int_{s_p}$ to the homomorphism $\Z/2\Z\to\Aut(\Spin(p,\C)\times\Spin(q,\C))$ whose image of $-1$ is $\Int_{s_p}$ by abuse of notation.
    
    \item When $n=4m$ there are several possibilities depending on the image of the monodromy representation of $L$: if the monodromy group is $\Z/2\Z\times\Z/2\Z$ then
    \begin{equation*}
    \widetilde{\mdl}(X_L,\Spin(2m,\C)\times\Spin(2m,\C),\Z/2\Z\times\Z/2\Z,\tau,c)\subset \mdl(X,\Spin(n,\C))^L
    \end{equation*}
    and
    \begin{equation*}
        \mdl_{ss}(X,\Spin(n,\C))^L\subset\widetilde{\mdl}(X_L,\Spin(2m,\C)\times\Spin(2m,\C),\Z/2\Z\times\Z/2\Z,\tau,c).
    \end{equation*}
    If the generators of $(\Z/2\Z)^2$ are $a$ and $b$ then $\tau_a=\Int_{J}$ and $\tau_b=\Int_{s_{2m}}$.
    Moreover, $c$ is the 2-cocycle in $Z^2_{\tau}(\Z/2\Z\times\Z/2\Z,Z(\Spin(2m,\C)\times\Spin(2m,\C)))$ which is equal to $-1$ at  $(b,a)$, $(a,a)$, $(ba,b)$ and $(b,ab)$ and trivial at every other pair.

    \item If $n=4m$ and the monodromy group of $L$ is a subgroup of order 2 in $\Z/2\Z\times\Z/2\Z$ whose image under $f$ is $\pm1\in\SO(n,\C)$ then:
    \begin{align*}
    &\widetilde{\mdl}(X_L,\Spin(2m,\C)\times\Spin(2m,\C),\Z/2\Z,\Int_{M},-1)\cup\\ &\cup\widetilde{\mdl}(X_L,\GL(2m,\C),\Z/2\Z,\theta,1)\subset \mdl(X,\Spin(n,\C))^L
    \end{align*}
    and
    \begin{align*}
    \mdl_{ss}(X,\Spin(n,\C))^L\subset&\widetilde{\mdl}(X_L,\Spin(2m,\C)\times\Spin(2m,\C),\Z/2\Z,\Int_M,-1)\cup\\ &\cup\widetilde{\mdl}(X_L,\GL(2m,\C),\Z/2\Z,\theta,1),
    \end{align*}
    where $M=J'$ or $J's_{2m}$ depending on the actual monodromy group and $\theta$ consists of taking transpose and inverse.

    \item If $n=4m$ and the image of the monodromy is $f^{-1}(1)=\pm1\in\Spin(4m,\C)$ then:
    \begin{equation*}
        \bigcup_{p\;\text{even},\, p+q=n}\widetilde{\mdl}(X_L,\Spin(p,\C)\times\Spin(q,\C),\Z/2\Z,\Int_{s_p},1)\subset \mdl(X,\Spin(n,\C))^L
    \end{equation*}
    and
    \begin{equation*}
        \mdl_{ss}(X,\Spin(n,\C))^L\subset\bigcup_{p\;\text{even},\, p+q=n}\widetilde{\mdl}(X_L,\Spin(p,\C)\times\Spin(q,\C),\Z/2\Z,\Int_{s_p},1).
    \end{equation*}
\end{enumerate}
\end{proposition}

\section{Action of a line bundle of order two on \texorpdfstring{$E_7$}{E7}}\label{section-e7}
Let $E_7$ be the simply connected group with exceptional Lie algebra $\lie e_7$. We briefly review the construction of $E_7$ ---for more details see \cite{exceptional}. Recall that the Cayley algebra $\mathfrak C$ is the $\R$-algebra generated by the group of octonions, which is equipped with a conjugation. The exceptional Jordan algebra $\mathfrak J$ consists of $3\times 3$ hermitian matrices over $\mathfrak C$, and we may construct its complexification $\mathfrak J^{\C}$ and define the Freudenthal vector space 
$$\mathfrak B^{\C}:=\mathfrak J^{\C}\oplus\mathfrak J^{\C}\oplus\C\oplus\C.$$
Given two elements $u$ and $v$ in $\mathfrak B^{\C}$, we may define a $\C$-linear mapping $u\times v:\mathfrak B^{\C}\to \mathfrak B^{\C}$. We define $E_7$ as the group of $\C$-linear automorphisms $f:\fr\to\fr$ such that 
$$f(u\times v)f^{-1}=f(u)\times f(v).$$
This has centre $\{\pm 1\}\cong \Z/2\Z$.
Two elements in $E_7$ are $\iota$ and $s$, defined by
$$\iota:\fr\to\fr;\,(u,v,a,b)\mapsto (-iu,iv,-ia,ib)
$$
and
$$
s:\fr\to\fr;\,(u,v,a,b)\mapsto (v,-u,b,-a).$$
They anticommute and their squares are both $-1$, so that the corresponding inner automorphisms have order 2. According to \cite{exceptional}, if $E_6$ is the simply connected group with Lie algebra $\mathfrak e_6$ then $E_7^{\Int_{\iota}}\cong (E_6\times\C^*)/(\Z/3\Z)$, where $\Z/3\Z\subset\C^*$ acts by simultaneous multiplication on both factors. The action of $\tau:=\Int_{s}$ on $(E_6\times\C^*)/(\Z/3\Z)$ is given by transposing and inverting the factor $E_6$ and inverting the factor $\C^*$, and we have
$$(E_7)_{\Int_\iota}\cong (E_6\times\C^*)/(\Z/3\Z)\times_{\tau,-1}\Z/2\Z.$$

Now let $L$ be a line bundle of order $2$ on $X$, which may be regarded as $\Z/2\Z$-bundle, and let $\llambda$ be the subgroup of $H^1(X,\Z/2\Z)$ generated by $L$. Let $X_L$ be associated étale cover. According to the previous paragraph and Theorem \ref{th-prym-narasimhan-ramanan-oscar-ramanan-higgs}, the image of
$$\mdl(X_L,(E_6\times\C^*)/(\Z/3\Z),\Z/2\Z,\tau,-1)$$
is contained in $\mdl(X,E_7)^{\llambda}$ and its intersection with the smooth locus is a union of connected components of $\mdl_{ss}(X,E_7)^{\llambda}$.

\newpage

\chapter{Actions of finite subgroups of the Jacobian on moduli spaces of Higgs bundles}\label{chapter-jacobian}
Let $X$ be a compact Riemann surface and consider the moduli space $\mdl(X,\GL(n,\C))$, isomorphic to the moduli space of Higgs bundles of rank $n$. Applying Theorem \ref{th-prym-narasimhan-ramanan-oscar-ramanan-higgs} we find a description for the locus of fixed points in $\mdl(X,\GL(n,\C))$ under the action of a finite subgroup $\llambda$ in the Jacobian $J(X)\cong H^1(X,\C^*)$. In other words, we give a description of $\mdl(X,\GL(n,\C))^{\llambda}$. For simplicity we make our arguments in the moduli space $M(X,\GL(n,\C))$ of vector bundles of rank $n$, which are also valid for Higgs bundles. We write $\llambda^*:=\Hom(\llambda,\C^*)$. This will be our $\gam$, the Galois group of the étale cover constructed in Section \ref{section-prym-narasimhan-ramanan}.

\section{Antisymmetric pairings and character varieties}The first step is to achieve a better understanding of the character variety
$$\x(\llambda,\Int(\GL(n,\C))):=\Hom(\llambda,\Int(\GL(n,\C)))/\Int(\GL(n,\C)),$$
where $\Int(\GL(n,\C))$ acts by conjugation.

Let $l:\llambda\to\llambda^*$ be an \textbf{antisymmetric pairing}, i.e. a homomorphism whose associated pairing satisfies $\langle\ambda,\ambda\rangle=1$ for each $\ambda\in\llambda$. 
In particular note that 
\begin{equation}\label{eq-antisymmetry}
    1=\langle\ambda\ambda',\ambda\ambda'\rangle=\langle\ambda,\ambda'\rangle\langle\ambda',\ambda\rangle
\end{equation}
for every $\ambda$ and $\ambda'\in\llambda$.
Consider a maximal subgroup $\iota:\Delta\hookrightarrow\llambda$ satisfying that the induced homomorphism $\Delta\to\Delta^*$ is trivial, which we call a \textbf{maximal isotropic subgroup}. In particular, the kernel of $\llambda\to\llambda^*\to\Delta^*$ is equal to $\Delta$. Indeed, if it were bigger than $\Delta$ then there would be an element $\ambda\in\llambda$ such that $\langle\ambda,\delta\rangle=1$ for each $\delta\in\Delta$, hence because of the antisymmetry of $l$ the subgroup of $\llambda$ generated by $\ambda$ and $\Delta$ would pair trivially with itself, contradicting the maximality of $\Delta$. Hence we get an injection
$$f:\llambda/\Delta\hookrightarrow\Delta^*.$$
Consider a homomorphism
$$s:\Delta\to\C^{*n}\subset\GL(n,\C)$$
landing in the subgroup of diagonal matrices of $\GL(n,\C)$. The set of weights for the action of $\Delta$ on $\C^n$ is a subset $\widehat{\Delta}\subseteq s(\Delta)^*\le\Delta^*$. We call $W_{\delta}$ to the weight space in $\C^n$ with weight $\delta\in\Delta^*$.
A matrix whose only non-zero entries are contained in blocks corresponding to each weight space of $s(\Delta)$ is called a \textbf{$\Delta$-matrix}. Given a $\Delta$-matrix $M$, we call $M_{\delta}$ to the block which is the restriction to the weight space $W_{\delta}$ for each $\delta\in\widehat{\Delta}$.
A matrix $M$ which is given by a set of linear isomorphisms $M_{\delta'\delta,\delta'}$ going from a weight space $W_{\delta'}$ to the weight space $W_{\delta'\delta}$, where $\delta$ and $\delta'\in\widehat{\Delta}$, is called a \textbf{permutation matrix}. The element $\delta\in\Delta^*$ is denoted $p(M)$. In particular, $\Delta$-matrices are permutation matrices with trivial $p$-image.

\begin{definition}\label{def-admissible-pair-representative-triple}
A \textbf{representative triple} for $\llambda$ is a triple $(l,\Delta,s)$, where $l:\llambda\to\llambda^*$ is an antisymmetric pairing, $\Delta\le\llambda$ is a maximal isotropic subgroup and $s:\llambda\to\GL(n,\C)$ is a map satisfying that:
\begin{enumerate}
    \item It restricts to a homomorphism $s\vert_{\Delta}:\Delta\to\C^{*n}\subset\GL(n,\C)$.
    \item\label{condition-ints-hom} The map $\Int_s$ is a homomorphism.
    \item The antisymmetric pairing 
    $$\llambda\to\llambda^*;\,\ambda\mapsto(\ambda'\mapsto \sg s_{\ambda'}\sg^{-1}s_{\ambda'}^{-1}),$$
    which is well defined because of (\ref{condition-ints-hom}) and the fact that $\llambda$ is abelian, is equal to $l$. In particular, for every $\ambda\in\llambda$ the matrix $\sg$ is a permutation matrix such that 
    $
    p(\sg)=\iota^*l(\ambda)^{-1}.    
    $
    Note that this condition only depends on the class of $\Int_s$ in $\x(\llambda,\GL(n,\C))$
    \item The image of $s$ consists of permutation matrices whose blocks are multiples of the identity.
    \end{enumerate}
\end{definition}

\begin{lemma}\label{lemma-class-representative-triple}
For every class in the character variety $\x(\llambda,\Int(\GL(n,\C)))$ there exists a representative triple $(l,\Delta,s)$ such that $\Int_s$ is in the class.
\end{lemma}
\begin{proof}
Let 
$$\theta:\llambda\to\Int(\GL(n,\C));\,\ambda\mapsto\Int_{\sg}$$
be a homomorphism.
Since $\llambda$ is abelian, we get an antisymmetric pairing
$$l:\llambda\to\llambda^*;\,\ambda\mapsto(\ambda'\mapsto \sg s_{\ambda'}\sg^{-1}s_{\ambda'}^{-1}).$$
Choose a maximal isotropic subgroup $\iota:\Delta\hookrightarrow\llambda$ and call the corresponding injection
$f:\llambda/\Delta\hookrightarrow\Delta^*.$ Since the elements in $s(\Delta)$ are semisimple (a finite power of each of them is in $\C^*$) and commute with each other, they can be simultaneously diagonalised and so we may assume (after conjugating $\theta$ if necessary) that they are all diagonal. Moreover, we may assume after rescaling that the map $s\vert_{\Delta}$ is a homomorphism and every element of $s(\Delta)$ has some diagonal entries equal to one. The set of weights is a subset $\widehat{\Delta}\subseteq s(\Delta)^*\le\Delta^*$ containing 1. A matrix $M$ whose conjugation by the elements in $s(\Delta)$ induces an element $\delta\in\Delta^*$ must be a permutation matrix with $p(M)=\delta$. Therefore, the homomorphism 
$\llambda/\Delta\xrightarrow{ps}\Delta^*$
induced by $p$ is precisely equal to the multiplicative inverse of $f$. Given a weight $\delta\in\widehat{\Delta}\subset\Delta^*$ there is a coset of weights $\delta f(\llambda/\Delta)$, and the dimensions of all the weight spaces in a given coset must be equal. The subgroup $f(\llambda/\Delta)$ of $\Delta^*$ preserves $\widehat{\Delta}$, since the corresponding cosets are the orbits of the conjugation action of $s(\llambda)$.

%The group $f(\llambda/\Delta)$ acts simply, transitively and freely on each orbit. 
We show that there exists a $\Delta$-matrix $S$ such that $\Int_S\theta\Int_S^{-1}(\llambda)$ consists of conjugations by permutation matrices whose blocks are multiples of the identity. Since conjugating by a $\Delta$-matrix does not change the elements of $s(\Delta)$, it is enough to choose representatives $\ambda\in \llambda$ of each coset $\ol\in\llambda/\Delta$ and find a $\Delta$-matrix $S$ so that $Ss(\ambda)S^{-1}$ is a permutation matrix with blocks in $\C^*$ for each $\ol\in\llambda/\Delta$. Choose a representative $\delta\in\widehat{\Delta}$ of each orbit $\delta f(\llambda/\Delta)\in\widehat{\Delta}/f(\llambda/\Delta)$. Consider the $\Delta$-matrix $S$ determined by $S_{\delta\delta'}=s(\ambda_{\delta'})_{\delta\delta',\delta}^{-1},$
where $\ambda_{\delta'}\in\llambda$ is any element in $f^{-1}(\delta')$ and $\delta\in\widehat{\Delta}$ is the representative of $\delta f(\llambda/\Delta)$. Here $s(\ambda_{\delta'})_{\delta\delta',\delta}$ is just the square matrix representing the restriction of $s(\ambda_{\delta'})$ to the $\delta$-weight space, whose image is the $\delta\delta'$-weight space. We show that $S$ satisfies the claim. 

Indeed, let $\ol\in\llambda/\Delta$ be represented by $\ambda\in\llambda$. We want to show that $$S_{\delta'f(\ol)}s(\ambda)_{\delta'f(\ol),\delta'}S_{\delta'}^{-1}=s(\ambda'')_{\delta'f(\ol),\delta}^{-1}s(\ambda)_{\delta'f(\ol),\delta'}s(\ambda')_{\delta',\delta}$$
is a multiple of the identity for each $\od\in\widehat{\Delta}/f(\llambda/\Delta)$ and $\delta'\in \od$, where $f(\ambda''\Delta)=\delta^{-1}\delta'f(\ol)$ and $f(\ambda' \Delta):=\delta^{-1}\delta'$. But, since $\theta$ is a homomorphism, $s(\ambda'')=s(\ambda)s(\ambda')d$ for some $\Delta$-matrix $d$ such that $d_{\delta}\in\C^*$ for each $\delta\in\widehat{\Delta}$, hence 
$$s(\ambda'')_{\delta'f(\ol),\delta}^{-1}s(\ambda)_{\delta'f(\ol),\delta'}s(\ambda')_{\delta',\delta}=d_{\delta}^{-1}s(\ambda')_{\delta',\delta}^{-1}s(\ambda)_{\delta'f(\ol),\delta'}^{-1}s(\ambda)_{\delta'f(\ol),\delta'}s(\ambda')_{\delta',\delta}=d_{\delta}^{-1},$$
as required.
\end{proof}

\section{The homomorphism \texorpdfstring{$\cct$}{c sub theta}}
Let $\theta:\llambda\to \Int(\GL(n,\C))$ be a homomorphism. When studying fixed points we only care about the class of $\theta$ in $\x(\llambda, \Int(\GL(n,\C)))$ by Theorem \ref{th-fixed-points-oscar-ramanan-principal}, hence by Lemma \ref{lemma-class-representative-triple} we may assume that $\theta=\Int_s$ for some representative triple $(l,\Delta,s)$. The group $\GL(n,\C)^{\theta}$ defined in Section \ref{section-Gtheta} consists of all the invertible $\Delta$-matrices $M$ such that $M_{\delta}=M_{\delta'}$ whenever $\delta$ and $\delta'$ are elements of $\widehat{\Delta}$ in the same coset of $\widehat{\Delta}/f(\llambda/\Delta)$. From Section \ref{section-Gtheta} we have a homomorphism
$$\cct:\GL(n,\C)_{\theta}\to \llambda^*.$$
Recall that the objects $\md_{\llambda}(X,\GL(n,\C)_{\theta})$ in the statement of Theorem \ref{th-fixed-points-oscar-ramanan-principal} are defined as the subvariety of $\md(X,\GL(n,\C)_{\theta})$ consisting of $\GL(n,\C)_{\theta}$-bundles $E$ such that $\cct(E)\cong\llambda$. Thus the second step in the description of $\md(X,\GL(n,\C))^{\llambda}$ is to understand $\cct$.

\begin{lemma}\label{lemma-gltheta}
The image of $\cct$ is equal to
$$\sett:=\{\gamma\in\llambda^*\suhthat \gamma\vert_{\Delta}\widehat{\Delta}=\widehat{\Delta}\andd \dim W_{\gamma\delta}=\dim W_{\delta}\forevery \delta\in\widehat{\Delta}\},$$
where $W_{\delta}$ is the $\delta$-weight space.
Moreover, there is a subgroup $\setp<\GL(n,\C)_{\theta}$ containing the centre$Z(\GL(n,\C)^{\theta})$ of $\GL(n,\C)^{\theta}$ such that the restriction $\cct\vert_{\setp}$ induces an isomorphism
$$\setp/Z(\GL(n,\C)^{\theta})\cong \sett.$$
\end{lemma}
\begin{proof}
For each element $\gamma\in \sett$ we define a permutation matrix $\mtau$ such that $p(\mtau)=\gamma\vert_{\Delta}$ and
\begin{equation}\label{eq-mtau}
    \Int_{s(\ambda)}(\mtau)=\gamma(\ambda)\mtau
\end{equation}
for each $\ambda\in\llambda$. First choose a representative $\delta\in\widehat{\Delta}$ of each coset $\delta f(\llambda/\Delta)\in\widehat{\Delta}/f(\llambda/\Delta)$. Then, for each $\delta'\in \delta f(\llambda/\Delta)$, choose $\ambda\in\llambda$ such that $f(\ol)=\delta'\delta^{-1}$ and define 
$$\mtau_{\delta'\gamma,\delta'}:=s(\ambda)_{\delta'\gamma,\delta\gamma}s(\ambda)_{\delta',\delta}^{-1}\gamma(\ambda)^{-1}.$$ 
This is independent of the choice of $\ambda$ since, for every $\delta_0\in \Delta=\ker f$, we have
\begin{align*}
    s(\ambda\delta_0)_{\delta'\gamma,\delta\gamma}s(\ambda\delta_0)_{\delta',\delta}^{-1}\gamma(\ambda\delta_0)^{-1}=ds(\ambda)_{\delta'\gamma,\delta\gamma}s(\delta_0)_{\delta\gamma}d^{-1}s(\ambda)_{\delta',\delta}^{-1}s(\delta_0)^{-1}_{\delta}\gamma(\ambda)^{-1}\gamma(\delta_0)^{-1}=\\
    s(\ambda)_{\delta'\gamma,\delta\gamma}s(\ambda)_{\delta',\delta}^{-1}\gamma(\ambda)^{-1}\gamma(\delta_0)\gamma(\delta_0)^{-1}=s(\ambda)_{\delta'\gamma,\delta\gamma}s(\ambda)_{\delta',\delta}^{-1}\gamma(\ambda)^{-1},
\end{align*}
where $d\in\C^*$ depends on $\ambda$ and $\delta_0$.
Moreover, for each element $\ambda'\in\llambda$, we have
\begin{align*}
s(\ambda')_{f(\ambda'\Delta)\delta'\gamma,\delta'\gamma}\mtau_{\delta'\gamma,\delta'}s(\ambda')_{{f(\ambda'\Delta)\delta',\delta'}}^{-1}=\\
s(\ambda')_{f(\ambda'\Delta)\delta'\gamma,\delta'\gamma}s(\ambda)_{\delta'\gamma,\delta\gamma}s(\ambda)_{\delta',\delta}^{-1}s(\ambda')_{{f(\ambda'\Delta)\delta',\delta'}}^{-1}\gamma(\ambda)^{-1}
\end{align*}
which, since $s(\ambda)s(\ambda')$ and $s(\ambda\ambda')$ differ by a constant, is equal to
$$s(\ambda\ambda')_{f(\ambda'\Delta)\delta'\gamma,\delta\gamma}s(\ambda\ambda')_{f(\ambda'\Delta)\delta',\delta}^{-1}\gamma(\ambda)^{-1}=\mtau_{f(\ambda'\Delta)\delta'\gamma,f(\ambda'\Delta)\delta'}\gamma(\ambda').
$$
This shows that $\mtau$ satisfies $\Int_{\sg}(\mtau)=\gamma(\ambda)\mtau$ for every $\ambda\in\llambda$. 

Note that, if $\gamma\in\llambda^*$ did not restrict to an element of $\Delta^*$ preserving $\widehat{\Delta}$, there would be no matrix $M$ satisfying $\Int_{\delta}M=\gamma(\delta)M$ for each $\delta\in\Delta$, since this implies that $M$ is a permutation matrix such that $p(M)=\gamma\vert_{\Delta}$. The automorphism $M$ would then send some non-zero weight space to a trivial weight space via an isomorphism, which is absurd. Something similar happens when $W_{\delta}$ and $W_{\gamma\delta}$ have different dimension for some $\delta\in\widehat{\Delta}$. Therefore, the map
$$\{\mtau\}_{\gamma\in \sett}\xrightarrow{\cct}\gamt$$
is a bijection. 

Now we prove that, if $Z(\glt)\cong (\C^*)^{\vert\widehat{\Delta}/f(\llambda/\Delta)\vert}$ is the centre of $\glt$, the set $\setp:=Z(\glt)\{\mtau\}_{\gamma\in\sett}$ is actually a subgroup of $\GL(n,\C)$, so that $\cct$ induces an isomorphism $\setp/Z(\glt)\cong\gamt$. Indeed, it is enough to see that whenever $\gamma$ and $\gamma'\in\llambda^*$ preserve $\widehat{\Delta}$ and the corresponding weight spaces have the same dimension, the matrix $\mtau\mtauu$ is equal to $\mttauu$ multiplied by an element of $Z(\glt)$, which is a $\Delta$-matrix whose restriction to each orbit of $\widehat{\Delta}/f(\llambda/\Delta)$ is constant. Given $\delta'\in\Delta^*$, the chosen representatives $\delta\in\delta'f(\llambda/\Delta)$ and $\delta_0\in\delta'\gamma' f(\llambda/\Delta)$ and elements $\ambda$ and $\ambda_0\in\llambda$ satisfying $f(\ol)=\delta'\delta^{-1}$ and $f(\ambda_0\Delta)=\delta\delta_0^{-1}\gamma'$, we have
$$\mtau_{\delta'\gamma\gamma',\delta'\gamma'}\mtauu_{\delta'\gamma',\delta'}=
s(\ambda\ambda_0)_{\delta'\gamma\gamma',\delta_0\gamma}s(\ambda\ambda_0)_{\delta'\gamma',\delta_0}^{-1}\gamma(\ambda\ambda_0)^{-1}
s(\ambda)_{\delta'\gamma',\delta\gamma'}s(\ambda)_{\delta',\delta}^{-1}\gamma'(\ambda)^{-1},$$
which is equal to
\begin{align*}
    s(\ambda)_{\delta'\gamma\gamma',\delta\gamma\gamma'}s(\ambda_0)_{\delta\gamma\gamma',\delta_0\gamma}s(\ambda_0)^{-1}_{\delta\gamma',\delta_0}s(\ambda)_{\delta'\gamma',\delta\gamma'}^{-1}
    s(\ambda)_{\delta'\gamma',\delta\gamma'}s(\ambda)_{\delta',\delta}^{-1}\gamma\gamma'(\ambda)^{-1}\gamma(\ambda_0)^{-1}=\\
    s(\ambda)_{\delta'\gamma\gamma',\delta\gamma\gamma'}s(\ambda)_{\delta',\delta}^{-1}\gamma\gamma'(\ambda)^{-1}[s(\ambda_0)_{\delta\gamma\gamma',\delta_0\gamma}s(\ambda_0)^{-1}_{\delta\gamma',\delta_0}\gamma(\ambda_0)^{-1}]=\\
    \mttauu_{\delta'\gamma\gamma',\delta'}[s(\ambda_0)_{\delta\gamma\gamma',\delta_0\gamma}s(\ambda_0)^{-1}_{\delta\gamma',\delta_0}\gamma(\ambda_0)^{-1}].
\end{align*}
The expression in brackets only depends on $\delta,\delta_0,\gamma$ and $\gamma'$, so that $\mtau_{\delta'\gamma\gamma',\delta'\gamma'}\mtauu_{\delta'\gamma',\delta'}$ and $\mttauu_{\delta'\gamma\gamma',\delta'}$ differ by an element of $Z(\glt)$ as required.
\end{proof}

\begin{corollary}\label{cor-surjective-cct}
The homomorphism $\cct:\gls\to\llambda^*$ is surjective if and only if the set of weights is the whole $\Delta^*$ and the weight spaces have all the same dimension. In particular, under this assumption $s\vert_{\Delta}$ is injective, $\widehat{\Delta}$ is identified with $\Delta^*$ via $s^*$ and the order of $\Delta$ must divide $n$.
\end{corollary}

% \begin{proof}
% Assume that $\cct$ is surjective. Note that the restriction homomorphism $\iota^*:\llambda^*\to\Delta^*$ is surjective, which means that $\iota^*\cct(P)$ must be the whole group $\Delta^*$. By Lemma \ref{lemma-gltheta} this implies that the set of elements $\delta\in\Delta^*$ preserving $\widehat{\Delta}$ is the whole $\Delta^*$, i.e. $\widehat{\Delta}=\Delta^*$, and the weight spaces all have the same dimension.

% Conversely, since the composition $P\xrightarrow{\cct}\llambda^*\to\Delta^*$ is surjective, it is enough for $\cct$ to be surjective to show that $\cct(s(\Delta))=(\llambda/\Delta)^*<\llambda^*$. Note that $s(\Delta)$ is a subgroup of $\gls$, since the adjoint action of $s(\llambda)$ is given by the multiplicative inverse of $f$, and moreover $\cct(s(\Delta))\subseteq(\llambda/\Delta)^*$ because $\cct(s(\Delta))$ is trivial on $\Delta$. But $\cct\vert_{s(\Delta)}$ is precisely the dual of $f$, since
% $$\cct(s(\delta))(\ambda)=\langle\ambda,\delta\rangle=f(\ambda\Delta)(\delta)=f^*(\delta)(\ambda\Delta)$$
% for every $\delta\in\Delta$ and $\ambda\in\llambda$, so its surjectivity onto $(\llambda/\Delta)^*$ follows from the injectivity of $f$.
% \end{proof}

We call $(l,\Delta,s)$ an \textbf{admissible triple} if any of the two equivalent conditions in the statement of Corollary \ref{cor-surjective-cct} are met. In particular this implies, by injectivity of $s\vert_{\Delta}$ and $f=ps$, that $\theta=\Int_s$ is injective. 

\section{Admissible triples and components of the fixed point variety}
With definitions as in Section \ref{section-fixed-moduli} we show that $\widetilde{M}_{\llambda}(X,\GL(n,\C)_{\theta})$ is empty unless the triple $(l,\Delta,s)$ is admissible. Moreover, we parametrize the components of the fixed point locus using solely antisymmetric pairings.

\begin{corollary}\label{cor-moduli-non-empty}
If $M(X,\GL(n,\C))^{\llambda}$ is the stable fixed point locus, the intersection $$M(X,\GL(n,\C))^{\llambda}\cap \widetilde{M}(X,\GL(n,\C)_{\theta})$$
is empty unless $(l,\Delta,s)$ is an admissible triple.
\end{corollary}
\begin{proof}
The whole point here is that the monodromy of $\llambda$ when considered as an element of $H^1(X,\Hom(\llambda,\C^*))\cong\Hom(\llambda,H^1(X,\C^*))$ is precisely $\llambda^*$ by Proposition \ref{prop-XGamma-connected}. Therefore, according to Theorem \ref{th-fixed-points-oscar-ramanan-principal}, in order for the smooth fixed point locus $M(X,\GL(n,\C))^{\llambda}$ to be non-empty we need $\gamt$ to be isomorphic to $\llambda^*$ via the homomorphism $\cct:\GL(n,\C)_{\theta}\to\llambda^*$. Equivalently, $\cct$ must be surjective.
\end{proof}

Let $p:\xg\to X$ be the étale cover which underlies the $\llambda^*$-bundle $\llambda$.

\begin{proposition}\label{prop-XGamma-connected}
    The cover $\xg$ is connected. In other words, the monodromy of $\llambda$ is equal to $\llambda^*$. 
\end{proposition}
\begin{proof}
We prove this by induction on the minimal number of generators of $\llambda$. Let $\llambda=\langle\ambda_1,\dots,\ambda_k\rangle$ be a choice of generators identifying $\llambda$ with a product of cyclic groups. Let $\llambda':=\langle\ambda_1,\dots,\ambda_{k-1}\rangle\le \llambda$ (this may be trivial). Let $p':X_{\llambda'}\to X$ be the connected étale cover determined by $\llambda$. By induction hypothesis $p':X_{\llambda'}\to X$ is connected and has Galois group $\llambda'^*$, so it is enough to show that the monodromy of $\llambda$ contains an element of $\llambda^*$ which is trivial on $\llambda'$ and has order equal to the order of $\ambda_k$.

First note that the kernel of the pullback $p'^*:J(X)\to J(X_{\llambda'})$ is equal to $\llambda'$: indeed, there is an equivalence of categories between line bundles on $X$ and $\llambda'^*$-equivariant bundles on $X_{\llambda'}$. The only possible actions on the trivial bundle are elements of $(\llambda'^*)^*\cong\llambda'$ multiplied by the trivial action (holomorphic functions are trivial on $X_{\llambda'}$, and the composition of any action with the inverse of the trivial action is just a group of holomorphic functions), and each of those actions determines an element of $\llambda'$ and vice versa. Thus, since no power of $p'^*\ambda_k$ is in $\llambda'$ by assumption, its order is equal to the order of $\ambda_k$. Hence the monodromy of $p'^*\ambda_k$ must contain an element of $\C^*$ with order equal to the order of $\ambda_k$, in other words there is an element of $\pi_1(\xg)\le\pi_1(X)$ whose image $\gamma\in\llambda^*$ via holonomy satisfies that $\gamma(\ambda_k)$ has this order. On the other hand, the fact that $p'^*\llambda'$ is trivial implies that $\gamma$ is trivial on $\llambda'$. 
\end{proof}

\begin{remark}
In fact the proof of Proposition \ref{prop-XGamma-connected} provides an inductive construction of $\xg$ in terms of étale covers associated to line bundles with finite order. It also shows that we may characterize $\xg$ by the condition that $\ker(\pg^*:J(X)\to J(\xg))=\llambda$.
\end{remark}

\begin{lemma}\label{lemma-admissible}
Given an antisymmetric pairing $l:\llambda\to\llambda^*$ such that the order of a maximal isotropic subgroup $\Delta\le\llambda$ divides $n$, there exists a map $s:\llambda\to\GL(n,\C)$ making $(l,\Delta,s)$ an admissible triple. Moreover, the class of $\theta:=\Int_s$ in $\x(\Gamma,\Int_s)$ is unique, i.e. it only depends on $l$.
\end{lemma}
\begin{remark}\label{remark-maximal-isotropic-same-order}
The assumption is independent of the choice of $\Delta$, since $\vert\Delta\vert$ only depends on $n$ and $l$: given two maximal isotropic groups $\Delta$ and $\Delta'$ in $\llambda$, the pairing $l$ induces an injection $\Delta'/\Delta\cap\Delta'\hookrightarrow(\Delta/\Delta\cap\Delta')^*$. Indeed, the kernel of $\Delta'\to\Delta^*$ is $\Delta\cap\Delta'$ because otherwise the group generated by the kernel and $\Delta$ would be isotropic and would strictly contain $\Delta$, and $\Delta\cap\Delta'$ is in the kernel of $\iota^*l(\Delta')$, where $\iota:\Delta\hookrightarrow\llambda$ is the inclusion. This implies that $\vert\Delta'\vert\le\vert\Delta\vert$, and by symmetry we get $\vert\Delta'\vert=\vert\Delta\vert$.
\end{remark} 
\begin{proof}[Proof of Lemma \ref{lemma-admissible}]
Let $\langle\cdot,\cdot\rangle$ be the pairing associated to $l$. We define the restriction of $s$ to $\Delta$ to be an isomorphism to a subgroup of diagonal matrices of $\GL(n,\C)$ whose set of weights is $\Delta^*$ and whose weight spaces have the same dimension (this is possible because $\vert\Delta\vert=\vert\Delta^*\vert$ divides $n$). It is easy to see that any such choice of $s\vert_{\Delta}$ is in the same class of $\x(\Delta,\GL(n,\C))$.

Let $\overline{\ambda}_1,\dots\overline{\ambda}_m$ be a minimal set of generators for $\llambda/\Delta$ (this gives an isomorphism between $\llambda/\Delta$ and a product of finite cyclic groups). Consider the quotient $\kappa:\llambda\to\llambda/\Delta$ and assume that $s$ has been defined for a subgroup $\kappa^{-1}(H)$, where $H<\llambda/\Delta$ is generated by the first $k-1$ generators. Set $\overline{\ambda}:=\overline{\ambda_k}$ with representative $\ambda\in\llambda$ and let $H'\le \llambda/\Delta$ be the subgroup generated by $H$ and $\overline{\ambda}$. We extend $s$ over the subgroup $\kappa^{-1}(H')$ of $\llambda$ as follows: first consider a permutation matrix $\mtau$ defined as in the proof of Lemma \ref{lemma-gltheta}, where $\gamma:=l(\ambda)\in H^*$. The matrix $\mtau$ satisfies 
\begin{equation}\label{eq-commutativity-relations}
    s(h)\mtau s(h)^{-1}=\langle h,\ambda\rangle\mtau
\end{equation}
for each $h\in \kappa^{-1}(H)$ and, in particular, $p(\mtau)=\gamma\vert_{\Delta}$. Note that multiplying $\mtau$ by a $\Delta$-matrix $d$ whose restriction to
\begin{equation*}
    W_{\delta f(H)}:=\bigoplus_{\delta'\in\delta f(H)}W_{\delta'}
\end{equation*}
is a multiple of the identity for each $\delta\in\Delta^*$ preserves (\ref{eq-commutativity-relations}), where $f:H\to\Delta^*$ is the homomorphism induced by $l$. We may choose $d$ such that $(\mtau d)^{\vert \overline{\ambda}\vert}$ is a multiple of $s(\ambda^{\vert \overline{\ambda}\vert})$. Set $s(\ambda):=\mtau d$ and define $s(h\ambda^r):=s(h)s(\ambda)^r$ for each $h\in \kappa^{-1}(H)$ and natural number $r$. Note that the elements of $s(\kappa^{-1}(H))$ commute up to multiplication by $\C^*$. Moreover, given an element $h\in H$, the fact that $s(h\ambda^r)^{\vert \kappa(h\ambda^r)\vert}$ is a multiple of the the correct diagonal matrix (namely $s((h\ambda^r)^{\vert \kappa(h\ambda^r)\vert})$, where $(h\ambda^r)^{\vert \kappa(h\ambda^r)\vert}\in\Delta$) follows from the corresponding properties of the generators and the fact that $s\vert_{\Delta}$ is a homomorphism. The minimality of the generators chosen in $\llambda/\Delta$ implies that the new definitions are compatible with the old ones. The fact that the map $s$ induces the pairing $l$ follows from (\ref{eq-antisymmetry}) and (\ref{eq-commutativity-relations}). 

Finally note that, given $s\vert_{\kappa^{-1}(H)}$, the choice of $s(\ambda)$ (or rather $\Int_{s(\ambda)}$) is determined by $l(\ambda)\vert_{\kappa^{-1}(H)}$ up to conjugation by a $\Delta$-matrix which is constant on each set of blocks defined by a coset of $\Delta^*/f(H)$. Of course, conjugating $s$ by such a matrix does not change $s\vert_{\kappa^{-1}(H)}$. Moreover, by induction $\Int_s\vert_{\kappa^{-1}(H)}$ is unique up to conjugation by an element $\Int_g\in\Int(\GL(n,\C))$, and it can be seen that $gs(\ambda)g^{-1}$ satisfies (\ref{eq-commutativity-relations}) if $s\vert_{\kappa^{-1}(H)}$ is replaced by $gs\vert_{\kappa^{-1}(H)}g^{-1}$. Thus the class of $\theta=\Int_s$ in $\x(\llambda,\GL(n,\C))$ only depends on $l$ and $\Delta$. To see that it does not depend on $\Delta$ notice that, given another maximal isotropic subgroup $\Delta'\le\llambda$, by Lemma \ref{lemma-class-representative-triple} there exists a homomorphism $\theta'=\Int_{s'}$ in the class of $\theta$ such that $(l,\Delta',s')$ is admissible, and by the previous argument this only depends on $l$ and $\Delta'$. Thus the class of $\theta$ is the same for $\Delta$ and $\Delta'$, as required.
\end{proof}

\section{Fixed points as twisted equivariant bundles}

Let $(l,\Delta,s)$ be an admissible triple. In order to apply Theorem \ref{th-prym-narasimhan-ramanan-oscar-ramanan-higgs} we need to describe the homomorphism $\tau$ fitting in the commutative diagramme
\begin{equation}\label{eq-commutative-diagramme-gls-glt}
    \begin{tikzcd}
\setp\arrow[r,"\Int"]\arrow[d,"\cct"]&\Aut(\glt)\\
\llambda^*\arrow[ru,"\tau"]\arrow[u,bend left,"\phi"]
\end{tikzcd}.
\end{equation}
This map lifts the characteristic homomorphism $\gamt\to\Out(\glt)$ of the extension $\GL(n,\C)_{\theta}$. Recall that $\setp$ consists of permutation matrices with constant blocks and the conjugation of a $\Delta$-matrix $M$ by an element $a\in \setp$ is described by the equation
$(aMa^{-1})_{\delta}=M_{(\iota^*\cct(a))\delta}$
for every $\delta\in\Delta^*$. Hence, for each $\ambda\in\llambda$, $\tau(\ambda)$ permutes the blocks of an element in $\Aut(\GL(n,\C)^{\theta})$ according to the multiplication by $\iota^*(\ambda)^{-1}$ in $\Delta^*$. Since an element $M\in\GL(n,\C)^{\theta}$ satisfies $M_{\delta}=M_{\delta'}$ for each $\delta$ and $\delta'$ in the same coset of $\Delta^*/f(\llambda/\Delta)$, the automorphism $\tau(\ambda)$ is actually determined by the coset of $\iota^*(\ambda)^{-1}$ in $\Delta^*/f(\llambda/\Delta)$. Choosing a map $\phi:\llambda^*\to \setp$ such that $\cct \phi=1$ gives an isomorphism $\glt\times_{\tau,c}\llambda^*$ for a suitable 2-cocycle $c\in Z_{\tau}^2(\llambda^*,Z(\glt))$ as in Proposition $\ref{prop-extensions-isomorphic-twisted-group}$. 

Let $\xg$ be the (connected by Proposition \ref{prop-XGamma-connected}) étale cover determined by the $\llambda^*$-bundle $\llambda$. We may simplify the description of polystable $(\tau,c)$-twisted $\llambda^*$-equivariant $\glt$-bundles over the étale cover $\xg\to X$ determined by the $\llambda^*$-bundle $\llambda$ as follows:

\begin{definition}\label{def-chi-equivariant}
The antisymmetric homomorphism $l$ induces a pairing on $\iota^{*-1}(f(\llambda/\Delta))=l(\llambda)\cong\llambda/\ker l$, which we will also denote $\langle\cdot,\cdot\rangle$. Let $E$ be a vector bundle of rank $n/\vert\Delta\vert$ over $\xg$ equipped with an action $\chi$ of $l(\llambda)$ descending to the inverse of the action on $\xg$ and satisfying
    \begin{equation}\label{eq-commutativity-lgamma}
    \chi_{\gamma}\chi_{\gamma'}= \langle\gamma,\gamma'\rangle\chi_{\gamma'}\chi_{\gamma}\forevery\gamma\,\text{and}\,\gamma'\in l(\llambda).
    \end{equation}
We also assume that $\chi_{\gamma}^{\vert\gamma\vert}$ is a fixed multiple of the identity for each $\gamma\in l(\llambda)$.
We call the pair $(E,\chi)$ \textbf{stable} if every $\chi$-invariant sub-bundle has smaller slope (which is the degree divided by the rank), and we call it \textbf{polystable} if it is a direct sum of stable $\chi$-invariant sub-bundles with the same slope. We denote the moduli space parametrizing isomorphism classes of such polystable pairs $M(\xg,l,n)$.
\end{definition}

\begin{lemma}\label{lemma-twisted-equivariant-are-vector-bundles}
    There is a surjective morphism $M(\xg,l,n)\to M(\xg,\glt,\llambda^*,\tau,c)$, where the moduli spaces are given in Definition \ref{def-chi-equivariant} and the end of Section \ref{section-twisted-equivariant-higgs-pairs} resp. The preimage of the stable locus of $M(\xg,\tau,c,\glt,\llambda^*)$ is the subvariety $U$ of stable pairs $(E,\chi)\in M(\xg,l,n)$ such that $E$ is not $\chi\vert_{l(\llambda)}$-equivariantly isomorphic to $\gamma^*E$ for any $\gamma\in \llambda^*/l(\llambda)$. The restriction of the surjection induces an isomorphism 
    $$U/(\llambda^*/l(\llambda))\xrightarrow{\sim} M_{ss}(\xg,\tau,c,\glt,\llambda^*),$$ 
    where $\llambda^*/l(\llambda)$ acts via pullback.
\end{lemma}
\begin{proof}
First let $E'$ be a $(\tau,c)$-twisted $\llambda^*$-equivariant $\glt$-bundle with $\llambda^*$-action $\wchi$. The associated vector bundle $E$ is equal to a direct sum of $\vert\Delta^*\vert=\vert\Delta\vert$ vector bundles of rank $m:=n/\vert\Delta\vert$ such that any two summands corresponding to elements $\delta$ and $\delta'$ in the same coset of $\Delta^*/f(\llambda/\Delta)$ are isomorphic. Set
$E=\bigoplus_{\delta\in\Delta^*}E_{\delta},$
where $E_{\delta'}\cong E_{\delta}$ for every $\delta$ and $\delta'\in\Delta^*$ such that $\delta^{-1}\delta'\in f(\llambda/\Delta)$. The $(\tau,c)$-twisted $\llambda^*$-equivariant action on $E'$ induces, by Remark \ref{remark-associated-vector-bundle}, an action $\chi:l(\llambda)=\iota^{*-1}(f(\llambda/\Delta))\to\Aut(E_1)$ on the summand $E_1$. Note that, for each $\gamma\in l(\llambda)$, the element $\phi(\gamma)\in P$ is a permutation matrix with blocks equal to multiples of the identity and satisfying the commutativity relations (\ref{eq-commutativity-relations}), which implies that it is equal to $s_{\llambda}D$ for any $\ambda\in l^{-1}(\gamma)$ and some $D\in Z(\glt)$ (depending on $\ambda$). But, for any pair of elements $\ambda$ and $\ambda'\in\llambda$, we have $s_{\ambda}s_{\ambda'}=\langle\ambda,\ambda'\rangle s_{\ambda'}s_{\ambda}$, and this does not change if we introduce elements of $Z(\glt)$ because such elements commute with $s_{\ambda'}$ and $s_{\ambda}$. This implies (\ref{eq-commutativity-lgamma}).

Conversely, let $(E_1,\chi)$ be a vector bundle of rank $m$ with an $l(\llambda)$-action $\chi$ satisfying (\ref{eq-commutativity-lgamma}). In particular we have isomorphisms $\gamma^*E_1\cong\gamma'^*E_1$ for each $\gamma$ and $\gamma'$ in the same coset of $\llambda^*/l(\llambda)$, so that we write $\overline{\gamma}^*E_1$ instead of $\gamma^*E_1$, where $\overline{\gamma}$ is the coset of $\gamma\in\llambda^*$. Let $F:=\bigoplus_{\overline{\gamma}\in\llambda^*/l(\llambda)}\overline{\gamma}^*E_1$. We define an action $\wchi$ on $F$ as follows: first we define the action of $l(\llambda)$ on $E_1$ so that 
$
    \wchi_{\gamma}\vert_{E_1}\wchi_{\gamma'}\vert_{E_1}=c(\gamma,\gamma')_1\wchi_{\gamma\gamma'}\vert_{E_1}
$
for every $\gamma$ and $\gamma'\in l(\llambda)$, which is done rescaling the action $\chi$. For example, we may choose a minimal set of generators $\gamma_1,\dots,\gamma_m$ of $\llambda^*$ defining an isomorphism with a product of cyclic groups. We define $\wchi(\gamma_k)\vert_{E_1}:=z_k\chi(\gamma_k)$, where $z_k\in\C^*$ is such that $z_k^{\vert\gamma_k\vert}=\phi(\gamma_k)_1^{\vert\gamma_k\vert}\chi(\gamma_k)^{-\vert\gamma_k\vert}$. Then define $\wchi(\gamma_1^{r_1}\dots\gamma_m^{r_m}):=\gamma\wchi(\gamma_1)^{r_1}\dots\chi(\gamma_m)^{r_m}$, where $\gamma=[\phi(\gamma_1)^{r_1}\dots\phi(\gamma_m)^{r_m}\phi(\gamma_1^{r_1}\dots\gamma_m^{r_m})^{-1}]_1$ is completely determined by $c$. Then the complex number $\wchi(\gamma)\wchi(\gamma')\wchi(\gamma\gamma')^{-1}$ depends solely on the commutators of $\chi(\llambda^*)$, which are the ones of $\phi(\llambda^*)$, and so it is equal to $c(\gamma,\gamma')_1$. We then extend $\wchi\vert_{E_1}$ arbitrarily (but independently of the choice of $(E_1,\chi)$) to get a $\llambda^*$-action on $E_1$ (note that the images of elements in $\llambda^*\setminus l(\llambda)$ are not in $E_1$ though) and, for each $\gamma$ and $\gamma'\in\llambda^*$, we set $\wchi(\gamma')\vert_{\overline{\gamma}^{*}E_1}:=[\phi(\gamma')\phi(\gamma)\phi(\gamma\gamma')^{-1}]_{\gamma}\wchi(\gamma\gamma')\wchi(\gamma)^{-1}.$
Note that this is compatible with the existing definitions because of the definition of $\wchi(l(\llambda))\vert_{E_1}$, and moreover it is independent of the choice of $\gamma$ in $\overline{\gamma}\in\llambda^*/l(\llambda)$. Then we have
\begin{align*}
    \wchi(\gamma'')\vert_{(\overline{\gamma}\overline{\gamma}')^{*}E_1}\wchi(\gamma')\vert_{\overline{\gamma}^{*}E_1}=\\
    [\phi(\gamma'')\phi(\gamma\gamma')\phi(\gamma\gamma'\gamma'')^{-1}\phi(\gamma')\phi(\gamma)\phi(\gamma\gamma')^{-1}]_{\gamma}\wchi(\gamma\gamma'\gamma'')\wchi(\gamma\gamma')^{-1}\wchi(\gamma\gamma')\wchi(\gamma)^{-1}=\\
    [\phi(\gamma'')\phi(\gamma')\phi(\gamma)\phi(\gamma\gamma'\gamma'')^{-1}]_{\gamma\gamma'\gamma''}\wchi(\gamma\gamma'\gamma'')\wchi(\gamma)^{-1}=\\
    [\phi(\gamma'')\phi(\gamma')\phi(\gamma'\gamma'')^{-1}]_{\gamma'\gamma''}[\phi(\gamma'\gamma'')\phi(\gamma)\phi(\gamma\gamma'\gamma'')^{-1}]_{\gamma\gamma'\gamma''}\wchi(\gamma\gamma'\gamma'')\wchi(\gamma)^{-1}=\\
    c(\gamma'',\gamma')\wchi(\gamma'\gamma'')\vert_{\overline{\gamma}^{*}E_1}.
\end{align*}
This shows that, if $E$ is a direct sum of $\vert l(\llambda)/(\llambda/\Delta)^*\vert$ copies of $F$ and $E'$ is the $\glt$-bundle whose extension of structure group to $\GL(n,\C)$ is the bundle of frames of $E$, then $\wchi$ induces a $(\tau,c)$-twisted $\llambda^*$-equivariant action on $E'$.

To see the compatibility between the stability notions suppose that $E_1$ is polystable. The decomposition of $E_1$ into stable $l(\llambda)$-sub-bundles with the same slope induces a decomposition of $F$ into $l(\llambda)$-invariant sub-bundles with the same slope, so we may further assume that $E_1$ is stable. Note that polystability of $E'$ is equivalent to the fact that $F$ is a direct sum of $\llambda^*$-invariant sub-bundles, each of which has no proper $\llambda^*$-invariant sub-bundle with larger or equal slope. First let $\Gamma\le \llambda^*/l(\llambda)$ be the subgroup of elements $\gamma$ such that there are isomorphisms $\psi_{\gamma}:E_1\cong\gamma^*E_1$ compatible with the respective restrictions of the $l(\llambda)$-action. For each homomorphism $\mu:\llambda\to \C^*$ set 
$$F_{\mu}:=\left\{\sum_{\gamma\in\Gamma}\mu_{\gamma}\psi_{\gamma}(e):\;e\in E_1\right\}\subset F.$$
This has no proper $\Gamma$-invariant proper sub-bundles with slope bigger than or equal to the slope of $E_1$ (which is equal to the slope of $F$): any sub-bundle with bigger or equal slope has non-trivial projection on $\gamma^*E_1$ for some $\gamma\in\Gamma$ and so the only possibility is that it has the same slope and the projection is an isomorphism, which implies that it is the whole $F_{\mu}$.
Moreover $\bigoplus \gamma^*E_1\cong\bigoplus_{\mu\in\llambda^*}F_{\mu}$ is a decomposition into $\llambda$-invariant sub-bundles which are isomorphic to $E_1$ via projection, thus having the same slope as $F$. 

We claim that $F_{\mu}':=\sum_{\gamma l(\llambda)\in\llambda^*/l(\llambda)}\wchi_{\gamma}(F_{\mu})\subseteq F$ is stable as a $\llambda^*$-invariant sub-bundle of $F$. Since $F= \bigoplus_{\mu\in\llambda^*} F_{\mu}'$, this will be enough. Let $F_0$ be a $\llambda^*$-invariant sub-bundle of $F_{\mu}'$ with slope equal or greater. The projection on some summand of $F_{\mu}'$ must be non-trivial, and the image must have slope smaller or equal to the summand because it is $\llambda$-invariant. Thus the kernel must have slope greater than or equal to the slope of $F_{\mu}$. Repeating the same argument with the kernel we arrive at a situation where the projection on a summand has trivial kernel. At this point we conclude that the slope of $F_0$ must be equal to the slope of $F_{\mu}$ and the projection is an isomorphism, let's say to $F_{\mu}$ without loss. If the projection to every other summand is zero then we get a sub-bundle of $F_0$ contained in $F_{\mu}$ and, by $\llambda^*$-equivariance, $F_0$ must be the whole $F_{\mu}'$. Otherwise we get a second summand, say $\wchi_{\gamma}(F_{\mu})$ with $\gamma\in\llambda^*\setminus\llambda$, such that the corresponding projection is also an isomorphism. Since projections are $l(\llambda)$-equivariant this yields an $l(\llambda)$-equivariant isomorphism $F_{\mu}\cong\wchi_{\gamma}(F_{\mu})$, which contradicts the definition of $\llambda$.

Finally, given two pairs $(E_1,\chi)$ and $(E_1',\chi')$ in $U\subseteq M(\xg,l,n)$ such that the corresponding bundles $F:=\bigoplus_{\overline{\gamma}\in\llambda^*/l(\llambda)}\overline{\gamma}^*E_1$ and $F':=\bigoplus_{\overline{\gamma}\in\llambda^*/l(\llambda)}\overline{\gamma}^*E'_1$ are $\llambda^*$-equivariantly isomorphic. Because of the definition of $U$ this implies that $E_1$ is isomorphic to $\overline{\gamma}^*E'_1$ for some $\gamma\in\llambda^*/l(\llambda)$, which proves the last statement.

%By Proposition 4.1 in \cite{oscar-suratno} $E_1$ is polystable, hence we may assume after rescaling that $\psi_{\gamma\gamma'}=\psi_{\gamma}\psi_{\gamma'}$ for every $\gamma$ and $\ambda'\in\llambda$.

\end{proof}

\section{Fixed points as pushforwards}

Let $\pd:\xd\to X$ be the étale cover of $X$ determined by the $\Delta^*$-bundle $\Delta<J(X)$. Note that there is a natural isomorphism $\Delta^*\cong\llambda^*/l(\Delta)$ given by the short exact sequence of abelian groups
\begin{equation*}
    1\to (\llambda/\Delta)^*\to\llambda^*\to\Delta^*\to 1
\end{equation*}
and the fact that the homomorphism $\Delta\to (\llambda/\Delta)^*$ induced by $l$ is surjective (this follows from injectivity of $f$). The pairing $l$ determines an isomorphism $\llambda/\Delta\cong l(\llambda)/l(\Delta)$ (this is injective because $l$ induces the injection $f:\llambda/\Delta\to\Delta^*$). 

\begin{definition}\label{def-moduli-vector-bundles-xd}
Consider the moduli space $M(\xd,\GL(n/\vert\Delta\vert,\C))$ of vector bundles of rank $n/\vert\Delta\vert$ over $\xd$. We define a subvariety $M(\xd,\GL(n/\vert\Delta\vert,\C))^{l(\llambda)}$ consisting of isomorphism classes of polystable vector bundles $E$ such that $E\cong l(\ambda)^*E\otimes \pd^*\ambda$ for every $\ambda\in\llambda$. There is a natural pushforward morphism
\begin{equation*}
p_{\Delta*}:M(\xd,\GL(n/\vert\Delta\vert,\C))^{l(\llambda)}\to M(X,\GL(n,\C)).
\end{equation*}
\end{definition}

On the other hand, by the equivalence of categories between $l(\Delta)$-equivariant bundles over $\xg$ and bundles over $\xd$, we have a morphism 
\begin{equation*}
    p_{l}:M(\xg,l,n)\to M(\xd,\GL(n/\vert\Delta\vert,\C))^{l(\llambda)}.
\end{equation*}
Indeed, given a twisted equivariant bundle $E\in M(\xg,l,n)$, the $l(\Delta)$-action on $E$ is equivariant and so $E/l(\Delta)$ is a vector bundle over $\xd$. Moreover, the $l(\llambda)$-action satisfies (\ref{eq-commutativity-lgamma}) and so, for each $\ambda\in\llambda$, $E\cong l(\ambda)^*E$ and the induced $l(\Delta)$-action on $l(\ambda)^*E$ is defined, for each $\delta\in\Delta$, as the pullback of the $l(\delta)$-action on $E$ multiplied by $\langle\ambda,\delta\rangle=(l(\delta)(\ambda))^{-1}$. In other words, $E$ is equivariantly isomorphic to $l(\ambda)^*E$ tensored by the trivial line bundle on $\xg$ equipped with the $l(\Delta)$-action determined by $\ambda^{-1}$. But this is the natural $l(\Delta)$-action on $\pg^*(\ambda)$, so we get $l(\Delta)$-equivariant isomorphisms $E/l(\Delta)\cong l(\ambda)^*(E/l(\Delta))\otimes\pd^*\ambda$ as required. A stable bundle $E\in M(\xg,l,n)$ contains no proper $l(\Delta)$-invariant sub-bundles with greater or equal slope and so the corresponding image in $M(\xd,\GL(n/\vert\Delta\vert,\C))^{l(\llambda)}$ contains no proper sub-bundles with greater or equal slope, so the stability notions are compatible.

\begin{lemma}\label{lemma-twisted-equivariant-pushforward}
The composition
\begin{align*}
    M(X_{\llambda},l,n)\to M(\xg,\tau,c,\glt,\llambda^*)\to M(X,\gls) \to M(X,\GL(n,\C)),
\end{align*}
where the last morphism is defined in Proposition \ref{prop-polystability-extension-structure-group} and the second one is defined in Theorem \ref{th-prym-narasimhan-ramanan-higgs}, factors as
\begin{equation*}
    M(X_{\llambda},l,n)\to M(\xd,\GL(n/\vert\Delta\vert,\C))^{l(\llambda)}\xrightarrow{p_{\Delta*}}M(X,\GL(n,\C)).
\end{equation*}
In particular, by Lemma \ref{lemma-admissible} and Proposition \ref{prop-polystability-extension-structure-group}, the image of the pushforward is independent of the choice of maximal isotropic subgroup $\Delta$.
\end{lemma}

The morphism $p_l$ is not surjective. Indeed, (\ref{eq-commutativity-lgamma}) implies that, for every polystable vector bundle $E\in p_l(M(\xg,l,n))$, the corresponding isomorphisms $f_{\ambda}:E\xrightarrow{\sim}l(\ambda)^*E\otimes \pd^*\ambda$ satisfy
$$(l(\ambda')^*f_{\ambda}\otimes\id_{\pd^*\ambda'})\circ f_{\ambda'}=\langle\ambda,\ambda'\rangle(l(\ambda)^*f_{\ambda'}\otimes\id_{\pd^*\ambda})\circ f_{\ambda}$$
for every $\ambda$ and $\ambda'\in\llambda$, so that $l$ fixes the "commutators" of the isomorphisms. On the other hand, any polystable bundle $E\in M(\xd,\GL(n/\vert\Delta\vert,\C))^{l(\llambda)}$ has isomorphisms $f_{\ambda}:E\xrightarrow{\sim}l(\ambda)^*E\otimes \pd^*\ambda$ such that 
$$f_{\ambda\ambda'}=c(l(\ambda)^*f_{\ambda'}\otimes\id_{\pd^*\ambda})\circ f_{\ambda}$$
for some $c\in\C^*$ depending on $\ambda$ and $\ambda'\in\llambda$. This is trivial if $E$ is stable because $E$ is simple, and in general this follows by redefining the original isomorphisms so that both the left and right hand sides send each stable summand of $E$ to the same stable summand in $l(\ambda\ambda')^*E\otimes\pd^*(\ambda\ambda')$ (and maybe rescaling the restriction to each summand). Thus either way we may choose the isomorphisms $f_{\ambda}$ so that the "commutators" are multiples of the identity, which defines an antisymmetric pairing $l':\llambda\to \llambda^*$. Note that this pairing is trivial on $\Delta$, so that $E$ is actually in the image of $M(\xg,l',\Delta)$. We conclude that
\begin{equation}\label{eq-union-twisted-equivariant-equal-pushforward}
    M(\xd,\GL(n/\vert\Delta\vert,\C))^{l(\llambda)}=\bigcup_{l}p_l(\xg,l,n),
\end{equation}
where $l$ runs over all antisymmetric pairings of $\llambda$ having $\Delta$ as a maximal isotropic subgroup.

We conclude:

\begin{theorem}\label{th-finite-group-jacobian-principal}
We have the inclusions
$$\bigcup_{l}p_{\Delta*}M(\xd,\GL(n/\vert\Delta\vert,\C))^{l(\Gamma)}\subset M(X,\GL(n,\C))^{\Gamma}$$
and
$$M_{ss}(X,\GL(n,\C))^{\Gamma}\subset\bigcup_{l}p_{\Delta*}M(\xd,\GL(n/\vert\Delta\vert,\C))^{l(\Gamma)},$$
where $M(\xd,\GL(n/\vert\Delta\vert,\C))^{l(\Gamma)}$ is defined in Definition \ref{def-moduli-vector-bundles-xd}. The parameter $l$ runs over all antisymmetric pairings on $\llambda$ such that the order of a maximal isotropic subgroup $\Delta$ divides $n$. The choice of $\Delta$ is fixed for each $l$.
\end{theorem}

Similar arguments (replacing Theorems \ref{th-fixed-points-oscar-ramanan-principal} and \ref{th-prym-narasimhan-ramanan-oscar-ramanan-principal} by Theorems \ref{th-fixed-points-oscar-ramanan-higgs} and \ref{th-prym-narasimhan-ramanan-oscar-ramanan-higgs} respectively) give the result for Higgs bundles, which we introduce next.

\begin{definition}\label{def-moduli-higgs-bundles-xd}
Consider the moduli space $\mdl(\xd,\GL(n/\vert\Delta\vert,\C))$ of Higgs bundles of rank $n/\vert\Delta\vert$ over $\xd$. We define a subvariety $\mdl(\xd,\GL(n/\vert\Delta\vert,\C))^{l(\llambda)}$ consisting of isomorphism classes of polystable Higgs bundles $(E,\phi)$ such that $(E,\phi)\cong (l(\ambda)^*E\otimes \pd^*\ambda,l(\ambda)^*\phi)$ for every $\ambda\in\llambda$. There is a natural pushforward morphism
\begin{equation*}
p_{\Delta*}:\mdl(\xd,\GL(n/\vert\Delta\vert,\C))^{l(\llambda)}\to\mdl(X,\GL(n,\C)).
\end{equation*}
\end{definition}

\begin{theorem}\label{th-finite-group-jacobian-higgs}
    We have the inclusions
$$\bigcup_{l}p_{\Delta*}\mdl(\xd,\GL(n/\vert\Delta\vert,\C))^{l(\Gamma)}\subset \mdl(X,\GL(n,\C))^{\Gamma}$$
and
$$\mdl_{ss}(X,\GL(n,\C))^{\Gamma}\subset\bigcup_{l}p_{\Delta*}\mdl(\xd,\GL(n/\vert\Delta\vert,\C))^{l(\Gamma)},$$
where $\mdl(\xd,\GL(n/\vert\Delta\vert,\C))^{l(\Gamma)}$ is defined in Definition \ref{def-moduli-higgs-bundles-xd}. The parameter $l$ runs over all antisymmetric pairings on $\llambda$ such that the order of a maximal isotropic subgroup $\Delta$ divides $n$. The choice of $\Delta$ is fixed for each $l$.
\end{theorem}

\section{The case \texorpdfstring{$\llambda\cong\Z/2\times\Z/2$}{the finite group is generated by two line bundles of order 2}}

We apply Proposition \ref{th-finite-group-jacobian-higgs} to the case when $\llambda$ has order 4 with two generators $a$ and $b$. We consider its action on the moduli space of Higgs bundles of rank $n$ and degree $d$. There are two possible antisymmetric pairings, the trivial one and the non-degenerate one $l$. Maximal isotropic subgroups have order 4 and 2 resp. Thus the smooth fixed point locus is empty unless $n=4m$ or $4m+2$ for some natural number $m$. 
\begin{enumerate}
    \item If $n=4m+2$ then $4$ does not divide $n$ and so the only relevant pairing in the decomposition of Theorem \ref{th-finite-group-jacobian-higgs} is $l$. Let $p_a:X_a\to X$ be the étale cover of $X$ associated to $a$. Call $\mdl(X_a,m+1,d)^b$ to the variety of polystable Higgs bundles $(E,\phi)$ of rank $m+1$ and degree $d$ on $X_{\llambda}$ equipped with an isomorphism $(E,\phi)\cong \gamma^*(E\otimes p_a^*b,\phi)$, where $\gamma$ is the generator of the Galois group of $X_a$. Then $p_{a*}\mdl(X_a,m+1,d)^b$ is contained in the fixed point locus and contains its intersection with the smooth locus of $\mdl(X,n,d)$.
    
    \item If $n=4m$ then we have a second component corresponding to the trivial pairing, namely the pushforward of the moduli space of Higgs bundles of rank $m$ and degree $d$ over $\xg$. In this case the smooth fixed point locus is contained in the union of $p_{\llambda*}(\md(\xg,m,d))$ and $p_{a*}(\md(X_a,2m,d)^b)$, and this union is contained in the fixed point locus.
\end{enumerate}

\newpage

\chapter{Action of a finite group of line bundles of order 2 on the moduli space of symplectic Higgs bundles}\label{chapter-symplectic}

Let $\Sp(2n,\C)$ be the symplectic group. Its centre is $\{1,-1\}\cong\Z/2$. Let $\llambda\subset H^1(X,\Z/2)$ be a finite subgroup. Note that, via the Abel--Jacobi Theorem, we may identify $H^1(X,\Z/2)$ with the subgroup of $J(X)$ consisting of elements of order at most 2. Since $J(X)$ is an abelian variety of dimension equal to the genus $g$ of $X$, this is isomorphic to $(\Z/2\Z)^{2g}$. We adapt the definitions and arguments of Section \ref{chapter-jacobian} to describe $M(X,\Sp(2n,\C))^{\llambda}$, and we divide it into analogous sections to emphasize the parallelism. The same arguments give the result for $\mdl(X,\Sp(2n,\C))^{\llambda}$, which we just state for simplicity. 

Throughout this chapter we regard $\Sp(2n,\C)$ as the subgroup of $\GL(2n,\C)$ preserving the standard symplectic form on $\C^{*2n}$.
%Section \ref{chapter-jacobian}.

\section{Antisymmetric pairings and character varieties}

\begin{definition}\label{def-admissible-pair-representative-triple-sp}
Recall that an \textbf{antisymmetric pairing} on $\llambda$ is a homomorphism $l:\llambda\to\llambda^*:=\Hom(\llambda,\Z/2\Z)\cong\Hom(\llambda,\C^*)$ such that every element of $\llambda$ pairs trivially with itself. We may choose a maximal \textbf{isotropic} subgroup $\iota:\Delta\hookrightarrow\llambda$ where the pairing is trivial, and we have an induced injection $f:\llambda/\Delta\hookrightarrow\Delta^*$.

A \textbf{representative quadruple} for $\llambda$ is a quadruple $(l,\Delta,q,s)$, where $l:\llambda\to\llambda^*$ is an antisymmetric pairing, $\Delta$ is a maximal isotropic subgroup, $q\in\Delta^*$ and $s:\llambda\to\Sp(2n,\C)$ is a map satisfying that:
\begin{enumerate}
    \item It induces a homomorphism $\Int_s:\llambda\to\Int(\Sp(2n,\C))$.
    \item It restricts to a map $\Delta\to\C^{*2n}\subset\GL(2n,\C)$.
    \item The antisymmetric pairing
    $$\llambda\to\llambda^*;\,\ambda\mapsto(\ambda'\mapsto \sg s_{\ambda'}\sg^{-1}s_{\ambda'}^{-1})$$
    is equal to $l$. In particular, for every $\ambda\in\llambda$ the element $\sg$ is a permutation matrix such that
    $
    p(\sg)=\iota^*l(\ambda).    
    $
    \item The image of $s$ consists of permutation matrices whose blocks are multiples of the identity.
    \item For each $\delta\in\Delta$, the eigenspaces of $\C^{2n}$ of the automorphism $s(\delta)$ are either isotropic (i.e. the restriction of the symplectic form is trivial) or symplectic (i.e., the restriction of the symplectic form is non-degenerate). This provides a homomorphism $\Delta\to\Z/2\Z$ which maps $\ambda$ to $-1$ if the eigenspaces are isotropic and to $1$ otherwise, which must be equal to $q$.
    \end{enumerate}
We call $q$ the \textbf{characteristic homomorphism} of $s$.
\end{definition}

The only statement in (5) which is not tautological is the fact that the characteristic homomorphism is equal to $q$: note that (1) imposes that the eigenvalues of $s(\delta)$ be either $\pm 1$ or $\pm i$, thus there are only two eigenspaces. The fact that $s(\delta)$ preserves the symplectic form implies that the eigenspaces are isotropic if and only if the eigenvalues are $\pm i$ (with each isotropic subspace having opposite eigenvalue), and they are $\pm 1$ if and only if the eigenspaces are symplectic. It follows from (2) that the characteristic homomorphism is an actual homomorphism: given $\delta$ and $\delta'\in\Delta$, the eigenspaces of $s(\delta\delta')$ are isotropic/symplectic if and only if the eigenspaces of $s(\delta)s(\delta')$ are isotropic/symplectic (by (1) they differ by a constant), if and only if the eigenvalues of $s(\delta)s(\delta')$ are $\pm i$/$\pm 1$. These are the product of the eigenvalues of $s(\delta)$ and $s(\delta')$, hence the result reduces to checking the possibilities.

\begin{lemma}\label{lemma-class-representative-triple-sp}
For every class in the character variety $\x(\llambda,\Int(\Sp(2n,\C)))$ there exists a representative quadruple $(l,\Delta,q,s)$ such that $\Int_s$ is in the class.
\end{lemma}
\begin{proof}
Let 
$$\theta:\llambda\to\Int(\Sp(2n,\C));\,\ambda\mapsto\Int_{\sg}$$
be a homomorphism.
Since $\llambda$ is abelian, we get an antisymmetric pairing
$$l:\llambda\to\llambda^*;\,\ambda\mapsto(\ambda'\mapsto \sg s_{\ambda'}\sg^{-1}s_{\ambda'}^{-1}).$$
Choose a maximal isotropic subgroup $\iota:\Delta\hookrightarrow\llambda$ and consider the corresponding injection
$f:\llambda/\Delta\hookrightarrow\Delta^*.$ Since the elements in $s(\Delta)$ are semisimple (the square of each of them is in $\Z/2\Z\subset\C^*$) and commute with each other, they can be simultaneously diagonalised by symplectic matrices (this follows, for example, from the fact that every two maximal tori in $\Sp(2n,\C)$ are conjugate to each other). Thus we may assume, after conjugating $\theta$ if necessary, that they are all diagonal.

Let $q\in \Delta^*$ be the characteristic homomorphism of $s$. It follows from the fact that every element of $\Delta$ has order two that, for each $\delta\in\Delta$, the diagonal matrix $s(\delta)$ has eigenvalues $\pm 1$ if $q(\delta)=1$ and $\pm i$ otherwise. Thus we may further assume after rescaling by $\pm 1$ that the first vector of the standard basis of $\C^{2n}$ has eigenvalue $1$ or $-i$. For convenience we redefine the eigenvalues of $s(\delta)$ so that they are always $\pm1$ by multiplying them by $i$ if $q(\delta)=-1$. Note that, under this convention, the first vector of the standard basis always has eigenvalue $1$. We call a simultaneous eigenspace for the action of $s(\Delta)$ a weight space, whose corresponding weight is an element of $s(\Delta)^*\le\Delta^*$ which associates its eigenvalue (in this new sense, so that it takes values $\pm 1$) to each $\delta\in\Delta$. The fact that each weight is a homomorphism $\Delta\to\C^*$ follows from the first vector of the standard basis having weight 1, since $s(\delta\delta')=\pm s(\delta)s(\delta')$ for every $\delta$ and $\delta'\in\Delta$. Note that, since the only possible eigenvalues are $1$ and $-1$, the eigenspaces of each element in $s(\Delta)$ determine that element. This implies that the set of weights is precisely equal to $s(\Delta)^*\le\Delta^*$. 

%Note that, even after rescaling elements of $s(\Delta)$ so that eigenvalues only take values $\pm 1$, we may choose the elements of $s(\llambda)$ such that $s(\gamma)s(\gamma')=\pm s(\gamma\gamma')$ for all $\gamma$ and $\gamma'\in\Gamma$. Indeed, if we choose generators $\gamma_1,\dots,\gamma_k,\dots,\gamma_m$ identifying $\Gamma$ as a product of $\Z/2\Z$'s, where $\langle\gamma_1,\dots,\gamma_l\rangle=\Delta$, we may rescale each element $\gamma_1^{r_1}\dots\gamma_m^{r_m}$ by $c_1^{r_1}\dots c_k^{r_k}$, where $c_i$ is the factor we are multiplying $\gamma_i$ by. The resulting matrices are not symplectic, but rather they multiply the symplectic form by $\pm 1$.
%If $q$ is not trivial we may assume that there exists some $\delta_{q}\in\Delta$ such that 
%$$s(\delta_{q})=
%\begin{pmatrix}
 %   1 & 0\\
  %  0 & -1
%\end{pmatrix}
%$$

As in the proof of Lemma \ref{lemma-class-representative-triple}, the homomorphism 
$\llambda/\Delta\xrightarrow{ps}\Delta^*$
induced by $p$ is equal to the multiplicative inverse of $f$, which in this case is equal to $f$. Given a weight $\delta\in s(\Delta)^*\subset\Delta^*$ we get an orbit of weights $\delta f(\llambda/\Delta)$, and the dimensions of all the weight spaces in a given orbit must be equal. In particular the subgroup $f(\llambda/\Delta)$ of $\Delta^*$ must be contained in $s(\Delta)^*$. For each $\delta\in s(\Delta)^*$ we call $W_{\delta}$ to the corresponding weight space.

%The group $f(\llambda/\Delta)$ acts simply, transitively and freely on each orbit. 
As in the proof of Lemma \ref{lemma-class-representative-triple}, we choose representatives $\ambda\in \llambda$ of each coset $\ambda\Delta\in\llambda/\Delta$ and show that there exists a symplectic $\Delta$-matrix $S$ such that $Ss(\ambda)S^{-1}$ is a permutation matrix with blocks in $\C^*$ for each $\ambda\Delta\in\llambda/\Delta$. Choose a representative $\delta\in s(\Delta)^*$ of each coset $\delta f(\llambda/\Delta)\in s(\Delta)^*/f(\llambda/\Delta)$. We split the argument into three cases: first assume that the characteristic homomorphism $q\in \Delta^*$ is non-trivial and $q\in f(\llambda/\Delta)$. Choose an element $\delta_q\in q^{-1}(-1)$ and consider the subgroup $\ker \delta_q<f(\llambda/\Delta)$, which has degree 2 (because $\delta_q\in f(\llambda/\Delta)^*$ has order at most two and $q(\delta_q)=-1$). Consider the $\Delta$-matrix $S$ determined by $S_{\delta\delta'}=s(\ambda_{\delta'})_{\delta\delta',\delta}^{-1}$ and
$S_{(q\delta)\delta'}=s(\ambda_{\delta'})_{q\delta\delta',q\delta}^{-1},$
where $\delta'\in \ker \delta_q$, $\ambda_{\delta'}\in \llambda$ is any element whose coset $\ambda_{\delta'}\Delta\in\llambda/\Delta$ is equal to $f^{-1}(\delta')$ and $\delta\in s(\Delta)^*$ is the representative of $\delta f(\llambda/\Delta)$.
Note that there is a decomposition into symplectic subspaces
$$\C^{2n}=\bigoplus_{\delta\in\ker \delta_q}W_{\delta}\oplus W_{q\delta},$$ 
where $W_{\delta}$ is isotropic. This follows from the fact that $W_{q\delta}$ is contained in the same eigenspaces as $W_{\delta}$ for elements in $\ker q$, which have symplectic eigenspaces, and in different eigenspaces for elements in $q^{-1}(-1)$, which decompose $\C^{2n}$ into maximal isotropic subspaces. Thus $s(\ambda_{\delta'})_{\delta\delta',\delta}^{-1}\oplus s(\ambda_{\delta'})_{q\delta\delta',q\delta}^{-1}$ may be regarded as an automorphism of $\C^{\dim W_{\delta}}\oplus\C^{\dim W_{\delta}}$ preserving the standard symplectic form, which shows that $S\in \Sp(2n,\C)$. The same calculation as in the proof of Lemma \ref{lemma-class-representative-triple} shows that $Ss(\ambda)S^{-1}$ has constant blocks for every $\ambda\in\llambda$ such that $l(\ambda)\in \ker \delta_q$. We take this for granted hereafter.

Now choose $\ambda_q\in \llambda$ such that $f(\ambda_q\Delta)=q$. Since $\ambda_q$ has order two we know that, for each $\delta\in\ker\delta_q\le s(\Delta)^*$, $s(\ambda_q)_{q\delta,\delta}s(\ambda_q)_{\delta,q\delta}$ is equal to $1$ or $-1$. Since $s(\ambda_q)$ preserves the symplectic form, we have $s(\ambda_q)_{q\delta,\delta}=- s(\ambda_q)_{\delta,q\delta}^{t-1}$ (recall that the restriction of the standard symplectic form to $W_{\delta}\oplus W_{q\delta}$ is the standard one in lower dimension). Thus $s(\ambda_q)_{\delta,q\delta}$ is either symmetric or skew-symmetric and so it is diagonalizable, which implies that it has a square root $s(\ambda_q)_{\delta,q\delta}^{1/2}$ in $\GL(\dim W_{\delta},\C)$. Now let $S'\in\GL(2n,\C)$ be the $\Delta$-matrix such that $S'_{\delta}=s(\ambda_q)_{\delta,q\delta}^{-1/2}$ and $S'_{q\delta}=s(\ambda_q)_{\delta,q\delta}^{1/2}$ for every $\delta\in \ker \delta_q\le s(\Delta)^*$. The matrix $S'$ is a multiple of a symplectic matrix because, for every (skew-)symmetric matrix $A$,
\begin{align*}
    \begin{pmatrix}
    A^{1/2}   &   0\\
    0   &   A^{-1/2}
\end{pmatrix}
\begin{pmatrix}
    0 &   1\\
    -1   &   0
\end{pmatrix}
\begin{pmatrix}
    (A^{1/2})^t   &   0\\
    0   &   (A^{-1/2})^t
\end{pmatrix}&=
\begin{pmatrix}
    0   &   A^{1/2}\\
     -A^{-1/2}   &  0
\end{pmatrix}
\begin{pmatrix}
    \pm A^{1/2}   &   0\\
    0   &   \pm A^{-1/2}
\end{pmatrix}\\&=
\begin{pmatrix}
    0 &   \pm1\\
    \mp1   &   0
\end{pmatrix}.
\end{align*}
Moreover, since $\Int_s$ is a homomorphism, $s(\ambda)$ has constant blocks for every $\ambda\in\llambda$ such that $l(\ambda)\in \ker \delta_q$ and $f(\llambda/\Delta)=\ker \delta_q\sqcup q\ker \delta_q$, we know that $s(\ambda_q)_{\delta\delta',q\delta\delta'}$ is a multiple of $s(\ambda_q)_{q\delta,\delta}$ for every $\delta'\in f(\llambda/\Delta)$. This implies that the restrictions of $S'$ to $W_{\delta}$ and $W_{\delta\delta'}$ differ by a constant and so a permutation matrix $M$ with $p(M)\in f(\llambda/\Delta)$ satisfies $(S'MS'^{-1})_{\delta}=c M_{\delta}$ for some $c\in\C^*$. Therefore, $S's(\ambda)S'^{-1}$ still has constant blocks for every $\ambda\in\llambda$ such that $l(\ambda)\in \ker \delta_q$. Hence, since $\Int_s$ is a homomorphism, in order to prove that the image of $S'sS'^{-1}$ consists of permutation matrices with constant blocks it is enough to show that $S's(\ambda_q)S'^{-1}$ has constant blocks:
$$(S's(\ambda_q)S'^{-1})_{\delta,q\delta}=S'_{\delta}s(\ambda_q)_{\delta,q\delta}S'^{-1}_{q\delta}=s(\ambda_q)_{\delta,q\delta}^{-1/2}s(\ambda_q)_{\delta,q\delta}s(\ambda_q)_{\delta,q\delta}^{-1/2}=1$$
and
\begin{align*}
  (S's(\ambda_q)S'^{-1})_{q\delta,\delta}=S'_{q\delta}s(\ambda_q)_{q\delta,\delta}S'^{-1}_{\delta}=s(\ambda_q)_{\delta,q\delta}^{1/2}s(\ambda_q)_{q\delta,\delta}s(\ambda_q)_{\delta,q\delta}^{1/2}=\\
  \pm s(\ambda_q)_{\delta,q\delta}^{1/2}s(\ambda_q)_{\delta,q\delta}^{-1}s(\ambda_q)_{\delta,q\delta}^{1/2}=\pm1
\end{align*}
for each $\delta\in\ker\delta_q$.

It remains to consider the cases when either $q$ is trivial or $q\notin f(\llambda/\Delta)$. If $q$ is trivial then the weight spaces are all symplectic vector spaces with the standard symplectic form. The elements of $s(\llambda)$ are permutation matrices whose blocks preserve this form, hence the matrix $S$ defined in the proof of Lemma \ref{lemma-class-representative-triple}, which is a $\Delta$-matrix built up from these blocks, must be symplectic. If $q$ is not trivial and $q\notin f(\llambda/\Delta)$, choose $\delta_q\in q^{-1}(-1)$ as before. In this situation the kernel of $\delta_q$ in $f(\llambda/\Delta)$ is equal to the whole $f(\llambda/\Delta)$, so the matrix $S$ which we defined when addressing the first case (with non-trivial $q\in f(\llambda/\Delta)$) does the trick.
\end{proof}

\section{The homomorphism \texorpdfstring{$\cct$}{c sub theta}}

Let $\theta=\Int_s$ be the homomorphism corresponding to a representative quadruple $(l,\Delta,q,s)$.
The group $\Sp(2n,\C)^{\theta}$ consists of all the symplectic $\Delta$-matrices $M$ such that $M_{\delta}=M_{\delta'}$ whenever $\delta$ and $\delta'$ are elements of $ s(\Delta)^*$ in the same orbit of $ s(\Delta)^*/f(\llambda/\Delta)$. From Section \ref{section-Gtheta} we have a homomorphism
$$\cct:\Sp(2n,\C)_{\theta}\to \llambda^*.$$

\begin{lemma}\label{lemma-sptheta}
The image of $\cct$ is equal to
$$\sett:=\{\gamma\in\llambda^*\suhthat \gamma\vert_{\Delta}\in s(\Delta)^*\andd \dim W_{\gamma\delta}=\dim W_{\delta}\forevery \delta\in s(\Delta)^*\}.$$
Moreover, there is a subgroup $\setp<\Sp(2n,\C)_{\theta}$ containing the centre$Z(\Sp(2n,\C)^{\theta})$ of $\Sp(2n,\C)^{\theta}$ such that the restriction $\cct\vert_{\setp}$ induces an isomorphism
$$\setp/Z(\Sp(2n,\C)^{\theta})\cong \sett.$$
\end{lemma}
\begin{proof}
Choose representatives $\delta\in s(\Delta)^*$ for each coset $ \delta f(\llambda/\Delta)\in s(\Delta)^*/f(\llambda/\Delta)$ and define
$$\mtau_{\delta'\gamma,\delta'}:=s(\ambda)_{\delta'\gamma,\delta\gamma}s(\ambda)_{\delta',\delta}^{-1}\gamma(\ambda)^{-1}=s(\ambda)_{\delta'\gamma,\delta\gamma}s(\ambda)_{\delta',\delta}^{-1}\gamma(\ambda)$$ 
as in the proof of Lemma \ref{lemma-gltheta}, where $\delta'\in \delta f(\llambda/\Delta)$ and $\ambda\in\llambda$ is such that $p(s(\ambda))=\delta'\delta^{-1}=\delta'\delta$. Recall that this definition is independent of the choice of $\ambda$. We start by showing that, for each $\gamma\in \sett$, the matrix $M^{\gamma}$ is equal to a symplectic matrix multiplied by an element of the centre$Z(\GL(2n,\C)^{\Int_s})$ of the fixed point subgroup $\GL(2n,\C)^{\Int_s}$ of $\GL(2n,\C)$. Note that this symplectic matrix still satisfies (\ref{eq-mtau}). 

First suppose that the characteristic homomorphism $q$ is trivial. In this situation every weight space for the action of $s(\Delta)$ is symplectic with standard symplectic form. Thus $s(\ambda)_{\delta'\gamma,\delta\gamma}$ and $s(\ambda)_{\delta',\delta}$ are both symplectic. Moreover $\gamma(\ambda)=\pm 1$, which also preserves the symplectic form, so $\mtau\in\Sp(2n,\C)$. 

Now let $q$ be non-trivial. If $q\notin f(\llambda/\Delta)$ then we may assume that, if $\delta$ is the chosen representative for $\delta f(\llambda/\Delta)$, the element $q\delta\in q\delta f(\llambda/\Delta)$ also represents its class and so, on the one hand, for every $\delta'\in \delta f(\llambda/\Delta)$,
$$\mtau_{\delta'\gamma,\delta'}=s(\ambda)_{\delta'\gamma,\delta \gamma}s(\ambda)_{\delta',\delta}^{-1}\gamma(\ambda)$$
and, on the other,
$$\mtau_{\delta'q\gamma,\delta'q}=s(\ambda)_{\delta'q\gamma,\delta q\gamma}s(\ambda)_{\delta'q,\delta q}^{-1}\gamma(\ambda).$$
If $\gamma\ne q$ then, since $s(\ambda)_{\delta'\gamma,\delta \gamma}\oplus s(\ambda)_{\delta'q\gamma,\delta q\gamma}$ and $s(\ambda)_{\delta',\delta}\oplus s(\ambda)_{\delta'q,\delta q}$ preserve the standard symplectic form, the restriction of $\mtau$ to $W_{\delta'}\oplus W_{q\delta'}$ is symplectic for every $\delta'\in s(\Delta)^*$. 

If $q=\gamma$ then $\mtau_{\delta'\gamma,\delta'}$ and $\mtau_{\delta'q\gamma,\delta'q}$ are inverses of each other, in other words $\mtau$ exchanges $W_{\delta'}$ and $W_{q\delta'}$ with the restrictions being inverses of each other. By antisymmetry of the symplectic form the restriction to $W_{\delta'}\oplus W_{q\delta'}$ multiplies the symplectic form by $-1$. If we choose representatives $\delta'$ of each coset $\delta'\{1,q\}\subset\Delta^*$ in such a way that the action of $f(\llambda/\Delta)$ preserves representatives (recall that $q\notin f(\llambda/\Delta)$) then we may define a $\Delta$-matrix $D$ which is equal to 1 when restricted to $W_{\delta'}$ and $-1$ when restricted to $W_{\delta'q}$. It is clear that $D$ is constant on 
\begin{equation*}
    W_{\delta f(\llambda/\Delta)}:=\bigoplus_{\delta'\in\delta f(\llambda/\Delta)}W_{\delta'}
\end{equation*}
for every $\delta\in s(\Delta)^*$, hence it is in $Z(\GL(2n,\C)^{\Int_s})$. Moreover $D$ multiplies the symplectic matrix by $-1$, so $D\mtau$ is symplectic as required.

Now assume that $q\in f(\llambda/\Delta)$ and choose $\ambda_q\in \llambda$ such that $f(\ambda_q)=q$.
Let $d=\pm 1$ such that $s(\ambda\ambda_q)=ds(\ambda)s(\ambda_q)$. Then
\begin{align*}
\mtau_{\delta'q\gamma,\delta'q}&=s(\ambda\ambda_q)_{\delta'q\gamma,\delta\gamma}s(\ambda\ambda_q)_{\delta'q,\delta}^{-1}\gamma(\ambda\ambda_q)\\&=ds(\ambda)_{\delta'q\gamma,\delta q\gamma}s(\ambda_q)_{\delta q\gamma,\delta\gamma}d^{-1}s(\ambda)_{\delta'q,\delta q}^{-1}s(\ambda_q)_{\delta q,\delta}^{-1}\gamma(\ambda\ambda_q)\\&=
[s(\ambda_q)_{\delta q\gamma,\delta\gamma}s(\ambda_q)_{\delta q,\delta}^{-1}\gamma(\ambda_q)]s(\ambda)_{\delta'q\gamma,\delta q\gamma}s(\ambda)_{\delta'q,\delta q}^{-1}\gamma(\ambda),
\end{align*}
where the expression in square brackets only depends on $q,\gamma$ and the coset $\delta f(\llambda/\Delta)$. Since $s(\ambda)_{\delta'\gamma,\delta \gamma}\oplus s(\ambda)_{\delta'q\gamma,\delta q\gamma}$ and $s(\ambda)_{\delta',\delta }\oplus s(\ambda)_{\delta'q,\delta q}$ are symplectic, this shows that the restriction of $\mtau$ to $\bigoplus_{\delta'\in \delta f(\llambda/\Delta)}W_{\delta'}\oplus W_{q\delta'}$ multiplies the symplectic form by a constant, which implies that multiplying this restriction by a suitable complex number (namely $[s(\ambda_q)_{\delta q\gamma,\delta\gamma}s(\ambda_q)_{\delta q,\delta}^{-1}\gamma(\ambda_q)]^{-1/2}$) yields a symplectic transformation. In other words, $\mtau$ is equal to a symplectic matrix multiplied by a diagonal matrix which is constant on $W_{\delta f(\llambda/\Delta)}$ for every $\delta\in s(\Delta)^*$ yields a symplectic matrix. But such diagonal matrix is in $Z(\GL(2n,\C)^{\Int_s})$, as required.

Hereafter we rename the matrices $\mtau$ so that they are symplectic.

Note that, if $\gamma\in\llambda^*$ did not restrict to an element of $ s(\Delta)^*$, there would be no matrix $M$ satisfying $\Int_{\delta}M=\gamma(\delta)M$ for each $\delta\in\Delta$, since this implies that $M$ is a permutation matrix such that $p(M)=\gamma\vert_{\Delta}$. The automorphism $M$ would then send some non-zero weight space to a trivial weight space via an isomorphism, which is absurd. Similarly, given $\delta$ and $\gamma\in s(\Delta)^*,$ the weight spaces $W_{\delta}$ and $W_{\delta\gamma}$ must have the same dimension if there is a permutation invertible matrix $M$ with $p(M)=\gamma$, since the image of $W_{\delta}$ after applying the linear transformation $M$ is $W_{\delta\gamma}$. Therefore, the map
$$\{\mtau\}_{\gamma\in \sett}\xrightarrow{\cct}\gamt$$
is a bijection. 

In the proof of Lemma \ref{lemma-gltheta} we showed that, for every $\gamma$ and $\gamma'\in \sett$, $\mtau M^{\gamma'}=zM^{\gamma\gamma'}$ for some $z\in Z(\GL(2n,\C)^{\Int_s})$. The new definition of $\mtau$ only differs from the old one by an element of $Z(\GL(2n,\C)^{\Int_s})$, so this still holds. Moreover, since $\mtau$ is symplectic for every $\gamma\in \sett$, it follows that $z\in Z(\GL(2n,\C)^{\Int_s})\cap \Sp(2n,\C)$. But, on the one hand, $Z(\GL(2n,\C)^{\Int_s})\cap \Sp(2n,\C)\subset \GL(2n,\C)^{\Int_s}\cap \Sp(2n,\C)=\spt$. On the other, the adjoint action of $Z(\GL(2n,\C)^{\Int_s})$ on $\spt<\GL(2n,\C)^{\Int_s}$ is trivial. Thus $z\in Z(\spt)$ and so we may define $\setp:=Z(\spt)\{\mtau\}_{\gamma\in \sett}$.
\end{proof}

\begin{corollary}\label{cor-surjective-cct-sp}
The homomorphism $\cct:\sps\to\llambda^*$ is surjective if and only if $ s(\Delta)^*$ is identified with $\Delta^*$ via $s^*$ and all the weight spaces have the same dimension. In particular, under this assumption $s\vert_{\Delta}$ is injective and the order of $\ker q\le\Delta$ must divide $n$.
\end{corollary}

\begin{proof}
The first statement follows from Lemma \ref{lemma-sptheta}. To show that the order of $\ker q$ divides $n$ we distinguish two cases: if $q$ is trivial then the restriction of the symplectic form to every weight space is the standard one, which implies that the intersection of each weight space with a copy of $\C^n$ in $\C^{2n}$ has half its dimension. Dividing $n$ by this dimension (which is equal for every weight space) gives the number of weights, which is equal to the order of $\Delta=\ker q$. If $q$ is non-trivial then $\ker q$ is a subgroup of $\Delta$ of degree 2. Since the quotient of $2n$ by the dimension of any weight space gives the number of weights, which is equal to $\vert\Delta\vert=2\vert\ker q\vert$, the order of $\ker q$ must divide $n$ as required.
\end{proof}

We call $(l,\Delta,q,s)$ an \textbf{admissible quadruple} if any of the two equivalent conditions in the statement of Corollary \ref{cor-surjective-cct-sp} are met. In particular this implies, by injectivity of $s\vert_{\Delta}$ and $f=ps$, that $\theta=\Int_s$ is injective.

\section{Admissible quadruples as components of the fixed point variety}

Consider the image $\widetilde{M}(X,\Sp(2n,\C)_{\theta})^{\llambda}$ of the extension of structure group morphism
$$M(X,\Sp(2n,\C)_{\theta})\to M(X,\Sp(2n,\C))$$ 
given in Proposition \ref{prop-polystability-extension-structure-group}.

\begin{corollary}\label{cor-moduli-non-empty-sp}
If $M_{ss}(X,\Sp(2n,\C))^{\llambda}$ is the moduli space of stable fixed points, the intersection $M_{ss}(X,\Sp(2n,\C))^{\llambda}\cap \widetilde{M}(X,\Sp(2n,\C)_{\theta})$ is empty unless $(l,\Delta,q,s)$ is an admissible quadruple.
\end{corollary}
\begin{proof}
The monodromy of $\llambda$ when considered as an element of $$H^1(X,\llambda^*)\cong\Hom(\llambda,H^1(X,\Z/2))\cong\Hom(\llambda,H^1(X,\C^*)),$$
where the last isomorphism follows from the Abel--Jacobi Theorem, is equal to $\llambda^*$ by Proposition \ref{prop-XGamma-connected}. Therefore, according to Theorem \ref{th-fixed-points-oscar-ramanan-principal}, in order for the smooth fixed point locus $M_{ss}(X,\Sp(2n,\C))^{\llambda}$ to be non-empty we need $\gamt$ to be isomorphic to $\llambda^*$ via the homomorphism $\cct:\Sp(2n,\C)_{\theta}\to\llambda^*$. Equivalently, $\cct$ must be surjective.
\end{proof}

\begin{lemma}\label{lemma-admissible-sp}
Given an antisymmetric pairing $l:\llambda\to\llambda^*$, a maximal isotropic subgroup $\Delta\le\llambda$ and an element $q\in\Delta^*$ such that the order of $\ker q$ divides $n$, there exists a map $s:\llambda\to\Sp(2n,\C)$ making $(l,\Delta,q,s)$ an admissible quadruple.
\end{lemma}
\begin{proof}
Let $\langle\cdot,\cdot\rangle$ be the pairing associated to $l$. Choose an injective map $s':\ker q\hookrightarrow\GL(n,\C)$ whose image consists of diagonal matrices with eigenvalues $\pm 1$ and weight spaces of dimension $n/\ker q$. Take the composition 
$$s\vert_{\ker q}:\ker q\xrightarrow{s'\oplus s'}\GL(n,\C)\oplus\GL(n,\C)\hookrightarrow\GL(2n,\C).$$
If $q$ is non-trivial we then choose $\delta_q\in q^{-1}(-1)$, set 
$$s(\delta_q):=\begin{pmatrix}
iI_n & 0\\
0   & -iI_n
\end{pmatrix}
$$
and define $s\vert_{\Delta}$ so that $s$ preserves the group multiplication up to factors of $\pm 1$ (note that $\Delta=\ker q\sqcup\delta_q\ker q$).

Since every element of $\llambda$ has order two, we may find a set of generators $\ambda_1,\dots,\ambda_m$ identifying $\llambda$ with a product of $\Z/2\Z$'s and such that $\Delta$ is generated by the first $k$ of them. We prove the statement by induction on the number of generators of $\llambda/\Delta$, so assume that we have a map $s:\llambda':=\langle\ambda_1,\dots,\ambda_{m-1}\rangle\to\Sp(2n,\C)$ making $(l\vert_{\llambda'},\Delta,q,s)$ an admissible quadruple for some $m>k$. Let $\gamma:=l(\ambda_m)\vert_{\llambda'}$ and $\mtau\in\Sp(2n,\C)$ as defined in the proof of Lemma \ref{lemma-sptheta}. We claim that $(\mtau)^2=\pm 1$. Note that $\gamma$ cannot be in $l(\llambda')$, since $f:\llambda/\Delta\to\Delta^*$ is an injection. Hence we may choose the representative of each coset in $\Delta^*/f(\llambda'/\Delta)$ so that, if $\delta$ represents $\delta f(\llambda'/\Delta)$, so does $\gamma\delta$ in $\gamma\delta f(\llambda'/\Delta)$. Recall that if $q$ is trivial or $q\notin f(\llambda'/\Delta)$ then 
$$\mtau_{\delta'\gamma,\delta'}=s(\ambda)_{\delta'\gamma,\delta\gamma}s(\ambda)_{\delta',\delta}^{-1}\gamma(\ambda)\quad\text{or}\quad s(\ambda)_{\delta'\gamma,\delta\gamma}s(\ambda)_{\delta',\delta}^{-1}\gamma(\ambda)D_{\delta'\gamma,\delta'},$$
where $D$ is defined in the proof of Lemma \ref{lemma-sptheta}, depending on whether $q\ne\gamma$ or $q=\gamma$ resp. On the other hand, if $q$ is non-trivial and $q\in f(\llambda'/\Delta)$,
$$\mtau_{\delta'\gamma,\delta'}=[s(\ambda_q)_{\delta q\gamma,\delta\gamma}s(\ambda_q)_{\delta q,\delta}^{-1}\gamma(\ambda_q)]^{-1/2}s(\ambda)_{\delta'\gamma,\delta\gamma}s(\ambda)_{\delta',\delta}^{-1}\gamma(\ambda),$$
where $f(\ambda)=\delta'\delta$ and $f(\ambda_q\Delta)=q$. In the first case, if $q\ne\gamma$ then
$$(\mtau)^2_{\delta',\delta'}=\mtau_{\delta',\delta'\gamma}\mtau_{\delta'\gamma,\delta'}=s(\ambda)_{\delta',\delta}s(\ambda)_{\delta'\gamma,\delta\gamma}^{-1}\gamma(\ambda)s(\ambda)_{\delta'\gamma,\delta\gamma}s(\ambda)_{\delta',\delta}^{-1}\gamma(\ambda)=1.$$
If $q=\gamma$ then, since $\mtau$ and $D$ anticommute (the restrictions of $D$ to $W_{\delta}$ and $W_{q\delta}$ differ by $-1$ for every $\delta\in\Delta^*$ and $\mtau=M^q$ permutes them), we have, by the previous calculation,
$$1=(\mtau D)^2=-(\mtau)^2D^2=-(\mtau)^2.$$
In the second case, i.e. if $q\in f(\llambda/\Delta)$ is non-trivial,
\begin{align*}
  (\mtau)^2_{\delta',\delta'}=[s(\ambda_q)_{\delta q,\delta}s(\ambda_q)_{\delta q\gamma ,\delta\gamma}^{-1}\gamma(\ambda_q)]^{-1/2}s(\ambda)_{\delta',\delta}s(\ambda)_{\delta'\gamma,\delta\gamma}^{-1}\gamma(\ambda)\\
  [s(\ambda_q)_{\delta q\gamma,\delta\gamma}s(\ambda_q)_{\delta q,\delta}^{-1}\gamma(\ambda_q)]^{-1/2}s(\ambda)_{\delta'\gamma,\delta\gamma}s(\ambda)_{\delta',\delta}^{-1}\gamma(\ambda)=1,
\end{align*}
assuming that we have chosen the square roots in a compatible way.

Now it can be seen that the map $s:\llambda\to\Sp(2n,\C)$ which sends $\ambda \ambda_m^k$ to $s(\ambda)(\mtau)^k$ for each $\ambda\in\llambda'$ induces a homomorphism $\Int_s$: indeed, $(s(\ambda)\mtau)^2$ is equal to $\pm 1$ because the square of each factor is $\pm 1$ and they commute up to multiplication by $\pm 1$. Moreover, $s$ yields an admissible quadruple: for example, the antisymmetry of $l$ and the construction of $s$ imply that $l$ is induced by $s$.
\end{proof}

\begin{remark}\label{remark-non-uniqueness-ints-sp}
Unlike the case of $\GL(n,\C)$ (see Lemma \ref{lemma-admissible}), given the data $l,\Delta$ and $q$, the class of $\theta=\Int_s$ in $\x(\Gamma,\Int(\Sp(2n,\C)))$ is not unique. Indeed, on each inductive step in the proof of Lemma \ref{lemma-admissible-sp} we have a choice: having already defined $s\vert_{\ambda'}$ and letting $\gamma:=l(\ambda_m)\vert_{\llambda'}$, let us think about the possible choices for a permutation matrix in $\Sp(2n,\C)$ with constant blocks satisfying (\ref{eq-commutativity-relations}) and having square $\pm 1$ (any such a matrix would extend $s$ to $\llambda$ by setting $s(\ambda_m):=\mtau$, in such a way that $(l,\Delta,q,s)$ is an admissible quadruple). We know by Lemma \ref{lemma-admissible-sp} that this matrix exists, namely the matrix $\mtau$ defined in the proof of Lemma \ref{lemma-sptheta}. Another symplectic matrix satisfying the requirements is equal to $\mtau C$, where $C$ is a symplectic $\Delta$-matrix which is equal to a constant in $\C^*$ when restricted to $W_{\delta f(\llambda/\Delta)}:=\bigoplus_{\delta'\in\delta f(\llambda/\Delta)}W_{\delta'}$ for every $\delta\in\Delta^*$. We split the argument into three different cases.

If $q\in f(\llambda/\Delta)$ then, according to the proof of Lemma \ref{lemma-admissible-sp} we may choose $\mtau$ to have order 2. Moreover, $C$ is symplectic if and only if its restriction to $W_{\delta f(\llambda/\Delta)}$ is equal to $\pm 1$ for every $\delta\in\Delta^*$, since $W_{\delta f(\llambda/\Delta)}$ is symplectic. The condition $(C\mtau)^2=\pm 1$ implies that the restriction of $C$ to $W_{\delta f(\llambda/\Delta)}$ is equal to $\pm 1$ multiplied by its restriction to $W_{\gamma\delta f(\llambda/\Delta)}$. A representative of the $+1$ case is just $C=1$, whereas we call a representative for the $-1$ case $C^-$. Each sign yields a different class in $\x(\Gamma,\Int(\Sp(2n,\C)))$: the matrices $\mtau$ and $\mtau C^-$ have different order (2 and 4 resp.), since $(\mtau C^-)^2=-(\mtau)^{2}(C^-)^2=-1$. 
Now, if $\Int_{\mtau}$ and $\Int_{\mtau C^-}$ where conjugate in $\Int(\Sp(2n,\C))$ then $\mtau$ would be conjugate to $\pm \mtau C^-$, which has different order, so this is impossible. Hence we have two different choices for the class of $\Int_s$ completely determined by the order of $s(\ambda_m)$.

If $q\notin f(\llambda'/\Delta)$ and $q\ne\gamma$ then, for every $\delta\in\Delta^*$, the subspaces $W_{\delta\llambda'/\Delta}$ and $W_{\delta q\llambda'/\Delta}$ are different. The automorphism $C\vert_{W_{\delta\gamma\llambda'/\Delta}}$ may differ from $C\vert_{W_{\delta\llambda'/\Delta}}$ by a factor, say $k$, but then the fact that $C$ is symplectic implies that $C\vert_{W_{\delta\gamma q\llambda'/\Delta}}=k^{-1}C\vert_{W_{\delta q\llambda'/\Delta}}$. However, conjugation by the matrix
\begin{equation*}
    \begin{pmatrix}
        1 & 0 & 0 & 0 \\
        0 & k^{-1/2} & 0 & 0 \\
        0 & 0 & 1 & 0 \\
        0 & 0 & 0 & k^{1/2}
    \end{pmatrix},
\end{equation*}
which is defined on $W_{\delta\llambda'/\Delta}\oplus W_{\delta\gamma\llambda'/\Delta}\oplus W_{\delta q\llambda'/\Delta}\oplus W_{\delta\gamma q\llambda'/\Delta}$ (note that this preserves the standard symplectic form), reduces the possible cases to $k=1$. Then $C$ is of the form $c\oplus c\oplus c^{-1}\oplus c^{-1}$ on $W_{\delta\llambda'/\Delta}\oplus W_{\delta\gamma\llambda'/\Delta}\oplus W_{\delta q\llambda'/\Delta}\oplus W_{\delta\gamma q\llambda'/\Delta}$, and so $\mtau C$ is conjugate to $\mtau$ via the matrix
\begin{equation*}
    \begin{pmatrix}
        c^{-1/2} & 0 & 0 & 0 \\
        0 & c^{-1/2}& 0 & 0 \\
        0 & 0 & c^{1/2} & 0 \\
        0 & 0 & 0 & c^{1/2} \\
    \end{pmatrix},
\end{equation*}
so there is only one possible choice for the class of $\Int_s$ in this case.

Finally, if $q=\gamma$, the matrix $C$ is symplectic if and only if its restriction to $W_{\delta\llambda'/\Delta}\oplus W_{\delta q\llambda'/\Delta}$ is equal to $c\oplus c^{-1}$ for some $c\in\C^*$. Then $\mtau C$ is conjugate to $\mtau$ via the matrix
\begin{equation*}
    \begin{pmatrix}
        c^{-1/2} & 0 \\
        0 & c^{1/2} \\
    \end{pmatrix},
\end{equation*}
so again there is only one possible choice for the class of $\Int_s$.
\end{remark}

\section{Fixed points as twisted equivariant bundles}

Let $(l,\Delta,q,s)$ be an admissible quadruple. As we saw in section \ref{chapter-jacobian}, the group $\gltt$ consists of $\Delta$-matrices $M$ such that $M_{\delta}=M_{\delta'}$ whenever $\delta\delta'\in f(\llambda/\Delta)$. We have a similar description for $\spt=\gltt\cap\Sp(2n,\C)$.
Let $\setp:=Z(\spt)\{\mtau\}_{\gamma\in \sett}<\sps$ as in the proof of Lemma \ref{lemma-sptheta}. According to Lemma \ref{lemma-sptheta} the group $\sps$ is generated by $\spt$ and $\setp$ and the commutative diagramme (\ref{eq-commutative-diagramme-gls-glt}) restricts to
$$
    \begin{tikzcd}
\setp\arrow[r,"\Int"]\arrow[d,"\cct"]&\Aut(\spt)\\
\llambda^*\arrow[ru,"\tau"]\arrow[u,bend left,"\phi"]
\end{tikzcd},
$$
where $\phi$ is a section of $\cct$. Let $\xg$ be the (connected by Proposition \ref{prop-XGamma-connected}) étale cover determined by the $\llambda^*$-bundle $\llambda$. According to Lemma \ref{lemma-twisted-equivariant-are-vector-bundles} $(\tau,c)$-twisted $\llambda^*$-equivariant $\glt$-bundles over $\xg$ are in correspondence with twisted $l(\llambda)$-equivariant vector bundles of rank $2n/\vert\Delta\vert$ over $\xg$, where the action satisfies (\ref{eq-commutativity-lgamma}). Hence $(\tau,c)$-twisted $\llambda^*$-equivariant $\spt$-bundles may be described in terms of twisted $l(\llambda)$-equivariant vector bundles equipped with some "symplectic form". There are two main possibilities depending on $q$:
\begin{enumerate}
    \item If $q$ is trivial then we get symplectic vector bundles of rank $2n/\vert\Delta\vert$ equipped with an $l(\llambda)$-action satisfying (\ref{eq-commutativity-lgamma}) which respects the symplectic form. Note that a symplectic form on a vector bundle $E$ may be regarded as an isomorphism $\psi:E\cong E^*$ such that, for every $e$ and $e'\in E$, 
    \begin{equation*}
        \psi(e)(e')=-\psi(e')(e)=-\psi^*(e)(e'),
    \end{equation*}
    i.e. $\psi^*=-\psi$.
    \item\label{item-q-sp-2} If $q$ is not trivial then, because of the description of $\spt$, we get a symplectic form on $E\oplus q^*E$, where by abuse of notation $q\in\llambda^*$ is an extension of $q$ and $E$ is a vector bundle of rank $2n/\vert\Delta\vert$ with a symplectic $l(\llambda)$-action satisfying (\ref{eq-commutativity-lgamma}). The symplectic form $\omega$ is codified by an isomorphism $\psi:E\xrightarrow{\sim}q^*E^*$ which is equivariant with respect to the $l(\Delta)$-action on $E$ and the dual of the pullback of this action on $q^*E^*$. The restriction of $\omega$ to $E$ is equal to $\psi$, whereas the restriction to $q^*E$ is equal to $q^*\psi$. 
    The antisymmetry of $\omega$ is equivalent to $q^*\psi^*=-\psi$ since, for every $e\in E$ and $e'\in q^*E$,
    \begin{equation*}
        q^*\psi^*(e)(e')=q^*\psi(e')(e)=-\psi(e)(e').
    \end{equation*}
\end{enumerate}

\begin{remark}
A subtle point of contrast with the case of $\GL(2n,\C)$ is the ambig\"uity in the choice of $\theta$ for each triple $(l,\Delta,q)$, see Remark \ref{remark-non-uniqueness-ints-sp}. This implies that we have different choices for the orders of the different elements of the $l(\llambda)$-action. We have not made explicit mention of this so far, but see Remark \ref{remark-non-unique-ints-pushforward} for a more clear interpretation in terms of pushforwards.
\end{remark}

For each antisymmetric pairing $l$, choose a maximal isotropic subgroup $\Delta\le\Gamma$ and let $M(\xg,l,\Delta,2n,q)$ be the moduli space of triples $(E,\psi,\chi)$, where $E$ is a vector bundle of rank $2n/\vert\Delta\vert$, $\psi:E\xrightarrow{\sim}q^*E^*$ is an isomorphism such that $q^*\psi^*=-\psi$ and $\chi$ is an $l(\llambda)$-action on $E$ respecting $\psi$ and satisfying (\ref{eq-commutativity-lgamma}). Cases (1) and (2) form a partition of these moduli spaces.
We have morphisms $M(\xg,l,\Delta,2n,q)\to \mdl(\xg,\tau,c,\spt,\llambda^*)$ which follow as in Lemma \ref{lemma-twisted-equivariant-are-vector-bundles}
% and Proposition 4.1 in \cite{oscar-suratno}.

\section{Fixed points as pushforwards}
As in the case of $\GL(n,\C)$ we have an interpretation of the moduli spaces $M(\xg,l,\Delta,2n,q)$ in terms of pushforwards.

\begin{definition}\label{def-moduli-vector-bundles-xd-sp}
Let $M(\xd,\GL(2n/\vert\Delta\vert,\C),q)$ be the moduli space of pairs $(E,\psi)$ consisting of a vector bundle $E$ of rank $2n/\vert\Delta\vert$ and an isomorphism $\psi:E\xrightarrow{\sim} q^*E$ such that $q^*\psi^*=-\psi$ (this may be constructed using the techniques in \cite{schmitt:2008}). We define 
$$M(\xd,\GL(2n/\vert\Delta\vert,\C),q)^{l(\llambda)}$$ 
to be the subvariety parametrizing pairs $(E,\psi)$ such that $l(\ambda)^*(E\otimes\pd^*\ambda,\psi)\cong(E,\psi)$ for each $\ambda\in\llambda$ (we define the pullback $l(\ambda)^*\psi$ as it is usual for homomorphisms of vector bundles). The pushforward of vector bundles induces a morphism 
\begin{equation*}
    p_{\Delta^*}: M(\xd,\GL(2n/\vert\Delta\vert,\C),q)^{l(\llambda)}\to M(X,\Sp(2n,\C)).
\end{equation*}
\end{definition}

Note that we are calling $l(\ambda)$ to its coset in $\llambda^*/l(\Delta)$ by abuse of notation, see the discussion before the statement of Lemma \ref{lemma-twisted-equivariant-pushforward}. As in such discussion the correspondence between $l(\Delta)$-equivariant vector bundles on $\xg$ and vector bundles over $\xd$ implies the existence of a morphism $p_{l,q}:M(\xg,l,\Delta,2n,q)\to M(\xd,\GL(2n/\vert\Delta\vert,\C),q)^{l(\llambda)}$. As in the case of $\GL(n,\C)$, this determines a decomposition
\begin{equation}\label{eq-union-twisted-equivariant-equal-pushforward-sp}
    M(\xd,\GL(2n/\vert\Delta\vert,\C),q)^{l(\llambda)}=\bigcup_{l,q} p_{q,l}M(\xg,l,\Delta,2n,q),
\end{equation}
where $l$ runs over all antisymmetric pairings of $\Gamma$ with maximal isotropic subgroup $\Delta$ and $q$ runs over $\Delta^*$.

\begin{remark}\label{remark-non-unique-ints-pushforward}
This decomposition of $M(\xd,\GL(2n/\vert\Delta\vert,\C),q)^{l(\llambda)}$ manifests the different choices that we have for the commutators between the different isomorphisms $f_{\ambda}:(E,\psi)\xrightarrow{\sim}l(\ambda)^*(E\otimes\pd^*\ambda,\psi)$, where $\ambda\in \llambda$ and $(E,\psi)$ represents an element of the moduli space $M(\xd,\GL(2n/\vert\Delta\vert,\C),q)^{l(\llambda)}$, as it is the case with (\ref{eq-union-twisted-equivariant-equal-pushforward}). However, the moduli space $M(\xg,l,\Delta,2n,q)$ itself has different components corresponding to different choices for the orders of this isomorphisms because of the considerations of Remark \ref{remark-non-uniqueness-ints-sp}. In other words, for each non-trivial $\ambda$ we have the possibilities $(l(\ambda)^*f_{\ambda}\otimes\id_{\pd^*\ambda})\circ f_{\ambda}=\pm1$ and in general only one of them is possible (note that multiplying $if_{\ambda}$ does not preserve $\psi$, which is something that we do not need to worry about for $\GL(n,\C)$).
\end{remark}

\begin{lemma}\label{lemma-twisted-equivariant-pushforward-sp}
The morphism
\begin{align*}
    M(X_{\llambda},l,n,q)\to M(\xg,\tau,c,\spt,\llambda^*)\to M(X,\sps) \to M(X,\Sp(n,\C)),
\end{align*}
where the last morphism is defined in Proposition \ref{prop-polystability-extension-structure-group} and the second one is defined in Theorem \ref{th-prym-narasimhan-ramanan-higgs}, factors as
\begin{equation*}
    M(X_{\llambda},l,n,q)\to M(\xd,\GL(2n/\vert\Delta\vert,\C),q)^{l(\llambda)}\xrightarrow{p_{\Delta*}}M(X,\Sp(n,\C)).
\end{equation*}
In particular, by Lemma \ref{lemma-admissible} and Proposition \ref{prop-polystability-extension-structure-group}, the image of the pushforward is independent of the choice of maximal isotropic subgroup $\Delta$.
\end{lemma}

We conclude:

\begin{theorem}\label{th-finite-group-jacobian-sp}
We have the inclusions
$$\bigcup_{l,q}p_{\Delta*}M(\xd,\GL(2n/\vert\Delta\vert,\C),q)^{l(\Gamma)}\subset M(X,\Sp(2n,\C))^{\Gamma}$$
and
$$M_{ss}(X,\Sp(2n,\C))^{\Gamma}\subset\bigcup_{l,q}p_{\Delta*}M(\xd,\GL(2n/\vert\Delta\vert,\C),q)^{l(\Gamma)},$$
where $M(\xd,\GL(2n/\vert\Delta\vert,\C),q)^{l(\Gamma)}$ is defined in Definition \ref{def-moduli-vector-bundles-xd-sp}. The parameter $l$ runs over all antisymmetric pairings on $\llambda$ such that the order of a maximal isotropic subgroup $\Delta$ divides $n$. The choice of $\Delta$ is fixed for each $l$. The parameter $q$ runs over elements of $\Delta^*$.
\end{theorem}

\begin{proof}
Follows from Lemmas \ref{lemma-class-representative-triple-sp}, \ref{lemma-admissible-sp} and \ref{lemma-twisted-equivariant-pushforward-sp}, Theorem \ref{th-fixed-points-oscar-ramanan-principal}, Corollary \ref{cor-moduli-non-empty-sp} and (\ref{eq-union-twisted-equivariant-equal-pushforward-sp}).
\end{proof}

\begin{definition}\label{def-moduli-vector-bundles-xd-sp-higgs}
Let $\mdl(\xd,\GL(2n/\vert\Delta\vert,\C),q)$ be the moduli space of triples $(E,\phi,\psi)$ consisting of a Higgs bundle $(E,\phi)$ of rank $2n/\vert\Delta\vert$ and an isomorphism $\psi:(E,\phi)\xrightarrow{\sim} q^*(E,\phi)$ such that $q^*\psi^*=-\psi$ (this may be constructed using the techniques in \cite{schmitt:2008}). We define $\mdl(\xd,\GL(2n/\vert\Delta\vert,\C),q)^{l(\llambda)}$ to be the subvariety parametrizing triples $(E,\phi,\psi)$ such that $l(\ambda)^*(E,\phi\otimes\pd^*\ambda,\phi,\psi)\cong(E,\phi,\psi)$ for each $\ambda\in\llambda$ (we define the pullback $l(\ambda)^*\psi$ as it is usual for homomorphisms of vector bundles). The pushforward of vector bundles induces a morphism 
\begin{equation*}
    p_{\Delta^*}: \mdl(\xd,\GL(2n/\vert\Delta\vert,\C),q)^{l(\llambda)}\to \mdl(X,\Sp(2n,\C)).
\end{equation*}
\end{definition}

\begin{theorem}\label{th-finite-group-jacobian-sp-higgs}
We have the inclusions
$$\bigcup_{l,q}p_{\Delta*}\mdl(\xd,\GL(2n/\vert\Delta\vert,\C),q)^{l(\Gamma)}\subset \mdl(X,\Sp(2n,\C))^{\Gamma}$$
and
$$\mdl_{ss}(X,\Sp(2n,\C))^{\Gamma}\subset\bigcup_{l,q}p_{\Delta*}\mdl(\xd,\GL(2n/\vert\Delta\vert,\C),q)^{l(\Gamma)},$$
where $\mdl(\xd,\GL(2n/\vert\Delta\vert,\C),q)^{l(\Gamma)}$ is defined in Definition \ref{def-moduli-vector-bundles-xd-sp-higgs}. The parameter $l$ runs over all antisymmetric pairings on $\llambda$ such that the order of a maximal isotropic subgroup $\Delta$ divides $n$. The choice of $\Delta$ is fixed for each $l$. The parameter $q$ runs over elements of $\Delta^*$.
\end{theorem}

\section{An abelianization phenomenon}
In particular, if $2n =2^m$ is some power of $2$ and the genus of $X$ is greater than $m/2$, there are subgroups $\llambda\le H^1(X,\Z/2\Z)$ of order $2^{m}$. If $l$ is the trivial pairing then the pushforward $p_{\Delta*}M(\xd,\C^*,q)^{l(\Gamma)}$ is a component of the fixed point locus, where $q$ is any non-trivial element of $\llambda^*$. In this situation $M(\xd,\C^*,q)^{l(\Gamma)}$ is just the moduli space of isomorphism classes of pairs $(L,\psi)$, where $L$ is line bundle over $\xg$ and $\psi:q^*L\xrightarrow{\sim}L^*$ is an isomorphism satisfying $q^*\psi=-\psi^*$, and the image $p_{\Delta*}(L,\psi)$ is just the pushforward of $L$ equipped with the induced symplectic form. This "abelianization" phenomenon is a manifestation, with the extra structure of the symplectic form, of the corresponding description of a component in the fixed point locus of $M(X,\GL(2n,\C))$ under the action of a subgroup of order $2n$ in $J(X)$.

\newpage

\chapter{Fixed points when the action does not involve tensorization}\label{chapter-alpha-trivial}

Throughout this Section $G$ is a connected semisimple complex Lie group with centre $Z$ and Lie algebra $\lie g$, $X$ is a compact Riemann surface and $\Gamma$ is a subgroup of $\Aut(X)\times\Out(G)\times\C^*$. From Section \ref{section-action} we have an action of $\Gamma$ on $\cM(X,G)$ given by homomorphisms $\eta:\Gamma\to\Aut(X), a:\Gamma\to\Out(G)$ and $\mu:\Gamma\to\C^*$. We also fix a lift $\theta$ of $a$. Note the absence of $\alpha\in Z^1_{a}(\Gamma,H^1(X,Z))$, which is trivial in this Chapter. 

Let us denote the fixed-point subvariety of $\cM(X,G)$ under the above defined action of $\Gamma$ by $\cM(X,G)^{\Gamma}$. This section gives a description of $\cM(X,G)^{\Gamma}$ which is a particular case of the answer in Chapter \ref{chapter-general} when $\eta$ is injective. When $\eta$ is not injective it gives a description which is refined in Chapter \ref{chapter-general}.

\section{The forgetful morphism}
\label{section-forgetful-morphism}
%%%%%%%%%%%%%%%%%%%%%%%%%%%%%%%%%%%%%%%%%%%
Let $c$ a 2-cocycle in $Z^2_{a}(\Gamma,Z)$ and $\mu$ a character of $\Gamma$.

\begin{proposition}\label{underlyingsemi}
Let $(E,\cdot,\phi)$ be a $(\theta,c,\mu)$-twisted $\Gamma$-equivariant $G$-Higgs bundle as defined in Section \ref{section-twisted-equivariant-higgs-pairs}.
\begin{enumerate}
  \item  If $(E,\phi)$ is (semi)stable then $(E,\cdot,\phi)$ is (semi)stable.
  \item If $(E,\cdot,\phi)$ is semistable then $(E,\phi)$ is semistable.
  \item If  $(E,\cdot,\phi)$ is polystable 
        then $(E,\phi)$ is polystable.
 \end{enumerate}
\end{proposition}
\begin{proof}
The statement $(1)$ is obvious. For the statement $(3)$ let $(E,\cdot,\phi)$ be a polystable twisted equivariant $G$-Higgs bundle. Then by Theorems \ref{EH1-equivariant} and \ref{EH1} the underlying $G$-Higgs bundle is polystable.

Now we prove $(2)$. Suppose $(E,\cdot,\phi)$ is semistable
 but $(E,\phi)$ is not semistable. Note that, in char $0$, $(E,\phi)$ is semistable if and only if 
 $(\ad(E),\ad(\phi))$ is semistable ( see \cite[Lemma 2.10]{AB}). 
 Thus $(\ad(E),\ad(\phi))$ is not semistable. Then there is a unique filtration of $(\ad(E),\ad(\phi))$ 
 by $\ad(\phi)$ invariant sub-bundles
 \[0=F_0\subset F_1\subset \cdots F_{n-1}\subset F_{n}=\ad(E)\] such that each $(F_{i}/F_{i-1},\ad(\phi)|_{F_i/F_{i-1}})$ is semistable and
 $\mu(F_{i}/F_{i-1})<\mu(F_{i-1}/F_{i-2})$, for all $i=1,2,\cdots,n$. Then $n$ is odd and $F_{n+1/2}$ is a parabolic subalgebra
 bundle of $\ad(E)$ (see proof of \cite[Lemma 2.5]{bhp} and \cite[Lemma 2.11]{AB}).
 Moreover, as $(\ad(E),\ad(\phi))$ is not semistable we have $\mbox{deg}(F_{n+1/2})> 0$. Let $\Ad(E):=E\times^{G} G$ be the group scheme associated 
 to $E$ for the action of $G$ on itself.
 By \cite[Lemma 4]{AAB} there exists a parabolic 
 sub group scheme $\underline{P}\subset \Ad(E)$ such that the associated Lie algebra bundle is $F_{n+1/2}$. By uniqueness of Harder Narashimahn 
 filtration we get $F_{n+1/2}\cdot\gamma=F_{n+1/2}$ for all $\gamma\in \Gamma$.
 Now we can show that there exists a parabolic subgroup $P\subset G$ and a reduction of structure group $E_{\sigma}\subset E$ to $P$ such that 
 $\Ad(E_{\sigma})=\underline{P}$ (see the proof \cite[Lemma 2.11]{AB}. Therefore, $F_{n+1/2}=E_{\sigma}(\mathfrak{p})$.
 But this would contradict our assumption that $(E,\cdot,\phi)$ is semistable. Therefore, $(\ad(E),\mbox{ad}(\phi))$ is 
 semistable.
 
 %The sub-bundle $F_{n+1/2}$ is of the form $\mbox{ad}(E_{P})$ where $P$ is a parabolic subgroup and 
 %$E_P\subset E$ is a holomorphic reduction of the structure group. Using the uniqueness of Harder-Narashimahn filtration 
 %we can easily show that the $P$-bundle $E_{P}\subset E$ is defined by a $\Gamma$-invariant reduction $\sigma$ of
 %structure group of $E$ to $P$ and clearly, by construction $\phi\in H^0(X,E_{P}(\mathfrak{p})\otimes K_X)$. 
 %But this would contradict our assumption that $(E,\{\widetilde{\eta}_{\gamma}\},\phi)$ is semistable.

\end{proof}

Let $\sigma_{x_i}\in Z^1_{c_{x_i}}(\Gamma_{x_i},G)$ for each isotropy point $x_i\in X$ (notation as in Section \ref{section-isotropy}). By Proposition \ref{underlyingsemi} here exists a forgetful morphism
\begin{equation}\label{eq-forgetful-morphism}
    \cM(X,G,\Gamma,\theta,c,\mu,\sigma)\to \cM(X,G).
\end{equation}
We denote the image of the forgetful map inside  
$\cM(X,G)$ by  $\widetilde{\cM}(X,G,\Gamma,\theta,c,\mu,\sigma)$.

The image of the forgetful map consists of those isomorphism classes of polystable $G$-Higgs bundles which admit a 
$(\theta,c,\mu)$-twisted $\Gamma$-equivariant structure.
Now if a $G$-Higgs bundle $(E,\phi)$ admits a $(\theta,c,\mu)$-twisted $\Gamma$-equivariant structure then, by definition of twisted equivariant structures, 
we have 
\[(E,\phi)\cong (\eta_{\gamma}^*\theta^{-1}_{\gamma}(E),\mu(\gamma)\eta_{\gamma}^*\theta^{-1}_{\gamma}(\phi)),\] where $\cong$ denotes isomorphism of $G$-Higgs bundles.
As points of $\mathcal{M}(X,G)$ consist of isomorphism classes of polystable $G$-bundles we immediately have the following.
\begin{proposition}\label{obvious-fix}
 $\widetilde{\cM}(X,G,\Gamma,\theta,c,\mu,\sigma)\subset \cM(X,G)^{\Gamma}$.
\end{proposition}

%%%%%%%%%%%%%%%%%%%%%%%%%%%%%%%%%%%%%%%%%
\section{Fixed points and simplicity}\label{section-fixed-points-and-simplicity-alpha-trivial}
%%%%%%%%%%%%%%%%%%%%%%%%%%%%%%%%%%%%%%%%%

Recall that a $G$-Higgs bundle $(E,\phi)$ is said to be {\bf simple} if $\mbox{Aut}(E,\phi)=Z$.
% Let $(E,\phi)$ be a $G$-Higgs bundle. 
% Note that an isomorphism $E\cong \eta_{\gamma}^*\theta^{-1}_{\gamma}(E)$, for $\gamma\in \Gamma$, produces a Higgs a field 
% $\eta_{\gamma}^*\theta^{-1}_{\gamma}(\phi)$ of $\eta_{\gamma}^*\theta^{-1}_{\gamma}(E)$.  
% We say that $(E,\phi)\cong  (\eta_{\gamma}^*\theta^{-1}_{\gamma}(E),\mu(\gamma)\eta_{\gamma}^*\theta^{-1}_{\gamma}(\phi))$
% if  $E\cong \eta_{\gamma}^*\theta^{-1}_{\gamma}(E)$ and the Higgs field $\phi$ makes the following diagramme commutative

% $$
% \begin{matrix}
% E(\mathfrak{g})\otimes K_X & \cong & 
% \eta_\gamma^*\theta_{\gamma}(E)(\mathfrak{g})\otimes K_X\\
% ~\Big\uparrow\varphi && ~\,\text{  }~\,\text{ }\Big\uparrow \mu_{\gamma}\eta_{\gamma}^*\theta^{-1}_{\gamma}(\phi)\\
% X & = & X
% \end{matrix}.
% $$
%Let $\alpha\in Z^1(\Gamma,H^1(X,Z))$ be a 
%$1$-cocycle where $\Gamma$ acts on $H^1(X,Z)$ via the action of $\Aut(X)\times \Out(G)$ on $H^1(X,Z)$.

\begin{proposition} \label{simple}
Let $\theta:\Gamma\to \Aut(G)$ be a lift of $a:\Gamma\to \Out(G)$, and 
let $(E,\phi)$ be a simple $G$-Higgs bundle over $X$ such that 
 $(E,\phi)\cong (\eta_{\gamma}^*\theta^{-1}_{\gamma}(E),\mu_{\gamma}\eta_{\gamma}^*\theta^{-1}_{\gamma}(\phi))$ for each $\gamma\in\Gamma$. 
Then $(E,\phi)$ admits a $(\theta,c,\mu)$-twisted $\Gamma$-equivariant structure for some
$c\in Z^2_\theta(\Gamma,Z)$.
\end{proposition} 
 \begin{proof}
 Let $(E,\phi)$ be a  $G$-Higgs bundle over $X$ such that 
 \begin{equation}\label{sim1} 
 (E,\phi)\cong (\eta_{\gamma}^*\theta^{-1}_{\gamma}(E),\mu_{\gamma}\eta_{\gamma}^*\theta^{-1}_{\gamma}(\phi))
 \end{equation}
for all $\gamma \in \Gamma$.  Let $\Aut(E,\phi)$ be the group of automorphisms covering the identity of $X$, and $\Aut_{\Gamma,\eta,\theta,\mu}(E,\phi)$ be the group
defined in Remark \ref{remark-exact-sequence-aut}. 
The simplicity of $(E,\varphi)$ implies that 
$\Aut(E,\varphi)\cong Z$,  and hence (\ref{exact-aut})
gives an extension
$$
1\to Z\to \Aut_{\Gamma,\eta,\theta,\mu}(E,\varphi) \to \Gamma\to 1. 
$$
This extension defines a cocycle $c\in Z^2_\theta(\Gamma,Z)$,  and a
$c$-twisted homomorphism 
$$
\Gamma\to \Aut_{\Gamma,\eta,\theta,\mu}(E,\varphi)
$$ 
with cocycle $c$, that is, a $(\theta,c,\mu)$-twisted $\Gamma$-equivariant
structure on $(E,\varphi)$. 
\end{proof}

We have the following.

\begin{theorem}[Proposition 5.1 and Theorem 5.7 in \cite{oscar-suratno}]\label{main}
 Let ${\cM}_{ss}(X,G)\subset \cM(X,G)$ be the subvariety of $\cM(X,G)$ 
consisting of those $G$-Higgs bundles
 which are stable and simple. Fix a homomorphism $\theta:\Gamma\to \Aut(G)$ lifting $a:\Gamma\to \Out(G)$. Then 
 $$
{\cM}_{ss}(X,G)^{\Gamma}\subset \bigcup_{[c]\in H^2_a(\Gamma,Z), 
[\sigma]\in \{H^1_{c_{x_i}}(\Gamma_{x_i},G)\}} 
\widetilde{\cM}(X,G,\Gamma,\theta,c,\mu,\sigma)
$$
and 
$$ \bigcup_{[c]\in H^2_a(\Gamma,Z), 
[\sigma]\in \{H^1_{c_{x_i}}(\Gamma_{x_i},G)\}} 
\widetilde{\cM}(X,G,\Gamma,\theta,c,\mu,\sigma)
\subset
{\cM}(X,G)^{\Gamma}.
$$

\end{theorem}
\begin{proof}
 Let $(E,\phi)\in {\cM}_{ss}(G)^{\Gamma}$. 
 Then, by Proposition \ref{simple}, $(E,\phi)$ admits a $(\theta,c,\mu)$-twisted $\Gamma$-equivariant structure, where $c\in Z^2_\theta(\Gamma,Z)$, the set of all $2$-cocycles where 
 $\Gamma$ acts on $Z$ via $\theta$.
  Thus $(E,\phi)\in \widetilde{\cM}(X,G,\Gamma,\theta,c,\mu,\sigma)$. 
It follows from Proposition  \ref{cohom} that the
union should run over  $[c]\in H^2_a(\Gamma,Z)$ and
$[\sigma]\in \{H^1_{c_{x_i}}(\Gamma_{x_i},G)\}$.
\end{proof}

\begin{theorem}\label{main-principal}
 Let ${M}_{ss}(X,G)\subset M(X,G)$ be the subvariety of $M(X,G)$ 
consisting of those $G$-bundles
 which are stable and simple. Fix a homomorphism $\theta:\Gamma\to \Aut(G)$ lifting $a:\Gamma\to \Out(G)$.  Then 
 $$
{M}_{ss}(X,G)^{\Gamma}\subset \bigcup_{[c]\in H^2_a(\Gamma,Z), 
[\sigma]\in \{H^1_{c_{x_i}}(\Gamma_{x_i},G)\}} 
\widetilde{M}(X,G,\Gamma,\theta,c,\sigma)
$$
and 
$$ \bigcup_{[c]\in H^2_a(\Gamma,Z), 
[\sigma]\in \{H^1_{c_{x_i}}(\Gamma_{x_i},G)\}} 
\widetilde{M}(X,G,\Gamma,\theta,c,\sigma)
\subset
{M}(X,G)^{\Gamma}.
$$

\end{theorem}

%%%%%%%%%%%%%%%%%%%%%%%%%%%%%%%%%%%%%%%%%%%%%%%%%%%%%%%%%%%%%%%%%%%%%%%%%%%%
\section{Fixed points in the character variety}\label{section-character-variety-alpha-trivial}
%%%%%%%%%%%%%%%%%%%%%%%%%%%%%%%%%%%%%%%%%%%%%%%%%%%%%%%%%%%%%%%%%%%%%%%%%%%

We study now the action of $\Gamma$ on the character variety $\calR(X,G)$ and 
describe the fixed points in terms of twisted equivariant representations. Recall that we are given  homomorphisms  $\eta:\Gamma\to \Aut(X)$, 
$a:\Gamma\to \Out(G)$ and $\theta:\Gamma\to \Aut(G)$, where $\theta$ is a lift of $a$. 

Fix a point $x\,\in\, X$. The automorphism  $\eta_\gamma$ of $X$ produces a
homomorphism
$$
{\eta_\gamma}_*\, :\, \pi_1(X, x)\to \pi_1(X, \eta_\gamma(x))\, .
$$
This induces an automorphism of $\calR(X,G)$ since 
the quotient $\Hom(\pi_1(X,x),G)/G$ is independent of the base point of $X$.
Combining this with the action of $\Aut(G)$ given in Section \ref{section-prym-narasimhan-ramanan-character-varieties},
for every $\gamma\in \Gamma$ and  $\rho\in \Hom(\pi_1(X),G)$ we have 
$\rho\cdot\gamma \in \Hom(\pi_1(X),G)$ given by 
$$
\rho\cdot\gamma=\theta_\gamma^{-1}\circ \rho\circ {\eta_\gamma}_*.
$$

% In other words we have the following.

% \begin{proposition}\label{action-nahc}
% Denote  the action of $\Gamma$ on $\cM(X,G)$ and $\calR(X,G)$ by $\gamma\cdot$, and consider
% the non-abelian Hodge correspondence $\cM(X,G)\cong \calR(X,G)$ given by Theorem
% \ref{rep1}. Then, for every $\gamma\in\Gamma$  the  following diagramme commutes:

% \begin{equation}\label{eq-induced-action-character-variety}
%     \xymatrix{
% \cM(X,G) \ar[rr]^{\cong} \ar[d]_{\cdot\gamma} & & \calR(X,G) 
% \ar[d]^{\cdot\gamma}
% \\
% \cM (G)\ar[rr]^{\cong} & & \calR(X,G).
% }
% \end{equation}

% \end{proposition}

Combining Theorem \ref{rep1} with Theorems \ref{equivariant-nahc} and \ref{main} and Proposition 
\ref{obvious-fix} we have the following.

\begin{theorem}\label{main-rep}
Let ${\calR}_{\irr}(X,G)\subset \calR(X,G)$ be the subvariety of $\calR(X,G)$ 
consisting of isomorphism classes of irreducible representations, and let 
$\widetilde{\calR}(X,G,\Gamma,\theta,c)$ 
be the image of $\calR(X,G,\Gamma,\theta,c)$ in $\calR(X,G)$ under the 
natural map defined by diagramme  (\ref{compatible-rep}). Let $\calR(X,G)^{\Gamma}$ be
the fixed-point subvariety for the action of $\Gamma$ defined by the homomorphism $(\eta,a): \Gamma \to \Aut(X)\times \Out(G)$.
Then

$$
\bigcup_{[c]\in H^2_a(\Gamma,Z)}\widetilde{\calR}(X,G,\Gamma,\theta,c)\subset \calR(X,G)^{\Gamma}  
$$

and

$$
{\calR}_{\irr}(X,G)^{\Gamma}\subset \bigcup_{[c]\in H^2_a(\Gamma,Z)} 
\widetilde{\calR}(X,G,\Gamma,\theta,c).
$$
Here $\theta:\Gamma\to \Aut(G)$ is 
any  lift of $a:\Gamma\to \Out(G)$.
\end{theorem}

\begin{remark}
We could have considered also a non-trivial character $\mu:\Gamma\to \C^*$ 
with image the subgroup of $\C^*$ given by $\{\pm 1\}\cong \Z/2\Z$. With a minor
modification in the definition of the group $G\times_{\theta,c}\Gamma$ 
one would obtain similar results (see  
\cite{biswas-calvo-García-prada,PR,ow} 
for an analogous situation).
\end{remark}

%%%%%%%%%%%%%%%%%%%%%%%%%%
\section{Example 1}\label{section-alfa-trivial-example1}
%%%%%%%%%%%%%%%%%%%%%%%%%%

 Let $G=\SL(2,\C)$ and $(X,\sigma)$ be a hyperelliptic curve together with the 
hyperelliptic involution $\sigma$.
 In this case let $\Gamma=\Z/2\Z$  and 
consider  the homomorphism $\eta:\Z/2\Z\to \Aut(X)$ defined by sending 
$-1\mapsto \sigma$.
In this case $\Out(G)=1$ and $Z=\Z/2\Z$,  hence $\Aut(G)=\Int(G)$, 
and therefore
 $\Aut(G)$ acts trivially on the centre $Z$. So, in this case, 
we have $H^2(\Z/2\Z,Z)=\Z/2\Z$. 
Also,
there are only two characters $\mu^\pm$, defined by  $\mu^\pm(-1)=\pm 1$. 
We can then define actions on the moduli space of $\SL(2,\C)$-Higgs bundles
defined by $\eta$ and $\mu^\pm$. The case in which $\eta$ is the trivial homomorphism from $\Gamma$ to $\Aut(X)$ and $\mu=\mu^-$ is studied in
\cite{hitchin1987,García-prada,García-prada-ramanan,PR}.

%%%%%%%%%%%%%%%%%%%%%%%%%%%%%%%%%%%%%%%
\section{Example 2}\label{section-alfa-trivial-example2}
%%%%%%%%%%%%%%%%%%%%%%%%%%%%%%%%%%%%%%

Let $G=\SL(n,\C)$, with $n>2$ and $X$ a hyperelliptic curve together with the 
hyperelliptic involution $\sigma$, as above. Let $\Gamma=\Z/2\Z$ and 
$\eta:\Z/2\Z\to \Aut(X)$ be the 
homomorphism defined by sending $-1\mapsto \sigma$. 
In this case $\Out(G)\cong \Z/2\Z$ and $Z\cong \Z/n\Z$.
Let us denote the 
trivial homomorphism from $\Gamma\to \Out(G)$ by $a^+$ and the homomorphism which sends $-1$ to $b$, where
$\Out(G)=\langle b \rangle$, by $a^-$. 
In the first case we have the trivial action of $\Gamma$ on the 
centre $Z$ via $a^+$. To compute the second group cohomology of $\Z/2\Z$ with coefficients in $\Z/n\Z$ we will use the following fact:
Let $C$ be a cyclic group of order $r$ generated by $t$ and $A$ be a finite abelian group with a $C$-action. Let 
$N=1+t+\cdots+t^{r-1}\in \Z[\Gamma]$ then obviously $Na$, $a\in A$, is fixed by all $\alpha\in C$. With these notations we have 
$H^p(C, A)=\frac{A^{\Gamma}}{NA}$, $p=2,4,6...$.
Thus we have , in this case, $H^2(\Gamma,Z)=0$ when $n$ 
is odd and  $H^2(\Gamma,Z)=(\Z/2\Z)^r$ when $n$ is even,
where $r$ is the minimal number of generators of the $2$-Sylow subgroups of 
$\Z/n\Z$. 
On the other 
hand the action of the generator of $\Z/2\Z$ on $Z$ induced by $a^-$ 
sends $x\in Z$ to $x^{-1}$. In this case we have $H^2_{a^-}(\Gamma, Z)=Z^{\Z/2\Z}$.
Thus the action is trivial when $n$ is odd and 
hence $H^2_{a^-}(\Gamma, Z)=0$, and if $n$ is even then $H^2_{a^-}(\Gamma,Z)$ consists of  
all order $2$ elements of $\Z/n\Z$.
As in the previous example, we have $\mu^\pm$ as possible characters.

The cases in which $\eta$ is the trivial homomorphism from $\Gamma$ to $\Aut(X)$ and $\mu=\mu^\pm$ is studied in
\cite{García-prada,PR} (see also \cite{heller-schaposnik,schaffhauser} 
for related work).

%%%%%%%%%%%%%%%%%%%%%%%%%%%%%
\section{Example 3}\label{section-alfa-trivial-example3}
%%%%%%%%%%%%%%%%%%%%%%%%%%%%%

Let $G=\Spin(8,\C)$. Then $Z=\Z/2\Z\times \Z/2\Z$, $\Out(G)\cong S_3$. 
In \cite{PR} the authors consider various 
actions of cyclic subgroups  $\Gamma$ of $\Out(G)$, with 
$\Gamma$ acting trivially on  $X$, and identify the fixed-point subvarieties.  

Now in our situation the following three cases are relevant.

Case (I): Let $X$ be a compact Riemann surface of genus $g>2$ and $\Gamma:=S_3$.
Let $\eta:\Gamma \to \Aut(X)$ be an injective homomorphism in other words the action of $\Gamma$ on $X$ is faithful.
 Let $\sigma$ and $\tau$ generate the group $\Out(\Spin(8,\C))\cong S_3$. 
Let $a: \Gamma\to \Out(\Spin(8,\C))$ be the isomorphism defined by sending an order $2$ generator to $\sigma$
and an order $3$ generator to $\tau$. Let $\mu:\Gamma\to \mathbb{C}^*$ be a  
character of $S_3$. We know that $S_3$ has three non-equivalent conjugacy classes. Let 
$\mu_i$, $i=1,2,3$, be the corresponding characters. 
We define homomorphisms 
$F_i=(\eta,a,\mu_i):\Gamma \to \Aut(X)\times \Out(G)\times \C^*$, $i=1,2,3$. Then each $F_i$  determines an action on 
the moduli space of $G$-Higgs bundles. Let $H=\langle\tau\rangle$ be the 
normal subgroup of $G$ generated by $\tau$. 
Then by (\cite[Lemma 2.2.4]{oxr}) $H^2(\Gamma,Z)=H^2(\Gamma/H,Z^{H})$. As $\tau$ is an element of order $3$ either $Z^{H}=(e)$ or $Z^{H}=Z$.
So, we have either $H^2(\Gamma,Z)=0$ or $H^2(\Gamma,Z)=\Z/2\Z$.

Case (II): Let $X$ be a hyperelliptic curve and $\Gamma:=S_3$. We 
define a homomorphism $\eta:S_3\to \Aut(X)$ by sending $\sigma$ to $\text{Id}$ and $\tau$ to an order $2$ hyperelliptic 
involution. Let $b_i\in \Out(G)$ be the class of an order two automorphism of $G$ and $\theta_i:\Gamma \to \Out(G)$
be the homomorphism defined by $\tau \mapsto 1$ and $\sigma \mapsto b_i$. 
We define $F_i:=(\eta,\theta_i,\mu_i):\Gamma\to \text{Aut}(X)\times \text{Out}(G)$, $i=1,2,3$. Then the respective actions of $\Gamma$ on the moduli space 
of $G$-Higgs bundles are determined by $F_i$. In this case the subgroup $H$ 
acts on $Z$ trivially, therefore $Z^{H}=Z$, and hence 
$H^2(\Gamma,Z)=\Z/2\Z$.

Case (III): $X$ is a cyclic trigonal curve.  In other words we assume that the subgroup $<\tau>$ acts trivially on $X$ and
$f$ is an order $3$ automorphism such that $X/<f>\cong \mathbb{P}^1$.
This case is related to the work of Oxbury and Ramanan \cite{oxr}, and to
what they refer as Galois $\Spin(8,\mathbb{C})$-bundles.
We 
define a homomorphism $\eta:S_3\to \Aut(X)$ by sending $\sigma$ to $\Id$ and 
$\tau$ to the order $3$ automorphism $f$.
Let $b$ be the class of unique order $3$ automorphism of $X$ and $\theta:S_3\to \Out(G)$ be defined
by sending $\sigma$ to $I$ and $\tau$ to $b$.  As in the previous case we define $F_i:=(\eta,\theta_i,\mu_i):\Gamma\to \text{Aut}(X)\times \text{Out}(G)$, $i=1,2,3$. 
Then the action of $\Gamma$ on the moduli space 
of $G$-Higgs bundles is determined by $F_i$. Since $\Z/3\Z$ and 
$Z=\Z/2\Z\times \Z/2\Z$ have coprime 
order, by \cite[Lemma 2.2.4]{oxr} $H^2(\Z/3\Z,Z)=0$.

%%%%%%%%%%%%%%%%%%%%%%%%%%%%%
\section{Example 4}\label{section-alfa-trivial-example4}
%%%%%%%%%%%%%%%%%%%%%%%%%%%%%

Let $G$ be a group of type $E_6$. In this case $\text{Out}(G)=\Z/2\Z$.
Let $X$ be a hyperelliptic curve together with a hyperelliptic involution $\sigma$ and $\Gamma=\Z/2\Z$.
We define a homomorphism $\eta:\Gamma\to \text{Aut}(X)$ by sending $-1\mapsto \sigma$. As in the case 
of example $1$ we have two homomorphisms $a^\pm:\Gamma\to \text{Out}(G)$ and two characters $\mu^{\pm}$.

\newpage

\chapter{Fixed points when the action is general}\label{chapter-general}
Throughout this chapter we fix a finite group $\Gamma$, a compact Riemann surface $X$ and a connected semisimple complex Lie group $G$ with centre $Z$ and Lie algebra $\lie g$.

\section{Étale covers and lifts of \texorpdfstring{$a$}{a}}\label{section-group-theory}
Let $\eta:\Gamma\to\Aut(X)$ and $a:\Gamma\to\Out(G)$ be homomorphisms. By \cite{de-siebenthal}, there exists a homomorphism
$$\theta:\ker\eta\to\Aut(G)$$
lifting $a\vert_{\ker\eta}$. With definitions as in Section \ref{section-Gtheta} we have an extension
\begin{equation}\label{eq-extension-connected-component-gt}
    1\to\gt_0\to\gt\xrightarrow{\pt}\gamtz\to1,
\end{equation}
restricting (\ref{eq-extension-connected-component}) where $\gt_0$ is the connected component of the identity and $\gamtz$ is a (finite) subgroup of $\gamtt$. Let $\otau$ be the characteristic homomorphism of (\ref{eq-extension-connected-component-gt}). By \cite{de-siebenthal} there exists a lift $\tau:\gamtz\to\Aut(\gt_0)$ of the characteristic homomorphism of (\ref{eq-extension-connected-component-gt}), and consequently by Proposition \ref{prop-extensions-isomorphic-twisted-group} we can find $c\in\zzgt$ such that we have an isomorphism $\gt\cong\gt_0\times_{\tau,c}\gamtz$ as extensions of $\gt_0$, i.e. we have a commutative diagramme
\begin{equation*}
    \begin{tikzcd}[ar symbol/.style = {draw=none,"#1" description,sloped},
  isomorphic/.style = {ar symbol={\cong}},
  equals/.style = {ar symbol={=}},
  ]
        1\arrow[r]\ar[equal]{d}  &   \gt_0\arrow[r]\ar[equal]{d} &   \gt\arrow[r]\arrow[d,"\sim" labl]   & \gamtz  \arrow[r]\ar[equal]{d} & 1\ar[equal]{d}\\
        1 \arrow[r]  &   \gt_0 \arrow[r]  &   \gt_0\times_{\tau,c}\gamtz \arrow[r]    &     \gamtz \arrow[r]  &   1
    \end{tikzcd}.
\end{equation*}

Now let $\Lambda\le \gamtz$ be a subgroup, $\gtl:=\pt^{-1}(\Lambda)$ and let $p:Y\to X$ be a connected étale cover associated to a $\Lambda$-bundle over $X$ and consider the subgroup $\wgam\le\Aut(Y)$ lifting $\eta(\Gamma)$. This contains $\Lambda$, the Galois group of $Y$ over $X$, as a normal subgroup ($\Lambda$ is the kernel of the projection $\wgam\to\eta(\Gamma)$). Let 
\begin{equation*}
    \wgame:=\{(\gamma,\wga)\in\Gamma\times\wgam\suhthat \etag=p(\wga)\}.
\end{equation*}
This contains the subgroup $\ker\eta\times 1$, which is a copy of $\ker\eta$. Let $p_{\Gamma}:\wgame\to\Gamma$ be the projection on the first factor. We have the following commutative diagramme:
\begin{equation}\label{eq-extension-wgam}
    \begin{tikzcd}[
      ar symbol/.style = {draw=none,"#1" description,sloped},
      equals/.style = {ar symbol={=}},
      ]
        &   &   1\arrow[d]&   1\arrow[d]&   \\
        &   &   \Lambda\arrow[d]\ar[equal]{r}&   \Lambda\arrow[d]&   \\
        1\arrow[r]&  \ker\eta\ar[equal]{d}\arrow[r] &   \wgame\arrow[d,"p_{\Gamma}"]\arrow[r]&   \wgam\arrow[d,"p"]\arrow[r]&   1\\
        1\arrow[r]&   \ker\eta\arrow[r]&   \Gamma\arrow[d]\arrow[r,"\eta"]&   \eta(\Gamma)\arrow[d]\arrow[r]&   1\\
        &   &   1&   1&   \\
    \end{tikzcd},
\end{equation}
whose rows and columns are exact.

% A homomorphism $\tau:\wgam\to\Aut(\gt_0)$ is said to be a \textbf{restriction} of a map $\tilde \tau:\wgame\to\Aut(G)$ if $\tilde\tau(\wgame)$ preserves $\gt_0$ and factors through a map $\wgam=\wgame/\ker\eta\to \Aut(\gt_0)$ which is equal to $\tau$. 

We say that a map $\tilde\tau:\wgame\to\Aut(G)$ is a \textbf{$c$-twisted homomorphism} if it satisfies
\begin{equation}\label{eq-c-twisted-hom}
    \tilde\tau_{\gamma,\wga}\tilde\tau_{\gamma',\wga'}=\Int_{c(\gamma,\gamma')}\tilde\tau_{\gamma\gamma',\wga\wga'}
\end{equation}
for every $(\gamma,\wga)$ and $(\gamma',\wga')\in\wgame$, where $c\in Z^2_{\tilde\tau}(\wgam,Z(\gt_0))$; of course we are assuming that $\tilde\tau(\wgam)$ preserves $Z(\gt_0)$.
Equivalently, the associated (left) action of $\wgame$ on $G$ is $c$-twisted. We have a similar notion of a $c$-twisted homomorphism $\wgam\to \gt$. Denote by $\Hom_c(A,B)$ the set of $c$-twisted homomorphisms from $A$ to $B$ and let $\Hom_c(\wgame,\Aut(G))^{\gtl}$ be the set of $c$-twisted homomorphisms whose associated $c$-twisted $\wgame$-action on $G$ preserves $\gtl$. We have a restriction map $r_{\ker\eta}:\Hom_c(\wgame,\Aut(G))\to\Hom(\ker\eta,\Aut(G))$. The image consists of homomorphisms because $c$ is trivial on $\ker\eta\in\wgame/\ker\eta\cong\wgam$. Let $\Hom_{\theta,c}(\wgame,\Aut(G)):=\Hom_c(\wgame,\Aut(G))^{\gt}\cap r_{\ker\eta}^{-1}(\theta)$. We also have a restriction map $r_{\gtl}:\Hom_{\theta,c}(\wgame,\Aut(G))\to\Hom_c(\wgam,\Aut(\gtl))$: indeed, given $\tilde\tau\in \Hom_{\theta,c}(\wgame,\Aut(G))$, $(\gamma,\wga)\in\wgame$, $\gamma'\in\ker\eta$ and $g\in\gtl$, we have
\begin{equation*}
    \tilde\tau_{(\gamma\gamma',\wga)}(g)=\Int_{c(\wga,1)}^{-1}\tilde\tau_{(\gamma,\wga)}\tilde\tau_{\gamma'}(g)=\tilde\tau_{(\gamma,\wga)}\theta_{\gamma'}(g)=\tilde\tau_{(\gamma,\wga)}(g),
\end{equation*}
hence the induced map $\tilde\tau:\wgame\to\Aut(\gtl)$ factors through $\wgam$. 

Note that any automorphism of $\gtl$ preserves the connected component $\gt_0$ of $\gtl$, hence we have a map $r_{\gt_0}:\Hom_c(\wgam,\Aut(\gtl))\to\Hom(\wgam,\Aut(\gt_0))$, where the fact that $Z(\gt_0)$ acts trivially on $\gt_0$ by conjugation implies that the image consists of (honest) homomorphisms. 
% Given a homomorphism $\tau:\wgam\to\Aut(\gt_0)$, we denote the preimage $r_{\gt_0}^{-1}(\tau)$.
Conversely, given a homomorphism $\tau:\wgam\to\Aut(\gt_0)$ and a 2-cocycle $c\in Z^2_{\tau}(\wgam,Z(\gt_0))$ such that $\gt_0\times_{\tau,c}\Lambda$ and $\gtl$ are isomorphic as extensions of $\gt_0$ (see Proposition \ref{prop-extensions-isomorphic-twisted-group}), there is a $c$-twisted extension $e_{\tau,c}:\wgam\to\Aut(\gtl)$ given as follows:
\begin{equation*}
    \begin{tikzcd}[ar symbol/.style = {draw=none,"#1" description},
    isomorphic/.style = {ar symbol={\cong}},
      ]
        \wgam\arrow[r]\arrow[rrrr,bend right=15,"e_{\tau,c}", swap] & \gt_0\times_{\tau,c}\wgam\arrow[r] & \Int(\gt_0\times_{\tau,c}\wgam)\arrow[r] & \Aut(\gt_0\times_{\tau,c}\Lambda)\ar[isomorphic]{r} &  \Aut(\gtl),
    \end{tikzcd}
\end{equation*}
where the first map sends $\gamma\in\Gamma$ to $(1,\gamma)\in\gt_0\times_{\tau,c}\wgam$ and the existence of the second map follows from the fact that $\Lambda$ is normal in $\wgam$. The map $\etc$ is $c$-twisted, since $\Int_{(1,\gamma)}\Int_{(1,\gamma')}=\Int_{c(\gamma,\gamma')}\Int_{1,\gamma\gamma'}$ by definition of the group multiplication on $\gt_0\times_{\tau,c}\wgam$.

In summary, we have the following diagramme:
\begin{equation}\label{eq-restriction-extension maps}
    \begin{tikzcd}
        \Hom(\wgame,\Out(G)) &
        \Hom_{\theta,c}(\wgame,\Aut(G)) \arrow[l,"q_*"]\arrow[r,"r_{\gtl}"] & \Hom_c(\wgam,\Aut(\gtl))\arrow[d,"r_{\gt_0}"]\\
        \Hom(\Gamma,\Out(G))\arrow[u,"p_{\Gamma}^*"]&       &   \Hom(\wgam,\Aut(\gt_0))
    \end{tikzcd},
 \end{equation}
where $q_*$ is the pushforward of the natural projection $\Aut(G)\to\Out(G)$ and $p_{\Gamma}^*$ is the pullback of $p_{\Gamma}:\wgame\to\Gamma$. For each homomorphism $\tau:\wgam\to\Aut(\gt_0)$ and each 2-cocycle $c\in Z^2_{\tau}(\wgam,Z(\gt_0))$ as above we set 
\begin{align*}
    &\Hom_{\theta,\tau,c}(\wgame,\Aut(G)):=r_{\gt_{\Lambda}}^{-1}(e_{\tau,c})\subset\Hom_{\theta,c}(\wgame,\Aut(G))\andd\\
    &\Hom_{\theta,\tau,c}(\wgame,\Out(G)):=q_*(\Hom_{\theta,\tau,c}(\wgame,\Aut(G))).
\end{align*}

% Let $S_{\tau,c}$ be the subset of $\outgtz\times\zzgtw$ extending $(\otau,c)$. We construct an \textbf{extension map}
% \begin{equation}\label{eq-extension-map}
    % e:S_{\tau,c}\to\Hom(\eta(\Gamma),\outgt).
% \end{equation}
% Let $\otau':\wgam\to \outgtz$ be a homomorphism extending $\otau$ and $c'\in \zzgtw$ a 2-cocycle extending $c$. Consider a lift $\tau':\wgam\to\Aut(\gt_0)$ which restricts to $\tau$.
% Let $\underline{G}^{\theta}:=\gt_0\times_{\tau',c'}\wgam$, which contains $\gt$ as a normal subgroup. Given $\gamma\in\wgam$, the inner automorphism $\Int_{(1,\gamma)}\in\Int(\underline{G}^{\theta})$ restricts to $\tau\in\Aut(\gt_0)$. Moreover, for every $\gamma,\gamma'\in \wgam$ we have
% \begin{equation*}
    % \Int_{\gamma}\Int_{\gamma'}=\Int_{(\oc(\gamma,\gamma'),\gamma\gamma')}=\Int_{c'(\gamma,\gamma')}\Int_{(1,\gamma\gamma')},
% \end{equation*}
% in particular the composition with the natural homomorphism $\Aut(\gt)\to\Out(\gt)$ induces a homomorphism $e(\otau',c'):\wgam\to \Out(\gt)$. It is also clear that the image of $\Lambda$ in $\Out(\gt)$ is trivial, thus this descends to a homomorphism $\eta(\Gamma)\to\outgt$ as required. Note that $e(\otau',c')$ does not depend on the choice of $\tau$ and $\tau'$, since their ambigüity relies on the composition with conjugations by elements of $\gt_0<\gt$.

Now let $\tilde\tau:\wgame\to\Aut(G)$ be a $c$-twisted homomorphism in $\Hom_{\theta,\tau,c}(\wgame,\Aut(G))$ preserving $\liegm$. There exists an associated left $(\tau,c)$-twisted action (with respect to the adjoint action of $\gt_0$ on $\liegm$) of $\wgame$ on $\liegm$, namely
\begin{equation}\label{eq-action-tildetau}
    \rhotm:=\pg^*(\mu^{-1})\tilde\tau:\wgame\to\GL(\liegm);\,(\gamma,\wga)\mapsto (v\mapsto\mug^{-1}\tilde\tau_{\gamma,\wga}(v)).
\end{equation}
Moreover, since by definition of $\liegm$ and the fact that $r_{\ker\eta}(\tilde\tau)=\theta$ the action is trivial on $\ker\eta$, and $c$ only depends on the coset in $\wgam$, this factors through a $(\tau,c)$-twisted action of $\wgam$ on $\liegm$ which we also call $\rhotm$. 

\begin{remark}\label{remark-not-preserve-liegm}
    In Sections \ref{section-fixed-points-general-alpha-trivial} and \ref{section-general-theorem} we will encounter $c$-twisted homomorphisms $\tilde\tau$ living in the set $\Hom_{\theta,\tau,c}(\wgame,\Aut(G))$ which do not necessarily preserve $\liegm$. However, we still have a well defined linear homomorphism 
\begin{equation*}
    \rhotm:\wgam\to\Hom(\liegm,\lie g).
\end{equation*}
Given a $\gt_0$-bundle $F$ over $Y$, this induces a map $H^0(Y,F(\liegm)\otimes K_Y)\to H^0(Y,F(\lie g)\otimes K_Y)$. Hence we have a notion of \textbf{$(\tau,c,\rhotm)$-twisted $\wgam$-equivariant $(\gt_0,\liegm)$-Higgs pair} over $Y$, which is given precisely as in Definition \ref{mega} by regarding the images of the Higgs field under the action of $\wgam$ as sections of $F(\lie g)\otimes K_Y$ and imposing that they actually live in $F(\liegm)\otimes K_Y$ and are equal to the Higgs field.
\end{remark}

\section{Fixed points when \texorpdfstring{$\alpha$}{alpha} is trivial}\label{section-fixed-points-general-alpha-trivial}

Fix a homomorphism
$$\Gamma\to\Aut(X)\times\Out(G)\times\C^*;\,\gamma\mapsto(\fg,\ag,\mug).$$ Keep the notation of Section \ref{section-group-theory}. We refine the results of Chapter \ref{chapter-alpha-trivial}.

\begin{proposition}\label{prop-fixed-points-reduction-alpha-trivial-principal}
Consider a lift $\theta:\ker\eta\to\Aut(G)$ of $a\vert_{\ker\eta}$ and a subgroup $\Lambda\le \gamtz$. Take a connected étale cover $p:Y\to X$ associated to a $\Lambda$-bundle over $X$ and the group $\wgam\le\Aut(Y)$ fitting in (\ref{eq-extension-wgam}). Let $\tau:\wgam\to\Aut(\gt_0)$ be a homomorphism and $c\in\zzgtw$ a 2-cocycle such that there is an isomorphism of extensions $\gtl:=\pt^{-1}(\Lambda)\cong \gt_0\times_{\tau,c}\Lambda$. Assume that $\pg^*a\in\outc$ (see Section \ref{section-group-theory}).

Let $F$ be a $(\tau,c)$-twisted $\wgam$-equivariant $\gt_0$-bundle over $Y$. Then $F$ can be regarded as a $\gtl$-bundle over $X$ via Theorem \ref{th-prym-narasimhan-ramanan-principal}, and its extension of structure group $E$ to $G$ is isomorphic to $E\gamma$ for each $\gamma\in\Gamma$.
\end{proposition}

\begin{proof}
% Let $(F,\psi)$ be a $(\tau,c)$-twisted $\wgam$-equivariant $(\gt_0,\liegm)$-Higgs pair over $Y$. Then $(F,\psi)$ is a $(\gt,\liegm)$-Higgs pair over $X$ when we regard the action of $\gamtz$ as multiplication by the corresponding element of $\gt$. Let $(E,\phi)$ be its extension of structure group to $G$, which is a $G$-Higgs bundle over $X$.
Throughout the proof we call $\tau:\wgam\to\Aut(\gt)$ to the extension $e_{\tau,c}$ of the given $\tau:\wgam\to\Aut(\gt_0)$ by abuse of notation. We have to show that $E\cong E\gamma$ for every $\gamma\in\Gamma$. Pick $\wga\in\wgam$ such that $\eta(\gamma)=p(\wga)$ (in other words, $(\gamma,\wga)\in\wgame$). Consider the automorphism of $F$ given by $\wga$. We want to know how it interacts with the $\gtl$-action on $F$. Consider the twisted product $\gt_0\times_{\tau,c}\wgam$. For each $e\in F$ and $(g,\lambda)\in\gt_0\times_{\tau,c}\Lambda\cong \gtl$,
\begin{align*}
    (e(g,\lambda))\cdot\wga=((eg)\cdot\lambda)\cdot\wga=((((eg)\cdot\wga)\cdot\wga^{-1})\cdot\lambda)\cdot\wga=((eg)\cdot\wga)\cdot(1,\wga^{-1})(1,\lambda)(1,\wga)=\\
    (e\cdot\wga)\tau^{-1}_{\wga}(g)\tau^{-1}_{\wga}(1,\lambda)=(e\cdot\wga)\tau^{-1}_{\wga}(g,\lambda),
\end{align*}
thus $\wga$ induces an isomorphism 
\begin{equation*}
    \hg:F\xrightarrow{\sim}\etag^{*-1}\tau_{\wga}(F);\,e\mapsto \etag^{*-1}e\cdot \wga,
\end{equation*}
where $F$ is regarded as a $\gtl$-bundle over $X$. 

Now let $\tilde\tau\in q_*^{-1}(\pg^*a)\subset\homtc$. Using the extension $\tilde\tau_{\gamma,\wga}\in\Aut(G)$ of $\tau$ we may extend the above isomorphism to $E$, thus getting $E\cong \etag^*\tilde\tau_{\gamma,\wga}(F)$. Since $\tilde\tau_{\gamma,\wga}$ is a lift of $a_{\gamma}$, we are done. 
\end{proof}

\begin{proposition}\label{prop-fixed-points-reduction-alpha-trivial-higgs}
Let $\Lambda\le \gamtz$ be a subgroup. Consider a lift $\theta:\ker\eta\to\Aut(G)$ of $a\vert_{\ker\eta}$, a connected étale cover $p:Y\to X$ associated to a $\Lambda$-bundle over $X$ and the group $\wgam\le\Aut(Y)$ fitting in (\ref{eq-extension-wgam}). Let $\tau:\wgam\to\Aut(\gt_0)$ be a homomorphism and $c\in\zzgtw$ a 2-cocycle such that there is an isomorphism of extensions $\gtl:=\pt^{-1}(\Lambda)\cong \gt_0\times_{\tau,c}\Lambda$. Assume that $\pg^*a\in\outc$ (see Section \ref{section-group-theory}) and pick $\tilde\tau\in \homtc$ such that $q_*(\tilde\tau)=\pg^*a$.

Let $(F,\psi)$ be a $(\tau,c,\rhotm)$-twisted $\wgam$-equivariant $(\gt_0,\liegm)$-Higgs pair over $Y$. Then $(F,\psi)$ can be regarded as a $(\gtl,\liegm)$-Higgs pair over $X$ via Proposition \ref{prop-twisted-equivariant-higgs pairs-one-to-one}, and its extension of structure group $(E,\phi)$ to $G$ is isomorphic to $(E,\phi)\gamma$ for each $\gamma\in\Gamma$.
\end{proposition}

\begin{proof}
It is left to show that the isomorphism $\hg:E\xrightarrow{\sim} E\gamma$ defined in the proof of Proposition \ref{prop-fixed-points-reduction-alpha-trivial-principal} for each $\gamma\in\Gamma$ respects the higgs field.  
Let $(e,v)\otimes k$ be a local expression for $\psi$, where $e\in F$, $v\in\liegm$ and $k\in K_Y$. By $\wgam$-invariance of $\psi$ we have
\begin{equation*}
    \psi=\psi\cdot\wga=(e\cdot\wga,\mug^{-1}d\tilde\tau_{\gamma,\wga}v)\otimes\etag^*k=\mug^{-1}\etag^{*-1}d\tilde\tau_{\gamma,\wga}(\hg(\psi)),
\end{equation*}
thus we conclude that $\hg(\psi)=\mug\etag^{*}\tilde\tau_{\gamma,\wga}^{-1}(\psi)$ and so $\hg$ sends $\phi$ to $\phi\cdot\gamma$ as required.
\end{proof}

\begin{proposition}\label{prop-simple-fixed-points-oscar-ramanan-alpha-trivial-principal}
Let $E$ be a simple $G$-bundle over $X$ which is isomorphic to $E\gamma$ for every $\gamma\in\Gamma$. Then there exist a lift $\theta$ of $a\vert_{\ker\eta}$ and a reduction of structure group $F$ of $E$ to $\gtl:=\pt^{-1}(\Lambda)$ satisfying the following: let $p:Y\to X$ be the connected étale cover associated to the $\Lambda$-bundle $F/\gt_0\to X$ and $\wgam$ the subgroup of $\Aut(Y)$ lifting $\eta(\Gamma)$. Then there is a homomorphism $\otau:\wgam\to\Out(\gt_0)$ such that, for every lift $\tau:\wgam\to\Aut(\gt_0)$ (which exists by \cite{de-siebenthal}), we can find a 2-cocycle $c\in Z^2_{\tau}(\wgam,Z(\gt_0))$ such that:
\begin{enumerate}
    \item We have an isomorphism $\gtl\cong\gt_0\times_{\tau,c}\Lambda$ as extensions of $\gt_0$.
    \item $\pg^*a\in\outc$ (see Section \ref{section-group-theory}).
    \item The tautological reduction of $p^*F$ to $\gt_0$ is a $(\tau,c)$-twisted $\wgam$-equivariant $\gt_0$-bundle.
\end{enumerate}
\end{proposition}

\begin{proof}
Let $E$ be a simple $G$-bundle which is isomorphic to $E\gamma$ for each $\gamma\in\Gamma$. According to Proposition \ref{prop-simple-fixed-points-principal} there is a lift 
$$\theta:\ker\eta\to\Aut(G)$$
of $a\vert_{\ker\eta}$ and a reduction $F$ with structure group
$\gt$. Another lift of $a\vert_{\ker\eta}$ gives such a reduction if and only if it is in the orbit of $\theta$ under the conjugation action of $\Int(G)$.
On the other hand, according to \cite{de-siebenthal} there exists a homomorphism
$$\wt:\Gamma\to\Aut(G)$$
lifting $a$. For each $\gamma\in \ker\eta$ we set $\tg=\Int_{\sg}\wtg$, where 
$$\Int_s:\ker\eta\to\Int(G);\,\gamma\mapsto \Int_{\sg}$$
is an element of $Z^1_{\wt}(\ker\eta,\Int(G))$ by Lemma \ref{lemma-lifts-vs-non-abelian-cohomology}. Fix an element $\beta\in\Gamma$. We claim that $\wtb(\gt)=G^{\theta'}$, where \begin{equation}\label{eq-def-theta'}
    \theta':\ker\eta \to\Aut(G);\,\gamma\mapsto\tg':=\Int_{\wtb(s_{\beta\gamma\beta^{-1}})}\wtg.
\end{equation}

Indeed, for every $\gamma\in\ker\eta$ and $g\in\gt$,
\begin{align*}
    \Int_{\wtb(s_{\beta\gamma\beta^{-1}})}\wtg\wtb(g)=
    \Int_{\wtb(s_{\beta\gamma\beta^{-1}})}\wtb\wt_{\beta\gamma\beta^{-1}}=
    \wtb\Int_{s_{\beta\gamma\beta^{-1}}}\wt_{\beta\gamma\beta^{-1}}(g)=\\\wtb\theta_{\beta\gamma\beta^{-1}}(g)=\wtb(g),
\end{align*}
so the inclusion $\wtb(\gt)\subseteq G^{\theta'}$ follows. Note that the same reasoning applied to $\beta$ instead of $\beta^{-1}$ and $\theta'$ instead of $\theta$ gives $\wt_{\beta}(\gtt)\subseteq\gt$, so $\gtt\subseteq\wtb(\gt)$ also follows. Moreover, we can see that $\theta'$ is a homomorphism, or equivalently (see Lemma \ref{lemma-lifts-vs-non-abelian-cohomology}) that the map 
$\Int_{\wtb(s_{\beta\bullet\beta^{-1}})}:\ker\eta\to\Int(G)$ is an element of $Z^1_{\wt}(\ker\eta,\Int(G))$: for every $\gamma$ and $\gamma'\in\ker\eta$ we have
$$\Int_{\wtb(s_{\beta\gamma\beta^{-1}})\wtg(\wtb(s_{\beta\gamma'\beta^{-1}}))}=\Int_{\wtb(s_{\beta^{-1}\gamma\beta}\wt_{\beta^{-1}\gamma\beta}(s_{\beta\gamma'\beta^{-1}}))}=\Int_{\wtb(s_{\beta\gamma\gamma'\beta^{-1}})},$$
where the last equality follows from the fact that $\Int_s\in Z^1_{\wt}(\ker\eta,\Int(G))$.

For each $\gamma\in\Gamma$ there is an isomorphism (unique up to multiplication by an element of $Z$)
$$\hg:E\to \etag^{*-1}\wtg E.$$
Consider the sub-bundle $F':=\fb^{*}\hb(F)$ of $E$. For every $e\in E$ and $g\in G$ we have
$$\fb^{*}\hb(eg)=\fb^{*}\hb(e)\wtb(g),$$
and so by the previous paragraph $F'$ is a reduction of $E$ with structure group $\wtb(\gt)=\gtt$. 

Thus, we have shown that $F'$ is a reduction of structure group of $E$ to $\gtt$. Since $\theta'$ is a lift of $a\vert_{\ker\eta}$, it must be a conjugate of $\theta$ by an element of $\Int(G)$. Moreover, $F't_{\beta}=F$ for some element $t_{\beta}\in G$ by Proposition \ref{prop-simple-fixed-points-principal} and so $\wt_{\beta}\Int_{t_{\beta}}$ preserves $\gt$. Note that, given another element $t'_{\beta}\in G$ such that $F't'_{\beta}=F$, we must have $t_{\beta}=t'_{\beta}g$ for some $g\in\gt$.

Now let $p:Y\to X$ be the étale cover of $X$ defined by $F/\gt_0$. For simplicity we assume that it is connected, so that $\Lambda=\gamtz$ and $\gtl=\gt$. The general case follows using the same argument after taking a connected component of $Y$ and the corresponding monodromy group $\Lambda$. The map $t$ determines a map from $\Gamma$ to the group of automorphisms $\Aut(Y)$ of $Y$ which is a homomorphism if we equip $\Aut(Y)$ with its trasposed multiplication, defined as follows: for each $\gamma\in\Gamma$ take an element of $Y=F/\gt_0$, choose a representative in $F$, apply $\etag^{*}\hg(\bullet)t_{\gamma}$ and take the image in $F/\gt_0$. Different choices of $t$ provide different choices of the map $\Gamma\to\Aut(Y)$ differing by elements of $\gamtz$, which is the Galois group of $Y$ over $X$. Consider the subgroup $\wgamm_Y\le\Aut(Y)$ consisting of the lifts of $\eta(\Gamma)$ to $Y$. By the previous discussion this is the subgroup generated by $\gamtz$ and the image of $\Gamma$ under any of the maps $\Gamma\to\Aut(Y)$ that we have defined. Recall that we are assuming that $\wgamm_Y$ acts on $Y$ on the right (in fact $\gamtz$ acts naturally on the right, since it comes from a principal bundle action). The rest of the proof is committed to defining a suitable action of $\wgamm_Y$ on the tautological reduction of $p^*F$ to $\gt_0$, which we also call $F$ (they have the same total space).

Let $\wgame\subseteq\Gamma\times\wgamm_Y$ be the subset of pairs $(\gamma,\wga)$ such that $\eta(\Gamma)=p(\wga)$. First note that for each pair $(\gamma,\wga)\in \wgame$ we can define an automorphism $\etag^{*}\hg(\bullet)t_{\gamma,\wga}$ of $F$ lifting $\wga$, where $t_{\gamma,\wga}\in G$ is chosen suitably. This is an automorphism of the total space of $F$ as a complex variety, and in general it does not preserve the $\gt_0$-action. Note that another such choice $t'_{\gamma,\wga}$ is equal to $t_{\gamma,\wga}g$ for some $g\in\gt_0$. We claim that the map $$\wt\Int_t:\wgame\to\Aut(G);\,(\gamma,\wga)\mapsto \wtg\Int_{t_{\gamma,\wga}},$$ 
induces a homomorphism $\wgamm_Y\to\Out(\gt_0).$
Indeed, for every $(\gamma,\wga)\in \wgame$, $e\in F$ and $g\in\gt_0$ we have 
\begin{align*}
    \etag^{*}\hg(eg)t_{\gamma,\wga}=\etag^{*}\hg(e)t_{\gamma,\wga}(\Int^{-1}_{t_{\gamma,\wga}}\wtg^{-1}(g))=\etag^{*}\hg(e)t_{\gamma,\wga}((\wtg\Int_{t_{\gamma,\wga}})^{-1}(g))
\end{align*}
and so, since both $\etag^{*}\hg(e)t_{\gamma,\wga}$ and $\etag^{*}\hg(eg)t_{\gamma,\wga}$ lie in $F$, the automorphism $\wtg\Int_{t_{\gamma,\wga}}$ of $G$ must preserve $\gt_0$. Moreover, the image of the restriction $\wtg\Int_{t_{\gamma,\wga}}\vert_{\gt_0}$ in $\Out(\gt_0)$ does not depend on $\gamma$: given another $\gamma'\in\Gamma$ such that $\eta(\gamma')=p(\wga)$, the element $\gamma^{-1}\gamma'$ must lie in $\ker\eta$. By simplicity of $(E,\phi)$, the morphism
$$h_{\gamma^{-1}\gamma'}(\bullet)s_{\gamma^{-1}\gamma'}:(E,\phi)\to\eta_{\gamma^{-1}\gamma'}^{*-1} \theta_{\gamma^{-1}\gamma'}(E,\mu_{\gamma^{-1}\gamma'}\phi)=\theta_{\gamma^{-1}\gamma'}(E,\mu_{\gamma^{-1}\gamma'}\phi)$$
must be equal to the isomorphism induced by the identity on $(F,\psi)$ up to an element of $Z$. Thus the restrictions of $\etag^{*}\hg(\bullet)t_{\gamma,\wga}$ and $\eta_{\gamma'}^{*}h_{\gamma'}(\bullet)t_{\gamma',\wga}$ to $F$ differ by an element of $G$ and, since they both preserve $F$ and lift $\wga$, this element must actually be in $\gt_0$. The compatibility of the $\gt_0$-actions then implies that $\wtg\Int_{t_{\gamma,\wga}}$ and $\wt_{\gamma'}\Int_{t_{\gamma',\wga}}$ must differ by an element of $\Int(\gt_0)$, as required. 

Therefore, we get a map $\wgamm_Y\to\Out(\gt_0)$. It is left to show that it is a homomorphism. For every $(\gamma,\wga)$ and $(\gamma',\wga')\in \wgame$ we have
\begin{align*}
    (\eta_{\gamma\gamma'}^{*}h_{\gamma\gamma'}(F))t_{\gamma\gamma',\wga\wga'}=F=\eta_{\gamma}^{*} h_{\gamma}(F)t_{\gamma,\wga}&=\eta_{\gamma'}^{*} h_{\gamma'}(\eta_{\gamma}^{*-1} h_{\gamma}(F)t_{\gamma,\wga})t_{\gamma',\wga'}\\&=
(\eta_{\gamma\gamma'}^{*}h_{\gamma\gamma'}(F))\wt_{\gamma'}^{-1}(t_{\gamma,\wga})t_{\gamma',\wga'}z
\end{align*}
for some element $z\in Z$. Hence there exists $g\in\gt_0$ such that $$\wt_{\gamma'}(t_{\gamma,\wga})t_{\gamma',\wga'}z=t_{\gamma\gamma',\wga\wga'}g,$$
and so
$$\wtg\Int_{t_{\gamma,\wga}}\wt_{\gamma'}\Int_{t_{\gamma',\wga'}}=\wt_{\gamma\gamma'}\Int_{\wt_{\gamma'}^{-1}(t_{\gamma,\wga})t_{\gamma',\wga'}}=\wt_{\gamma\gamma'}\Int_{t_{\gamma\gamma',\wga\wga'}}\Int_g,$$
as required.

Thus we have obtained a homomorphism 
$$\otau:\wgamm_Y\to\Out(\gt_0).$$
% which is induced by a homomorphism $\wgame\to\Aut(G)^{\theta}/\Int(\gt_0)'$, where $\Aut(G)^{\theta}$ is the subgroup of $\Aut(G)$ preserving $\gt_0$ and $\Int(\gt_0)'$ is the normal subgroup of inner automorphisms by elements of $\gt_0$. 
By \cite{de-siebenthal} $\otau$ lifts to a homomorphism $\wgamm_Y\to\Aut(\gt_0)$. In other words, we may rechoose the map $t:\wgame\to G$ to impose
\begin{equation}\label{eq-def-t}
    \tau:=\wt\Int_t\vert_{\gt_0}.
\end{equation}
More precisely, for each coset $\wga\in\wgamm_Y\cong\wgame/\ker\eta$ we may choose a representative $(\gamma,\wga)\in\wgame$ and define $t_{\gamma,\wga}$ so that $\tau_{\gamma}:=\wtg\Int_{t_{\gamma,\wga}}\vert_{\gt_0}$ and then, for each $\gamma'\in\gamma\ker\eta$, take the unique $t_{\gamma',\wga}$ that fits into the equation 
\begin{equation}\label{eq-c-doesnt-depend-on-ker}
    \etag^{*}\hg(\bullet)t_{\gamma,\wga}\vert_{F}=\eta_{\gamma'}^{*}h_{\gamma'}(\bullet)t_{\gamma',\wga}\vert_{F}
\end{equation}
(this exists because of simplicity of $E$ and because $\ker\eta$ acts trivially on $\gt_0$ via $\theta$). Let $c:\wgamm_Y\times\wgamm_Y\to Z(\gt_0)$ be the (unique) map satisfying
\begin{equation}\label{eq-def-c}
    \eta_{\gamma'}^{*}h_{\gamma'}(\etag^{*}\hg(\bullet )t_{\gamma,\wga}) t_{\gamma',\wga'}\vert_{F}= h_{\gamma\gamma'}(\bullet c(\wga,\wga'))t_{\gamma\gamma',\wga\wga'}\vert_{F}
\end{equation}
for each $\wga,\wga'\in\wgamm_Y$ and any $\gamma,\gamma'\in\Gamma$ satisfying $\eta(\gamma)=p(\wga)$ and $\eta(\gamma')=p(\wga')$. Note that $c$ is well defined: both sides of (\ref{eq-def-c}) are independent of the choice of $\gamma$ and $\gamma'$ by (\ref{eq-c-doesnt-depend-on-ker}), and they are both $\gt_0$-equivariant with respect to the $\gt_0$ action given by $\wt_{\gamma}\Int_{t_{\gamma}}\wt_{\gamma'}\Int_{t_{\gamma'}}=\wt_{\gamma\gamma'}\Int_{t_{\gamma\gamma'}}$, hence they defer by an element of $\gt_0$ commuting with $\gt_0$ (in other words, an element in $Z(\gt_0)$). Because of associativity of the composition of homomorphisms of $\gt_0$ bundles, $c\in Z^2_{\tau}(\wgamm_Y,Z(\gt_0))$ is a 2-cocycle. 

Define a right action of $\wgamm_Y$ on $F$ as follows:
\begin{equation}\label{eq-def-twisted-equivariant-actiom-F}
     F\times\wgamm_Y\to F;\,(e,\wga)\mapsto e\cdot\wga:= \etag^{*-1}\hg(e)t_{\gamma,\wga},\,\eta(\gamma)=p(\wga).
\end{equation}
This is independent of the choice of $\gamma$ by (\ref{eq-c-doesnt-depend-on-ker}), it descends to the action of $\wgamm_Y$ by the construction of $t$ and it is $(\tau,c)$-twisted by (\ref{eq-def-t}) and (\ref{eq-def-c}). Thus it is a right $(\tau,c)$-twisted $\wgamm_Y$-equivariant action on $F$. 

This finishes the proof of (3). Statements (1) and (2) follow by construction, so we are done.
\end{proof}

\begin{proposition}\label{prop-simple-fixed-points-oscar-ramanan-alpha-trivial-higgs}
Let $(E,\phi)$ be a simple $G$-Higgs bundle over $X$ which is isomorphic to $(E,\phi)\gamma$ for every $\gamma\in\Gamma$. Then there exist a lift $\theta$ of $a\vert_{\ker\eta}$ and a reduction of structure group $(F,\psi)$ of $(E,\phi)$ to $\gtl:=\pt^{-1}(\Lambda)$ satisfying the following: let $p:Y\to X$ be the étale cover associated to the $\Lambda$-bundle $F/\gt_0\to X$ and $\wgam$ the subgroup of $\Aut(Y)$ lifting $\eta(\Gamma)$. Then there is a homomorphism $\otau:\wgam\to\Out(\gt_0)$ such that, for every lift $\tau:\wgam\to\Aut(\gt_0)$ (which exists by \cite{de-siebenthal}), we can find a 2-cocycle $c\in Z^2_{\tau}(\wgam,Z(\gt_0))$ such that:
\begin{enumerate}
    \item We have an isomorphism $\gtl\cong\gt_0\times_{\tau,c}\Lambda$ as extensions of $\gt_0$.
    \item $\pg^*a\in\outc$ (see Section \ref{section-group-theory}).
    \item There exists $\tilde\tau\in\homtc$ 
    % preserving $\liegm$ 
    such that $q_*(\tilde\tau)=\pg^*a$ and the tautological reduction of $p^*(F,\psi)$ to $\gt_0$ is a $(\tau,c,\rhotm)$-twisted $\wgam$-equivariant $(\gt_0,\liegm)$-Higgs pair.
\end{enumerate}
\end{proposition}

\begin{proof}
The argument is the same as the proof of Proposition \ref{prop-simple-fixed-points-oscar-ramanan-alpha-trivial-principal} after replacing Proposition \ref{prop-simple-fixed-points-principal} with Proposition \ref{prop-simple-fixed-points-higgs}. However, we need to check several things related to the Higgs field.

First recall that at the beginning of the proof we used Proposition \ref{prop-reduction-principal} to get a lift $\theta:\ker\eta\to\Aut(G)$ of $a\vert_{\ker\eta}$ and a reduction of structure group $F$ of $E$ to $\gt$. In this case we use Proposition \ref{prop-simple-fixed-points-higgs}, so that we obtain a reduction of structure group $(F,\psi)$ of $(E,\phi)$ to $(\gt,\liegm)$. On the other hand take any lift $\wt:\Gamma\to\Aut(G)$ of $a$. By Lemma \ref{lemma-lifts-vs-non-abelian-cohomology} there is a 1-cocycle $\Int_s\in Z^1_{\wt}(\ker\eta,\Int(G))$ such that $\theta=\Int_s\wt$. By assumption we have an isomorphism 
$$\hg:(E,\phi)\to\etag^{*-1}\wtg(E,\mug^{-1}\phi)$$
for every $\gamma\in\Gamma$. In the proof of Proposition \ref{prop-simple-fixed-points-oscar-ramanan-alpha-trivial-principal} we showed that $F':=\etag^*\hg(F)$ is a reduction of $E$ to $G^{\theta'}$, where $\theta'$ is given by (\ref{eq-def-theta'}). It is left to show that there exists a section $\psi'\in H^0(X,F'(\lieg^{\theta'}_{\mu})\otimes K_X)$ such that the Higgs pair $(F',\psi')$ is a reduction of structure group of $(E,\phi)$ ---in other words, that the Higgs field $\phi$ actually lies in $H^0(X,F'(\lieg^{\theta'}_{\mu})\otimes K_X)$.

Using the equation $\hb(\phi)=\mub^{-1}\fb^{*}\wt_{\beta}(\phi)$ we know that $\phi=\mub\wtb\fb^{*}\hb(\phi)\in H^0(X,E(\lieg)\otimes K)$. Given $x\in X$, $e\in \eta_x$, $k\in K_x$ and an element $v\in \liegm$ such that $\phi_x=(e,v)\otimes k$, we have 
$$\phi_{\eta_{\beta}(x)}=[(\fb^{*}\hb(e),\mub d\wtb(v))\otimes \fb^{*}k]_{\fb(x)},$$
hence it is enough to show that $d\wtb(\liegm)\subseteq\lieg^{\theta'}_{\mu}$ (in fact they are equal). But, for every $\gamma\in\ker\eta$ and $v\in\liegm$, we have
\begin{align*}
    d\ttg d\wtb(v)= \Ad_{\wtb(s_{\beta\gamma\beta^{-1}})}d\wtg d\wtb(v)=d\wtb \Ad_{s_{\beta\gamma\beta^{-1}}} d\wt_{\beta\gamma\beta^{-1}} (v)=d\wtb \theta_{\beta\gamma\beta^{-1}} (v)=\\
    \mu_{\beta\gamma\beta^{-1}} d\wtb (v)=\mug d\wtb(v),
\end{align*}
where the last equality follows from the fact that $\C^*$ is abelian and $\mu$ is a homomorphism. 

Next we defined an étale cover $Y:=F/\gt_0$, which we assume for simplicity that is connected, and a $(\theta,c)$-twisted $\wgam$-action on $F$, which is thought of as a $\gt_0$-bundle over $Y$. 
According to Section \ref{section-group-theory}, in order to define the action of $\wgamm_Y$ on $H^0(Y,F(\liegm)\otimes K_X)$ we need $\tau$ to be the restriction of a $c$-twisted homomorphism $\tilde{\tau}:\wgame\to\Aut(G)$ such that $q_*\tilde{\tau}=p_{\Gamma^*}a$ (which is statement (2)). We claim that $\tilde\tau:=\wt\Int_t:\wgame\to\Aut(G)$ is a $c$-twisted homomorphism: we have an action
\begin{equation*}
     E\times\wgame\to E;\,(e,(\gamma,\wga))\mapsto e\cdot (\gamma,\wga):= \etag^{*}\hg(e)t_{\gamma,\wga}.
\end{equation*}
By simplicity of $E$ we have 
$$eg_{(\gamma,\wga),(\gamma',\wga')}\cdot (\gamma,\wga)\cdot (\gamma',\wga')=e\cdot (\gamma\gamma',\wga\wga')$$
for each $e\in E$, $(\gamma,\wga)$ and $(\gamma',\wga')\in\wgame$ and some $g_{(\gamma,\wga),(\gamma',\wga')}\in G$ (depending on $(\gamma,\wga)$ and $(\gamma',\wga')$ but not on $e$). But this action restricts to (\ref{eq-def-twisted-equivariant-actiom-F}) on $F$, hence in fact $g_{(\gamma,\wga),(\gamma',\wga')}=c(\wga,\wga')$. The claim follows by considering the $G$-action on $E$ and noting that $q_*(\tilde\tau)=\pg^*a$ is implied by the fact that $\tilde\theta$ lifts $a$. 

Thus we have constructed a $c$-twisted homomorphism $\tilde\tau\in\homtc$ 
such that $q_*(\tilde\tau)=\pg^*a$. Moreover, $F$ is a $(\tau,c)$-twisted $\wgam$-equivariant $\gt_0$-bundle. To finish the proof of (3) we have to show that the Higgs field $\psi$ such that $(F,\psi)$ is a reduction of structure group of $(E,\phi)$ is $\wgam$-invariant with respect to the twisted equivariant action of $\wgam$ on $F$ and $\rhotm$ (see Section \ref{section-group-theory}). Let $(e,v)\otimes k$ be a local expression for $\psi$, where $e\in F$, $v\in\liegm$ and $k\in K_Y$. Keeping the notation of the proof of Proposition \ref{prop-simple-fixed-points-oscar-ramanan-alpha-trivial-principal}, we have
\begin{align*}
    ((e,v)\otimes k)\cdot\wga
    &=(e\cdot\wga,\rhotm^{-1}(v))\otimes \etag^{*}k
    \\&=(e\cdot\wga,\mug d\tilde\taug^{-1}(v))\otimes \etag^{*}k
    \\&=(\etag^{*}\hg(e)t_{\gamma,\wga},\mug\Ad_{t_{\gamma,\wga}}^{-1}d\wtg(v))\otimes \etag^{*}k
    \\&=
    (\etag^{*}\hg(e),\mug d\wtg(v))\otimes \etag^{*-1}k
    \\&=(e,v)\otimes k
\end{align*}
for every $(\gamma,\wga)\in\wgame$, where the last equation follows from the definition of $h$. 

This finishes the proof of (3). Statements (1) and (2) follow by construction, so we are done.
\end{proof}

\begin{proposition}\label{prop-prym-narasimhan-ramanan-alpha-trivial-higgs}
Consider a lift $\theta:\ker\eta\to\Aut(G)$ of $a\vert_{\ker\eta}$, a subgroup $\Lambda\le\gamtz$, a connected étale cover $p:Y\to X$ with Galois group $\Lambda$ and the group $\wgam\le\Aut(Y)$ fitting in (\ref{eq-extension-wgam}). Let $\tau$ be a homomorphism $\tau:\wgam\to\Aut(\gt_0)$ and $c\in\zzgtw$ a 2-cocycle such that there is an isomorphism of extensions $\gtl=\pt^{-1}(\Lambda)\cong \gt_0\times_{\tau,c}\Lambda$. Assume that $\pg^*a\in\outc$ (see Section \ref{section-group-theory}) and pick $\tilde\tau\in \homtc$ such that $q_*(\tilde\tau)=\pg^*a$.

We have a morphism
\begin{equation}
    \cM(Y,\gt_0,\wgam,\tau,c,\liegm,\rhotm)\to \cM(X,G),
\end{equation}
given by Theorem \ref{th-prym-narasimhan-ramanan-higgs} and extension of structure group.
\end{proposition}
\begin{proof}
Let $(E,\cdot,\phi)$ be a polystable $(\tau,c,\rhotm)$-twisted $\wgam$-equivariant $(\gt_0,\liegm)$-Higgs pair over $Y$. By Theorem \ref{EH1-equivariant} there exists a $\wgam$-invariant metric satisfying the Hitchin Equation (\ref{hitchin-equation}). In particular it is $\gamtz$-invariant, hence by Theorem \ref{EH1-equivariant} the underlying twisted $\Lambda$-equivariant Higgs pair is polystable and, by Proposition \ref{prop-polystability-extension-of-structure-group-non-connected} and Theorem \ref{th-prym-narasimhan-ramanan-higgs}, the $(\gt,\liegm)$-Higgs pair over $X$ given by Propositions \ref{prop-twisted-equivariant-higgs pairs-one-to-one} and extension of structure group is polystable. By Proposition \ref{prop-polystability-extension-structure-group} the extension of structure group to $G$, which is a $G$-Higgs bundle, is polystable as required.
\end{proof}

\begin{corollary}\label{cor-prym-narasimhan-ramanan-alpha-trivial-principal}
    Consider a lift $\theta:\ker\eta\to\Aut(G)$ of $a\vert_{\ker\eta}$, a subgroup $\Lambda\le\gamtz$, a connected étale cover $p:Y\to X$ with Galois group $\Lambda$ and the group $\wgam\le\Aut(Y)$ fitting in (\ref{eq-extension-wgam}). Let $\tau$ be a homomorphism $\tau:\wgam\to\Aut(\gt_0)$ and $c\in\zzgtw$ a 2-cocycle such that there is an isomorphism of extensions $\gtl=\pt^{-1}(\Lambda)\cong \gt_0\times_{\tau,c}\Lambda$. Assume that $\pg^*a\in\outc$ (see Section \ref{section-group-theory}).

We have a morphism
\begin{equation}
    M(Y,\gt_0,\wgam,\tau,c)\to M(X,G),
\end{equation}
given by Theorem \ref{cor-prym-narasimhan-ramanan-principal} and extension of structure group.
\end{corollary}

Let $\wcM(Y,\gt_0,\wgam,\tau,c,\liegm,\rhotm)$ and $\widetilde M(Y,\gt_0,\wgam,\tau,c)$ be the images of\\ $\cM(Y,\gt_0,\wgam,\tau,c,\liegm,\rhotm)$ and $M(Y,\gt_0,\wgam,\tau,c)$ in $\cM(X,G)$ and $M(X,G)$ respectively, and similarly when the isotropy data is fixed.

\begin{theorem}\label{th-prym-narasimhan-ramanan-alpha-trivial-higgs}
Fix $\theta\in\Hom(\ker\eta,\Aut(G))$ lifting $a$. We have the following relations between moduli spaces:
\begin{enumerate}
    \item $$\bigcup_{[\beta],Y,[\tau^{\beta\theta}],[c^{\beta\theta}],\tilde\tau,[\sigma]}\wcM(Y,G^{\beta\theta}_0,\wgam,\tau^{\beta\theta},c^{\beta\theta},\lie g^{\beta\theta}_{\mu},\rhotm,\sigma)
    \subset\cM(X,G)^{\Gamma}. $$
    
    \item $$\cM_{ss}(X,G)^{\Gamma}\subset
    \bigcup_{[\beta],Y,[\tau^{\beta\theta}],[c^{\beta\theta}],\tilde\tau,[\sigma]}\wcM(Y,G^{\beta\theta}_0,\wgam,\tau^{\beta\theta},c^{\beta\theta},\lie g^{\beta\theta}_{\mu},\rhotm,\sigma).$$
    
\end{enumerate}
Here $[\beta]$ runs through $H^1_{\theta}(\Gamma,\Int(G))$, $Y$ runs over étale covers of $X$ with Galois group
$\Lambda\le\widehat{\Gamma}^{\beta\theta}$, $[\tau^{\beta\theta}]\in \Hom(\wgam,\Out(G^{\beta\theta}_0))$ and $[c^{\beta\theta}]\in H^2_{\tau^{\beta\theta}}(\wgam,Z(G^{\beta\theta}_0))$ are such that $\pg^*a\in (\Hom_{\beta\theta,\tau^{\beta\theta},c^{\beta\theta}}(\wgame,\Out(G)))$, their restrictions to $\Lambda$ satisfy $G^{\beta\theta}_0\times_{\tau^{\beta\theta},c^{\beta\theta}}\Lambda\cong \gtl$ as extensions and $[\sigma]\in H^1_{c^{\beta\theta}_{x_i}}(\wgamm_{Y,x_i},G)$. Moreover, for each choice of $[\beta]$, $[\tau^{\beta\theta}]$ and $[c^{\beta\theta}]$, the element $\tilde\tau\in\Hom_{\beta\theta,\tau^{\beta\theta},c^{\beta\theta}}(\wgame,\Aut(G))$ 
% preserves $\liegm$ and
satisfies $q_*\tilde\tau=\pg^*a$.

\end{theorem}

\begin{proof}
Follows from Propositions \ref{prop-fixed-points-reduction-alpha-trivial-higgs}, \ref{prop-simple-fixed-points-oscar-ramanan-alpha-trivial-higgs} and \ref{prop-prym-narasimhan-ramanan-alpha-trivial-higgs}.
\end{proof}

\begin{theorem}\label{th-prym-narasimhan-ramanan-alpha-trivial-principal}
Fix $\theta\in\Hom(\ker\eta,\Aut(G))$ lifting $a\vert_{\ker\eta}$. We have the following relations between moduli spaces:
\begin{enumerate}
    \item $$\bigcup_{[\beta],Y,[\tau^{\beta\theta}],[c^{\beta\theta}],[\sigma]}\widetilde M(Y,G^{\beta\theta}_0,\wgam,\tau^{\beta\theta},c^{\beta\theta},\sigma)
    \subset M(X,G)^{\Gamma}. $$
    
    \item $$M_{ss}(X,G)^{\Gamma}\subset
    \bigcup_{[\beta],Y,[\tau^{\beta\theta}],[c^{\beta\theta}],[\sigma]}\widetilde M(Y,G^{\beta\theta}_0,\wgam,\tau^{\beta\theta},c^{\beta\theta},\sigma).$$
    
\end{enumerate}
Here $[\beta]$ runs through $H^1_{\theta}(\Gamma,\Int(G))$, $Y$ runs over étale covers of $X$ with Galois group
$\Lambda\le\widehat{\Gamma}^{\beta\theta}$, $[\tau^{\beta\theta}]\in \Hom(\wgam,\Out(G^{\beta\theta}_0))$ and $[c^{\beta\theta}]\in H^2_{\tau^{\beta\theta}}(\wgam,Z(G^{\beta\theta}_0))$ are such that $\pg^*a\in (\Hom_{\beta\theta,\tau^{\beta\theta},c^{\beta\theta}}(\wgame,\Out(G)))$, their restrictions to $\Lambda$ satisfy $G^{\beta\theta}_0\times_{\tau^{\beta\theta},c^{\beta\theta}}\Lambda\cong \gtl$ as extensions and $[\sigma]\in H^1_{c^{\beta\theta}_{x_i}}(\wgamm_{Y,x_i},G)$.

\end{theorem}
\begin{proof}
Follows from Propositions \ref{prop-fixed-points-reduction-alpha-trivial-principal} and \ref{prop-simple-fixed-points-oscar-ramanan-alpha-trivial-principal} and Corollary \ref{cor-prym-narasimhan-ramanan-alpha-trivial-principal}.
\end{proof}

\section{The general theorem}\label{section-general-theorem}

Now we tackle the general case. Let $X$ be a compact Riemann surface, $G$ a connected semisimple complex Lie group with centre $Z$ and $\Gamma$ a finite subgroup of $H^1(X,Z)\rtimes(\Aut(X)\times\Out(G))\times\C^*$. By Section \ref{section-action} we have a right action of $\Gamma$ on $\mdl(X,G)$. The projections on each factor provide homomorphisms $\eta:\Gamma\to\Aut(X)$, $a:\Gamma\to\Out(G)$ and $\mu:\Gamma\to\C^*$, together with a 1-cocycle $\alpha\in Z^1_{a,\eta}(\Gamma,H^1(X,Z))$, where the action of $\Gamma$ on $H^1(X,Z)$ is determined by $a$ (via extension of structure group) and $\eta$ (via pullback). In other words, this is a map $\alpha:\Gamma\to H^1(X,Z)$ satisfying
\begin{equation}\label{eq-1-cocycle-eta-a}
    \alpha_{\gamma\gamma'}=\alg\etag^{*-1}\ag(\alpha_{\gamma'})
\end{equation}
for each $\gamma$ and $\gamma'\in\Gamma$.
The restriction $\alpha\vert_{\ker\eta}$ is 1-cocycle in $Z^1_{a}(\ker\eta,H^1(X,Z))\cong H^1(X,Z^{1}_a(\ker\eta,Z))$, thus any of its connected components provides an étale cover $X_{\alpha,\eta}\to X$ with Galois group $\Gamma_{\alpha,\eta}\le Z^{1}_a(\ker\eta,Z)$. 

Now pick a lift $\theta:\ker\eta\to\Aut(G)$ of $a\vert_{\ker\eta}$. Let $p:Y\to X_{\alpha,\eta}\to X$ be a connected component of a $\gamtt$-bundle in $\qqt^{-1}(\alpha\vert_{\ker\eta})$ (see (\ref{eq-def-qt})), and set $\olambda:=\gal(Y/X)\le \gamtt$. Consider the subgroup $\halpha\le H^1(X,Z)$ generated by the image of $\alpha$, which is finite because both $\Gamma$ and $Z$ are finite (thus any element of $H^1(X,Z)$ has finite order). Its image $p^*\halpha\le H^1(Y,Z)$ via pullback is also a finite subgroup determining a connected étale cover $p_{Y_{\alpha}}:Y_{\alpha}\xrightarrow{p_{\halpha}} Y\to X$. Like any pullback, this also has a projection $Y_{\alpha}\to\hat\alpha\in H^1(X,\Hom(\hat\alpha,Z))$, where $\hat\alpha$ is regarded as an étale cover of $X$. Let $\Lambda:=\gal(Y_{\alpha}/X)$, call $\wgam$ to the group of automorphisms of $Y_{\alpha}$ lifting $\eta(\Gamma)\le\Aut(X)$ and let $\wgame:=\{(\gamma,\wga)\in\Gamma\times\wgam\suhthat \eta(\gamma)=p(\wga)\}$. The commutative diagramme (\ref{eq-extension-wgam}) still holds and it has exact rows and columns. We also have a diagramme (\ref{eq-restriction-extension maps}), with the same notation. We may also define $\homtc$. 

Given $\taut:\wgam\to\Aut(\gt_0)$ and $\ct\in Z^2_{\taut}(\wgam,Z(\gt_0))$ whose restrictions to $\Lambda$ factor through $\olambda$ and satisfy $\gtll\cong \gt_0\times_{\taut,\ct}\olambda$, together with an element $\tilde\tau\in\homtc$, there is a $(\taut,\ct)$-twisted $ \wgam$-right action on $\liegm$ defined by (\ref{eq-action-tildetau}), which we call $\rhotm:\wgam\to\Hom(\liegm,\lieg)$.

As in Proposition \ref{prop-prym-narasimhan-ramanan-alpha-trivial-higgs} we have a morphism 
\begin{equation*}
    \cM(Y_{\alpha},G^{\theta}_0,\wgam,\tau^{\theta},c^{\theta},\lie g^{\theta}_{\mu},\rhotm,\sigma)
    \to\cM(X,G)
\end{equation*}
for each $\sigma\in \{Z^1_{\taut}(\Gamma_{x_i},\gt_0)\}$ (here $x_i$ are the isotropy points of $Y_{\alpha}$, with isotropy groups $\Gamma_{x_i}$), whose image we call $\wcM(Y_{\alpha},G^{\theta}_0,\wgam,\tau^{\theta},c^{\theta},\lie g^{\theta}_{\mu},\rhotm,\sigma)
    $. We define $$\widetilde{M}(Y_{\alpha},G^{\theta}_0,\wgam,\tau^{\theta},c^{\theta}
    ,\sigma)\subset M(X,G)$$ similarly.

\begin{theorem}\label{th-prym-narasimhan-ramanan-general}
Fix $\theta\in\Hom(\ker\eta,\Aut(G))$ lifting $a$. We have the following relations between moduli spaces:
\begin{enumerate}
    \item $$\bigcup_{[\beta],Y,[\tau^{\beta\theta}],[c^{\beta\theta}],\tilde\tau,[\sigma]}\wcM(Y_{\alpha},G^{\beta\theta}_0,\wgam,\tau^{\beta\theta},c^{\beta\theta},\tilde\tau,\lie g^{\beta\theta}_{\mu},\rhotm,\sigma)
    \subset\cM(X,G)^{\Gamma}. $$
    
    \item $$\cM_{ss}(X,G)^{\Gamma}\subset
    \bigcup_{[\beta],Y,[\tau^{\beta\theta}],[c^{\beta\theta}],\tilde\tau,[\sigma]}\wcM(Y_{\alpha},G^{\beta\theta}_0,\wgam,\tau^{\beta\theta},c^{\beta\theta},\tilde\tau,\lie g^{\beta\theta}_{\mu},\rhotm,\sigma).$$
    
\end{enumerate}
Here $[\beta]$ runs through $H^1_{\theta}(\Gamma,\Int(G))$, $Y$ runs over étale covers of $X_{\alpha,\eta}$ which are connected components of $\gamtt$-bundles in $\qt^{-1}(\alpha\vert_{\ker\eta})$, $[\tau^{\beta\theta}]\in \Hom(\wgam,\Out(G^{\beta\theta}_0))$ and $[c^{\beta\theta}]\in H^2_{\tau^{\beta\theta}}(\Gamma,Z(G^{\beta\theta}_0))$ are such that their restrictions to $\Lambda$ factor through $\olambda$ and satisfy 
\begin{equation}\label{iso-Gbetatheta-twisted-product}
G^{\beta\theta}_0\times_{\tau^{\beta\theta},c^{\beta\theta}}\olambda\cong \gtll,
\end{equation}
and $[\sigma]\in H^1_{c^{\beta\theta}_{x_i}}(\wgamm_{Y,x_i},G)$.
Moreover, for each choice of $[\beta]$, $[\tau^{\beta\theta}]$ and $[c^{\beta\theta}]$, the twisted homomorphism $\tilde\tau\in\Hom_{\beta\theta,\tau^{\beta\theta},c^{\beta\theta}}(\wgame,\Aut(G))$ is chosen so that it preserves $\lie g^{\beta\theta}_{\mu}$ and $q_*\tilde\tau=\pg^*a$ (if such a choice exists) and, if $t:\olambda\to\gtll$ is the map realizing (\ref{iso-Gbetatheta-twisted-product}) as in the proof of Proposition \ref{prop-extensions-isomorphic-twisted-group}, we have
\begin{equation}\label{eq-2-cocycle-vs-alpha}
    c^{\gamma\theta}(\hat\gamma^{-1},\hat\gamma)^{-1}c^{\gamma\theta}(\hat\gamma^{-1},\lambda)c^{\gamma\theta}(\hat\gamma^{-1}\lambda,\gamma)\tilde\tau^{-1}_{\gamma,\hat\gamma}(t_{\lambda})t_{\hat\gamma^{-1}\lambda\hat\gamma}^{-1}=\langle\alpha_{\gamma},\lambda\rangle,
\end{equation}
for every $\gamma\in \Gamma$ and $\lambda\in\Lambda$, where $(\gamma,\hat\gamma)\in\wgame$ and $\langle\alpha_{\gamma},\lambda\rangle$ is the evaluation of $p_{Y}\circ\lambda\in\Hom(\hat\alpha,Z)$ at $\alpha_{\gamma}$.
\end{theorem}
\begin{proof}[Sketch of the proof]
(1) follows by an argument analogous to the proof of Proposition \ref{prop-simple-fixed-points-oscar-ramanan-alpha-trivial-higgs}, except for (\ref{eq-2-cocycle-vs-alpha}). To prove (2) let $(E,\phi)$ be a simple $G$-bundle which is fixed by the $\Gamma$-action. Using Proposition \ref{prop-simple-fixed-points-higgs} we get a reduction of structure group $(F,\psi)$ to a $(G_{\beta\theta},\lie g^{\beta\theta}_{\mu})$-Higgs pair for some lift $\beta\theta$ of $a\vert_{\ker\eta}$. Let $Y$ be the étale cover of $X$ given by a connected component of the $\gamtt$-bundle $F(\gs/\gt_0)$. Using the equivalence between Higgs bundles on $X$ and $\Lambda$-equivariant Higgs bundles on $Y_{\alpha}$, together with the fact that $p_{Y_{\alpha}}^*\hat\alpha$ is trivial on $Y_{\alpha}$, we may get a twisted equivariant $\wgam$-action on $p_{Y_{\alpha}}^*(F,\psi)$ as in the proof of Proposition \ref{prop-simple-fixed-points-oscar-ramanan-alpha-trivial-principal}. 

We get (\ref{eq-2-cocycle-vs-alpha}) by considering each $p_{Y_{\alpha}}^*\alg$ as the trivial bundle with action twisted by the pairing of $\alg$ with the elements of the Galois group $\Lambda$.
\end{proof}

\begin{theorem}\label{th-prym-narasimhan-ramanan-general-principal}
Fix $\theta\in\Hom(\ker\eta,\Aut(G))$ lifting $a$. We have the following relations between moduli spaces:
\begin{enumerate}
    \item $$\bigcup_{[\beta],Y_{\alpha},[\tau^{\beta\theta}],[c^{\beta\theta}],[\sigma]}\widetilde M(Y,G^{\beta\theta}_0,\wgam,\tau^{\beta\theta},c^{\beta\theta},\sigma)
    \subset M(X,G)^{\Gamma}. $$
    
    \item $$M_{ss}(X,G)^{\Gamma}\subset
    \bigcup_{[\beta],Y_{\alpha},[\tau^{\beta\theta}],[c^{\beta\theta}],[\sigma]}\widetilde M(Y,G^{\beta\theta}_0,\wgam,\tau^{\beta\theta},c^{\beta\theta},\sigma).$$
    
\end{enumerate}
Here $[\beta]$ runs through $H^1_{\theta}(\Gamma,\Int(G))$, $Y$ runs over étale covers of $X$ with Galois group equal to a subgroup
$\Lambda\le\widehat{\Gamma}^{\beta\theta}$, $[\tau^{\beta\theta}]\in \Hom(\wgam,\Out(G^{\beta\theta}_0))$ and $[c^{\beta\theta}]\in H^2_{\tau^{\beta\theta}}(\wgam,Z(G^{\beta\theta}_0))$ are such that $\pg^*a\in (\Hom_{\beta\theta,\tau^{\beta\theta},c^{\beta\theta}}(\wgame,\Out(G)))$,  their restrictions to $\Lambda$ factor through $\olambda$, satisfy 
\begin{equation*}
    G^{\beta\theta}_0\times_{\tau^{\beta\theta},c^{\beta\theta}}\olambda\cong \gtll,
\end{equation*}
and $[\sigma]\in H^1_{c^{\beta\theta}_{x_i}}(\wgamm_{Y,x_i},G)$.
Moreover, if $t:\olambda\to\gtll$ is the map realizing (\ref{iso-Gbetatheta-twisted-product}) as in the proof of Proposition \ref{prop-extensions-isomorphic-twisted-group}, we have (\ref{eq-2-cocycle-vs-alpha}).
% \begin{equation}\label{eq-2-cocycle-vs-alpha}
%     c^{\gamma\theta}(\hat\gamma^{-1},\hat\gamma)^{-1}c^{\gamma\theta}(\hat\gamma^{-1},\lambda)c^{\gamma\theta}(\hat\gamma^{-1}\lambda,\gamma)\tilde\tau^{-1}_{\gamma,\hat\gamma}(t_{\lambda})t_{\hat\gamma^{-1}\lambda\hat\gamma}^{-1}=\langle\alpha_{\gamma},\lambda\rangle,
% \end{equation}
% for every $\gamma\in \Gamma$ and $\lambda\in\Lambda$, where $(\gamma,\hat\gamma)\in\wgame$ and $\langle\alpha_{\gamma},\lambda\rangle$ is the evaluation of $p_{Y}\circ\lambda\in\Hom(\hat\alpha,Z)$ at $\alpha_{\gamma}$.
\end{theorem}

If $\mu$ is trivial, we may use the notation of Sections \ref{section-twisted-equivariant-higgs-pairs-and-hitchin-equations} and \ref{section-character-variety-alpha-trivial} to get a generalization of Theorems \ref{main-rep} and \ref{th-prym-narasimhan-ramanan-character-varieties}:

\begin{theorem}\label{th-prym-narasimhan-ramanan-character-general}
Let $\mu:\Gamma\to\C^*$ be trivial. Fix $\theta\in\Hom(\ker\eta,\Aut(G))$ lifting $a$. We have the following relations between character varieties:
\begin{enumerate}
    \item $$\bigcup_{[\beta],Y,[\tau^{\beta\theta}],[c^{\beta\theta}]}\wcalR(Y,G^{\beta\theta}_0,\wgam,\tau^{\beta\theta},c^{\beta\theta})
    \subset\calR(X,G)^{\Gamma}. $$
    
    \item $$\calR_{\irr}(X,G)^{\Gamma}\subset
    \bigcup_{[\beta],Y,[\tau^{\beta\theta}],[c^{\beta\theta}]}\wcalR(Y,G^{\beta\theta}_0,\wgam,\tau^{\beta\theta},c^{\beta\theta}).$$
    
\end{enumerate}
Here $[\beta]$ runs through $H^1_{\theta}(\Gamma,\Int(G))$, $Y$ runs over étale covers of $X_{\alpha,\eta}$ which are connected components of $\gamtt$-bundles in $\qt^{-1}(\alpha\vert_{\ker\eta})$, and $[\tau^{\beta\theta}]\in \Hom(\wgam,\Out(G^{\beta\theta}_0))$ and $[c^{\beta\theta}]\in H^2_{\tau^{\beta\theta}}(\wgam,Z(G^{\beta\theta}_0))$ are such that $\pg^*a\in (\Hom_{\beta\theta,\tau^{\beta\theta},c^{\beta\theta}}(\wgame,\Out(G)))$ and their restrictions to $\Lambda$ factor through $\olambda$ and satisfy 
(\ref{iso-Gbetatheta-twisted-product}).
Moreover,
% for each choice of $[\beta]$, $[\tau^{\beta\theta}]$ and $[c^{\beta\theta}]$, we choose a twisted homomorphism $\tilde\tau\in\Hom_{\beta\theta,\tau^{\beta\theta},c^{\beta\theta}}(\wgame,\Aut(G))$ so that $q_*\tilde\tau=\pg^*a$ (if such a choice exists) and,
if $t:\olambda\to\gtll$ is the map realizing (\ref{iso-Gbetatheta-twisted-product}) as in the proof of Proposition \ref{prop-extensions-isomorphic-twisted-group}, we have (\ref{eq-2-cocycle-vs-alpha})
for every $\gamma\in \Gamma$ and $\lambda\in\Lambda$, where $(\gamma,\hat\gamma)\in\wgame$.
\end{theorem}
\begin{proof}
Use Theorems \ref{equivariant-nahc} and \ref{th-prym-narasimhan-ramanan-general}.
\end{proof}

\newpage
\chapter{Fixed points on moduli spaces of arbitrary Higgs pairs}\label{chapter-pairs}
Even though we are particularly interested in the study of fixed points in moduli spaces of Higgs bundles, the same arguments may be performed for more general moduli spaces of Higgs pairs. Throughout this section $X$ will be a compact Riemann surface, $G$ a connected semisimple complex Lie group with centre $Z$ and Lie algebra $\lie g$ and $\rho:G\to\GL(V)$ a representation in a complex vector space $V$. 

We assume that $Z\subset\ker\rho$, i.e. the action of $Z$ on $V$ is trivial ---this is of course true in the case of the adjoint representation. It implies, for example, that for each $z\in Z$ and each $(G,V)$-Higgs pair $(E,\phi)$, we have an automorphism sending $e\in E$ to $ez$ which preserves the Higgs field. As usual, we call the Higgs pairs with automorphism group equal to $Z$ \textbf{simple}. 

\begin{remark}
    The previous assumption is mostly a matter of notation, since it can be omitted by replacing $Z$ with $Z\cap\ker\rho$. With this minor change, all the results of this section hold.
\end{remark}

\section{Group actions on moduli spaces of Higgs pairs}\label{section-action-pairs}

\begin{definition}
    We denote by $\GL_G(V)\subset \GL(V)\times\Aut(G)$ the set of pairs $(\da,\theta)$, where $\theta\in\Aut(G)$ is a group automorphism and $\da\in\GL(V)$ is a $\theta$-twisted linear automorphism, i.e. a linear automorphism satisfying
    \begin{equation}\label{eq-compatibility-rhog-theta}
        \da\rho(g)v=\rho(\theta(g))\da v
    \end{equation}
    for each $v\in V$ and $g\in G$.
    
    This has a group structure given by the restriction of the product multiplication on $\GL(V)\times\Aut(G)$, since for every two pairs $(\da,\theta)$ and $(\da',\theta')\in\GL_G(V)$ we have
    \begin{equation*}
        \da\da'\rho(g)v=\da\rho(\theta'(g))\da'=
        \rho(\theta\theta'(g))\da\da'v.
    \end{equation*}
\end{definition}

We have a homomorphism 
\begin{equation*}
    \rho_G:G\to\GL_G(V);\,s\mapsto (\rho(s),\Int_s),
\end{equation*}
since, for every $v\in V$ and every two elements $s,g\in G$, we have
$$\rho(s)\rho(g)v=\rho(sgs^{-1})\rho(s)v.$$
The subgroup $\rho_G(G)\le \GL_G(V)$ is normal: 
\begin{align*}
    (\da,\theta)\rho_G(s)(\da,\theta)^{-1}
    =(\da,\theta)(\rho(s),\Int_s)(\da^{-1},\theta^{-1})
    =(\rho(\theta(s)),\Int_{\theta(s)})
    =\rho_G(\theta(s))
\end{align*}
for every $(\da,\theta)\in\GL_G(V)$ and $s\in G$.

The group $\GL_G(V)$ acts on the set of $(G,V)$-Higgs pairs on the left: each $(\da,\theta)$ sends a Higgs pair $(E,\phi)$ to $(\theta(E),\da(\phi))$, where $\theta(E)$ is given in Section \ref{section-action} and $\da(\phi)$ is defined using the action of $\da$ on $V$. We check that $\da(\phi)$ is well defined: if a local expression for $\phi\in H^0(X,E(V)\otimes K_X)$ is $(e,v)\otimes k$ for some $v\in V$ and local sections $e$ and $k$ of $E$ and $K_X$ respectively, we have
\begin{align*}
    (e g,\da(\rho({g^{-1}})v))&=(e \theta^{-1}(\theta(g)),\rho(\theta(g^{-1}))\da(v))
    \\&=(e\cdot\theta(g),\rho(\theta(g))^{-1}\da(v))
    \\&=(e,\da(v)),
\end{align*}
where the presence or absence of the dot denotes the $G$-action on $\theta(E)$ or $E$ respectively.

Moreover, the isomorphism class of $(E,\phi)$ is preserved by $\rho_G(G)$: for each $g\in G$, the morphism
\begin{equation*}
    E\to\Int_g(E);\,e\mapsto eg
\end{equation*}
induces the Higgs field $\rho(g)(\phi)$. Thus, we get a left action of 
\begin{equation}\label{eq-def-Out(G,V)}
    \Out(G,V):=\GL_G(V)/\rho_G(G)
\end{equation} 
on the set of isomorphism classes of $(G,V)$-Higgs pairs.

We also have a (right and left) action of $H^1(X,Z)$, such that $\alpha\in H^1(X,Z)$ sends $(E,\phi)$ to $(E\otimes\alpha,\phi)$. This is due to the fact that $Z$ acts trivially on $V$, so that $E\times_{\rho}V\cong(E\otimes \alpha)\times_{\rho}V$. Finally, $\Aut(X)$ and $\C^*$ act on the right by pullback and rescaling of the Higgs field as in Section \ref{section-action}. 

Altogether, we get a right group action of $H^1(X,Z)\rtimes(\Out(G,V)\times\Aut(X))\times\C^*$ on the set of isomorphism classes on $(G,V)$-Higgs pairs, so that $(\alpha,a,\eta,\mu)$ sends the class of $(E,\phi)$ to the class of
\begin{equation*}
    \eta^*(\theta^{-1}(E\otimes\alpha),\mu\da^{-1}(\phi)),
\end{equation*}
where $(\da,\theta)\in\GL_G(V)$ is any element in the coset $a\in\Out(G,V)$. Here $\Out(G,V)$ acts on $H^1(X,Z)$ via its projection on $\Out(G)$. We check that this is a well defined right group action: for every two elements $(\alpha,(\da,\theta),\eta,\mu)$ and $(\alpha',(\da',\theta'),\eta',\mu')$ in $ H^1(X,Z)\rtimes(\GL_G(V)\times\Aut(X))\times\C^*$, we have
\begin{align*}
    &((E,\phi)\cdot (\alpha,(\da,\theta),\eta,\mu))\cdot (\alpha',(\da',\theta'),\eta',\mu')\\
    &=[\eta^*(\theta^{-1}(E\otimes\alpha),\mu\da^{-1}(\phi))]\cdot (\alpha',(\da',\theta'),\eta',\mu')\\
    &=\eta'^*\eta^*(\theta'^{-1}(\theta^{-1}(E\otimes\alpha)\otimes\alpha'),\mu'\mu\da'^{-1}\da^{-1}(\phi))\\
    &=(\eta\eta')^*((\theta\theta')^{-1}(E\otimes\alpha\theta(\alpha')),\mu\mu'(\da\da')^{-1}(\phi))\\
    &=(E,\phi)\cdot (\alpha\theta(\alpha'),(\da\da',\theta\theta'),\eta\eta',\mu\mu')\\
    &= (E,\phi)\cdot ((\alpha,(\da,\theta),\eta,\mu) (\alpha',(\da',\theta'),\eta',\mu')).
\end{align*}
This action preserves polystability, hence we have a right group action of $H^1(X,Z)\rtimes(\Out(G,V)\times\Aut(X))\times\C^*$ on the moduli space $\mdl(X,G,V)$ of $(G,V)$-Higgs pairs over $X$.

\section{Fixed points when the action on the curve is trivial}

Let $\Gamma$ be a finite subgroup of $H^1(X,Z)\rtimes\Out(G,V)\times\C^*$. Projections on the second, third and first factors provide homomorphisms $a:\Gamma\to \Out(G,V)$, $\mu:\Gamma\to\C^*$ and a 1-cocycle $\alpha\in Z^1_a(\Gamma,H^1(X,Z))$ respectively ---here $\Gamma$ acts on $Z$ via the projection of $a$ on $\Out(G)$. Assume that there is a homomorphism $\Gamma\to\GL_G(V)$ lifting $a$ and fix one, say $(\da,\theta)$, where $\theta:\Gamma\to\Aut(G)$ is a homomorphism and $\da:\Gamma\to\GL(V)$ is a $\theta$-twisted representation. 

We denote by $V_{\mu}^{\da}$ the $\mu$-weight space of the action of $\Gamma$ on $V$, i.e. the subspace of $V$ consisting of vectors $v\in V$ such that $\daga(v)=\mug v$. Note that the action of $\gs$ on $V$ via $\rho$ preserves $\lievm$, since 
\begin{equation*}
    \daga\rho(g)v=\rho(g)\rho(z)\daga v=\mug\rho(g)v
\end{equation*}
for every $\gamma\in\Gamma$, $g\in\gs$ and $v\in V$ and some $z\in Z$. Here we are using that $\rho(Z)$ is trivial.

\begin{proposition}\label{prop-reduction-pairs}
Let $(E,\phi) $ be a $(G,V)$-Higgs pair. With notation as in Section \ref{section-Gtheta}, assume that there is a $(\gs,\lievm)$-Higgs pair $(F,\psi)$ which is a reduction of structure group of $(E,\phi)$ satisfying (\ref{eq-c-theta}):
\begin{equation*}
    \ctt(F)\cong\alpha.
\end{equation*}
Then $(E,\phi)$ is isomorphic to $(E,\phi)\gamma= (\tg^{-1}(E\otimes\ambda),\mug\daga^{-1}(\phi))$ for every $\ambda\in\llambda$.
\end{proposition}
\begin{proof}
    Follows from the proof of Proposition \ref{prop-reduction-principal} as the proof of Proposition \ref{prop-reduction-higgs} after replacing the action $\theta$ of $\Gamma$ on $\lie g$ with the action $\da$ on $V$. 
\end{proof}

Let $S^{\Gamma}$ be the set of isomorphism classes of $(G,V)$-Higgs pairs which are fixed by $\Gamma$. As in Section \ref{section-simple-and-H} we may construct a map
\begin{equation*}
    \wf:S^{\llambda}\to 
    H^1_{\theta}(\Gamma,G/Z).
\end{equation*}

\begin{proposition}\label{prop-simple-fixed-points-pairs}
Let $(E,\phi)$ be a simple $(G,V)$-Higgs pair over $X$ which is isomorphic to $(E,\phi) \ambda$ for every $\ambda$ in $\llambda$. Then a 1-cocycle $\beta\in Z^1_{\theta}(\Gamma,G/Z)$ is in $\wf(E,\phi)\in H^1_{\theta}(X,\Int(G))$ if and only if there exists a $G_{\beta\theta}$-bundle $F$ which is a reduction of structure group of $ E $ and satisfies (\ref{eq-c-betatheta}):
\begin{equation*}
    \tilde c_{\beta\theta}(F)\cong\alpha.
\end{equation*}
For each $\beta\in Z^1_{\theta}(\Gamma,G/Z)$ such a reduction is unique.
\end{proposition}
\begin{proof}
    Follows from the proof of Proposition \ref{prop-simple-fixed-points-principal} as the proof of Proposition \ref{prop-simple-fixed-points-higgs} after replacing the action $\theta'$ of $\Gamma$ on $\lie g$ with the action $\rho(s)\da$ on $V$, where $\beta=\Int_s$. 
\end{proof}

Fix a non-degenerate $G$-invariant pairing on $\lie g$. Choose a maximal compact subgroup $K$ of $G$ satisfying Lemma \ref{lemma-compact-involution} and a hermitian metric $h_V$ on $V$ such that $\rho(K)$ is contained in the group $U(V)$ of unitary automorphisms of $V$. Given a $(G,V)$-Higgs pair $(E,\phi)$, choose a metric $h\in \Omega^0(X,E(G/K))$ on $E$. We define
\begin{equation}\label{eq-def-mu(phi)}
    \mu_h(\phi):=d\rho^*(-\frac i2\phi\otimes\phi^{*h,h_V}),
\end{equation}
where $\phi^{*h,h_V}\in \Omega^0(X,E(V)^*\otimes K_X^*)$ is defined using the metric on $E$ and the hermitian metric on $V$. Here we identify $-\frac i2\phi\otimes\phi^{*h,h_V}$ as a skew symmetric section of $\End(E(V)\otimes K)^*=\End(E(V))^*$, hence a section of $E_h(\lie u(V))^*$, where $\lie u(V)$ is the Lie algebra of $U(V)$. The map $d\rho^*:E_h(\lie u(V))^*\to E_h(\lie k)^*$ is induced by the dual of the infinitesimal action $d\rho$ of $\lie k$ on $V$.

\begin{theorem}\label{EH1-pairs}
Let $(E,\phi)$ be a $(G,V)$-Higgs pair on $X$. 
Then $(E,\phi)$ is polystable if and only if the $G$-bundle $E$ and the vector space $V$ 
admit compatible metrics $h$ satisfying the {Hermite--Yang--Mills--Higgs equation}
\begin{equation}\label{hitchin-equation-pairs}
\Lambda F_h+\mu_h(\phi)=0,
\end{equation}
where $F_h\in \Omega^2(X,E_h(\lie k))$ is the Chern curvature and $\Lambda:\Omega^2(X)\to\Omega^0(X)$ is the adjoint of wedging with the volume form on $X$. The left hand side of (\ref{hitchin-equation-pairs}) may be regarded as a moment map, see \cite{PBI}.
\end{theorem}
\begin{proof}
    See \cite{PBI}.
\end{proof}

\begin{lemma}\label{lemma-invariant-metric}
Given a homomorphism $(\da,\theta):\Gamma\to\GL_G(V)$ and a $\theta(\Gamma)$-invariant maximal compact subgroup $K$ of $\gs$, there exists a $K$-invariant hermitian metric on $V$ which is also $\da(\Gamma)$-invariant.
\end{lemma}
\begin{proof}
    Take any hermitian metric $h_0$ on $V$ and define a new metric $h$ by
    \begin{equation*}
h(v,v')=\sum_{\gamma\in\Gamma}\int_{K}h_0(\daga\rho(k)v,\daga\rho(k)v')dk.
    \end{equation*}
    This is both $K$ and $\Gamma$-invariant, as required.
\end{proof}

% involution $\sigma$ of $G$ preserving $\gs$ such that, for every character $\mu:\Gamma\to\C^*$,
% $$d\sigma(\lieg^{\theta}_{\mu})=\lieg^{\theta}_{\mu^{-1}}.$$

\begin{proposition}\label{prop-polystability-extension-structure-group-pairs}
Let $(\da,\theta)\in\homgh$ be a lift of $a$. Then:
\begin{enumerate}
    \item If a $(\gs,\lievm)$-Higgs pair $(F,\psi)$ is polystable, the $(G,V)$-Higgs pair obtained by extension of structure group is also polystable.
    \item If $(E,\phi)$ is a (semi,poly)stable $(G,V)$-Higgs pair with a reduction of structure group to a $(\gs,\lievm)$-pair $(F,\psi)$, then $(F,\psi)$ is (semi,poly)stable.
    \item Given $g\in G$ and $(\da',\theta'):=\rho_G(g)(\da,\theta)\rho_G(g)^{-1}$,
    there is a canonical isomorphism between $\cM(X,\gs,\lievm)$ and $\cM(X,G_{\theta'},V^{\da'}_{\mu})$ making the following diagramme commute:
\begin{equation}\label{eq-diagramme-extension-theta-theta'-pairs}
    \begin{tikzcd}
\cM(X,\gs,\lievm)\arrow{r}\arrow{d} &
\cM(X,G,V)\\
\cM(X,G_{\theta'},V^{\da'}_{\mu})\arrow{ur}
\end{tikzcd},
\end{equation}
where the morphisms to $\cM(X,G,V)$ are given by extension of structure group. For each $\alpha\in Z^1_{a}(\Gamma,H^1(X,Z))$, it restricts to a diagramme
\[\begin{tikzcd}
\cM_{\alpha}(X,\gs,\lievm)\arrow{r}\arrow{d} &
\cM(X,G,V)\\
\cM_{\alpha}(X,G_{\theta'},\lieg^{\da'}_{\mu})\arrow{ur}
\end{tikzcd},
\]
where $\cM_{\alpha}(X,\gs,\lievm)$ is the moduli space of $(\gs,\lievm)$-Higgs pairs $(F,\psi)$ such that $\ctt(F)\cong\alpha$. 
\end{enumerate}
(1), (2) and (\ref{eq-diagramme-extension-theta-theta'-pairs}) are also true after replacing $\gs$ and $G_{\theta'}$ by $\gt$ and $G^{\theta'}$ respectively. 
\end{proposition}
\begin{proof}
The proofs of (2) and (3) are analogous to the proofs of the respective statements in Proposition \ref{prop-polystability-extension-structure-group}. In order to prove (1) fix a maximal $\Gamma$-invariant compact subgroup $K_{\theta}$ of $\gs$ and consider a maximal compact subgroup $K$ of $G$ containing it, so that $K_{\theta}=K\cap\gs$. Note that $K_{\theta}$ exists because $\gs\rtimes_{\theta}\Gamma$ is reductive, hence it must have a maximal compact subgroup. Its intersection with $\gs$, which is a maximal compact subgroup of $\gs$, must then be $\Gamma$-invariant. By Lemma \ref{lemma-invariant-metric} there is a $\Gamma$-invariant metric $h_V$ on $V$ such that $\rho(K)\subset U(V)$.
    
Given a polystable $(\gs,\lievm)$-Higgs pair $(F,\psi)$, by Theorem \ref{EH1-equivariant} and Proposition \ref{prop-twisted-equivariant-bundles-one-to-one} there exists a $\gamtt$-invariant reduction $h_{F}\in \Omega^0(F/K_{\theta})$ satisfying the Hermite--Yang--Mills--Higgs Equation (\ref{hitchin-equation-pairs}). Let $(E,\phi)$ be the extension of structure group of $(F,\psi)$ to $G$. Using the inclusion $F/K_{\theta}\subset E/K$, we get a reduction of structure group $h_E\in \Omega^0(E/K)$. 

On the other hand, (\ref{hitchin-equation}) is an equation setting a moment map equal to $0$: if we consider the topological bundle underlying $E$ and the space of $G$-connections $\mathcal A$ on it, there is an action of the gauge group preserving the metric $h$ and this provides a moment map
$m:X\to E_h(\lie k)^*,$
where $E_h$ is the reduction of $E$ to $K$ given by $h$ and $\lie{k}$ is the Lie algebra of $K$. Using the Killing form, $m$ may be regarded as a map $X\to E_h(\lie{k})$. The space $\mathcal B$ of $(\gs,\lievm)$-Higgs pairs $(B,\psi)$, where $B$ is a hermitian $\gs$-connection and $\psi\in\Omega^1(F(\lievm))$ (here $F$ is the reduction of structure group to $\gs$ determined by the $\gs$-connection) is then embedded in $\mathcal A$, and the corresponding moment map $m_{\theta}$ is the restriction of
$$X\xrightarrow{m}E_h(\lie{k})\to F_{h_{\theta}}(\lie k^{\theta})$$
to $\mathcal B$, where the second homomorphism is given by orthogonal projection and we use the obvious notations. Given a polystable $(\gs,\lievm)$-Higgs pair $(F,\psi)$ as in the previous paragraph, the moment map at $(F,\psi)$ is given by:
$$m(E,\phi)=F_{h}+\mu_h(\psi).$$
But $F_h=F_{h_{\theta}}$ (the curvature of the Chern connection for the metric $h_{\theta}$), so if we can prove that $\mu_h(\psi)\in \Omega^0(X,F_{h_{\theta}}(\lie k^{\theta}))$ then the moment map $m_{\theta}$ would just be the restriction of $m$ and so $m(E,\phi)=m_{\theta}(F,\psi)$. Thus, $(E,\phi)$ would satisfy (\ref{hitchin-equation-pairs}) and so, by Theorem \ref{EH1-pairs}, it would be polystable.

Consider the $\C$-antilinear isomorphism $h_V:V\xrightarrow{\sim} V^*$ given by $h_V$. For each $v\in\lievm$ and each $\gamma\in\Gamma$ we have
\begin{equation*}
    \daga h_V(v)=h_V(\daga(v))=h_V(\mug v)=\mug^{-1}h_V(v),
\end{equation*}
where we are using that $\gamma$ has finite order and so $\mug$ is a root of unity, which satisfies $\overline{\mug}=\mug^{-1}$. Thus $h_V(\lievm)=(V^*)^{\da^{*}}_{\mu^{-1}}$, where $\da^*:\Gamma\to\GL(V^*)$ is the dual representation. This implies that $\psi\otimes\psi^{*h,h_v}\in \Omega^0(X,F_{h_{\theta}}(\lievm\otimes V^{*\da^*}_{\mu^{-1}}))\subset \Omega^0(X,F_{h_{\theta}}(\End(V)^{\da}))$, where $\End(V)^{\da}$ is the fixed point subspace of the action of $\Gamma$ on $\End(V)$ induced by its action on $V$. Since $d\rho$ is $\Gamma$-equivariant by (\ref{eq-compatibility-rhog-theta}), we conclude using (\ref{eq-def-mu(phi)}) that $\mu_h(\phi)\in \Omega^0(X,F_{h_{\theta}}(\lie k\cap\lie g^{\theta}))=\Omega^0(X,F_{h_{\theta}}(\lie k^{\theta}))$, as required.
\end{proof}

We call $\widetilde{\mdl}_{\alpha}(X,\gs,\lievm)$ to the image of ${\mdl}_{\alpha}(X,\gs,\lievm)$ in $\mdl(X,G,V)$. By Proposition \ref{prop-polystability-extension-structure-group}, if $(\da',\theta'):=\rho_G(g)(\da,\theta)\rho_G(g)^{-1}$ for some $g\in G$, we have 
$\widetilde{\mdl}_{\alpha}(X,\gs,\lievm)=\widetilde{\mdl}_{\alpha}(X,G_{\theta'},V^{\da'}_{\mu})$.

Using the notation of Sections \ref{section-action} and \ref{section-action-pairs}, we have the following:

\begin{lemma}\label{lemma-lifts-vs-non-abelian-cohomology-pairs}
Fix a lift $(\da,\theta)$ of $a:\Gamma\to\Out(G,V)$. Let $S_a$ be the set of lifts of $a$.
There is a $G$-equivariant bijection
\begin{equation}\label{eq-bijection-Sa-cocycles-pairs}
    \{\text{Lifts of $a$}\}\leftrightarrow Z^1_{(\da,\theta)}(\Gamma,\rho_G(G));\,(\Int_s\da,\rho(s)\theta)\mapsto(\Int_s,\rho(s)),
\end{equation}
where the action on the left hand side is given by conjugation using $\rho_G$ and the action on the right hand side is given by (\ref{eq-cohomologous}).
In particular, this induces a bijection
\begin{equation}
    \{\text{Lifts of $a$}\}/G\leftrightarrow H^1_{(\da,\theta)}(\Gamma,\rho_G(G)).
\end{equation}
\end{lemma}
\begin{proof}
    This is analogous to the proof of Lemma \ref{lemma-lifts-vs-non-abelian-cohomology}.
\end{proof}

Let $\mdl(X,G,V)^{\llambda}$ be the fixed point locus of $\mdl(X,G,V)$ under the action of $\llambda$, and let $\mdl_{ss}(X,G,V)^{\llambda}$ be the intersection with the stable and simple locus.

\begin{theorem}\label{th-fixed-points-oscar-ramanan-pairs}
Fix a homomorphism $(\da,\theta):\Gamma\to\GL_G(V)$ lifting $a$. We have the following relations between moduli spaces:
\begin{enumerate}
    \item $$\bigcup_{[\dab,\beta]\in H^1_{\theta}(\Gamma,\rho_G(G))}\widetilde{\cM}_{\alpha}(X,G_{\beta\theta},V^{\dab\da}_{\mu})\subset\cM(X,G,V)^{\Gamma}.$$
    
    \item $$\cM_{ss}(X,G,V)^{\Gamma}\subset\bigcup_{[\dab,\beta]\in H^1_{\theta}(\Gamma,\rho_G(G))}\widetilde{\cM}_{\alpha}(X,G_{\beta\theta},V^{\dab\da}_{\mu}).$$
    
\end{enumerate}

 Moreover, the intersections 
$$\cM_{ss}(X,G,V)\cap\bigcup_{[\dab,\beta]\in H^1_{\theta}(\Gamma,\rho_G(G))}\widetilde{\cM}_{\alpha}(X,G_{\beta\theta},V^{\dab\da}_{\mu})$$
are disjoint for different $[\dab,\beta]\in H^1_{\theta}(\Gamma,\rho_G(G))$.
\end{theorem}
\begin{proof}
    Follows from Propositions \ref{prop-reduction-pairs} and \ref{prop-simple-fixed-points-pairs}.
\end{proof}

With notation as in Section \ref{section-prym-narasimhan-ramanan}, for every lift $(\kappa,\theta)$ of $a$ and a section $t:\gamtt\to\gs$ of (\ref{eq-extension-connected-component}) inducing the isomorphism $\gs\cong  \gt_0\times_{(\taut,\ct)}\gamtt$ as in Proposition \ref{prop-extensions-isomorphic-twisted-group}, we have a $(\taut,\ct)$-twisted action $\rho(t)$ of $\gamtt$ on $\lievm$. Let $Y$ be a $\gam$-bundle over $X$, where $\gam\le\gamtt$. We denote by ${\mdl}(Y,\gt_0,\gam,\tau^{\beta\theta},c^{\beta\theta}, \lievm)$ the moduli space of $(\tau^{\beta\theta},c^{\beta\theta},\rho(t))$-twisted $\gam$-equivariant $(\gt_0,\liegm)$-Higgs pairs over $Y$ given by $\rho(t)$. We have the following Prym--Narasimhan--Ramanan construction:

\begin{theorem}\label{th-prym-narasimhan-ramanan-oscar-ramanan-pairs}
For each homomorphism $(\da,\theta):\Gamma\to\GL_G(V)$ lifting $a$ we have an isomorphism
\begin{equation}
    \bigsqcup_{\qqt(Y)\cong \alpha}\mdl(Y,\gt_0,\tau^{\theta},\gam,c^{\theta}, \lievm)/Z_{\gamtt}(\gam)\cong \mdl_{\alpha}(X,\gs,\lievm),
\end{equation}
where $Z_{\gamtt}(\gam)$ is the centralizer of $\gam$ in $\gamtt$, which acts on  by Proposition \ref{prop-action-centralizer-higgs}.

Fix such a lift $\theta$. Let $\widetilde{\mdl}(Y,\gt_0,\gam,\tau^{\beta\theta},c^{\beta\theta}, \lievm)/Z_{\gamtt}(\gam)$ be the image of the moduli space ${\mdl}(Y,\gt_0,\gam,\tau^{\beta\theta},c^{\beta\theta}, \lievm)$ in $\mdl(X,G,V)$ via the composition of the isomorphism given in Theorem \ref{th-prym-narasimhan-ramanan-higgs} and extension of structure group from $\gs$ to $G$. Then we have the following inclusions:

\begin{enumerate}
    \item $$\bigcup_{[\dab,\beta]\in H^1_{\theta}(\Gamma,\rho_G(G)),\qqt(Y)\cong \alpha}\widetilde{\mdl}(Y,\gt_0,\gam,\tau^{\beta\theta},c^{\beta\theta}, V^{\dab\da}_{\mu})/Z_{\gamtt}(\gam)\subset\cM(X,G,V)^{\Gamma}.$$
    
    \item $$\cM_{ss}(X,G,V)^{\Gamma}\subset\bigcup_{[\dab,\beta]\in H^1_{\theta}(\Gamma,\rho_G(G)),\qqt(Y)\cong \alpha}\widetilde{\mdl}(Y,\gt_0,\gam,\tau^{\beta\theta},c^{\beta\theta}, V^{\dab\da}_{\mu})/Z_{\gamtt}(\gam).$$
    
\end{enumerate}

The intersections 
$$\cM_{ss}(X,G,V)\cap\widetilde{\mdl}(Y,\gt_0,\gam,\tau^{\beta\theta},c^{\beta\theta}, V^{\dab\da}_{\mu})/Z_{\gamtt}(\gam)
% =\cM_{ss}(X,G)^{\Gamma}\cap\widetilde{\mdl}(Y,\tau^{\beta\theta},c^{\beta\theta}, \gt_0,\gam,\lieg^{\beta\theta}_{\mu},\mu)/Z_{\gamtt}(\gam)
$$
are disjoint for different $[\dab,\beta]\in H^1_{\theta}(\Gamma,\rho_G(G))$ and $Y$.

\end{theorem}

\section{Fixed points when there is no tensorization}\label{section-alpha-trivial-pairs}

Let $\Gamma$ be a finite subgroup of $\Aut(X)\times\Out(G,V)\times\C^*$. Projections on the second and third and first factors provide homomorphisms $\eta:\Gamma\to\Aut(X)$, $a:\Gamma\to \Out(G,V)$ and $\mu:\Gamma\to\C^*$, respectively. Assume that there is a homomorphism $\Gamma\to\GL_G(V)$ lifting $a$ and fix one, say $(\da,\theta)$, where $\theta:\Gamma\to\Aut(G)$ is a homomorphism and $\da:\Gamma\to\GL(V)$ is a $\theta$-twisted representation. 

\begin{proposition}\label{underlyingsemi-pairs}
If $(E,\cdot,\phi)$ is a polystable $(\theta,c,\mu^{-1}\da)$-twisted $\Gamma$-equivariant $(G,V)$-Higgs pair, so is the underlying $(G,V)$-Higgs pair $(E,\phi)$. Thus we have a forgetful map
\begin{equation}\label{eq-forgetful-morphism-pairs}
    \mdl(X,G,\Gamma,\theta,c,V,\mu^{-1}\da)\to\mdl(X,G,V).
\end{equation}
\end{proposition}
\begin{proof}
    Analogous to Proposition \ref{underlyingsemi}.
\end{proof}

We denote the image of (\ref{eq-forgetful-morphism-pairs}) with a tilde as usual.

\begin{proposition}\label{obvious-fix-pairs}
 $\widetilde{\cM}(X,G,\Gamma,\theta,c,V,\mu^{-1}\da,\sigma)\subset \cM(X,G,V)^{\Gamma}$.
\end{proposition}
\begin{proof}
    Straightforward from definitions.
\end{proof}

\begin{proposition} \label{simple-pairs}
Let $(E,\phi)$ be a simple $(G,V)$-Higgs pair over $X$ such that 
 $(E,\phi)\cong (\eta_{\gamma}^*\theta^{-1}_{\gamma}(E),\mu_{\gamma}\eta_{\gamma}^*\daga^{-1}(\phi))$ for each $\gamma\in\Gamma$. 
Then $(E,\phi)$ admits a $(\theta,c,\mu^{-1}\da)$-twisted $\Gamma$-equivariant structure for some
$c\in Z^2_\theta(\Gamma,Z)$.
\end{proposition} 
\begin{proof}
    Analogous to Proposition \ref{simple}.
\end{proof}

\begin{theorem}\label{main-pairs}
 Let ${\cM}_{ss}(X,G,V)\subset \cM(X,G,V)$ be the subvariety of $\cM(X,G,V)$ 
consisting of those $(G,V)$-Higgs pairs
 which are stable and simple. Fix a homomorphism $(\da,\theta):\Gamma\to \GL_G(V)$ lifting $a:\Gamma\to \Out(G,V)$. Then 
 $$
{\cM}_{ss}(X,G,V)^{\Gamma}\subset \bigcup_{[c]\in H^2_a(\Gamma,Z), 
[\sigma]\in \{H^1_{c_{x_i}}(\Gamma_{x_i},G)\}} 
\widetilde{\cM}(X,G,\Gamma,\theta,c,V,\mu^{-1}\da,\sigma)
$$
and 
$$ \bigcup_{[c]\in H^2_a(\Gamma,Z), 
[\sigma]\in \{H^1_{c_{x_i}}(\Gamma_{x_i},G)\}} 
\widetilde{\cM}(X,G,\Gamma,\theta,c,V,\mu^{-1}\da,\sigma)
\subset
{\cM}(X,G,V)^{\Gamma}.
$$

\end{theorem}
\begin{proof}
    Follows from propositions \ref{obvious-fix-pairs} and \ref{simple-pairs}.
\end{proof}

\section{Fixed points for general actions}
We keep the notation of Section \ref{section-alpha-trivial-pairs} for now. We keep the assumption that there is a homomorphism $\theta:\Gamma\to\GL_G(V)$ lifting $a$, which we fix.

Let $\Lambda\le \gamtz$ be a subgroup, $\gtl:=\pt^{-1}(\Lambda)$ and let $p:Y\to X$ be a connected étale cover associated to a $\Lambda$-bundle over $X$ and consider the subgroup $\wgam\le\Aut(Y)$ lifting $\eta(\Gamma)$. This contains $\Lambda$, the Galois group of $Y$ over $X$, as a normal subgroup. Let 
\begin{equation*}
    \wgame:=\{(\gamma,\wga)\in\Gamma\times\wgam\suhthat \etag=p(\wga)\}.
\end{equation*}
Recall the commutative diagramme (\ref{eq-extension-wgam}). 

We say that a map $(\tda,\tilde\tau):\wgame\to\GL_G(V)$ is a \textbf{$c$-twisted homomorphism} for some $c\in Z^2_{\tilde\tau}(\wgam,Z(\gt_0))$ if $\tilde\tau$ is $c$-twisted ---i.e. if it satisfies (\ref{eq-c-twisted-hom})--- and
\begin{equation}\label{eq-c-twisted-da}
\tda_{\gamma,\wga}\tda_{\gamma',\wga'}=\rho(c(\gamma,\gamma'))\tda_{\gamma\gamma',\wga\wga'}
\end{equation}
for every $(\gamma,\wga)$ and $(\gamma',\wga')\in\wgame$.
Equivalently, the associated (left) actions of $\wgame$ on $G$ and $V$ are $c$-twisted. As in Section \ref{section-group-theory} denote by $\Hom_c(A,B)$ the set of $c$-twisted homomorphisms from $A$ to $B$ and let $\Hom_c(\wgame,\GL_G(V))^{\gtl}$ be the set of $c$-twisted homomorphisms whose associated $c$-twisted $\wgame$-action on $G$ preserves $\gtl$. As in Section \ref{section-group-theory} we have a restriction map
$r_{\ker\eta}:\Hom_c(\wgame,\GL_G(V))\to\Hom(\ker\eta,\GL_G(V))$. Let $\Hom_{\theta,c}(\wgame,\GL_G(V)):=\Hom_c(\wgame,\GL_G(V))^{\gt}\cap r_{\ker\eta}^{-1}(\da,\theta)$. We also have a restriction map $r_{\gtl}:\Hom_{\theta,c}(\wgame,\GL_G(V))\to\Hom_c(\wgam,\GL_{\gtl}(V))$, since 
$\theta$ is trivial on $\gtl$. Any automorphism of $\gtl$ preserves the connected component $\gt_0$ of $\gtl$, hence we have a map $r_{\gt_0}:\Hom_c(\wgam,\GL_{\gtl}(V))\to\Hom(\wgam,\GL_{\gt_0}(V))$, where the fact that $Z(\gt_0)$ acts trivially on $\gt_0$ by conjugation implies that the image consists of (honest) homomorphisms. 

In summary, we have the following diagramme:
\begin{equation}\label{eq-restriction-extension maps-pairs}
    \begin{tikzcd}
        \Hom(\wgame,\Out(G,V)) &
        \Hom_{\theta,c}(\wgame,\GL_G(V)) \arrow[l,"q_*"]\arrow[r,"r_{\gtl}"] & \Hom_c(\wgam,\GL_{\gtl}(V))\arrow[d,"r_{\gt_0}"]\\
        \Hom(\Gamma,\Out(G,V))\arrow[u,"p_{\Gamma}^*"]&       &   \Hom(\wgam,\GL_{\gt_0}(V)))
    \end{tikzcd},
 \end{equation}
where $q_*$ is the pushforward of the natural projection $\GL_G(V)\to\Out(G,V)$, given by (\ref{eq-def-Out(G,V)}), and $p_{\Gamma}^*$ is the pullback of $p_{\Gamma}:\wgame\to\Gamma$.

Conversely, given a $c$-twisted homomorphism $(\da,\tau):\wgam\to\GL_{\gt_0}(V)$, where  $c$ is a a 2-cocycle in $Z^2_{\tau}(\wgam,Z(\gt_0))$, we have a representation 
\begin{equation*}
    \rho\vert_{\gt_0}\times\dac:\gt_0\times_{\tau,c}\Lambda\to\GL(V).
\end{equation*}
We say that the pairs $(\gt_0\times_{\tau,c}\Lambda,\rho\vert_{\gt_0}\times\dac)$ and $(\gtl,\rho\vert_{\gtl})$ are \textbf{isomorphic} if there is an isomorphism $f:\gt_0\times_{\tau,c}\Lambda\xrightarrow{\sim}\gtl$ of extensions of $\gt_0$ such that $\rho\circ f=\rho\vert_{\gt_0}\times\dac $, i.e. making the following diagramme commute:
\begin{equation*}
    \begin{tikzcd}
        \gt_0\times_{\tau,c}\Lambda \arrow[r,"\rho\vert_{\gt_0}\times\dac"]\arrow[d,"f"] & \GL(V) \\
        \gtl\arrow[ru,"\rho"]
    \end{tikzcd}.
 \end{equation*}
Assuming that $\tau,\dac$ and $c$ are chosen so that this is the case, there is a $c$-twisted extension $e_{\dac,\tau,c}:\wgam\to\GL_{\gtl}(V)$ given as the composition:
\begin{equation*}
    \begin{tikzcd}[ar symbol/.style = {draw=none,"#1" description},
    isomorphic/.style = {ar symbol={\cong}},
      ]
        \wgam\arrow[r]\arrow[rrr,bend right=15,"e_{\dac,\tau,c}", swap] & \gt_0\times_{\tau,c}\wgam\arrow[r,"\rho_G\times\dac"] & 
        \GL_{\gt_0\times_{\tau,c}\Lambda}(V)\arrow[isomorphic]{r} &  \GL_{\gtl}(V),
    \end{tikzcd}
\end{equation*}
where the first map sends $\gamma\in\Gamma$ to $(1,\gamma)\in\gt_0\times_{\tau,c}\wgam$.

 For each homomorphism $\tau:\wgam\to\GL_{\gt_0}(V)$ and each 2-cocycle $c\in Z^2_{\tau}(\wgam,Z(\gt_0))$ as above we set 
\begin{align*}
    &\Hom_{\theta,\dac,\tau,c}(\wgame,\GL_{G}(V)):=r_{\gt}^{-1}(e_{\dac,\tau,c})\subset \Hom_{\theta,c}(\wgame,\GL_G(V))
    \andd\\&
    \Hom_{\theta,\dac,\tau,c}(\wgame,\Out(G,V)):=q_*(\Hom_{\theta,\epsilon,\tau,c}(\wgame,\GL_{G}(V))).
\end{align*}

Now fix homomorphisms  $(\da,\theta):\ker\eta\to\GL_G(V)$ and $\tau:\wgam\to\GL_{\gt_0}(V)$, a 2-cocycle $c\in Z^2_{\tilde\tau}(\wgam,Z(\gt_0))$ as above and $(\tda,\tilde\tau)\in\Hom_{\theta,\epsilon,\tau,c}(\wgame,\GL_{G}(V))$. There exists an associated left $(\tau,c)$-twisted action of $\wgame$ on $V$, namely $\rholm:=\pg^*(\mu^{-1})\tda$.
The restriction to $\lievm$ factors through a $(\tau,c)$-twisted homomorphism 
\begin{equation*}
    \wgam\to\Hom(\lievm,V).
\end{equation*}
Consequently, as in Remark \ref{remark-not-preserve-liegm} we have a notion of $(\tau,c,\rholm)$-twisted $\wgam$-equivariant $(\gt_0,\lievm)$-Higgs pair over $Y$.

\begin{proposition}\label{prop-fixed-points-reduction-alpha-trivial-pairs}
Consider a lift $(\da,\theta):\ker\eta\to\GL_G(V)$ of $a\vert_{\ker\eta}$ and a subgroup $\Lambda\le \gamtz$. Take a connected étale cover $p:Y\to X$ associated to a $\Lambda$-bundle over $X$ and the group $\wgam\le\Aut(Y)$ fitting in (\ref{eq-extension-wgam}). Let $(\dac,\tau):\wgam\to\GL_{\gt_0}(V)$ be a homomorphism and $c\in\zzgtw$ a 2-cocycle such that there is an isomorphism $(\gt_0\times_{\tau,c}\Lambda,\rho\vert_{\gt_0}\times\dac)\cong(\gtl,\rho\vert_{\gtl})$. Assume that $\pg^*a\in\Hom_{\theta,\dac,\tau,c}(\wgame,\Out(G,V))$ and pick $\tilde\tau\in \Hom_{\theta,\dac,\tau,c}(\wgame,\GL_{G}(V))$ 
% preserving $\lievm$ 
such that $q_*(\tilde\tau)=\pg^*a$.

Let $(F,\psi)$ be a $(\tau,c,\rholm)$-twisted $\wgam$-equivariant $(\gt_0,\lievm)$-Higgs pair over $Y$. Then $(F,\psi)$ can be regarded as a $(\gtl,\lievm)$-Higgs pair over $X$ via Proposition \ref{prop-twisted-equivariant-higgs pairs-one-to-one}, and its extension of structure group $(E,\phi)$ to $G$ is isomorphic to $(E,\phi)\gamma$ for each $\gamma\in\Gamma$.
\end{proposition}
\begin{proof}
    Follows from Proposition \ref{prop-fixed-points-reduction-alpha-trivial-principal} like Proposition \ref{prop-fixed-points-reduction-alpha-trivial-higgs}.
\end{proof}

\begin{proposition}\label{prop-simple-fixed-points-oscar-ramanan-alpha-trivial-pairs}
Let $(E,\phi)$ be a simple $(G,V)$-Higgs pair over $X$ which is isomorphic to $(E,\phi)\gamma$ for every $\gamma\in\Gamma$. Then there exist a lift $(\da,\theta)$ of $a\vert_{\ker\eta}$ and a connected reduction of structure group $(F,\psi)$ of $(E,\phi)$ to $(\gtl:=\pt^{-1}(\Lambda),\lievm)$ satisfying the following: let $p:Y\to X$ be the étale cover associated to the $\Lambda$-bundle $F/\gt_0\to X$ and $\wgam$ the subgroup of $\Aut(Y)$ lifting $\eta(\Gamma)$. Then there is a homomorphism $(\overline{\dac},\otau):\wgam\to\Out(\gt_0,V)$ such that, for every homomorphism $\tau:\wgam\to\GL(\gt_0)$ lifting $\otau$ (which exists by \cite{de-siebenthal}), we can find a 2-cocycle $c\in Z^2_{\tau}(\wgam,Z(\gt_0))$ and a map $\dac:\wgam\to\GL(V)$ such that:
\begin{enumerate}
    \item The pair $(\dac,\tau)$ is a $c$-twisted map $\wgam\to\GL_G(V)$.
    \item We have an isomorphism $(\gt_0\times_{\tau,c}\Lambda,\rho\vert_{\gt_0}\times\dac)\cong(\gtl,\rho\vert_{\gtl})$.
    \item $\pg^*a\in\Hom_{\theta,\dac,\tau,c}(\wgame,\Out(G,V))$.
    \item There exists $(\tilde{\dac},\tilde\tau)\in\Hom_{\theta,\dac,\tau,c}(\wgame,\GL_{G}(V))$ 
    % preserving $\lievm$ 
    such that $q_*(\tilde{\dac},\tilde\tau)=\pg^*a$ and the tautological reduction of $p^*(F,\psi)$ to $\gt_0$ is a $(\tau,c,\rholm)$-twisted $\wgam$-equivariant $(\gt_0,\lievm)$-Higgs pair.
\end{enumerate}
\end{proposition}
\begin{proof}
    Follows from Proposition \ref{prop-simple-fixed-points-oscar-ramanan-alpha-trivial-principal} like Proposition \ref{prop-simple-fixed-points-oscar-ramanan-alpha-trivial-higgs}, after replacing Proposition \ref{prop-simple-fixed-points-principal} with Proposition \ref{prop-simple-fixed-points-pairs}. The existence of $\dac$ follows from the construction of $\tilde\tau$: take a lift $(\tilde{\da},\tilde\theta):\Gamma\to\GL_G(V)$ of $a$, which exists by assumption ---see the beginning  of the section. Following the proof of Proposition \ref{prop-simple-fixed-points-oscar-ramanan-alpha-trivial-principal}, $\tilde\tau=\tilde\theta\Int_t$ for some map $t:\wgame\to G$ and so we may set $\tilde{\dac}:=\tilde{\da}\rho(t)$.
\end{proof}

\begin{proposition}\label{prop-prym-narasimhan-ramanan-alpha-trivial-pairs}
Consider a lift $(\da,\theta):\ker\eta\to\GL_{G}(V)$ of $a\vert_{\ker\eta}$, a subgroup $\Lambda\le\gamtz$, a connected étale cover $p:Y\to X$ with Galois group $\Lambda$ and the group $\wgam\le\Aut(Y)$ fitting in (\ref{eq-extension-wgam}). Let $(\dac,\tau):\wgam\to\GL_{\gt_0}(V)$ be a $c$-twisted homomorphism, where $c\in\zzgtw$ is a 2-cocycle, such that there is an isomorphism $(\gt_0\times_{\tau,c}\Lambda,\rho\vert_{\gt_0}\times\dac)\cong(\gtl,\rho\vert_{\gtl})$. Assume that $\pg^*a\in\Hom_{\theta,\dac,\tau,c}(\wgame,\Out(G,V))$ and pick $(\tilde{\dac},\tilde\tau)\in \Hom_{\theta,\dac,\tau,c}(\wgame,\GL_{G}(V))$ such that $q_*(\tilde{\dac},\tilde\tau)=\pg^*a$.

We have a morphism
\begin{equation}
    \cM(Y,\gt_0,\wgam,\tau,c,\lievm,\rholm)\to \cM(X,G,V),
\end{equation}
given by Theorem \ref{th-prym-narasimhan-ramanan-higgs} and extension of structure group.
\end{proposition}
\begin{proof}
    Same as Proposition \ref{prop-prym-narasimhan-ramanan-alpha-trivial-higgs} after replacing Theorem \ref{EH1} with Theorem \ref{EH1-pairs}, and Proposition \ref{prop-polystability-extension-structure-group} with Proposition \ref{prop-polystability-extension-structure-group-pairs}.
\end{proof}

\begin{theorem}\label{th-prym-narasimhan-ramanan-alpha-trivial-pairs}
Fix $\theta\in\Hom(\ker\eta,\GL_{G}(V))$ lifting $a\vert_{\ker\eta}$. We have the following relations between moduli spaces:
\begin{enumerate}
    \item $$\bigcup_{[\dab,\beta],Y,[{\dac}^{\beta\theta},\tau^{\beta\theta}],[c^{\beta\theta}],(\tilde{\dac},\tilde\tau),[\sigma]}\wcM(Y,G^{\beta\theta}_0,\wgam,\tau^{\beta\theta},c^{\beta\theta},V^{\dab\da}_{\mu},\rholm,\sigma)
    \subset\cM(X,G,V)^{\Gamma}. $$
    
    \item $$\cM_{ss}(X,G)^{\Gamma}\subset
    \bigcup_{[\dab,\beta],Y,[{\dac}^{\beta\theta},\tau^{\beta\theta}],[c^{\beta\theta}],(\tilde{\dac},\tilde\tau),[\sigma]}\wcM(Y,G^{\beta\theta}_0,\wgam,\tau^{\beta\theta},c^{\beta\theta},V^{\dab\da}_{\mu},\rholm,\sigma)
    .$$
    
\end{enumerate}
Here $[\dab,\beta]$ runs through $H^1_{\da,\theta}(\Gamma,\rho_G(G))$, $Y$ runs over étale covers of $X$ with Galois group equal to a subgroup
$\Lambda\le\widehat{\Gamma}^{\beta\theta}$, the elements $[{\dac}^{\beta\theta},\tau^{\beta\theta}]\in \Hom_{c^{\beta\theta}}(\wgam,\Out(G^{\beta\theta}_0,V))$ and $[c^{\beta\theta}]\in H^2_{\tau^{\beta\theta}}(\wgam,Z(G^{\beta\theta}_0))$ are such that $\pg^*a\in (\Hom_{\beta\theta,\dac^{\beta\theta},\tau^{\beta\theta},c^{\beta\theta}}(\wgame,\Out(G)))$, their restrictions to $\Lambda$ satisfy 
$$(G^{\beta\theta}_0\times_{\tau^{\beta\theta},c^{\beta\theta}}\Lambda,\rho\vert_{G^{\beta\theta}_0}\times\dac)\cong(G^{\beta\theta}_{\Lambda},\rho\vert_{G^{\beta\theta}_{\Lambda}})$$ 
and $[\sigma]\in H^1_{c^{\beta\theta}_{x_i}}(\wgamm_{Y,x_i},G)$. Moreover, for each choice of $[\dab,\beta]$, $[{\dac}^{\beta\theta},\tau^{\beta\theta}]$ and $[c^{\beta\theta}]$, the element $\tilde\tau\in\Hom_{\beta\theta,{\dac}^{\beta\theta},\tau^{\beta\theta},c^{\beta\theta}}(\wgame,\GL_{G}(V))$ 
% preserves $\lievm$ and
satisfies $q_*\tilde\tau=\pg^*a$.

\end{theorem}

Now we tackle the most general case. Let $X$ be a compact Riemann surface, $G$ a connected semisimple complex Lie group with centre $Z$ and $\Gamma$ a finite subgroup of $H^1(X,Z)\rtimes(\Aut(X)\times\Out(G,V))\times\C^*$. By Section \ref{section-action-pairs} we have a right action of $\Gamma$ on $\mdl(X,G,V)$. The projections on each factor provide homomorphisms $\eta:\Gamma\to\Aut(X)$, $a:\Gamma\to\Out(G,V)$ and $\mu:\Gamma\to\C^*$, together with 1-cocycle $\alpha\in Z^1_{a,\eta}(\Gamma,H^1(X,Z))$ as defined in Section \ref{section-action}, where the action of $\Gamma$ on $H^1(X,Z)$ is determined by $a$ (via extension of structure group) and $\eta$ (via pullback). This is a map $\alpha:\Gamma\to H^1(X,Z)$ satisfying
\begin{equation*}
    \alpha_{\gamma\gamma'}=\alg\etag^{*-1}\ag(\alpha_{\gamma'})
\end{equation*}
for each $\gamma$ and $\gamma'\in\Gamma$, where $\ag$ acts on $Z$ via its projection on $\Out(G)$. In what follows we assume that there is a homomorphism $\Gamma\to\GL_G(V)$ lifting $a$. 

The restriction $\alpha\vert_{\ker\eta}$ is 1-cocycle in $Z^1_{a}(\ker\eta,H^1(X,Z))\cong H^1(X,Z^{1}_a(\ker\eta,Z))$, thus any of its connected components provides an étale cover $X_{\alpha,\eta}\to X$ with Galois group $\Gamma_{\alpha,\eta}\le Z^{1}_a(\ker\eta,Z)$. 

Pick a lift $(\da,\theta):\ker\eta\to\Aut(G)$ of $a\vert_{\ker\eta}$. Let $p:Y\to X_{\alpha,\eta}\to X$ be a connected component of a $\gamtt$-bundle in $\qqt^{-1}(\alpha\vert_{\ker\eta})$ (see (\ref{eq-def-qt})), and set $\olambda:=\gal(Y/X)\le \gamtt$. Consider the subgroup $\halpha\le H^1(X,Z)$ generated by the image of $\alpha$, which is finite because both $\Gamma$ and $Z$ are finite (thus any element of $H^1(X,Z)$ has finite order). Its image $p^*\halpha\le H^1(Y,Z)$ via pullback is also a finite subgroup determining a connected étale cover $$p_{Y_{\alpha}}:Y_{\alpha}\xrightarrow{p_{\halpha}} Y\to X.$$ Like any pullback, this also has a projection $Y_{\alpha}\to\hat\alpha$, where $\hat\alpha\in H^1(X,\Hom(\hat\alpha,Z))$ is regarded as an étale cover of $X$. Let $\Lambda:=\gal(Y_{\alpha}/X)$, call $\wgam$ to the group of automorphisms of $Y_{\alpha}$ lifting $\eta(\Gamma)\le\Aut(X)$ and let $\wgame:=\{(\gamma,\wga)\in\Gamma\times\wgam\suhthat \eta(\gamma)=p(\wga)\}$. The commutative diagramme (\ref{eq-extension-wgam}) still holds and it has exact rows and columns. We also have a diagramme (\ref{eq-restriction-extension maps-pairs}), with the same notation. We may also define $\Hom_{\theta,\dac,\tau,c}(\wgame,\GL_{G}(V))$ and $\Hom_{\theta,\dac,\tau,c}(\wgame,\Out(G,V))$. 

Given $(\dac^{\theta},\taut):\wgam\to\Aut_{\gt_0}(V)$ and $\ct\in Z^2_{\taut}(\wgam,Z(\gt_0))$ whose restrictions to $\Lambda$ factor through $\olambda$ and satisfy 
$$(\gtll,\rho)\cong(\gt_0\times_{\taut,\ct}\olambda,\rho\vert_{G^{\beta\theta}_0}\times\dac),$$
together with $\tilde\tau\in\Hom_{\theta,\dac,\tau,c}(\wgame,\GL_{G}(V))$, there is a $(\taut,\ct)$-twisted $ \wgam$-right action $\rholm:=\mu^{-1}\tilde{\dac}:\wgam\to\Hom(\lievm,V)$.
As in Proposition \ref{prop-prym-narasimhan-ramanan-alpha-trivial-pairs} we have a morphism 
\begin{equation*}
    \cM(Y_{\alpha},G^{\theta}_0,\wgam,\tau^{\theta},c^{\theta},\lievm,\rholm,\sigma)
    \to\cM(X,G,V)
\end{equation*}
for each $\sigma\in \{Z^1_{\taut}(\Gamma_{x_i},\gt_0)\}$ (here $x_i$ are the isotropy points of $Y_{\alpha}$, with isotropy groups $\Gamma_{x_i}$), whose image we call $\wcM(Y_{\alpha},G^{\theta}_0,\wgam,\tau^{\theta},c^{\theta},\lievm,\rholm,\sigma)
    $.

\begin{theorem}\label{th-prym-narasimhan-ramanan-general-pairs}
Fix $(\da,\theta)\in\Hom(\ker\eta,\GL_{G}(V))$ lifting $a$. We have the following relations between moduli spaces:
\begin{enumerate}
    \item $$\bigcup_{[\dab,\beta],Y,[{\dac}^{\beta\theta},\tau^{\beta\theta}],[c^{\beta\theta}],(\tilde{\dac},\tilde\tau),[\sigma]}\wcM(Y,G^{\beta\theta}_0,\wgam,\tau^{\beta\theta},c^{\beta\theta},\tilde\tau,V^{\beta\theta}_{\mu},\rholm,\sigma)
    \subset\cM(X,G,V)^{\Gamma}. $$
    
    \item $$\cM_{ss}(X,G)^{\Gamma}\subset
    \bigcup_{[\dab,\beta],Y,[{\dac}^{\beta\theta},\tau^{\beta\theta}],[c^{\beta\theta}],(\tilde{\dac},\tilde\tau),[\sigma]}\wcM(Y,G^{\beta\theta}_0,\wgam,\tau^{\beta\theta},c^{\beta\theta},\tilde\tau,V^{\beta\theta}_{\mu},\rholm,\sigma).$$
    
\end{enumerate}

Here $[\dab,\beta]$ runs through $H^1_{\da,\theta}(\Gamma,\rho_G(G))$, $Y$ runs over étale covers of $X$ with Galois group equal to a subgroup
$\Lambda\le\widehat{\Gamma}^{\beta\theta}$, $[{\dac}^{\beta\theta},\tau^{\beta\theta}]\in \Hom_{c^{\beta\theta}}(\wgam,\Out(G^{\beta\theta}_0,V))$ and $[c^{\beta\theta}]\in H^2_{\tau^{\beta\theta}}(\wgam,Z(G^{\beta\theta}_0))$ are such that $\pg^*a\in (\Hom_{\beta\theta,\dac^{\beta\theta},\tau^{\beta\theta},c^{\beta\theta}}(\wgame,\Out(G)))$, their restrictions to $\Lambda$ factor through $\olambda$ and satisfy 
\begin{equation*}
(G^{\beta\theta}_0\times_{\tau^{\beta\theta},c^{\beta\theta}}\Lambda,\rho\vert_{G^{\beta\theta}_0}\times\dac)\cong(G^{\beta\theta}_{\Lambda},\rho\vert_{G^{\beta\theta}_{\Lambda}})
\end{equation*}
and $[\sigma]\in H^1_{c^{\beta\theta}_{x_i}}(\wgamm_{Y,x_i},G)$.
Moreover, for each choice of $[\dab,\beta]$, $[{\dac}^{\beta\theta},\tau^{\beta\theta}]$ and $[c^{\beta\theta}]$, the twisted homomorphism $(\tilde{\dac},\tilde\tau)\in\Hom_{\beta\theta,\dac^{\beta\theta},\tau^{\beta\theta},c^{\beta\theta}}(\wgame,\GL_{G}(V))$ is chosen so that $q_*\tilde\tau=\pg^*a$ and, if $t:\olambda\to\gtll$ is the map realizing (\ref{iso-Gbetatheta-twisted-product}) as in the proof of Proposition \ref{prop-extensions-isomorphic-twisted-group}, we have
\begin{equation*}
    c^{\gamma\theta}(\hat\gamma^{-1},\hat\gamma)^{-1}c^{\gamma\theta}(\hat\gamma^{-1},\lambda)c^{\gamma\theta}(\hat\gamma^{-1}\lambda,\gamma)\tilde\tau^{-1}_{\gamma,\hat\gamma}(t_{\lambda})t_{\hat\gamma^{-1}\lambda\hat\gamma}^{-1}=\langle\alpha_{\gamma},\lambda\rangle,
\end{equation*}
for every $\gamma\in \Gamma$ and $\lambda\in\Lambda$, where $(\gamma,\hat\gamma)\in\wgame$ and $\langle\alpha_{\gamma},\lambda\rangle$ is the evaluation of $p_{Y}\circ\lambda\in\Hom(\hat\alpha,Z)$ at $\alpha_{\gamma}$.
\end{theorem}

\end{document}